\documentclass[12pt,reqno] {amsart}
\usepackage{amsmath}
\usepackage{bbm, dsfont}
\usepackage{amssymb}
\usepackage{color}
\usepackage{graphicx}%
\usepackage{hyperref}
\usepackage{cleveref}
\usepackage{comment}
 \usepackage{mathrsfs}
 \usepackage{float}
 \usepackage{tikz, tikz-cd}
  \usetikzlibrary{arrows.meta, knots, calc}
 \usepackage{bm}
 \usepackage{mathtools} 
\usepackage{extarrows} 
\usepackage{blindtext}

 \usepackage{ulem}
 \usepackage{enumerate}
  \usepackage{mathdots}
\usetikzlibrary{decorations.markings}
\usetikzlibrary{decorations.pathreplacing}

\tikzset{->-/.style={decoration={
  markings,
  mark=at position #1 with {\arrow{>}}},postaction={decorate}}}
\tikzset{-<-/.style={decoration={
  markings,
  mark=at position #1 with {\arrow{<}}},postaction={decorate}}}
  \tikzset{
    partial ellipse/.style args={#1:#2:#3}{
        insert path={+ (#1:#3) arc (#1:#2:#3)}
    }
}
\theoremstyle{plain}

\numberwithin{equation}{section}

\newtheorem{theorem}{Theorem}[section]
\newtheorem*{theorem*}{Theorem}
\newtheorem{corollary}[theorem]{Corollary}
\newtheorem{lemma}[theorem]{Lemma}
\newtheorem{proposition}[theorem]{Proposition}
\newtheorem{remark}[theorem]{Remark}
\newtheorem{definition}[theorem]{Definition}
\newtheorem{example}[theorem]{Example}
\newtheorem*{ac}{Acknowledgement}
\newcommand{\1}{\textbf{1}}

\newcommand{\Hom}{\mathrm{HOM}}
\newcommand{\rep}{\mathrm{Rep }}
\newcommand{\Vc}{\mathrm{Vec }}
\newcommand{\Irr}{\mathrm{Irr }}
\newcommand{\bC}{\mathbb{C}}
\newcommand{\cC}{\mathcal{C}}
\newcommand{\cD}{\mathcal{D}}
\newcommand{\cE}{\mathcal{E}}
\newcommand{\fM}{\mathfrak{M}}
\newcommand{\Cl}{\mathrm{Cl}}
\newcommand{\Schi}{\mathfrak{S}\chi}
\newcommand{\sming}{\textcolor{red}}
\newcommand{\SL}{\operatorname{SL}}
\newcommand{\GL}{\operatorname{GL}}
\newcommand{\ord}{\operatorname{ord}}
\newcommand{\Gal}{\operatorname{Gal}}
\newcommand{\MCG}{\operatorname{MCG}}
\newcommand{\RT}{\operatorname{RT}}
\newcommand{\red}[1]{\textcolor{red}{#1}}
\newcommand{\op}{\operatorname{op}}
\begin{document}

\title{Alterfold Theory and Topological Modular Invariance}

 \author{Zhengwei Liu}
 \address{Z. Liu, Yau Mathematical Sciences Center and Department of Mathematics, Tsinghua University, Beijing, 100084, China}
 \address{Beijing Institute of Mathematical Sciences and Applications, Huairou District, Beijing, 101408, China}
 \email{liuzhengwei@mail.tsinghua.edu.cn}

 \author{Shuang Ming}
 \address{S. Ming, Beijing Institute of Mathematical Sciences and Applications, Beijing, 101408, China}
 \email{sming@bimsa.cn}

\author{Yilong Wang}
\address{Y. Wang, Beijing Institute of Mathematical Sciences and Applications, Beijing, 101408, China}
\email{wyl@bimsa.cn}

 \author{Jinsong Wu}
 \address{J. Wu, Beijing Institute of Mathematical Sciences and Applications, Beijing, 101408, China}
 \email{wjs@bimsa.cn}

\begin{abstract}
We propose a topological paradigm in alterfold topological quantum field theory to explore various concepts, including modular invariants, $\alpha$-induction and connections in Morita contexts within a modular fusion category of non-zero global dimension over an arbitrary field.
Using our topological perspective, we provide streamlined quick proofs and broad generalizations of a wide range of results. These include all theoretical findings by B\"{o}ckenhauer, Evans, and Kawahigashi on $\alpha$-induction.
Additionally, we introduce the concept of double $\alpha$-induction for pairs of Morita contexts and define its higher-genus $Z$-transformation, which remains invariant under the action of the mapping class group.
Finally, we establish a novel integral identity for modular invariance across multiple Morita contexts, unifying several known identities as special cases.

\end{abstract}

\maketitle

\tableofcontents

\section{Introduction}

Quantum symmetries extend the concept of classical group symmetries and are supported by mathematical frameworks from various fields, including 
subfactor theory \cite{Jon83, Jon85, Jon99}, quantum groups \cite{Dri87, Jim85}, confromal field theory \cite{FMS97}, topological quantum field theory (TQFT) \cite{Ati88, Wit88}, and tensor category theory \cite{EGNO15}, among others. 

In recent papers, the authors developed 2+1 alterfold topological quantum field theory in \cite{LMWW23,LMWW23b} and a higher dimensional theory in \cite{Liu24}. 
The alterfold topological quantum field theory of a spherical fusion category $\mathcal{C}$ contains both the Turaev-Viro TQFT \cite{TurVir92} of $\mathcal{C}$ and the Reshetikhin-Turaev TQFT \cite{ResTur91} of the Drinfeld center of $\mathcal{C}$.
It also offers convenient topological interpretations of notions in fusion categories such as the half-braiding, the twist, the Frobenius-Schur indicators, the S-matrices, and the Kirby-color,  providing conceptual and efficient computational tools. 
When $\mathcal{C}$ is modular, its Drinfeld center $\mathcal{Z(C)}=\mathcal{C}\boxtimes \mathcal{C}^{op}$ has natural 3D topological presentations. 

In this paper, we propose conceptual topological interpretations of modular invariants, $\alpha$-induction and connections for Morita contexts in a modular fusion category with non-zero global dimension over arbitrary field, using the alterfold TQFT. 
Based on our topological interpretation, we provide quick proofs and various generalizations of a large class of relevant results, including all theoretical results of B\"{o}ckenhauer, Evans and Kawahigashi on $\alpha$-induction in \cite{BE98, BE99, BE00, BE99b, BEK99, BEK00}.
We also provide a quick proof of recent results of Kawahigashi on the equivalence between the commutativity of a Frobenius algebra and the flatness of the corresponding connection in \cite{Kaw23,Kaw24}, (see Theorem \ref{thm:flatequiv}), as well as a non-trivial generalization of this equivalence to any ribbon fusion category over any field in Theorem \ref{thm:flatequi}. 
We also propose the notions of higher genus modular invariants and double $\alpha$-induction for a pair of Morita contexts in Theorem \ref{thm:mcg0}, and prove a new integral identity for the modular invariants of multiple Morita contexts in Theorem \ref{thm:physical}.

First let us briefly review the development and importance of the theory of modular invariance. 
In 1987, Cappelli, Itzykson, and Zuber \cite{CIZ87} classified minimal conformal invariant theories based on an $A$-$D$-$E$ classification of modular invariants of quantum $\mathfrak{sl}_2$. 
The partition function of a torus is expended over the basis $\{ \chi_j\}_j$ of characters  as 
\begin{align}\label{Equ: Z=Zjk}
Z=\sum_{j,k} z_{jk} \chi_j \overline{\chi_k}.
\end{align}
The coefficients $z_{jk}$ are the anomalous dimensions and they form a matrix commuting with the modular transformation, called the modular invariant mass matrix. 
This relationship between modular invariants and Dynkin diagrams was categorified by Ocneanu as the $A$-$D$-$E$ classification of module categories of quantum $\mathfrak{sl}_2$ in terms of connections on Dynkin diagrams \cite{Ocn88,Ocn00}, see also \cite{Ost03m} without assuming unitarity. 
In this case, the entries of the modular invariant become the dimension of certain hom-spaces and the Dynkin diagrams become the fusion graphs of the type $A$ modular fusion category acting on its module categories.
It was surprising that the connection is flat if and only if the type of Dynkin diagram is $A_n$, $D_{2n}$, $E_6$ or $E_8$, which led to the classification of subfactors with index below 4 \cite{Ocn88, GHJ89, BN91, Izu91, Izu94,Kaw95}.

The construction of a module category of a unitary fusion category is equivalent to the construction of a Q-system, or a Frobenius algebra, as shown by Longo-Rehren \cite{LonReh95}. 
Xu constructed commutative Frobenius algebras $Q$ in the modular fusion category $\mathcal{C}$ from conformal inclusions of loop groups, and induction maps $\alpha_{\pm}$ from $\mathcal{C}$ to the $Q$-$Q$ bimodule category $\mathcal{D}$. As examples, he reconstructed the $E_6$ and $E_8$ subfactors from the subcategory of $\mathcal{D}$ generated by $\alpha_{+}(\mathcal{C})$, where $\mathcal{C}$ is the modular fusion category of $\mathfrak{sl}_2$.

For a Frobenius algebra $Q$ in a unitary modular fusion category $\mathcal{C}$, B\"{o}ckenhauer, Evans and Kawahigashi constructed the modular invariant mass matrix $Z$ through the induction maps $\alpha_{\pm}$ as
$z_{jk}=\dim \hom_{\mathcal{D}}(\alpha_+(X_j),\alpha_-(X_k))$ for simple objects $X_j$ of $\mathcal{C}$, and they systematically studied modular invariants in \cite{BE98, BE99, BE00, BE99b, BEK99, BEK00}.
The modular invariant is crucial in the classification of Frobenius algebras of modular fusion categories, see \cite{Gan94, Ocn00, Gan23,EdiGan24} for the classification up to $\mathfrak{sl}_7$. 

We first study the alterfold TQFT for a modular fusion category $\mathcal{C}$ and give a topological interpretation of the well-known result that the Drinfeld center $\mathcal{Z(C)}$ is braided equivalent to $\mathcal{C}\boxtimes \mathcal{C}^{op}$ \cite{Wal91, Tur92}.
In particular, the vector space with torus boundary in the TQFT has a basis $\{\xi_{j} \otimes \xi_k^{op}\}_{j,k}$ labeled by simple objects $X_j,X_k$ of $\mathcal{C}$.
For a spherical fusion category  $\mathcal{D}$ Morita equivalent to $\mathcal{C}$, we construct the partition function $Z^{\mathcal{D}}$ of a torus and its expansion over the basis of simple objects of the $Z(\mathcal{C})$ as
\begin{align}\label{Z=Zjk top}
Z^{\mathcal{D}}=\sum_{j,k}z_{jk}\xi_{j} \otimes \xi_{k}^{op},   
\end{align}
which is a topological analogue of Equation ~\eqref{Equ: Z=Zjk} in alterfold TQFT as follows.
\begin{align*}
\vcenter{\hbox{\scalebox{0.5}{
\begin{tikzpicture}
\begin{scope}[shift={(1, 1)}, scale=1.8]
\path [fill=brown!20!white] (-1,0.5)--(2, 0.5)--(1, -0.5)--(-2, -0.5)--cycle;
\draw (-1,0.5)--(2, 0.5) (-2,-0.5)--(1, -0.5);
\draw (-2,-0.5)--(-1, 0.5) (1,-0.5)--(2, 0.5) node [below] {\tiny $B$} node [above] {\tiny $A$};
\node at (0,0) { $\mathcal{D}$};
\end{scope}
\begin{scope}[shift={(1, -1)}, scale=1.8]
\draw (-1,0.5)--(2, 0.5) (-2,-0.5)--(1, -0.5);
\draw (-2,-0.5)--(-1, 0.5) (1,-0.5)--(2, 0.5) node [above] {\tiny $B$}; 
\node at (2, 0.5) [below right] {\tiny Time Boundary};
\end{scope}
\end{tikzpicture}}}}
=
\sum_{j,k} z_{jk} \frac{1}{\mu}
\vcenter{\hbox{\scalebox{0.5}{
\begin{tikzpicture}
\begin{scope}[shift={(1, 1)}, scale=1.8]
\draw (-1,0.5)--(2, 0.5) (-2,-0.5)--(1, -0.5);
\draw (-2,-0.5)--(-1, 0.5) (1,-0.5)--(2, 0.5) node [above] {\tiny $B$} node [below] {\tiny $A$};
\draw [red] (-0.5,-0.5)--(0.5, 0.5);
\path [fill=white] (0.2, 0.2) circle (0.2cm);
\draw [blue] (-1.3, 0.2)--(1.7, 0.2);
\draw [blue] (-1.7, -0.2)--(-0.4, -0.2) (0, -0.2) --(1.3, -0.2);
\end{scope}
\begin{scope}[shift={(1, -1)}, scale=1.8]
\draw (-1,0.5)--(2, 0.5) (-2,-0.5)--(1, -0.5);
\draw (-2,-0.5)--(-1, 0.5) (1,-0.5)--(2, 0.5) node [above] {\tiny $B$}; 
\node at (2, 0.5) [below right] {\tiny Time Boundary};
\end{scope}
\end{tikzpicture}}}},
\end{align*}
where $\mu$ is the global dimension of $\mathcal{C}$.

The modular invariance property of the matrix $Z^{\mathcal{D}}$ is transparent from its topological nature.
The coefficient $z_{jk}$ can be expressed pictorially as
\begin{align*}
z_{jk}=\langle \xi_j\otimes \xi_k^{op}, Z^\mathcal{D}\rangle=& 
\frac{1}{\mu^2}
\vcenter{\hbox{\scalebox{0.5}{
\begin{tikzpicture}
\begin{scope}[shift={(1, 1)}, scale=1.8]
\path [fill=brown!20!white] (-1,0.5)--(2, 0.5)--(1, -0.5)--(-2, -0.5)--cycle;
\draw (-1,0.5)--(2, 0.5) (-2,-0.5)--(1, -0.5);
\draw (-2,-0.5)--(-1, 0.5) (1,-0.5)--(2, 0.5) node [below] {\tiny $B$} node [above] {\tiny $A$};
\node at (0,0) { $\mathcal{D}$};
\end{scope}
\begin{scope}[shift={(1, -1)}, scale=1.8]
\draw (-1,0.5)--(2, 0.5) (-2,-0.5)--(1, -0.5);
\draw (-2,-0.5)--(-1, 0.5) (1,-0.5)--(2, 0.5) node [above] {\tiny $B$} node [below] {\tiny $A$};
\draw [red] (-0.5,-0.5)--(0.5, 0.5);
\path [fill=white] (0.2, 0.2) circle (0.2cm);
\draw [blue] (-1.3, 0.2)--(1.7, 0.2);
\draw [blue] (-1.7, -0.2)--(-0.4, -0.2) (0, -0.2) --(1.3, -0.2);
\end{scope}
\end{tikzpicture}}}}\\
= & \frac{1}{\mu^2} \vcenter{\hbox{\scalebox{0.7}{
\begin{tikzpicture}[scale=0.7]
\draw [line width=0.6cm, brown!20!white] (0,0) [partial ellipse=-0.1:180.1:2 and 1.5];
\begin{scope}[shift={(2.5, 0)}]
\draw [line width=0.6cm] (0,0) [partial ellipse=0:180:2 and 1.5];
\draw [white, line width=0.57cm] (0,0) [partial ellipse=-0.1:180.1:2 and 1.5];
\draw [blue] (0,0) [partial ellipse=0:180:2.15 and 1.65];
\draw [blue] (0,0) [partial ellipse=0:180:1.85 and 1.35];
\end{scope} 
\begin{scope}[shift={(2.5, 0)}]
\draw [line width=0.6cm] (0,0) [partial ellipse=180:360:2 and 1.5];
\draw [white, line width=0.57cm] (0,0) [partial ellipse=178:362:2 and 1.5];
\draw [blue, -<-=0.5] (0,0) [partial ellipse=178:362:2.15 and 1.65] node[black, pos=0.7,below ] {\tiny $X_k$}; 
\draw [blue, ->-=0.5] (0,0) [partial ellipse=178:362:1.85 and 1.35] node[black, pos=0.7,above ] {\tiny $X_j$};
\end{scope}
\draw [line width=0.6cm, brown!20!white] (0,0) [partial ellipse=180:360:2 and 1.5];
\begin{scope}[shift={(4.5, 0)}]
\draw [red, dashed](0,0) [partial ellipse=0:180:0.4 and 0.25] ;
\draw [white, line width=4pt] (0,0) [partial ellipse=290:270:0.4 and 0.25];
\draw [red] (0,0) [partial ellipse=260:360:0.4 and 0.25];
\draw [red] (0,0) [partial ellipse=180:230:0.4 and 0.25];
\end{scope}
\end{tikzpicture}}}}\\
=&\dim\hom_{\mathcal{Z(C)}} (I(\1_{\mathcal{D}})), X_j\boxtimes X_k^{op}),
\end{align*}
where $I$ is the induction functor.

Furthermore, we show that the $\alpha$-induction functors $\alpha_{\pm}$ are implemented as the forgetful functors from $\mathcal{C}\boxtimes \1$ and $\1 \boxtimes\mathcal{C}^{op}$ to $\mathcal{D}$ respectively in the alterfold TQFT. 
(The element in $\mathcal{Z(C)}$ is presented by a diagram in the bulk $\Sigma \times [0,1]$, for a torus $\Sigma$. 
The forgetful functor is pushing the diagram to the $\mathcal{D}$-colored boundary $\Sigma \times \{1\}$.)
Then we obtain that
\begin{align*}
\dim\hom_{\mathcal{Z(C)}} (I(\1_{\mathcal{D}})), X_j\boxtimes X_k^{op})
=&\dim\hom_{\mathcal{D}} (\1_{\mathcal{D}}, \alpha_+(X_j)\otimes \alpha_-(X_k^{*}))\\
=&\dim\hom_{\mathcal{D}} (\alpha_+(X_j), \alpha_-(X_k)),
\end{align*}
which recovers the B\"{o}ckenhauer, Evans and Kawahigashi definition of the modular invariant mass $Z$-matrix.

Next, we express $Z^{\mathcal{D}}$ as a solid Hopf link.
This leads to two kinds of ``torus actions'',
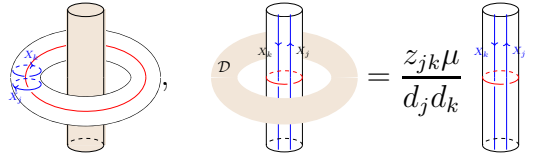
\begin{align*}
\vcenter{\hbox{\scalebox{0.5}{
\begin{tikzpicture}[xscale=0.8, yscale=0.6]
\draw [line width=0.8cm] (0,0) [partial ellipse=0:180:2 and 1.5];
\draw [white, line width=0.77cm] (0,0) [partial ellipse=-0.1:180.1:2 and 1.5];
\draw [red] (0,0) [partial ellipse=0:180:2 and 1.5];
\path [fill=brown!20!white](-0.65, -3) rectangle (0.65, 3);
\begin{scope}[shift={(0,3)}]
\path [fill=brown!20!white] (0,0) [partial ellipse=0:180:0.6 and 0.3];
\draw (0,0) [partial ellipse=0:360:0.6 and 0.3];
\end{scope}
\draw [line width=0.8cm] (0,0) [partial ellipse=180:360:2 and 1.5];
\draw (-0.6, 3)--(-0.6, 0) (0.6, 3)--(0.6, 0); 
\draw (-0.6, -3)--(-0.6, 0) (0.6, -3)--(0.6, 0);
\draw [white, line width=0.77cm] (0,0) [partial ellipse=178:362:2 and 1.5];
\draw [red] (0,0) [partial ellipse=178:362:2 and 1.5]; 
\begin{scope}[shift={(-1.95, 0.3)}]
\draw [blue, -<-=0.4, dashed](0,0) [partial ellipse=0:180:0.47 and 0.25] node [pos=0.4, above] {\tiny $X_k$};
\draw [white, line width=4pt] (0,0) [partial ellipse=290:270:0.47 and 0.25];
\draw [blue] (0,0) [partial ellipse=280:360:0.47 and 0.25];
\draw [blue] (0,0) [partial ellipse=180:260:0.47 and 0.25];
\end{scope}
\begin{scope}[shift={(-1.95, -0.3)}]
\draw [blue, dashed](0,0) [partial ellipse=0:180:0.47 and 0.25] ;
\draw [white, line width=4pt] (0,0) [partial ellipse=290:270:0.47 and 0.25];
\draw [blue, -<-=0.2] (0,0) [partial ellipse=180:360:0.47 and 0.25] node [pos=0.2, below] {\tiny $X_j$};
\end{scope}
\begin{scope}[shift={(0,-3)}]
\path [fill=brown!20!white] (0,0) [partial ellipse=180:360:0.6 and 0.3];
\draw [dashed](0,0) [partial ellipse=0:180:0.6 and 0.3];
\draw (0,0) [partial ellipse=180:360:0.6 and 0.3];
\end{scope}
\end{tikzpicture}}}}, \quad 
\vcenter{\hbox{\scalebox{0.5}{
\begin{tikzpicture}[xscale=0.8, yscale=0.6]
\draw [brown!20!white, line width=0.7cm] (0,0) [partial ellipse=0:180:2 and 1.5];
\node at (-2, 0.5) {$\mathcal{D}$};
\begin{scope}[shift={(0,3)}]
\draw (0,0) [partial ellipse=0:360:0.6 and 0.3];
\end{scope}
\path [fill=white] (-0.6, 0) rectangle (0.6, 2.7);
\draw [blue, ->-=0.5] (-0.2, 2.8)--(-0.2, 0) node [black, left, pos=0.6] {\tiny $X_k$} ;
\draw [blue, -<-=0.5] (0.2, 2.8)--(0.2, 0) node [black, right, pos=0.6] {\tiny $X_j$}; 
\draw (-0.6, 3)--(-0.6, 0) (0.6, 3)--(0.6, 0); 
\draw [blue] (-0.2, -3.2)--(-0.2, 0) (0.2, -3.2)--(0.2, 0);
\draw (-0.6, -3)--(-0.6, 0) (0.6, -3)--(0.6, 0);
\draw [brown!20!white, line width=0.7cm] (0,0) [partial ellipse=180:360:2 and 1.5];
\begin{scope}[shift={(0,-3)}]
\draw [dashed](0,0) [partial ellipse=0:180:0.6 and 0.3];
\draw (0,0) [partial ellipse=180:360:0.6 and 0.3];
\end{scope}
\draw [red, dashed](0,0) [partial ellipse=0:180:0.6 and 0.3]; 
\draw [white, line width=4pt]  (0,0) [partial ellipse=220:270:0.6 and 0.3];
\draw [red]  (0,0) [partial ellipse=180:280:0.6 and 0.3];
\draw [red]  (0,0) [partial ellipse=300:360:0.6 and 0.3];
\end{tikzpicture}}}}
=\displaystyle \frac{z_{jk}\mu}{d_jd_k}
\vcenter{\hbox{\scalebox{0.5}{
\begin{tikzpicture}[xscale=0.8, yscale=0.6]
\begin{scope}[shift={(0,3)}]
\draw (0,0) [partial ellipse=0:360:0.6 and 0.3];
\end{scope}
\path [fill=white] (-0.6, 0) rectangle (0.6, 2.7);
\draw [blue, ->-=0.5] (-0.2, 2.8)--(-0.2, 0) node [left, pos=0.6] {\tiny $X_k$} ;
\draw [blue, -<-=0.5] (0.2, 2.8)--(0.2, 0) node [right, pos=0.6] {\tiny $X_j$}; 
\draw (-0.6, 3)--(-0.6, 0) (0.6, 3)--(0.6, 0); 
\draw [blue] (-0.2, -3.2)--(-0.2, 0) (0.2, -3.2)--(0.2, 0);
\draw (-0.6, -3)--(-0.6, 0) (0.6, -3)--(0.6, 0);
\begin{scope}[shift={(0,-3)}]
\draw [dashed](0,0) [partial ellipse=0:180:0.6 and 0.3];
\draw (0,0) [partial ellipse=180:360:0.6 and 0.3];
\end{scope}
\draw [red, dashed](0,0) [partial ellipse=0:180:0.6 and 0.3]; 
\draw [white, line width=4pt]  (0,0) [partial ellipse=220:270:0.6 and 0.3];
\draw [red]  (0,0) [partial ellipse=180:280:0.6 and 0.3];
\draw [red]  (0,0) [partial ellipse=300:360:0.6 and 0.3];
\end{tikzpicture}}}}.
\end{align*}

From the first torus action, we can read out the minimal central idempotents in the fusion algebra of the dual fusion category and topological character.
Together with the $\alpha$-induction, we obtain a topological interpretation of its minimal idempotents \ref{thm:minimalidempotent}).
This yields a topological interpretation of the exponents of $K_0(\mathcal{C})$ acting on $K_0(\mathcal{M})$, where $\mathcal{M}$ is the module category (see Corollary \ref{cor:exponent}).

From the second torus action, we obtain $\displaystyle \frac{z_{j,k}\mu}{d_j d_k}$ as its spectrum (Lemma~\ref{lem:zcoefficient}), where $d_j$ is the quantum dimension. 
Considering the action of tori colored by multiple Morita context, 
we prove the following theorem of integral equalities for their modular invariant $Z$-matrices. 

\begin{theorem*}[Theorem \ref{thm:physical}]
Suppose $\mathcal{C}$ is a modular fusion category and $\mathcal{D}_s$, $s=1, \ldots, n$ are fusion categories Morita equivalent to $\mathcal{C}$.
We denote by $z_{jk}^{\mathcal{D}_s}$ the coefficients of the modular invariant mass matrix $Z^{\mathcal{D}_s}$.
Then we have that 
\begin{enumerate}
\item $\displaystyle \sum_{j, k=1}^r \prod_{s=1}^nz_{jk}^{\mathcal{D}_s}\frac{\mu^{n-2}}{d_k^{n-2}d_j^{n-2}}=\dim\hom_{\mathcal{Z(C)}}(\1_{\mathcal{Z(C)}}, I(\1_{\mathcal{D}_1})\otimes \cdots \otimes I(\1_{\mathcal{D}_n}))$.
\item  $\displaystyle \sum_{j=1}^r\prod_{s=1}^nz_{jj}^{\mathcal{D}_s}\frac{\mu^{n-2}}{d_j^{2n-4}}= \dim\hom_{\mathcal{Z(C)}}(I(\1_{\mathcal{C}}), I(\1_{\mathcal{D}_1})\otimes \cdots \otimes I(\1_{\mathcal{D}_n}))$.
\item $\displaystyle \sum_{j=1}^r \prod_{s=1}^nz_{j1}^{\mathcal{D}_s}\frac{\mu^{n-2}}{d_j^{n-2}}=\dim\hom_{\mathcal{Z(C)}}(\sum_j G^+(X_j), I(\1_{\mathcal{D}_1})\otimes \cdots \otimes I(\1_{\mathcal{D}_n}))$.
\item $\displaystyle \sum_{j=1}^r \prod_{s=1}^nz_{1j}^{\mathcal{D}_s}\frac{\mu^{n-2}}{d_j^{n-2}} =\dim\hom_{\mathcal{Z(C)}}(\sum_j G^-(X_j), I(\1_{\mathcal{D}_1})\otimes \cdots \otimes I(\1_{\mathcal{D}_n}))$, where $G^\pm$ is defined in Equation \eqref{eq:G}.
\end{enumerate}
\end{theorem*}

We also provide a quick proof of a recent result of Kawahigashi on the equivalence between the commutativity of a Frobenius algebra and the flatness of the corresponding connection in \cite{Kaw23,Kaw24}, (see Theorem \ref{thm:flatequiv}), as well as a non-trivial generalization of this equivalence to any ribbon fusion category over any field in Theorem \ref{thm:flatequi}.
We present a topological interpretation of the full center of a Frobenius algebra studied in \cite{KR08}, which is the $\mathcal{D}$-colored tube in the alterfold TQFT and characterize the equivalence of the isomorphism of the full centers and the isomorphism of the Frobenius algebras \cite{KR08}.
Inspired by the topological interpretation of the character for modular fusion category, we study the topological character for Morita equivalence and obtain the orthogonality of characters (Theorems \ref{thm:ortho1}, \ref{thm:ortho2}, \ref{thm:ortho3}). 
The first one is known in \cite{ENO21} and we have not found the others in the literature.

Finally, we propose the notion of double $\alpha$-induction by coloring the boundary surfaces $\Sigma \times {1}$ and $\Sigma \times {0}$ by spherical fusion categories $\mathcal{D}$ and $\mathcal{E}$ which are Morita equivalent to the modular fusion category $\mathcal{C}$. 
The $\alpha$-induction has tensor functors $\alpha_{\pm}$ from $\mathcal{C}$ to $\mathcal{D}$, while the double $\alpha$-induction has tensor functors from $\mathcal{C}$ to a multi-fusion category generated by $\mathcal{D}$, $\mathcal{E}$.
The corresponding $Z$-matrix is generalized to the higher-genus $Z$-transformation which is invariant under mapping class group action due to its topological nature (See Theorem \ref{thm:mcg0}). 
We shall investigate the analogy of the results in the higher alterfold TQFT \cite{Liu24} in the near future.

The paper is organized as follows.
Section 2 we will review some basic facts about the alterfold topological quantum field theory for Morita equivalence.
Section 3 we present a topological interpretation of the $\alpha$-induction and characterize the minimal idempotents of the dual fusion algebra.
Section 4 we recall the bi-invertible connection and show that the commutativity of the Frobenius algebra is equivalent to the flatness of the bi-invertible connection.
Section 5 we present a topological characterization of the isomorphism of the full centers.
Section 6 we study the orthogonalities of the topological characters.
Section 7 we obtain numbers of identities for modular invariance and related power series.
Section 8 we propose the double $\alpha$-induction by attaching one more $2D$ surface.
The equivariance of the double $\alpha$-induction is obtained.

\begin{ac}
We thank Professor Christoph Schweigert for helpful discussions and suggestions.  Z.~L., Y.~W. and J.~W. are supported by Beijing Natural Science Foundation Key Program (Grant No. Z220002). S.~M. and J.~W.; Y.~W.; J.~W. are supported by National Natural Science Foundation of China (Grant no. 12371124; 12301045; 12031004) respectively. 
Z.~L. is supported by Natural Science Foundation of Beijing Municipality (Grant No. Z221100002722017) and by National Key Programs of China (Grant no. 2020YFA0713000).
The authors are supported by grants from Beijing Institute of Mathematical Sciences and Applications. 
\end{ac}

\section{Preliminaries}
We shall briefly review some results for modular fusion category and the Morita equivalence in the alterfold TQFT.

\subsection{Morita Context}
We work with the alterfold TQFT decorated by a $2$-category of spherical Morita contexts defined in \cite{Mug03a, Mug03b}. 
This 2-category consists of two objects $\mathfrak{C}$ and $\mathfrak{D}$ with endo-categories $\mathcal{C}=\operatorname{END}(\mathfrak{C})$ and $\mathcal{D}=\operatorname{END}(\mathfrak{D})$, and we denote $\text{HOM}(\mathfrak{C}, \mathfrak{D})$ by ${}_{\mathcal{C}}\mathcal{M}_{\mathcal{D}}$, which indicates that $\mathcal{M}$ is equipped with a natural $\mathcal{C}$-$\mathcal{D}$ bimodule structure. 
Of course, conditions such as sphericality are required be satisfied, see \cite{Mug03a, Mug03b} for details. 
M\"uger showed that given such a spherical Morita context, for all simple generator $J$ in ${}_{\mathcal{C}}\mathcal{M}_{\mathcal{D}}$, $Q:=J\overline{J}$ is a Frobenius algebra in $\mathcal{C}$, and that $\mathcal{D}$ is equivalent to the category of $Q$-$Q$ bimodules in $\cC$, which is denoted by ${}_Q\cC_{Q}$ in \cite{EGNO15}. 
In this case, $\cC$ and $\cD$ are called Morita equivalent, and we will abuse notations by saying such a 2-category is a Morita context between $\cC$ and $\cD$.

In this paper, we consider all fusion categories Morita equivalent to a given braided fusion category $\cC$, which means, we consider all Morita contexts sharing a common object $\mathfrak{C}$ with endo-category $\operatorname{END}(\mathfrak{C}) = \cC$.  Since Morita equivalence is an equivalence relation, so a Morita context between $\cC$ and $\cD = \operatorname{END}(\mathfrak{D})$ and one between $\cC$ and $\cE = \operatorname{END}(\mathfrak{E})$ provides a Morita context between $\cD$ and $\cE$, and in this case, we denote $\operatorname{HOM}(\mathfrak{D}, \mathfrak{E})$ by ${}_{\cD}\mathcal{M}_{\cE}$.

Let us fix some notations: Denote by $d_J$ the quantum dimension of $J$.
It is shown in \cite{Mug03a} that the Drinfeld centers $\mathcal{Z(C)}$, $\mathcal{Z(D)}$ and $\mathcal{Z(E)}$ are braided equivalent to each other.
We always assume that the set $\Irr(\mathcal{C})$ of simple objects consists of elements $X_1, \ldots, X_r$, where $X_1=\1$.
Denote by $d_j$ the quantum dimension of $X_j$ and by $\mu$ the global dimension of $\mathcal{C}$ and assume that $\mu\neq 0$.
Suppose that $\Irr(\mathcal{D})=\{Y_1, \ldots, Y_{r'}\}$, where $Y_1=\1$.

\subsection{Alterfold TQFT}
We recall the partition function, topological moves, and alterfold TQFT \cite{LMWW23, LMWW23b} in the following.
A 3-alterfold is a closed compact 3-manifold including an embedded separating surface such that the connected components of the complement of the separating surface in the 3-manifold are alternatively colored by colors $A$ and $B$. 
The region colored by $B$ is denoted by $R_B$.
There are five topological moves including the graphical calculus for the Morita context.
\begin{itemize}
\item \textbf{Planar Graphical Calculus:} The planar graphical calculus for Morita context.
Let $D\subset \Sigma$ be an embedded disk and $\Gamma_{D}=D\cap \Gamma$. 
Suppose $\Gamma_{D}'=\Gamma_{D}$ as morphisms in the Morita context. 
We change the $\Gamma$ to $\Gamma'$ by replacing $\Gamma_{D}$ to $\Gamma_{D}'$.

\begin{align}\label{eq:calculus}
\begin{array}{ccc}
   \vcenter{\hbox{
\begin{tikzpicture}
\draw[->-=0.8, ->-=0.2, blue] (-.3, 1.2)--(-.3, -1.2);
\draw[->-=0.8, ->-=0.2, blue] (.3, 1.2)--(.3, -1.2);
\draw[blue] (0, -.75) node{$\ldots$};
\draw[blue] (0, .65) node{$\ldots$};
\draw[fill=white] (-.4, -.3) rectangle (.4, .3);
\draw (0, 0) node{$\Gamma_{D}$};
\draw[dashed] (-1.2, -1.2) rectangle (1.2, 1.2);
\draw (1, -1) node{$D$};
\end{tikzpicture}}}
&
\rightarrow
&
\vcenter{\hbox{
\begin{tikzpicture}    
\draw[->-=0.8, ->-=0.2, blue] (-.3, 1.2)--(-.3, -1.2);
\draw[->-=0.8, ->-=0.2, blue] (.3, 1.2)--(.3, -1.2);
\draw[blue] (0, -.75) node{$\ldots$};
\draw[blue] (0, .65) node{$\ldots$};
\draw[fill=white] (-.4, -.3) rectangle (.4, .3);
\draw (0, 0) node{$\Gamma_{D}'$};
\draw[dashed] (-1.2, -1.2) rectangle (1.2, 1.2);
\draw (1, -1) node{$D$};
\end{tikzpicture}}}\\
\end{array}
\end{align}

\item \textbf{Move 0:} Let $P$ be a point in the interior of $R_B$. 
We change the color of a tubular neighborhood $P_{\epsilon}$ of $P$ from $B$ to $A$ and add a factor of $\displaystyle \frac{1}{\mu}$ in front of it.
The boundary of the $A$-colored $0$-handle $P_\epsilon$ can be colored by either $\mathcal{C}$ or $\mathcal{D}$. 
We remark that
the boundary color can be switched to the other by applying planar graphical calculus, i.e. adding a closed string labeled by $J$, with its orientation property chosen.

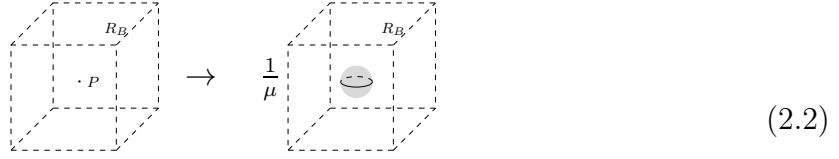
\begin{align}\label{eq:move0}
\begin{array}{ccc}
   \vcenter{\hbox{\scalebox{0.7}{
\begin{tikzpicture}
\draw[dashed] (0, 0) rectangle (2, 2);
\draw[dashed] (.8, .8) rectangle (2.8, 2.8);
\draw[dashed] (0, 2)--+(0.8, 0.8);
\draw[dashed] (2, 2)--+(0.8, 0.8);
\draw[dashed] (2, 0)--+(0.8, 0.8);
\draw[dashed] (0, 0)--+(0.8, 0.8);
\draw (2, 2.3) node{\tiny{$R_{B}$}};
\draw (1.3, 1.3) node[right]{\tiny{$P$}} node{$\cdot$};
\end{tikzpicture}}}}  & \rightarrow & \frac{1}{\mu}\vcenter{\hbox{\scalebox{0.7}{
\begin{tikzpicture}
\draw[dashed] (0, 0) rectangle (2, 2);
\draw[dashed] (.8, .8) rectangle (2.8, 2.8);
\draw[dashed] (0, 2)--+(0.8, 0.8);
\draw[dashed] (2, 2)--+(0.8, 0.8);
\draw[dashed] (2, 0)--+(0.8, 0.8);
\draw[dashed] (0, 0)--+(0.8, 0.8);
\filldraw[white!70!gray] (1.3, 1.3) circle (0.3);
\draw[dashed, opacity=0.3] (1.3, 1.3) [partial ellipse=0:180:0.3 and 0.1];
\draw[opacity=0.3] (1.3, 1.3) [partial ellipse=180:360:0.3 and 0.1];
\draw (2, 2.3) node{\tiny{$R_{B}$}};
\end{tikzpicture}}}}\\
\vspace{2mm}\\
(M, \Sigma, \Gamma)&\rightarrow & (M, \Sigma \sqcup \partial P_{\epsilon}, \Gamma)\\
\end{array}
\end{align}

\item \textbf{Move 1:} 
 Let $S$ be an embedded arc in $R_{B}$ with $\partial S$ meeting $\Sigma$ transversely and not intersecting $\Gamma$. We change the color of the tubular neighborhood $S_{\epsilon}$ of $S$ from $B$ to $A$, and put an $\Omega$-colored circle $C$ on the belt of $S_{\epsilon}$, where the $\Omega$-color on the belt is replaced by the $\Omega$-color in $\Hom(\mathfrak{C}, \mathfrak{C})$, $\Hom(\mathfrak{C}, \mathfrak{D})$, $\Hom(\mathfrak{D}, \mathfrak{C})$, $\Hom(\mathfrak{D}, \mathfrak{D})$, depending on the color on the attaching boundary of the $1$-handle. 
 Below is an example of adding $1$-handle attached to $\mathcal{D}$-colored regions.

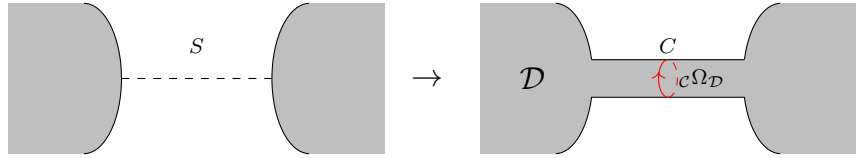
\begin{align*}
\begin{array}{ccc}
   \vcenter{\hbox{
\begin{tikzpicture}
\path[fill=gray!50!white]
(-1.5, 0) [partial ellipse=-90:90:0.5 and 1];
\draw (-1.5, 0) [partial ellipse=-90:90:0.5 and 1];
\node at (-1.8, 0) {$\mathfrak{B}$};
\node at (1.8, 0) {$\mathfrak{A}$};
\path[fill=gray!50!white]
(-2.5, -1) rectangle (-1.5, 1);
\begin{scope}[xscale=-1]
\path[fill=gray!50!white]
(-1.5, 0) [partial ellipse=-90:90:0.5 and 1];
\draw (-1.5, 0) [partial ellipse=-90:90:0.5 and 1];
\path[fill=gray!50!white]
(-2.5, -1) rectangle (-1.5, 1);
\end{scope}
\draw[dashed] (-1, 0)--(1, 0);
\draw (0, 0.2) node[above]{\tiny{$S$}};
\end{tikzpicture}}}  & \rightarrow & \vcenter{\hbox{
\begin{tikzpicture}
\path[fill=gray!50!white]
(-1.5, 0) [partial ellipse=-90:90:0.5 and 1];
\draw (-1.5, 0) [partial ellipse=-90:90:0.5 and 1];
\path[fill=gray!50!white]
(-2.5, -1) rectangle (-1.5, 1);
\node at (-1.8, 0) {$\mathcal{D}$};
\node at (1.8, 0) {$\mathcal{C}$};
\begin{scope}[xscale=-1]
\path[fill=gray!50!white]
(-1.5, 0) [partial ellipse=-90:90:0.5 and 1];
\draw (-1.5, 0) [partial ellipse=-90:90:0.5 and 1];
\path[fill=gray!50!white]
(-2.5, -1) rectangle (-1.5, 1);
\end{scope}
\draw[dashed] (-1, 0)--(1, 0);
\draw (0, 0.2) node[above]{\tiny{$C$}};
\path[fill=white]
(-1.1, -0.25) rectangle (1.1, 0.25);
\path[fill=gray!50!white]
(-1.1, -0.25) rectangle (1.1, 0.25);
\draw (-1.02, 0.25)--(1.02, 0.25);
\draw (-1.02, -0.25)--(1.02, -0.25);
\draw[red, dashed] (0, 0) [partial ellipse=-90:90:0.125 and 0.25];
\draw[red, -<-=0.5] (0, 0) [partial ellipse=90:270:0.125 and 0.25];
\draw (0, 0) node[right]{\tiny${}_{\mathcal{C}}\Omega_{\mathcal{D}}$};
\end{tikzpicture}}}\\
\vspace{2mm}\\
(M, \Sigma, \Gamma)&\rightarrow & (M, \partial(R_{B}\setminus S_{\epsilon}), \Gamma\sqcup C)\\
\end{array}
\end{align*}

\item \textbf{Move 2:}  Let $D$ be a disk in $R_{B}$ with $\partial D\subset \partial R_{B}$ intersect $\Gamma$ only at edges transversely. We change the color of a tubular neighborhood of $D$ from $B$ to $A$, then cut the diagram along $\partial D$ and put a pair of sum of dual basis $\phi$ and $\phi'$ on both side of $D_{\epsilon}$.

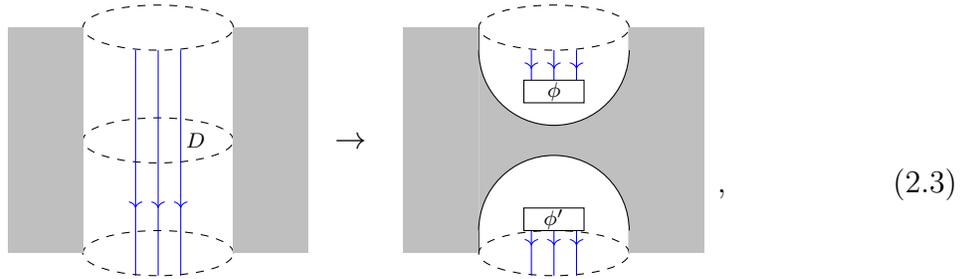
\begin{align}\label{eq:move2}
\begin{array}{ccc}
   \vcenter{\hbox{
\begin{tikzpicture}
\draw[dashed] (0, 1.5) [partial ellipse=0:360:1 and 0.3];
\draw[dashed] (0, 0) [partial ellipse=0:360:1 and 0.3];
\draw[dashed] (0, -1.5) [partial ellipse=0:360:1 and 0.3];
\draw (0.5, 0) node{\tiny{$D$}};
\draw (-1, -1.5)--(-1, 1.5);
\draw (1, -1.5)--(1, 1.5);
\path[fill=gray!50!white]
(-1, -1.5) rectangle (-2, 1.5);
\path[fill=gray!50!white]
(1, -1.5) rectangle (2, 1.5);
\draw [blue, ->-=0.7] (0, 1.2)--(0, -1.8);
\draw [blue, ->-=0.7] (-0.3, 1.2)--(-0.3, -1.8);
\draw [blue, ->-=0.7] (0.3, 1.2)--(0.3, -1.8);
\end{tikzpicture}}}  & \rightarrow & \vcenter{\hbox{
\begin{tikzpicture}
\draw[dashed] (0, 1.5) [partial ellipse=0:360:1 and 0.3];
\draw[dashed] (0, -1.5) [partial ellipse=0:360:1 and 0.3];
\draw (-1, -1.5)--(-1, -1.2);
\draw (1, 1.5)--(1, 1.2);
\draw (-1, 1.5)--(-1, 1.2);
\draw (1, -1.5)--(1, -1.2);
\path[fill=gray!50!white]
(-1, -1.5) rectangle (-2, 1.5);
\path[fill=gray!50!white]
(1, -1.5) rectangle (2, 1.5);
\draw [blue, ->-=0.5] (0, 1.2)--(0, 0.7);
\draw [blue, ->-=0.5] (-0.3, 1.2)--(-0.3, 0.7);
\draw [blue, ->-=0.5] (0.3, 1.2)--(0.3, 0.7);
\draw [blue, ->-=0.5] (0, -1)--(0, -1.8);
\draw [blue, ->-=0.5] (-0.3, -1)--(-0.3, -1.8);
\draw [blue, ->-=0.5] (0.3, -1)--(0.3, -1.8);
\path[fill=gray!50!white]
(-1, -1.2) arc (180:0:1)--(1, 1.2) arc (0:-180:1)--(-1, -1.2);
\draw (-1, -1.2) arc (180:0:1);
\draw (1, 1.2) arc (0:-180:1);
\draw [fill=white](-0.4, 0.5) rectangle (0.4, 0.8); \node at (0, 0.65) {\tiny{$\phi$}};
\draw [fill=white](-0.4, -1.2) rectangle (0.4, -0.9); \node at (0, -1.05) {\tiny{$\phi'$}};
\end{tikzpicture}}}\\
\vspace{2mm}\\
(M, \Sigma, \Gamma)&\rightarrow & (M, \partial(R_{B}\setminus D_{\epsilon}), \Gamma')
\end{array},
\end{align}
where $\phi$ and $\phi'$ denote the dual base $\{\phi_j\}_j$ and $\{\phi_j'\}_j$ by and suppress the summation.

\item \textbf{Move 3:}  Let $T$ be a $B$-colored $3$-ball with $\partial T\subset \Sigma$ and $\partial T\cap \Gamma=\emptyset$. 
We change the color of $T$ from $B$ to $A$,
no matter whether the boundary of the sphere is colored by $\mathcal{C}$ or $\mathcal{D}$.
\begin{align}\label{eq:move3}
\begin{array}{ccc}
\vcenter{\hbox{
\begin{tikzpicture}
\path[fill=gray!50!white]
(-2, -1.5) rectangle (2, 1.5);
\draw [fill=white] (0, 0) [partial ellipse=0:360:1];
\draw[opacity=0.3] (0, 0)[partial ellipse=0:-180:1 and 0.3];
\draw[dashed, opacity=0.3] (0, 0)[partial ellipse=0:180:1 and 0.3];
\draw (0.3, 0.3) node{\tiny{$T$}};
\end{tikzpicture}
}}
&\rightarrow&
\vcenter{\hbox{
\begin{tikzpicture}
\path[fill=gray!50!white]
(-2, -1.5) rectangle (2, 1.5);
\end{tikzpicture}
}}\\
\vspace{2mm}\\
(M, \Sigma, \Gamma)&\rightarrow& (M, \Sigma\setminus \partial T, \Gamma)\\
\end{array}
\end{align}
\end{itemize}

There exists a unique partition function invariant under the topological moves described as above.

\begin{proposition}[Partition Function, Theorem 6.1 in \cite{LMWW23b}] \label{thm:e partition function}
Let $\zeta$ be a nonzero scalar. 
There exists a unique partition function $Z_e$ from $3$-alterfolds to the ground field satisfying the following conditions.
\begin{itemize}
    \item If $R_{B}=\emptyset$, $Z_e(M,\emptyset,\emptyset)=1$.
    
    \item \textbf{Disjoint Union:} Suppose $(M, \Sigma, \Gamma)$ and $(M', \Sigma', \Gamma')$ are two $3$-alterfolds.
    Then
    $$Z_e(M\sqcup M', \Sigma \sqcup \Sigma', \Gamma\sqcup\Gamma')=Z_e(M, \Sigma, \Gamma)Z_e(M', \Sigma', \Gamma').$$
    
     \item \textbf{Homeomorphims:} The evaluation $Z_e(M, \Sigma, \Gamma)$ only depends on the orientation preserving homeomorphism class of $(R_{B}, \Gamma)$.
     
    \item \textbf{Planar Graphical Calculus:} If $\Gamma'$ and $\Gamma$ are diagrams on $\Sigma$ that are identical outside of a contractible region $D$, and $\Gamma_{D}:=\Gamma\cap D$ equals to $\Gamma_{D}':=\Gamma'\cap D$ as morphisms. 
    Then
    $$Z_e(M, \Sigma, \Gamma)=Z_e(M, \Sigma, \Gamma').$$
    \item \textbf{Move 0:} Suppose $(M, \Sigma', \Gamma')$ is derived from $(M, \Sigma, \Gamma)$ by applying Move $0$, then
    $$Z_e(M, \Sigma', \Gamma')=\zeta\mu Z_e(M, \Sigma, \Gamma).$$
    
    \item \textbf{Move 1:} Suppose $(M, \Sigma', \Gamma')$ is derived from $(M, \Sigma, \Gamma)$ by applying Move $1$, then
    $$Z_e(M, \Sigma', \Gamma')=\frac{1}{\zeta} Z_e(M, \Sigma, \Gamma).$$
    
    \item \textbf{Move 2:} Suppose $(M, \Sigma', \Gamma')$ is derived from $(M, \Sigma, \Gamma)$ by applying Move $2$, then
    $$Z_e(M, \Sigma', \Gamma')=\zeta Z_e(M, \Sigma, \Gamma).$$
    
    \item \textbf{Move 3:} Suppose $(M, \Sigma', \Gamma')$ is derived from $(M, \Sigma, \Gamma)$ by applying Move $3$, then
    $$Z_e(M, \Sigma', \Gamma')=\frac{1}{\zeta} Z_e(M, \Sigma, \Gamma).$$
\end{itemize}
\end{proposition}

We usually assume that $\zeta=1$ for computational convenience.
The partition function gives a topological quantum field theory.
\begin{proposition}[Alterfold TQFT,  Theorem 6.2, 6.3 in \cite{LMWW23b}]\label{thm:etqft}
The partition function $Z_e$ defined in Proposition \ref{thm:e partition function} extends to a topological quantum field theory.
If the Morita context is unitary, then the alterfold TQFT is unitary.
\end{proposition}

Recall that the Drinfeld center $\mathcal{Z(C)}$ is braided isomorphic to the tube category $\mathcal{A}$ defined in \cite{LMWW23} whose morphisms are depicted as cylinders in the alterfold TQFT. 
The $B$-colored region with time boundary represents the hom-space in the Reshetikhin-Turaev TQFT of the Drinfeld center $\mathcal{Z(C)}$.
The $\mathcal{C}$-colored surface is colored by white and the $\mathcal{D}$-colored surface is colored by brown in general.
Without specifying, the inside of a tube is colored by $A$ and the outside is colored by $B$.
We shall frequently use the topological moves 0,1,2,3 throughout this paper.

\subsection{Alterfold TQFT for Modular Fusion Category}
When $\mathcal{C}$ is a modular fusion category.
The Drinfeld center $\mathcal{Z(C)}$ as tube category is isomorphic to $\mathcal{C}\otimes \mathcal{C}^{op}$ through the following decomposition:
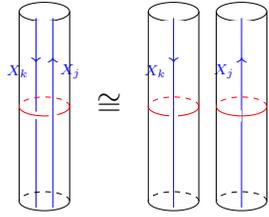
\begin{align*}
    \vcenter{\hbox{\scalebox{0.7}{
\begin{tikzpicture}[xscale=0.8, yscale=0.6]
\begin{scope}[shift={(0,3)}]
\draw (0,0) [partial ellipse=0:360:0.6 and 0.3];
\end{scope}
\path [fill=white] (-0.6, 0) rectangle (0.6, 2.7);
\draw [blue, ->-=0.5] (-0.2, 2.8)--(-0.2, 0) node [left, pos=0.6] {\tiny $X_k$} ;
\draw [blue, -<-=0.5] (0.2, 2.8)--(0.2, 0) node [right, pos=0.6] {\tiny $X_j$}; 
\draw (-0.6, 3)--(-0.6, 0) (0.6, 3)--(0.6, 0); 
\draw [blue] (-0.2, -3.2)--(-0.2, 0) (0.2, -3.2)--(0.2, 0);
\draw (-0.6, -3)--(-0.6, 0) (0.6, -3)--(0.6, 0);
\begin{scope}[shift={(0,-3)}]
\draw [dashed](0,0) [partial ellipse=0:180:0.6 and 0.3];
\draw (0,0) [partial ellipse=180:360:0.6 and 0.3];
\end{scope}
\draw [red, dashed](0,0) [partial ellipse=0:180:0.6 and 0.3]; 
\draw [white, line width=4pt]  (0,0) [partial ellipse=220:270:0.6 and 0.3];
\draw [red]  (0,0) [partial ellipse=180:280:0.6 and 0.3];
\draw [red]  (0,0) [partial ellipse=300:360:0.6 and 0.3];
\end{tikzpicture}}}}
\cong
\vcenter{\hbox{\scalebox{0.7}{
\begin{tikzpicture}[xscale=0.8, yscale=0.6]
\begin{scope}[shift={(0,3)}]
\draw (0,0) [partial ellipse=0:360:0.6 and 0.3];
\end{scope}
\path [fill=white] (-0.6, 0) rectangle (0.6, 2.7);
\draw [blue, ->-=0.5] (0, 2.8)--(0, 0) node [left, pos=0.6] {\tiny $X_k$} ;
\draw (-0.6, 3)--(-0.6, 0) (0.6, 3)--(0.6, 0); 
\draw [blue] (0, -3.2)--(0, 0);
\draw (-0.6, -3)--(-0.6, 0) (0.6, -3)--(0.6, 0);
\begin{scope}[shift={(0,-3)}]
\draw [dashed](0,0) [partial ellipse=0:180:0.6 and 0.3];
\draw (0,0) [partial ellipse=180:360:0.6 and 0.3];
\end{scope}
\draw [red, dashed](0,0) [partial ellipse=0:180:0.6 and 0.3]; 
\draw [red]  (0,0) [partial ellipse=180:260:0.6 and 0.3];
\draw [red]  (0,0) [partial ellipse=280:360:0.6 and 0.3];
\end{tikzpicture}}}}
\vcenter{\hbox{\scalebox{0.7}{
\begin{tikzpicture}[xscale=0.8, yscale=0.6]
\begin{scope}[shift={(0,3)}]
\draw (0,0) [partial ellipse=0:360:0.6 and 0.3];
\end{scope}
\path [fill=white] (-0.6, 0) rectangle (0.6, 2.7);
\draw [blue, -<-=0.5] (0, 2.8)--(0, 0) node [left, pos=0.6] {\tiny $X_j$} ;
\draw (-0.6, 3)--(-0.6, 0) (0.6, 3)--(0.6, 0); 
\draw [blue] (0, -3.2)--(0, 0);
\draw (-0.6, -3)--(-0.6, 0) (0.6, -3)--(0.6, 0);
\begin{scope}[shift={(0,-3)}]
\draw [dashed](0,0) [partial ellipse=0:180:0.6 and 0.3];
\draw (0,0) [partial ellipse=180:360:0.6 and 0.3];
\end{scope}
\draw [red, dashed](0,0) [partial ellipse=0:180:0.6 and 0.3]; 
\draw [white, line width=6pt]  (0,0) [partial ellipse=220:270:0.6 and 0.3];
\draw [red]  (0,0) [partial ellipse=180:270:0.6 and 0.3];
\draw [red]  (0,0) [partial ellipse=270:360:0.6 and 0.3];
\end{tikzpicture}}}}.
\end{align*}
Here the black strings indicate the boundary of the torus, the red strings indicate the Kirby color and the blue strings indicate the objects in the category.
In particular, we have that $\RT_{\mathcal{Z(C)}}(\mathcal{M})=\RT_{\mathcal{C}}(\mathcal{M}) \RT_{\mathcal{C}^{op}}(\mathcal{M})$ for the Reshetikhin-Turaev invariant $\RT$ of a closed compact 3-manifold $\mathcal{M}$, where was first proved in \cite{Wal91, Tur92}.
The following is a series of applications of the topological moves in the alterfold TQFT which are useful in the next sections.

 \begin{lemma}\label{lem:decomposition1}
Suppose $\mathcal{C}$ is a modular fusion category. 
Locally, we have
\begin{align*}
\vcenter{\hbox{\scalebox{0.7}{
}}},
\end{align*}

\end{lemma}

\begin{proof}
For the middle term in the first equation, we apply handle sliding to the inner red circle along the one parallel to the meridian, we obtain
$$\frac{1}{\mu}\vcenter{\hbox{\scalebox{0.7}{
%
}}}.$$
This other identity follows from rotating the diagrams anti-clockwisely by $90$ degree.
\end{proof}

 \begin{lemma}\label{lem:decomposition2}
We have that 
 \begin{align*}
 \vcenter{\hbox{\scalebox{0.7}{
%
}}},
 \end{align*}
 where the brown indicates $\mathcal{D}$-color.
 \end{lemma}
 \begin{proof}
 It follows from Move 1 and Move 2.
 \end{proof}

\begin{lemma}\label{lem:decomposition3}
Suppose $\mathcal{C}$ is a modular fusion category. 
We have locally
\begin{align*}
\vcenter{\hbox{\scalebox{0.6}{
%
}}}
\end{align*}
\end{lemma}
\begin{proof}
We use the following cylindrical presentation of a torus while remembering that the top and bottom of the cylinders below are to be identified in the torus.  Notice
\begin{align*}
\sum_{i=1}^r\quad\vcenter{\hbox{%
}},
\end{align*}
where the first equality is due to the modularity of $\mathcal{C}$, the second equality comes from handle sliding the red circle in the middle along the top one, the third equality comes from contracting diagram along the other side of the torus and the last equality comes from again the modularity of $\mathcal{C}$.
\end{proof}

\begin{lemma}\label{lem:schange}
We have that 
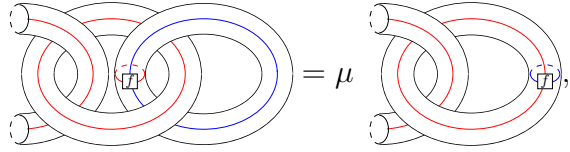
\begin{align*}
\vcenter{\hbox{\scalebox{0.7}{
%
}}},
\end{align*}
i.e., the $S$-matrix will switch the longitude and median up to a scalar.
\end{lemma}
\begin{proof}
It follows from Moves 1, 2 and handle sliding of the Kirby color.
\end{proof}

\begin{remark}
Note the $\mathcal{D}$-colored torus with empty diagrams is invariant under $S$-matrix. 
In fact, by Lemma \ref{lem:schange}, we have that 
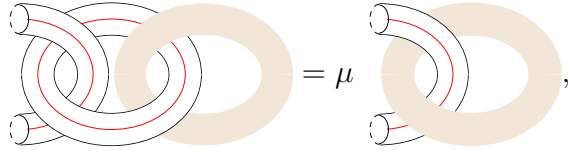
\begin{align*}
\vcenter{\hbox{\scalebox{0.7}{
%
}}},
\end{align*}
\end{remark}

\begin{proposition}\label{prop:dim}
Suppose $  _\mathcal{D}\mathcal{M}_{\mathcal{E}}$ is the module category. 
We have that 
\begin{align*}
\frac{1}{\mu}\vcenter{\hbox{\scalebox{0.8}{
%
}}}=|\Irr( _\mathcal{D}\mathcal{M}_{\mathcal{E}})|
\end{align*}
\end{proposition}
\begin{proof}
 It follows from Lemma 2.8 in \cite{LMWW23b}.   
\end{proof}

\begin{remark}
We remark that the above computation realizes the diagram on the left hand side as an endomorphism of the tensor unit in the Drinfeld center. 
If we realize it as an alterfold with ambient manifold $\mathbb{S}^3$, the result should be $\displaystyle \frac{1}{\mu}|\Irr( _\mathcal{D}\mathcal{M}_{\mathcal{E}})|$. 
The extra factor $\frac{1}{\mu}$ is due to the fact that the $B$-colored $\mathbb{S}^3$ is evaluate to $\displaystyle \frac{1}{\mu}$ in the alterfold TQFT.
In this paper, we always adopt the first convention.
\end{remark}

We end this section by recalling the factorization of the RT space for torus associated to the Drinfeld center $\mathcal{Z(C)}$.
Let $\Sigma$ be the torus in 3-dimensional sphere $\mathbb{S}^3$.
Suppose $\Sigma\times[0,1]$ in $\mathbb{S}^3$ such that $\Sigma\times \{0,1\}$ are time boundaries, $\Sigma\times (0,1)$ is colored by $B$ and $\Sigma \times \left((-\infty, 0)\cup (1, \infty)\right)$ is colored by $A$.
By the fact that $\mathcal{Z(C)}=\mathcal{C}\boxtimes \mathcal{C}^{op}$, we have that 
\begin{align*}
\RT_{\mathcal{Z(C)}}(\Sigma\times\{0,1\})=\RT_{\mathcal{C}}(\Sigma\times\{0,1\})\otimes \RT_{\mathcal{C}^{op}}(\Sigma\times\{0,1\}),
\end{align*}
where $\RT_{\mathcal{Z(C)}}$ is the RT space in the alterfold TQFT defined in \cite{LMWW23b}.
Let $\{\xi_j\}_j$ be a basis in $\RT_{\mathcal{C}}(\Sigma\times\{0,1\})$ and $\{\xi_j^{op}\}_j$ the dual basis such that the image of $\xi_j$ in $\RT_{\mathcal{C}}(\Sigma\times\{0,1\})$ is depicted as
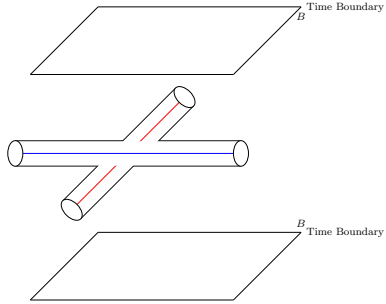
\begin{align*}
\vcenter{\hbox{\scalebox{0.5}{
\begin{tikzpicture}
\begin{scope}[shift={(1, 3)}, scale=1.8]
\draw (-1,0.5)--(2, 0.5) (-2,-0.5)--(1, -0.5);
\draw (-2,-0.5)--(-1, 0.5) (1,-0.5)--(2, 0.5) node [right] {\tiny Time Boundary}  node [below] {\tiny $B$};
\end{scope}
\draw [line width=0.7cm]  (-3,0)--(3, 0);
\draw [line width=0.7cm] (-1.5, -1.5)--(1.5, 1.5);
\draw [white, line width=0.65cm] (-1.75, -1.75)--(1.75, 1.75);
\draw [red] (-1.5, -1.5)--(1.5, 1.5);
\draw [white, line width=0.65cm]  (-3,0)--(3, 0);
\draw [blue] (-2.8, 0)--(2.8, 0);
\begin{scope}[shift={(3, 0)}]
\draw [fill=white] (0, 0) [partial ellipse=0:360:0.2 and 0.34];
\end{scope}
\begin{scope}[shift={(-3, 0)}]
\draw [fill=white] (0, 0) [partial ellipse=0:360:0.2 and 0.34];
\end{scope}
\begin{scope}[shift={(-1.5, -1.5)},rotate=45 ]
\draw [fill=white] (0, 0) [partial ellipse=0:360:0.2 and 0.34];
\end{scope}
\begin{scope}[shift={(1.5, 1.5)},rotate=45 ]
\draw [fill=white] (0, 0) [partial ellipse=0:360:0.2 and 0.34];
\end{scope}
\begin{scope}[shift={(1, 1)},rotate=45 ]
\end{scope}
\begin{scope}[shift={(1, -3)}, scale=1.8]
\draw (-1,0.5)--(2, 0.5) (-2,-0.5)--(1, -0.5);
\draw (-2,-0.5)--(-1, 0.5) (1,-0.5)--(2, 0.5) node [right] {\tiny Time Boundary} node [above] {\tiny $B$} ;
\end{scope}
\end{tikzpicture}}}}
\end{align*}
Suppose $f$ is an element in the mapping class group of $\Sigma\times\{0,1\}$.
Let $\rho(f)$ be the action of $f$ on $\RT_{\mathcal{Z(C)}}(\Sigma\times \{0,1\})$.
Let $f_{\mathcal{C}}$, $f_{\mathcal{C}^{op}}$ be the matrix representative of the actions of $f$ on $\RT_{\mathcal{C}}(\Sigma\times \{0,1\})$, $\RT_{\mathcal{C}^{op}}(\Sigma\times \{0,1\})$ respectively.

\section{Topological Modular Invariance and Dual Fusion Algebras}

In this section, we propose a topological interpretation of the modular invariance and the $\alpha$-induction for a Frobenius algebra in a modular fusion category $\mathcal{C}$ over a general field, using the alterfold TQFT of $\mathcal{C}$. 
Moreover, we give quick proofs of the results of B\"{o}ckenhauer, Evans and Kawahigashi, and their generalizations for modular fusion categories over a general field \cite{BE98,BE99,BE00,BEK99,BEK00}.

\subsection{Topological Modular Invariance}
In 1987, Cappelli, Itzykson, and Zuber \cite{CIZ87} classified minimal conformal invariant theories based on an $A$-$D$-$E$ classification of modular invariants of quantum $\mathfrak{sl}_2$.
The partition function of a torus is expended over the basis of characters as
\begin{align}\label{Equ: Z=Zjk1}
Z=\sum_{j,k} z_{jk} \chi_j \overline{\chi_k}.
\end{align}
The coefficients $z_{jk}$ are the anomalous dimensions and they form a matrix commuting with the modular transformation, called the modular invariance matrix.
In general, for a modular fusion category $\mathcal{C}$, a matrix $Z: \Irr(\mathcal{C})\times \Irr(\mathcal{C}) \to \mathbb{N}$ commuting with the modular $S$-matrix and $T$-matrix is called a modular invariant for $\mathcal{C}$.

We propose a topological analogue of the partition function $Z^{\mathcal{D}}$ of a torus and its expansion over the basis of simple objects of the $Z(\mathcal{C})$ as
\begin{align}\label{Z=Zjk top1}
Z^{\mathcal{D}}=\sum_{j,k}z_{jk}\xi_{j} \otimes \xi_{k}^{op},   
\end{align}
which is a topological analogue of Equation ~\eqref{Equ: Z=Zjk} in alterfold TQFT as follows.
\begin{align*}
\vcenter{\hbox{\scalebox{0.5}{
\begin{tikzpicture}
\begin{scope}[shift={(1, 1)}, scale=1.8]
\path [fill=brown!20!white] (-1,0.5)--(2, 0.5)--(1, -0.5)--(-2, -0.5)--cycle;
\draw (-1,0.5)--(2, 0.5) (-2,-0.5)--(1, -0.5);
\draw (-2,-0.5)--(-1, 0.5) (1,-0.5)--(2, 0.5) node [below] {\tiny $B$} node [above] {\tiny $A$};
\node at (0,0) { $\mathcal{D}$};
\end{scope}
\begin{scope}[shift={(1, -1)}, scale=1.8]
\draw (-1,0.5)--(2, 0.5) (-2,-0.5)--(1, -0.5);
\draw (-2,-0.5)--(-1, 0.5) (1,-0.5)--(2, 0.5) node [above] {\tiny $B$}; 
\node at (2, 0.5) [below right] {\tiny Time Boundary};
\end{scope}
\end{tikzpicture}}}}
=
\sum_{j,k} z_{jk} \frac{1}{\mu}
\vcenter{\hbox{\scalebox{0.5}{
\begin{tikzpicture}
\begin{scope}[shift={(1, 1)}, scale=1.8]
\draw (-1,0.5)--(2, 0.5) (-2,-0.5)--(1, -0.5);
\draw (-2,-0.5)--(-1, 0.5) (1,-0.5)--(2, 0.5) node [above] {\tiny $B$} node [below] {\tiny $A$};
\draw [red] (-0.5,-0.5)--(0.5, 0.5);
\path [fill=white] (0.2, 0.2) circle (0.2cm);
\draw [blue] (-1.3, 0.2)--(1.7, 0.2);
\draw [blue] (-1.7, -0.2)--(-0.4, -0.2) (0, -0.2) --(1.3, -0.2);
\end{scope}
\begin{scope}[shift={(1, -1)}, scale=1.8]
\draw (-1,0.5)--(2, 0.5) (-2,-0.5)--(1, -0.5);
\draw (-2,-0.5)--(-1, 0.5) (1,-0.5)--(2, 0.5) node [above] {\tiny $B$}; 
\node at (2, 0.5) [below right] {\tiny Time Boundary};
\end{scope}
\end{tikzpicture}}}},
\end{align*}

Here we shall present its topological definition in the alterfold TQFT.
We define the linear functional $\alpha_{\mathcal{D}}$ on $\RT(\Sigma\times\{0,1\})$ as follows:
\begin{align*}
\alpha_{\mathcal{D}}\left(\vcenter{\hbox{\scalebox{0.5}{
\begin{tikzpicture}
\begin{scope}[shift={(1, 1)}, scale=1.8]
\draw (-1,0.5)--(2, 0.5) (-2,-0.5)--(1, -0.5);
\draw (-2,-0.5)--(-1, 0.5) (1,-0.5)--(2, 0.5) node [right] {\tiny Time Boundary}  node [below] {\tiny $B$};
\end{scope}
\begin{scope}[shift={(1, -1)}, scale=1.8]
\draw (-1,0.5)--(2, 0.5) (-2,-0.5)--(1, -0.5);
\draw (-2,-0.5)--(-1, 0.5) (1,-0.5)--(2, 0.5) node [right] {\tiny Time Boundary} node [above] {\tiny $B$} ;
\node at (2, 1) { $\mathcal{Z(C)}$};
\end{scope}
\end{tikzpicture}}}}\right)
=\frac{1}{\mu}
\vcenter{\hbox{\scalebox{0.5}{
\begin{tikzpicture}
\begin{scope}[shift={(1, 1)}, scale=1.8]
\path [fill=brown!20!white] (-1,0.5)--(2, 0.5)--(1, -0.5)--(-2, -0.5)--cycle;
\draw (-1,0.5)--(2, 0.5) (-2,-0.5)--(1, -0.5);
\draw (-2,-0.5)--(-1, 0.5) (1,-0.5)--(2, 0.5)   node [below] {\tiny $B$} node [above] {\tiny $A$};
\node at (0,0) { $\mathcal{D}$};
\end{scope}
\begin{scope}[shift={(1, -1)}, scale=1.8]
\draw (-1,0.5)--(2, 0.5) (-2,-0.5)--(1, -0.5);
\draw (-2,-0.5)--(-1, 0.5) (1,-0.5)--(2, 0.5) node [above] {\tiny $B$} node [below] {\tiny $A$};
\node at (2, 1) { $\mathcal{Z(C)}$};
\end{scope}
\end{tikzpicture}}}}.
\end{align*}
Then
\begin{align*}
   \sum_{j,k=1}^r \alpha_{\mathcal{D}}(\xi_j\otimes\xi_k^{op})\xi_j\otimes \xi_k^{op}
\end{align*}
can be viewed as the partition function for the torus and the $Z$-matrix is $(\alpha_{\mathcal{D}}(\xi_j\otimes \xi_k^{op}))_{j,k}=(z_{j,k})_{j,k=1}^r$.
We then have that 
\begin{align*}
z_{jk}=\alpha_{\mathcal{D}}(\xi_j\otimes \xi_k^{op})=& 
\frac{1}{\mu^2}
\vcenter{\hbox{\scalebox{0.5}{
\begin{tikzpicture}
\begin{scope}[shift={(1, 1)}, scale=1.8]
\path [fill=brown!20!white] (-1,0.5)--(2, 0.5)--(1, -0.5)--(-2, -0.5)--cycle;
\draw (-1,0.5)--(2, 0.5) (-2,-0.5)--(1, -0.5);
\draw (-2,-0.5)--(-1, 0.5) (1,-0.5)--(2, 0.5) node [below] {\tiny $B$} node [above] {\tiny $A$};
\node at (0,0) { $\mathcal{D}$};
\end{scope}
\begin{scope}[shift={(1, -1)}, scale=1.8]
\draw (-1,0.5)--(2, 0.5) (-2,-0.5)--(1, -0.5);
\draw (-2,-0.5)--(-1, 0.5) (1,-0.5)--(2, 0.5) node [above] {\tiny $B$} node [below] {\tiny $A$};
\node at (2, 1) { $\mathcal{Z(C)}$};
\draw [red] (-0.5,-0.5)--(0.5, 0.5);
\path [fill=white] (0.2, 0.2) circle (0.2cm);
\draw [blue] (-1.3, 0.2)--(1.7, 0.2);
\draw [blue] (-1.7, -0.2)--(-0.4, -0.2) (0, -0.2) --(1.3, -0.2);
\end{scope}
\end{tikzpicture}}}}\\
= & \frac{1}{\mu^2} \vcenter{\hbox{\scalebox{0.7}{
\begin{tikzpicture}[scale=0.7]
\draw [line width=0.6cm, brown!20!white] (0,0) [partial ellipse=-0.1:180.1:2 and 1.5];
\begin{scope}[shift={(2.5, 0)}]
\draw [line width=0.6cm] (0,0) [partial ellipse=0:180:2 and 1.5];
\draw [white, line width=0.57cm] (0,0) [partial ellipse=-0.1:180.1:2 and 1.5];
\draw [blue] (0,0) [partial ellipse=0:180:2.15 and 1.65];
\draw [blue] (0,0) [partial ellipse=0:180:1.85 and 1.35];
\end{scope} 
\begin{scope}[shift={(2.5, 0)}]
\draw [line width=0.6cm] (0,0) [partial ellipse=180:360:2 and 1.5];
\draw [white, line width=0.57cm] (0,0) [partial ellipse=178:362:2 and 1.5];
\draw [blue, -<-=0.5] (0,0) [partial ellipse=178:362:2.15 and 1.65] node[black, pos=0.7,below ] {\tiny $X_k$}; 
\draw [blue, ->-=0.5] (0,0) [partial ellipse=178:362:1.85 and 1.35] node[black, pos=0.7,above ] {\tiny $X_j$};
\end{scope}
\draw [line width=0.6cm, brown!20!white] (0,0) [partial ellipse=180:360:2 and 1.5];
\begin{scope}[shift={(4.5, 0)}]
\draw [red, dashed](0,0) [partial ellipse=0:180:0.4 and 0.25] ;
\draw [white, line width=4pt] (0,0) [partial ellipse=290:270:0.4 and 0.25];
\draw [red] (0,0) [partial ellipse=260:360:0.4 and 0.25];
\draw [red] (0,0) [partial ellipse=180:230:0.4 and 0.25];
\end{scope}
\end{tikzpicture}}}}\\
=&\dim\hom_{\mathcal{Z(C)}} (I(\1_{\mathcal{D}})), X_j\boxtimes X_k^{op})
\end{align*}
is an integer.

\begin{proposition}\label{prop:z}
The $Z$-matrix is modular invariant, namely it commutes with the modular $S$-matrix and $T$-matrix and the entries are natural numbers.
\end{proposition}
\begin{proof}
Suppose $f$ is an element in the mapping class group of $\Sigma\times\{0,1\}$.
Note that $\mathcal{D}$-colored and $\mathcal{C}$-colored tori are invariant under the mapping class group of the torus.
We have that $\alpha_{\mathcal{D}}\rho(f)=\alpha_{\mathcal{D}}$.
Then 
\begin{align*}
\alpha_{\mathcal{D}}((f_{\cC}\xi_j) \otimes (f_{\cC^{\op}}\xi_k^{\op})) 
= \alpha_{\mathcal{D}}(\rho(f)(\xi_j \otimes \xi_k^{\op})) 
= \alpha_{\mathcal{D}}(\xi_j \otimes \xi_{k}^{\op}).
\end{align*}
This implies that $Z$-matrix is invariant under the mapping class group of the torus and the $Z$ matrix is modular invariant.
\end{proof}

\begin{proposition}\label{prop:tracedim}
We have that
\begin{align*}
  Tr(Z) = |\Irr(\mathcal{M})|, \quad Tr(ZZ^t) =|\Irr(\mathcal{D})|.
\end{align*}
\end{proposition}
\begin{proof}
Note that the identity in $\RT_{\mathcal{Z(C)}}(\Sigma)$ has decomposition $\sum_{j} \xi_j\otimes \xi_j^{\op}$.
By Proposition \ref{prop:dim}, we see that 
\begin{align*}
Tr(Z)=\frac{1}{\mu}
\vcenter{\hbox{\scalebox{0.5}{
\begin{tikzpicture}
\begin{scope}[shift={(1, 1)}, scale=1.8]
\path [fill=brown!20!white] (-1,0.5)--(2, 0.5)--(1, -0.5)--(-2, -0.5)--cycle;
\draw (-1,0.5)--(2, 0.5) (-2,-0.5)--(1, -0.5);
\draw (-2,-0.5)--(-1, 0.5) (1,-0.5)--(2, 0.5)   node [below] {\tiny $B$} node [above] {\tiny $A$};
\node at (0,0) { $\mathcal{D}$};
\end{scope}
\begin{scope}[shift={(1, -1)}, scale=1.8]
\draw (-1,0.5)--(2, 0.5) (-2,-0.5)--(1, -0.5);
\draw (-2,-0.5)--(-1, 0.5) (1,-0.5)--(2, 0.5) node [above] {\tiny $B$} node [below] {\tiny $A$};
\end{scope}
\end{tikzpicture}}}}=|\Irr(\mathcal{M})|.
\end{align*}
Similarly, we obtain that 
\begin{align*}
  Tr(ZZ^t)= \sum_{k,j} \alpha_{ \mathcal{D}}(\xi_j \otimes \xi_k^{\op})\alpha_{\mathcal{D}}(\xi_k\otimes \xi_j^{\op})
   =\frac{1}{\mu}\vcenter{\hbox{\scalebox{0.5}{
\begin{tikzpicture}
\begin{scope}[shift={(1, 1)}, scale=1.8]
\path [fill=brown!20!white] (-1,0.5)--(2, 0.5)--(1, -0.5)--(-2, -0.5)--cycle;
\draw (-1,0.5)--(2, 0.5) (-2,-0.5)--(1, -0.5);
\draw (-2,-0.5)--(-1, 0.5) (1,-0.5)--(2, 0.5)  node [below] {\tiny $B$} node [above] {\tiny $A$};
\node at (0,0) { $\mathcal{D}$};
\end{scope}
\begin{scope}[shift={(1, -1)}, scale=1.8]
\path [fill=brown!20!white] (-1,0.5)--(2, 0.5)--(1, -0.5)--(-2, -0.5)--cycle;
\draw (-1,0.5)--(2, 0.5) (-2,-0.5)--(1, -0.5);
\draw (-2,-0.5)--(-1, 0.5) (1,-0.5)--(2, 0.5) node [above] {\tiny $B$} node [below] {\tiny $A$};
\end{scope}
\end{tikzpicture}}}}=|\Irr(\mathcal{D})|.
\end{align*}
This completes the proof of the proposition.
\end{proof}

\subsection{Topological $\alpha$-Induction}
Suppose $\mathcal{C}$ is a modular fusion category and $Q$ is a Frobenius algebra in  $\mathcal{C}$, then the $Q$-$Q$ bimodule category $\mathcal{D}$ is a spherical fusion category Morita equivalent to $\mathcal{C}$.
Xu introduced the $\alpha$-induction for a commutative Frobenius algebra $Q$ in a unitary modular fusion category $\mathcal{C}$ in \cite{Xu98a}, which is a pair of induction functors of unitary fusion categories $\alpha_{\pm}:\mathcal{C}\to\mathcal{D}$, where the $\pm$ sign depends on the choice of the $\pm$ braids of $\mathcal{C}$.
B\"{o}ckenhauer, Evans and Kawahigashi proposed the modular invariant mass $Z$-matrix of $\mathcal{D}$ as 
\begin{align}\label{eq:BEKalpha}
z_{jk}=\dim \hom_{\mathcal{D}}(\alpha_+(X_j),\alpha_-(X_k))
\end{align}
for simple objects $X_j$ of $\mathcal{C}$. 
They proved that the $Z$-matrix is invariant under the action of the modular group, namely the genus-one mapping class group, which is generated by the modular $S$, $T$ matrices.
In addition to the modular invariant property, they studied the $\alpha$-induction and the $Z$-matrix systematically in \cite{BE98, BE99, BE00, BE99b, BEK99, BEK00} and they summarized the theoretical results as follows in \cite{BEK00},
\begin{enumerate}
\item Surjectivity of the $\alpha$-induction in $\mathcal{D}$.
\item The $Z$-matrix $Z$ is a modular invariance. (Theorem 5.7 in \cite{BEK00})
\item Characterize the matrix units of dual fusion algebra; (Theorem 6.8 in \cite{BEK00})
\item The Grothendieck ring $K_0(\mathcal{D})$ is commutative if and only if $z_{jk}\in \{0, 1\}$ (Corollary 6.9 in \cite{BEK00}).
\item The number of simple objects in $\mathcal{D}$ is the trace of $ZZ^t$ (Corollary 6.10 in \cite{BEK00}).
\item The diagonal entry $Z_{jj}$ is the dimension of the $j^{th}$ eigenspace of the commutative fusion algebra $K_0(\mathcal{C})\otimes_{\mathbb{Z}} \mathbb{C} $ acting on $K_0(\mathcal{M}) \otimes_{\mathbb{Z}} \mathbb{C}$, where $\mathcal{M}$ is the category of $\mathcal{C}$-$\mathcal{D}$ bimodules. (Theorem 6.12 in \cite{BEK00})
\item The number of simple objects in $\mathcal{M}$ is the trace of $Z$ (Corollary 6.13 in \cite{BEK00}). 
\end{enumerate}

We propose the topological interpretation of $\alpha$-induction in the alterfold TQFT and recover the definition of the modular invariant in Equation \eqref{eq:BEKalpha} and the results (1)-(7) of B\"{o}ckenhauer, Evans and Kawahigashi, and their generalizations for modular fusion categories over a general field. 
More precisely, Proposition \ref{prop:z} covers results (2); Proposition \ref{prop:tracedim} covers (5), (7); Theorem \ref{thm:minimalidempotent} covers result (3); Corollary \ref{cor:commdual} covers result (4); Corollary \ref{cor:exponent} covers result (6).
Note that the definition of the $\alpha$-induction and modular invariant $Z$-matrix for modular fusion categories over the complex field without unitary assumption can be found in \cite{Ost03m}.

Now let us introduce the definition of the $\alpha$-induction in the alterfold TQFT.
The $\alpha$-induction functors $\alpha_{\pm}$ are implemented as the forgetful functors from $\mathcal{C}\boxtimes \1$ and $\1 \boxtimes\mathcal{C}^{op}$ to $\mathcal{D}$ respectively.
The graphical calculus of $\mathcal{D}$ is on the plane colored by $\mathcal{D}$ and the object in the braided category $\mathcal{C}$ can be depicted as a tube in the $B$-colored region in the alterfold TQFT as follows:
\begin{align}\label{pic:alpha1}
\vcenter{\hbox{
\begin{tikzpicture}[rotate=-90, scale=1.1]
\path [fill=brown!20!white] (-1,0.5)--(2, 0.5)--(1, -0.5)--(-2, -0.5)--cycle;
\draw (-1,0.5)--(2, 0.5) (-2,-0.5)--(1, -0.5);
\draw (-2,-0.5)--(-1, 0.5) (1,-0.5)--(2, 0.5) node [right] {\tiny $B$} node [left] {\tiny $A$};
\node at (-0.9, 0.3) {\tiny $\mathcal{D}$};
\end{tikzpicture}}}
\hspace{1cm}
\vcenter{\hbox{\scalebox{1}{
\begin{tikzpicture}[xscale=0.8, yscale=0.6]
\begin{scope}[shift={(0,2)}]
\draw (0,0) [partial ellipse=0:360:0.6 and 0.3];
\end{scope}
\draw (-0.6, 2)--(-0.6, 0) (0.6, 2)--(0.6, 0); 
\draw (-0.6, -2)--(-0.6, 0) (0.6, -2)--(0.6, 0);
\draw [blue] (0, -2.3) --(0, 1.7);
\draw [line width=0.2cm,white] (0, -0.18) --(0, -0.38);
\draw [red] (0,0) [partial ellipse=180:360:0.6 and 0.3];
\draw [red, dashed] (0,0) [partial ellipse=0:180:0.6 and 0.3];
\begin{scope}[shift={(0,-2)}]
\draw [dashed](0,0) [partial ellipse=0:180:0.6 and 0.3];
\draw (0,0) [partial ellipse=180:360:0.6 and 0.3];
\end{scope}
\end{tikzpicture}}}}
\end{align}
The $\alpha$-induction in the alterfold TQFT is to attach the tube on the right hand side to the plane colored by $\mathcal{D}$.
The following is the definition of the tensor functors $\alpha_\pm$ in the alterfold TQFT:
\begin{align}\label{def:alpha11}
\alpha_+\left( \vcenter{\hbox{\begin{tikzpicture}
\draw [dashed] (-1, -1) rectangle (1, 1);
\draw[-<-=0.85, -<-=0.25,blue] (0, -1) node[black, black, below]{\tiny$X$}--(0, 1)node[black, black, above]{\tiny $X$};
\draw [fill=white] (-0.3, -0.3) rectangle (0.3, 0.3);
\node at (0,0) {\tiny $f$};
\end{tikzpicture}}}\right)=
\frac{1}{d_J}\vcenter{\hbox{\begin{tikzpicture}
\draw [dashed, fill=brown!20!white] (-1.5, -1.5) rectangle (1.5, 1.5);
\path [fill=white] (0,1.5) [partial ellipse=0:-180:1 and 1];
\draw [blue, ->-=0.15] (0,1.5) [partial ellipse=0:-180:1 and 1];
\path [fill=white] (0,-1.5) [partial ellipse=0:180:1 and 1];
\draw [blue, -<-=0.15] (0,-1.5) [partial ellipse=0:180:1 and 1];
\path [fill=white] (0.4, 1) .. controls +(0, -0.3) and +(0, 0.3).. (0.8, 0).. controls +(0, -0.3) and +(0, 0.3).. (0.4, -1)--
(-0.4, -1) .. controls +(0, 0.3) and +(0, -0.3).. (0, 0).. controls +(0, 0.3) and +(0, -0.3).. (-0.4, 1);
\draw[dashed] (0, 1) [partial ellipse=0:360:0.4 and 0.2];
\draw[dashed] (0, -1) [partial ellipse=0:360:0.4 and 0.2];
\draw (0.4, 1) .. controls +(0, -0.3) and +(0, 0.3).. (0.8, 0).. controls +(0, -0.3) and +(0, 0.3).. (0.4, -1);
\draw (-0.4, 1) .. controls +(0, -0.3) and +(0, 0.3).. (0, 0).. controls +(0, -0.3) and +(0, 0.3).. (-0.4, -1);
\draw [dashed, red] (0.4, 0) [partial ellipse=0:180:0.4 and 0.2];
\draw [red] (0.4, 0) [partial ellipse=180:250:0.4 and 0.2];
\draw [red] (0.4, 0) [partial ellipse=280:360:0.4 and 0.2];
\draw [blue, ->-=0.7] (0, 1.5)--(0, 1) .. controls +(0, -0.3) and +(0, 0.3).. (0.4, 0).. controls +(0, -0.3) and +(0, 0.3).. (0, -1)--(0, -1.5);
\begin{scope}[shift={(0.2, 0.5)}]
\draw [fill=white] (-0.2, -0.2) rectangle (0.2, 0.2);
\node at (0,0) {\tiny $f$};
\end{scope}
\node[black, above] at (0,1.5) {\tiny $X$};
\node[black, above] at (-1,1.5) {\tiny $\overline{J}$};
\node[black, above] at (1,1.5) {\tiny $J$};
\end{tikzpicture}}},
\end{align}
and
\begin{align}\label{def:alpha2}
\alpha_-\left( \vcenter{\hbox{\begin{tikzpicture}
\draw [dashed] (-1, -1) rectangle (1, 1);
\draw[-<-=0.85, -<-=0.25,blue] (0, -1) node[black, black, below]{\tiny$X$}--(0, 1)node[black, black, above]{\tiny $X$};
\draw [fill=white] (-0.3, -0.3) rectangle (0.3, 0.3);
\node at (0,0) {\tiny $f$};
\end{tikzpicture}}}\right)=
\frac{1}{d_J}
\vcenter{\hbox{\begin{tikzpicture}
\draw [dashed, fill=brown!20!white] (-1.5, -1.5) rectangle (1.5, 1.5);
\path [fill=white] (0,1.5) [partial ellipse=0:-180:1 and 1];
\draw [blue, ->-=0.15] (0,1.5) [partial ellipse=0:-180:1 and 1];
\path [fill=white] (0,-1.5) [partial ellipse=0:180:1 and 1];
\draw [blue, -<-=0.15] (0,-1.5) [partial ellipse=0:180:1 and 1];
\path [fill=white] (0.4, 1) .. controls +(0, -0.3) and +(0, 0.3).. (0.8, 0).. controls +(0, -0.3) and +(0, 0.3).. (0.4, -1)--
(-0.4, -1) .. controls +(0, 0.3) and +(0, -0.3).. (0, 0).. controls +(0, 0.3) and +(0, -0.3).. (-0.4, 1);
\draw[dashed] (0, 1) [partial ellipse=0:360:0.4 and 0.2];
\draw[dashed] (0, -1) [partial ellipse=0:360:0.4 and 0.2];
\draw (0.4, 1) .. controls +(0, -0.3) and +(0, 0.3).. (0.8, 0).. controls +(0, -0.3) and +(0, 0.3).. (0.4, -1);
\draw (-0.4, 1) .. controls +(0, -0.3) and +(0, 0.3).. (0, 0).. controls +(0, -0.3) and +(0, 0.3).. (-0.4, -1);
\draw [dashed, red] (0.4, 0) [partial ellipse=0:180:0.4 and 0.2];
\draw [blue, ->-=0.7] (0, 1.5)--(0, 1) .. controls +(0, -0.3) and +(0, 0.3).. (0.4, 0).. controls +(0, -0.3) and +(0, 0.3).. (0, -1)--(0, -1.5);
\begin{scope}[shift={(0.2, 0.5)}]
\draw [fill=white] (-0.2, -0.2) rectangle (0.2, 0.2);
\node at (0,0) {\tiny $f$};
\end{scope}
\draw [line width=2.5, white] (0.4, 0) [partial ellipse=200:320:0.4 and 0.2];
\draw [red] (0.4, 0) [partial ellipse=180:360:0.4 and 0.2];
\node[black, above] at (0,1.5) {\tiny $X$};
\node[black, above] at (-1,1.5) {\tiny $\overline{J}$};
\node[black, above] at (1,1.5) {\tiny $J$};
\end{tikzpicture}}}.
\end{align}
Note that the Drinfeld center $\mathcal{Z(D)}$ is $\mathcal{Z(C)}$ in the alterfold TQFT and $\mathcal{Z(C)}=\mathcal{C}\boxtimes \mathcal{C}^{op}$.
We see that $\alpha^\pm(X_j), j=1, \ldots, r$ generate $\mathcal{D}$.
By gluing the upside and downside of the right hand side of Equations \eqref{def:alpha11} and \eqref{def:alpha2}, and applying Move 1, we obtain that
\begin{align}\label{eq:alphatube}
\frac{1}{\mu}
\vcenter{\hbox{\scalebox{0.7}{
\begin{tikzpicture}[xscale=0.8, yscale=0.6]
\draw [line width=0.83cm] (0,0) [partial ellipse=-0.1:180.1:2 and 1.5];
\draw [white, line width=0.8cm] (0,0) [partial ellipse=-0.1:180.1:2 and 1.5];
\draw [blue] (0,0) [partial ellipse=-0.1:180.1:2 and 1.5];
\path [fill=brown!20!white](-0.65, -3) rectangle (0.65, 3);
\begin{scope}[shift={(0,3)}]
\path [fill=brown!20!white] (0,0) [partial ellipse=0:180:0.6 and 0.3];
\draw (0,0) [partial ellipse=0:360:0.6 and 0.3];
\end{scope}
\draw (-0.6, 3)--(-0.6, 0) (0.6, 3)--(0.6, 0); 
\draw (-0.6, -3)--(-0.6, 0) (0.6, -3)--(0.6, 0);
\draw [line width=0.83cm] (0,0) [partial ellipse=179:361:2 and 1.5];
\draw [white, line width=0.8cm] (0,0) [partial ellipse=177:363:2 and 1.5];
\draw [blue, ->-=0.5] (0,0) [partial ellipse=180:361:2 and 1.5] node[black, pos=0.5, below right] {\tiny $X$};
\begin{scope}[shift={(0,-3)}]
\path [fill=brown!20!white] (0,0) [partial ellipse=180:360:0.6 and 0.3];
\draw [dashed](0,0) [partial ellipse=0:180:0.6 and 0.3];
\draw (0,0) [partial ellipse=180:360:0.6 and 0.3];
\end{scope}
\begin{scope}[shift={(-2, 0)}]
\draw [red, dashed] (0,0) [partial ellipse=0:180:0.46 and 0.25];
\draw [red] (0,0) [partial ellipse=180:270:0.46 and 0.25] ;
\draw [red] (0,0) [partial ellipse=290:360:0.46 and 0.25];
\end{scope}
\end{tikzpicture}}}}
 =\vcenter{\hbox{\scalebox{0.7}{
\begin{tikzpicture}[xscale=1.2, yscale=0.6]
\path [fill=brown!20!white](-0.65, -3) rectangle (0.65, 3);
\begin{scope}[shift={(0,3)}]
\path [fill=brown!20!white] (0,0) [partial ellipse=0:180:0.6 and 0.3];
\draw (0,0) [partial ellipse=0:360:0.6 and 0.3];
\end{scope}
\draw (-0.6, 3)--(-0.6, 0) (0.6, 3)--(0.6, 0); 
\draw (-0.6, -3)--(-0.6, 0) (0.6, -3)--(0.6, 0);
\draw [dashed, blue](0,0) [partial ellipse=0:180:0.6 and 0.3];
\draw [blue, ->-=0.5](0,0) [partial ellipse=180:360:0.6 and 0.3];
\begin{scope}[shift={(0,-3)}]
\path [fill=brown!20!white] (0,0) [partial ellipse=180:360:0.6 and 0.3];
\draw [dashed](0,0) [partial ellipse=0:180:0.6 and 0.3];
\draw (0,0) [partial ellipse=180:360:0.6 and 0.3];
\end{scope}
\begin{scope}[shift={(0, -0.2)}]
\draw [fill=white] (-0.4, -0.3) rectangle (0.4, 0.3);
 \node at (0,0) {\tiny $\alpha_+(X)$};
\end{scope}
\end{tikzpicture}}}}, 
\quad
\frac{1}{\mu}
\vcenter{\hbox{\scalebox{0.7}{
\begin{tikzpicture}[xscale=0.8, yscale=0.6]
\draw [line width=0.83cm] (0,0) [partial ellipse=-0.1:180.1:2 and 1.5];
\draw [white, line width=0.8cm] (0,0) [partial ellipse=-0.1:180.1:2 and 1.5];
\draw [blue] (0,0) [partial ellipse=-0.1:180.1:2 and 1.5];
\path [fill=brown!20!white](-0.65, -3) rectangle (0.65, 3);
\begin{scope}[shift={(0,3)}]
\path [fill=brown!20!white] (0,0) [partial ellipse=0:180:0.6 and 0.3];
\draw (0,0) [partial ellipse=0:360:0.6 and 0.3];
\end{scope}
\draw (-0.6, 3)--(-0.6, 0) (0.6, 3)--(0.6, 0); 
\draw (-0.6, -3)--(-0.6, 0) (0.6, -3)--(0.6, 0);
\draw [line width=0.83cm] (0,0) [partial ellipse=179:361:2 and 1.5];
\draw [white, line width=0.8cm] (0,0) [partial ellipse=177:363:2 and 1.5];
\draw [blue, ->-=0.5] (0,0) [partial ellipse=192:361:2 and 1.5] node[black, pos=0.5, below] {\tiny $X$};
\begin{scope}[shift={(0,-3)}]
\path [fill=brown!20!white] (0,0) [partial ellipse=180:360:0.6 and 0.3];
\draw [dashed](0,0) [partial ellipse=0:180:0.6 and 0.3];
\draw (0,0) [partial ellipse=180:360:0.6 and 0.3];
\end{scope}
\begin{scope}[shift={(-2, 0)}]
\draw [red, dashed] (0,0) [partial ellipse=0:180:0.46 and 0.25];
\draw [red] (0,0) [partial ellipse=180:360:0.46 and 0.25] ;
\end{scope}
\end{tikzpicture}}}}
= \vcenter{\hbox{\scalebox{0.7}{
\begin{tikzpicture}[xscale=1.2, yscale=0.6]
\path [fill=brown!20!white](-0.65, -3) rectangle (0.65, 3);
\begin{scope}[shift={(0,3)}]
\path [fill=brown!20!white] (0,0) [partial ellipse=0:180:0.6 and 0.3];
\draw (0,0) [partial ellipse=0:360:0.6 and 0.3];
\end{scope}
\draw (-0.6, 3)--(-0.6, 0) (0.6, 3)--(0.6, 0); 
\draw (-0.6, -3)--(-0.6, 0) (0.6, -3)--(0.6, 0);
\draw [dashed, blue](0,0) [partial ellipse=0:180:0.6 and 0.3];
\draw [blue, ->-=0.5](0,0) [partial ellipse=180:360:0.6 and 0.3];
\begin{scope}[shift={(0,-3)}]
\path [fill=brown!20!white] (0,0) [partial ellipse=180:360:0.6 and 0.3];
\draw [dashed](0,0) [partial ellipse=0:180:0.6 and 0.3];
\draw (0,0) [partial ellipse=180:360:0.6 and 0.3];
\end{scope}
\begin{scope}[shift={(0, -0.2)}]
\draw [fill=white] (-0.4, -0.3) rectangle (0.4, 0.3);
 \node at (0,0) {\tiny $\alpha_-(X)$};
\end{scope}
\end{tikzpicture}}}}.
\end{align}

Now the entries $z_{jk}$ of the $Z$-matrix can also expressed as
\begin{equation}\label{eq:modularinv}
\begin{aligned}
z_{jk}
=& \alpha_{\mathcal{D}}(\xi_j\otimes \xi_k^{op})\\
=&\frac{1}{\mu^3}\vcenter{\hbox{\scalebox{0.7}{
\begin{tikzpicture}[scale=0.7]
\draw [line width=0.6cm, brown!20!white] (0,0) [partial ellipse=-0.1:180.1:2 and 1.5];
\begin{scope}[shift={(2.5, 0)}]
\draw [double distance=0.57cm] (0,0) [partial ellipse=-0.1:180.1:2 and 1.5];
\draw [blue] (0,0) [partial ellipse=0:180:2 and 1.5];
\end{scope} 
\begin{scope}[shift={(2.5, 0)}]
\draw [line width=0.6cm] (0,0) [partial ellipse=180:360:2 and 1.5];
\draw [white, line width=0.57cm] (0,0) [partial ellipse=178:362:2 and 1.5];
\draw [blue, ->-=0.5] (0,0) [partial ellipse=178:362:2 and 1.5] node[black, pos=0.7,below ] {\tiny $X_k$};
\end{scope}
\begin{scope}[shift={(-2.5, 0)}]
\draw [double distance=0.57cm] (0,0) [partial ellipse=-0.1:180.1:2 and 1.5];
\draw [blue, ->-=0.5] (0,0) [partial ellipse=0:180:2 and 1.5] node[black, pos=0.6, above] {\tiny $X_j$};
\end{scope} 
\begin{scope}[shift={(-2.5, 0)}]
\draw [line width=0.6cm] (0,0) [partial ellipse=180:360:2 and 1.5];
\draw [white, line width=0.57cm] (0,0) [partial ellipse=178:362:2 and 1.5];
\draw [blue] (0,0) [partial ellipse=178:362:2 and 1.5];
\end{scope}
\draw [line width=0.6cm, brown!20!white] (0,0) [partial ellipse=180:360:2 and 1.5];
\begin{scope}[shift={(4.5, 0)}]
\draw [red, dashed](0,0) [partial ellipse=0:180:0.37 and 0.25] ;
\draw [white, line width=4pt] (0,0) [partial ellipse=270:250:0.37 and 0.25];
\draw [red] (0,0) [partial ellipse=180:360:0.37 and 0.25];
\end{scope}
\begin{scope}[shift={(-4.5, 0)}]
\draw [red, dashed] (0,0) [partial ellipse=0:180:0.37 and 0.25];
\draw [red] (0,0) [partial ellipse=180:250:0.37 and 0.25];
\draw [red] (0,0) [partial ellipse=290:360:0.37 and 0.25];
\end{scope}
\end{tikzpicture}}}}\\
=& \dim\hom_{\mathcal{D}}(\alpha_+(X_j), \alpha_-(X_k)),
\end{aligned}
\end{equation}
which is the definition of the modular invariant mass matrix of  B\"{o}ckenhauer, Evans and Kawahigashi.

The topological interpretation of the modular invariance and the $\alpha$-induction can be generalized from the genus-one case to the higher-genus case which will be discussed in Section \ref{sec:alpha}.

\subsection{Minimal Idempotents in the Dual Fusion Algebra}
In the following, we shall present a topological construction of the minimal idempotents in the tube algebra of $\mathcal{D}$ in the alterfold TQFT. 
Note that the minimal idempotents in the tube algebra of $\mathcal{D}$ were characterized in \cite{BEK99}.
To begin the characterization, we recall Corollary 5.6 in \cite{LMWW23}, which states that
\begin{align*}
\vcenter{\hbox{\scalebox{0.7}{
}}}$ is either zero or a minimal central idempotent in $\displaystyle \text{End}\left(\vcenter{\hbox{\scalebox{0.7}{
%
}}}\right)$, where the brown circle is $I(\1_{\mathcal{D}})$ in the tube category associated to $\mathcal{D}$ (See \cite{LMWW23} for the details), $I$ is the induction functor from $\mathcal{D}\to \mathcal{Z(D)}$.

The following theorem indicates the size of the minimal central idempotents.

\begin{lemma}\label{lem:zrank}
 Suppose $\mathcal{C}$ is a modular fusion category.
 Then we have that
\begin{align*}
\frac{d_jd_k}{\mu^3}\vcenter{\hbox{\scalebox{0.7}{
%
}}}
=z_{jk}^2.
\end{align*}
\end{lemma}
\begin{proof}
The left hand side equals to
\begin{align*}
    \frac{d_jd_k}{\mu^4}\vcenter{\hbox{\scalebox{0.7}{
%
}}}.
\end{align*}
Now apply the hand slide to right most torus along the Kirby color in the Drinfeld center, we have that 
\begin{align*}
=&
\frac{d_jd_k}{\mu^4}\vcenter{\hbox{\scalebox{0.7}{
%
}}}
\end{align*}
Applying the Kirby color to the middle torus, we obtain that 
\begin{align*}
=& \frac{1}{\mu^4}\vcenter{\hbox{\scalebox{0.7}{
%
}}}\\
=& z_{jk}^2.
\end{align*}
This completes the proof of the lemma.
\end{proof}

\begin{theorem}\label{thm:rank}
 Suppose $\mathcal{C}$ is a modular fusion category.
Then the rank of $\displaystyle \frac{d_jd_k}{\mu^2}
\vcenter{\hbox{\scalebox{0.5}{
%
}}}$ is $z_{jk}$.
 \end{theorem}
\begin{proof}
Note that the rank of the left action of $\displaystyle \frac{d_jd_k}{\mu^2}
\vcenter{\hbox{\scalebox{0.5}{
%
}}}$ on the space $\displaystyle \text{End}\left(\vcenter{\hbox{\scalebox{0.7}{
%
}}}=z_{jk}^2.
\end{align*}
The last equality follows from Lemma \ref{lem:zrank}.
We then see that the rank of the projection is $z_{jk}$.
This completes the proof of the theorem.
\end{proof}

Recall that $\mathcal{M}$ is the $\mathcal{C}$-$\mathcal{D}$ bimodule category. 
Then $\mathbb{F}(\mathcal{M})=K_0(\mathcal{M})\otimes _{\mathbb{Z}}\mathbb{C}$ is a module of the fusion algebra $\mathbb{F}(\mathcal{C})=K_0(\mathcal{C})\otimes _{\mathbb{Z}}\mathbb{C}$. 
The action of $[X]$ on $[M]$ can be depicted as:
$$
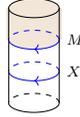
\vcenter{\hbox{\scalebox{0.7}{
%
}}}
$$
The modularity of $\mathcal{C}$ implies that the characters of $\mathbb{F}(\mathcal{C})$ are in one-to-one correspondence with the simple objects of $\mathcal{C}$.
For any simple object $X_j \in \mathcal{C}$, the corresponding character of $\mathbb{F}(\mathcal{M})$ is given by $\displaystyle \chi_j(X_k) = \frac{S_{jk}}{d_j}$ for any $X_k \in \Irr(\mathcal{C})$, where $S_{jk}$ are the entries of the $S$-matrix of $\mathcal{C}$. 
Note that 
\begin{align}\label{eq:charactermodular}
\frac{d_{j}}{\mu}\vcenter{\hbox{\scalebox{0.7}{
%
}}}
=S_{jk}
\end{align}
In the alterfold TQFT, the character $\chi_j$ can be depicted as
\begin{align*}
p_j=\frac{d_{j}}{\mu}\vcenter{\hbox{\scalebox{0.7}{
%
}}}
\end{align*}
which is a minimal central idempotent in $\text{End}(O)$, where $O$ is $I(\1_{\mathcal{C}})$ in the tube category associated to $\mathcal{C}$.
The multiplicity of $\chi_j$ in $\mathbb{F}(\mathcal{M})$ is the dimension of $p_j \mathbb{F}(\mathcal{M})$.

\begin{corollary}\label{cor:exponent}
With the notations above, for any $1 \le j \le r$, the multiplicity of the character $\chi_j$ in $\mathbb{F}(\mathcal{M})$ is the $Z$-matrix entry $z_{jj}$.
\end{corollary}
\begin{proof}
The trace of the projection $p_j$ restricted on $\mathbb{F}(\mathcal{M})$ is the trace of the following morphism
\begin{align*}
\frac{d_{j}}{\mu^2}\sum_{k}
\vcenter{\hbox{\scalebox{0.7}{
\begin{tikzpicture}[xscale=0.8, yscale=0.6]
\path [fill=brown!20!white](-0.65, -3) rectangle (0.65, -1);
\path [fill=brown!20!white](0, -1) [partial ellipse=0:360:0.6 and 0.3];
\draw [blue, -<-=0.5](0, -1) [partial ellipse=0:-180:0.6 and 0.3];
\draw [dashed, blue](0, -1) [partial ellipse=0:180:0.6 and 0.3];
\path [fill=brown!20!white](-0.65, 3) rectangle (0.65, 1);
\path [fill=brown!20!white](0, 1) [partial ellipse=0:360:0.6 and 0.3];
\draw [blue, ->-=0.5](0, 1) [partial ellipse=0:-180:0.6 and 0.3];
\draw [dashed, blue](0, 1) [partial ellipse=0:180:0.6 and 0.3];
\draw [red] (0, 0) [partial ellipse=-60:280:0.6 and 0.3];
\draw (0.6, -1) node[right]{\tiny$M_{k}$};
\draw (0.6, 1) node[right]{\tiny$M_{k}$};
\begin{scope}[shift={(0,1.5)}]
\end{scope}
\begin{scope}[shift={(0,0)}]
\end{scope}
\begin{scope}[shift={(0,3)}]
\path [fill=brown!20!white] (0,0) [partial ellipse=0:180:0.6 and 0.3];
\draw (0,0) [partial ellipse=0:360:0.6 and 0.3];
\end{scope}
\draw (-0.6, 3)--(-0.6, 0) (0.6, 3)--(0.6, 0); 
\draw (-0.6, -3)--(-0.6, 0) (0.6, -3)--(0.6, 0);
\begin{scope}[shift={(0,-3)}]
\path [fill=brown!20!white] (0,0) [partial ellipse=180:360:0.6 and 0.3];
\draw [dashed](0,0) [partial ellipse=0:180:0.6 and 0.3];
\draw (0,0) [partial ellipse=180:360:0.6 and 0.3];
\end{scope}
\draw[blue, -<-=0] (0, -0.3) [partial ellipse=0:170:0.2 and 0.45];
\draw[blue] (0, -0.3) [partial ellipse=190:360:0.2 and 0.45];
\end{tikzpicture}}}}
=\frac{d_{j}}{\mu^2}\vcenter{\hbox{\scalebox{0.7}{
\begin{tikzpicture}[xscale=0.8, yscale=0.6]
\draw [line width=0.8cm] (0,0) [partial ellipse=0:180:2 and 1.5];
\draw [white, line width=0.77cm] (0,0) [partial ellipse=-0.1:180.1:2 and 1.5];
\draw [red] (0,0) [partial ellipse=0:180:2 and 1.5];
\path [fill=brown!20!white](-0.65, -3) rectangle (0.65, 3);
\begin{scope}[shift={(0,1.5)}]
\end{scope}
\begin{scope}[shift={(0,0)}]
\end{scope}
\begin{scope}[shift={(0,3)}]
\path [fill=brown!20!white] (0,0) [partial ellipse=0:180:0.6 and 0.3];
\draw (0,0) [partial ellipse=0:360:0.6 and 0.3];
\end{scope}
\draw [line width=0.8cm] (0,0) [partial ellipse=180:360:2 and 1.5];
\draw (-0.6, 3)--(-0.6, 0) (0.6, 3)--(0.6, 0); 
\draw (-0.6, -3)--(-0.6, 0) (0.6, -3)--(0.6, 0);
\draw [white, line width=0.77cm] (0,0) [partial ellipse=178:362:2 and 1.5];
\draw [red] (0,0) [partial ellipse=280:362:2 and 1.5];
\draw [red] (0,0) [partial ellipse=178:270:2 and 1.5];
\begin{scope}[shift={(0,-3)}]
\path [fill=brown!20!white] (0,0) [partial ellipse=180:360:0.6 and 0.3];
\draw [dashed](0,0) [partial ellipse=0:180:0.6 and 0.3];
\draw (0,0) [partial ellipse=180:360:0.6 and 0.3];
\end{scope}
\draw[blue, -<-=0] (0, -1.5) [partial ellipse=0:170:0.2 and 0.45];
\draw[blue] (0, -1.5) [partial ellipse=190:360:0.2 and 0.45];
\end{tikzpicture}}}}
=
\frac{d_j^2}{\mu^2}
\vcenter{\hbox{\scalebox{0.7}{
\begin{tikzpicture}[xscale=0.8, yscale=0.6]
\draw [line width=0.8cm] (0,0) [partial ellipse=0:180:2 and 1.5];
\draw [white, line width=0.77cm] (0,0) [partial ellipse=-0.1:180.1:2 and 1.5];
\draw [red] (0,0) [partial ellipse=0:180:2 and 1.5];
\path [fill=brown!20!white](-0.65, -3) rectangle (0.65, 3);
\begin{scope}[shift={(0,1.5)}]
\end{scope}
\begin{scope}[shift={(0,0)}]
\end{scope}
\begin{scope}[shift={(0,3)}]
\path [fill=brown!20!white] (0,0) [partial ellipse=0:180:0.6 and 0.3];
\draw (0,0) [partial ellipse=0:360:0.6 and 0.3];
\end{scope}
\draw [line width=0.8cm] (0,0) [partial ellipse=180:360:2 and 1.5];
\draw (-0.6, 3)--(-0.6, 0) (0.6, 3)--(0.6, 0); 
\draw (-0.6, -3)--(-0.6, 0) (0.6, -3)--(0.6, 0);
\draw [white, line width=0.77cm] (0,0) [partial ellipse=178:362:2 and 1.5];
\draw [red] (0,0) [partial ellipse=178:362:2 and 1.5];
\begin{scope}[shift={(-1.95, 0.3)}]
\draw [blue, -<-=0.4, dashed](0,0) [partial ellipse=0:180:0.47 and 0.25] node [pos=0.4, above] {\tiny $X_j$};
\draw [white, line width=4pt] (0,0) [partial ellipse=290:270:0.47 and 0.25];
\draw [blue] (0,0) [partial ellipse=280:360:0.47 and 0.25];
\draw [blue] (0,0) [partial ellipse=180:260:0.47 and 0.25];
\end{scope}
\begin{scope}[shift={(-1.95, -0.3)}]
\draw [blue, dashed](0,0) [partial ellipse=0:180:0.47 and 0.25] ;
\draw [white, line width=4pt] (0,0) [partial ellipse=290:270:0.47 and 0.25];
\draw [blue, -<-=0.2] (0,0) [partial ellipse=180:360:0.47 and 0.25] node [pos=0.2, below] {\tiny $X_j$};
\end{scope}
\begin{scope}[shift={(0,-3)}]
\path [fill=brown!20!white] (0,0) [partial ellipse=180:360:0.6 and 0.3];
\draw [dashed](0,0) [partial ellipse=0:180:0.6 and 0.3];
\draw (0,0) [partial ellipse=180:360:0.6 and 0.3];
\end{scope}
\end{tikzpicture}}}},
\end{align*}
where the first equality follows from  the Move $1$ and Move $2$ and the second equality follows from the computations in the proof of Lemma \ref{lem:decomposition3}.
Now by Theorem \ref{thm:rank}, we see that the multiplicity of the character $\chi_j$ in $\mathbb{F}(\mathcal{M})$.
\end{proof}

\begin{remark}
Suppose $z_{jk}\le 1$ for all $j,k=1, \ldots, r$.
Here is an algebraic proof that the Grothendieck ring $K_0(\mathcal{D})$ is commutative.
We only need to prove $\text{End}(I(\1_{\mathcal{D}}))$ is commutative. 
It is semisimple, hence we only need to check it is diagonalizable, equivalently, we show all simple objects in $\mathcal{Z(D)}=\mathcal{Z(C)}$ has multiplicity at most $1$ in $I(\1_{\mathcal{D}})$.

Since $z_{jk}\le 1$ is equivalent to the condition that $\alpha_{+}(X_j)\otimes \alpha_{-}(X_k)$ has at most one copy of $\1_{\mathcal{D}}$ in $\mathcal{D}$.
In the case that $\mathcal{C}$ is a modular fusion category, the tensor $\alpha$-induction functor $\alpha_{-}\boxtimes \alpha_{+}$ factors through the center, and simple objects in $\mathcal{C}\boxtimes \mathcal{C}^{op}$ spans the simple objects in the center. 
This proves $I(\1_{\mathcal{D}})$ is multiplicity free.
\end{remark}

\begin{corollary}\label{cor:commdual}
Suppose $\mathcal{C}$ is a modular fusion category.
We have that the fusion algebra of $\mathcal{D}$ is commutative if and only if $z_{jk}\in \{0,1\}$ for all $j,k=1, \ldots, r$.
\end{corollary}
\begin{proof}
Suppose the fusion algebra of $\mathcal{D}$ is commutative.
Then $\displaystyle \frac{d_jd_k}{\mu^2}
\vcenter{\hbox{\scalebox{0.5}{
}}}$ is either zero or a minimal idempotent, i.e. the rank is at most $1$.
By Theorem \ref{thm:rank}, we have that $z_{jk}^2\leq 1$.
This implies that $z_{jk}\in \{0,1\}$ for all $j,k=1, \ldots, r$.

Suppose that $z_{jk}\in \{0,1\}$ for all $j,k=1, \ldots, r$.
Then the rank of $\displaystyle \frac{d_jd_k}{\mu^2}
\vcenter{\hbox{\scalebox{0.5}{
%
}}}$ is at most $1$.
Hence the fusion algebra of $\mathcal{D}$ is commutative.
\end{proof}

We see that it is important to study the idempotent $\displaystyle \frac{d_jd_k}{\mu^2}
\vcenter{\hbox{\scalebox{0.5}{
%
}}}$.
Reformulating the idempotent, we have that 
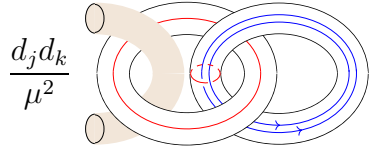
\begin{align*}
\frac{d_jd_k}{\mu^2}\vcenter{\hbox{\scalebox{0.7}{
%
}}}
\end{align*}
By cutting the idempotent horizontally, i.e. evaluating the Kirby color, we see that the idempotent is the sum of the following 
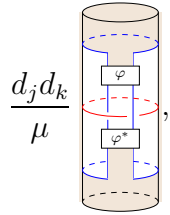
\begin{align*}
\frac{d_jd_k}{\mu}\vcenter{\hbox{\scalebox{0.7}{
%
}}},
\end{align*}
where $\{\varphi\}$ is a basis of $\hom_{\mathcal{C}}(J\overline{J}, X_jX_k)$.
However this is not idempotent in general.
It is an idempotent when the Frobenius algebra $Q$ is commutative, i.e.
   \begin{align}\label{eq:comm}
\vcenter{\hbox{\scalebox{0.6}{
%
}}}
\end{align}

\begin{proposition}
Suppose that $\mathcal{C}$ is a modular fusion category and the Frobenius algebra $Q=J\overline{J}$ is commutative.
Then $\displaystyle \frac{d_jd_k}{\mu}\vcenter{\hbox{\scalebox{0.7}{
%
}}}$ is either zero or a minimal idempotent in $\displaystyle \text{End}\left(\vcenter{\hbox{\scalebox{0.7}{
%
}}}\right)$ and any minimal idempotent is in this form.
\end{proposition}
\begin{proof}
It follows from the chiral relation directly.
\end{proof}

Note that the rank of $\displaystyle \frac{d_jd_k}{\mu^2}
\vcenter{\hbox{\scalebox{0.5}{
%
}}}$ is $z_{jk}$ by Theorem \ref{thm:rank}.
We should replace the basis in $\hom_{\mathcal{C}}(J\overline{J}, X_jX_k)$ by the basis in $\hom_{\mathcal{D}}(\alpha_+(X_j), \alpha_-(X_k))$ for general case.

\begin{theorem}\label{thm:minimalidempotent}
Suppose $\mathcal{C}$ is a modular fusion category and $\{\varphi\}$ is basis for $\hom_{\mathcal{D}}(\alpha_+(X_j), \alpha_-(X_k))$.
Then \begin{align}\label{eq:idem1}
\frac{d_jd_k}{\mu}\vcenter{\hbox{\scalebox{1}{
%
}}}
\end{align}
is zero or minimal idempotent in $\displaystyle \text{End}\left(\vcenter{\hbox{\scalebox{0.7}{
%
}}}\right)$ and any minimal idempotent is in this form.
\end{theorem}

\begin{proof}
By the handle slide, we have
\begin{align*}
\frac{d_jd_k}{\mu}\vcenter{\hbox{\scalebox{1}{
%
}}}.
\end{align*}
This implies that item \eqref{eq:idem1} is a minimal idempotent.
Note that 
\begin{align*}
 \frac{d_jd_k}{\mu^2}   \vcenter{\hbox{\scalebox{0.7}{
%
}}}.
\end{align*}
We see that the item \eqref{eq:idem1} is a sub-idempotent of $\displaystyle \frac{d_jd_k}{\mu^2}
\vcenter{\hbox{\scalebox{0.7}{
%
}}}$.
By counting the rank, we have that 
\begin{align*}
 \frac{d_jd_k}{\mu^2}
\vcenter{\hbox{\scalebox{0.7}{
%
}}}.
\end{align*}
This completes the proof of theorem.
\end{proof}

\section{Flat bi-invertible Connection}
Subfactors with index less than 4 have been classified in \cite{Ocn88, GHJ89, BN91, Izu91, Izu94,Kaw95}.
The principal graph of a subfactor $\mathcal{N}\subseteq \mathcal{M}$ is an induction-restriction graph of $\mathcal{N}$-$\mathcal{N}$ bimodules and $\mathcal{N}$-$\mathcal{M}$ bimodule with respect to the action of the generating $\mathcal{N}$-$\mathcal{M}$  bimodule $\mathcal{M}$.
The graph norm of the (finite) principal graph is the square root of the Jones index, which has to be an $A$-$D$-$E$ Dynkin diagram. 
It was surprising that only $A_n$-$D_{2n}$-$E_6$-$E_8$ appeared as the principal graph of a subfactor.
Ocneanu showed that a finite bipartite graph $\Gamma$ is a principal graph if and only if $\Gamma$ has a flat connection. Moreover, $D_{2n+1}$ and $E_7$ have connections but not flat. 
Solving flat connections became a major method to construct sub factors.
For $A$-$D$-$E$ Dynkin diagram, it is easy to solve the connection, but checking the flatness requires tedious computations. 
Xu proposed the $\alpha$-induction method to construct $E_6$ and $E_8$ subfactors from commutative Frobenius algebras $Q$ in the type $A$ modular fusion category, and constructed such commutative Frobenius algebras from conformal inclusions.
This construction from commutative Frobenius algebras avoided checking the flatness of the connection. 
Recently, Kawahigashi established the equivalence between the commutativity of a Frobenius algebra $Q$ in a unitary modular fusion category and the flatness of the bi-unitary connection induced by $Q$ (in Theorem 3.1 in\cite{Kaw24}).

In this section, we propose a non-trivial generalization of Kawahigashi's result to a Frobenius algebra in a modular fusion category over an arbitrary field. 
We prove the equivalence between the commutativity of a Frobenius algebra $Q$ in a modular fusion category and the flatness of the bi-invertible connection induced by $Q$, based our topological formalization in the previous sections.

We begin with the definition of bi-invertible connection.
Suppose that $\mathcal{G}, \mathcal{G}', \mathcal{H}, \mathcal{H}'$ are bipartite graphs such that the vertex set $V_0$ is common for even vertices of $\mathcal{G}$ and $\mathcal{H}$. Similarly, the vertex sets $V_1,V_2,V_3$ are common for odd vertices of $\mathcal{H}$ and $\mathcal{G}'$, even vertices of $\mathcal{G}'$ and $\mathcal{H}'$, and odd vertices of $\mathcal{G}$ and $\mathcal{H}'$, respectively. Graphically, the relation of the graphs and their shared vertices could be summarized in the following square:

$$
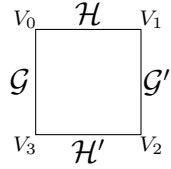
\begin{tikzpicture}[scale=0.7]
\draw (-1, -1) rectangle (1, 1);
\draw (-1.2, 1.2) node{\tiny $V_0$};
\draw (1.2, 1.2) node{\tiny  $V_1$};
\draw (-1.2, -1.2) node{\tiny  $V_3$};
\draw (1.2, -1.2) node{\tiny  $V_2$};
\draw (0, 1.3) node{$\mathcal{H}$};
\draw (0, -1.3) node{$\mathcal{H}'$};
\draw (-1.3, 0) node{$\mathcal{G}$};
\draw (1.3, 0) node{$\mathcal{G}'$};
\end{tikzpicture}$$

A choice of edges $T_1\in \mathcal{H}, T_2\in \mathcal{G}', T_3\in \mathcal{G}$ and $T_4 \in \mathcal{H}'$ such that $T_{2}\circ T_{1}=T_{4}\circ T_{3}$ is called a cell. 
A connection on the four bipartite graphs is an assignment of numbers to each cell. The function has an involution with respect to taking the mirror image. In the unitary context, the involution is the complex conjugate.
Equivalently, it is a 4-tensor defined on the vector spaces spanned by edges, together with an involution. 
These 4-tensors could be composed vertically and horizontally, in a way that one mirrors the square vertically and horizontally, and glue the square on the bottom and right respectively, to form a tensor of higher degree. To be precise:
$$\vcenter{\hbox{\begin{tikzpicture}
\draw (0, 0) rectangle (1, 1);
\draw (1, 0) rectangle (2, 1);
\draw (0, 1) node [left] {\tiny$v_0$};
\draw (1, 1) node [above]{\tiny$v_1$};
\draw (0, 0) node [left]{\tiny$v_3$};
\draw (1, 0) node [below]{\tiny$v_4$};
\draw (2, 1) node [right]{\tiny$v_5$};
\draw (2, 0) node [right]{\tiny$v_6$};
\end{tikzpicture}}}:=\sum_{\partial{T}=\{v_2, v_3\}}\vcenter{\hbox{\begin{tikzpicture}
\draw (0, 0) rectangle (1, 1);
\draw (0, 1) node[left]{\tiny$v_0$};
\draw (1, 1) node [right]{\tiny$v_1$};
\draw (0, 0) node[left]{\tiny$v_3$};
\draw (1, 0) node [right]{\tiny$v_4$};
\draw (1, 0.5) node[right]{\tiny$T$};
\end{tikzpicture}}}
\left(\vcenter{\hbox{\begin{tikzpicture}
\draw (0, 0) rectangle (1, 1);
\draw (0, 1) node[left]{\tiny$v_5$};
\draw (1, 1) node[right]{\tiny$v_1$};
\draw (0, 0) node[left]{\tiny$v_6$};
\draw (1, 0) node[right]{\tiny$v_4$};
\draw (1, 0.5) node[right]{\tiny$T$};
\end{tikzpicture}}}\right)^{*}$$
For details, see \cite{EK98, Kaw23, Kaw24}.
A connection is called bi-invertible if these numbers form a invertible matrix in two ways:
\begin{enumerate}
    \item the rows and columns paramatrized by triples 
\begin{align*}
\{(T_{2}, T_{4}, v_2)\in (\mathcal{G}', \mathcal{H}', V_2)|T_2\cap T_4=v_2\}.
\end{align*}
\item or by
\begin{align*}
\{(T_{2}, T_{1}, v_1)\in (\mathcal{G}', \mathcal{H}, V_2)|T_2\cap T_1=v_1\}.
\end{align*}
\end{enumerate}
and the inverse matrix is given by the involution.
A bi-invertible connection is called flat with respect to a vertex $\star\in V_{0}$ if the following equation is true:
$$\vcenter{\hbox{\begin{tikzpicture}
\draw (0, 0) node{$\star$};
\draw[->] (0, 0)--(1, 0);
\draw (1.5, 0) node{$\ldots$};
\draw (0, -2) node{$\star$};
\draw[->] (0, 0)--(0, -1);
\draw[->] (2, 0)--(2, -1);
\draw[->] (0, -2)--(1, -2);
\draw (1.5, -2) node{$\ldots$};
\draw (0, -1.3) node{$\vdots$};
\draw (2, -1.3) node{$\vdots$};
\draw (2, 0) node{$\star$};
\draw (2, -2) node{$\star$};
\draw (1, 0.2) node{$\sigma$};
\draw (1, -2.2) node{$\sigma'$};
\draw (-0.2, -1) node{$\rho$};
\draw (2.2, -1) node{$\rho'$};
\end{tikzpicture}}}=\delta_{\rho,\rho'}\delta_{\sigma,\sigma'}$$

Typical examples of flat bi-invertible(resp. bi-unitary) connections arise from spherical(resp. unitary) fusion categories. Indeed, given a spherical fusion category $\cC$, we can choose two tensor generators $\lambda$ and $\mu$, and let $V_{0}=V_{1}=V_{2}=V_{3}$ be the set of isomorphism classes of simple objects and let $\star$ be the tensor unit. Then the numbers that are assigned to cells are exactly the $6j$-symbols:
$$\left\{\begin{array}{ccc}
     v_0&v_1& \lambda \\
     v_2&v_3& \mu
\end{array}\right\}$$
with respect to a choice of basis in the corresponding morphism spaces. The involution is given by taking the dual basis.
The bi-invertible condition correspond to the orthogonality condition of a spherical fusion category, and the flatness condition implies the pentagon equation in an implicit way \cite{EK98}. The 4-tensors have the following pictorial presentation in the alterfold TQFT:

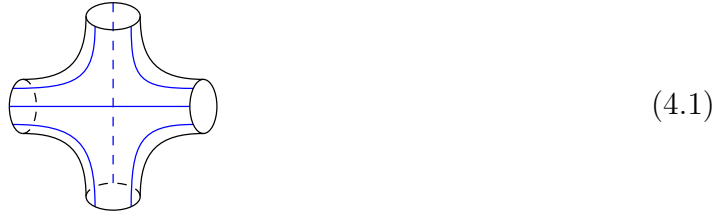
\begin{equation}\label{pic:4tensor}
\vcenter{\hbox{\scalebox{1.2}{\begin{tikzpicture}
\draw (-0.3, 1)..controls (-0.3, 0.5) and (-0.5, 0.3)..(-1, 0.3);
\draw (0.3, 1)..controls (0.3, 0.5) and (0.5, 0.3)..(1, 0.3);
\draw (0.3, -1)..controls (0.3, -0.5) and (0.5, -0.3)..(1, -0.3);
\draw (-0.3, -1)..controls (-0.3, -0.5) and (-0.5, -0.3)..(-1, -0.3);
\draw[blue] (-1.115, 0.2)..controls (-0.4, 0.2) and (-0.2, 0.3) ..(-0.2, 0.885);
\draw[blue] (0.2, 0.885)..controls (0.2, 0.3) and (0.3, 0.2)..(0.885, 0.2);
\draw[blue] (0.2, -1.115)..controls (0.2, -0.4) and (0.3, -0.2) ..(0.885, -0.2);
\draw[blue] (-1.115, -0.2)..controls (-0.4, -0.2) and (-0.2, -0.4)..(-0.2, -1.115);
\draw[blue] (-1.15, 0)--(0.85, 0);
\draw[blue, dashed] (0, 1.15) to (0, -0.85);
\draw (-1, 0) [partial ellipse=90:270:0.15 and 0.3];
\draw[dashed] (-1, 0) [partial ellipse=90:-90:0.15 and 0.3];
\draw (0, 1) [partial ellipse=0:360:0.3 and 0.15];
\draw[dashed] (0, -1) [partial ellipse=0:180:0.3 and 0.15];
\draw (0, -1) [partial ellipse=180:360:0.3 and 0.15];
\draw (1, 0) [partial ellipse=0:360:0.15 and 0.3];
\end{tikzpicture}}}}
\end{equation}
where the inside is colored by $B$ and the outside is colored by $A$, the blue strand is colored by the object in the fusion category $\mathcal{C}$
and the circles on the corner should be considered as coupons filled with the vector in the chosen basis. The involution is defined to be reversing all arrows for the objects and replacing the vector by its dual.

As for the bi-invertible connection, the diagram \eqref{pic:4tensor} can also be considered as a 4-tensor, it could be compose in an obvious way that the corners are glued together, we remark that the basis is normalized in a way such that when four of such diagrams composed together, the closed circle in the middle is summed up to the $\Omega$-color in the appropriate category, hence Move $1$ could be applied:

$$\vcenter
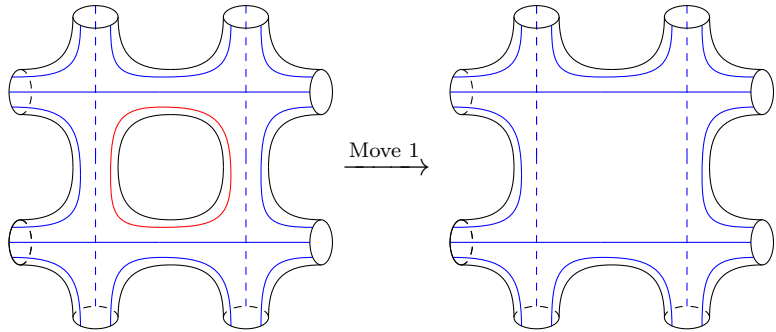
{\hbox{
\begin{tikzpicture}
\draw (-0.3, 1)..controls (-0.3, 0.5) and (-0.5, 0.3)..(-1, 0.3);
\draw (0.3, 1)..controls (0.3, 0.5) and (0.5, 0.3)..(1, 0.3);
\draw (0.3, -1)..controls (0.3, -0.5) and (0.5, -0.3)..(1, -0.3);
\draw (-0.3, -1)..controls (-0.3, -0.5) and (-0.5, -0.3)..(-1, -0.3);
\draw[blue] (-1.115, 0.2)..controls (-0.4, 0.2) and (-0.2, 0.3) ..(-0.2, 0.885);
\draw[blue] (0.2, 0.885)..controls (0.2, 0.3) and (0.3, 0.2)..(0.885, 0.2);
\draw[red] (0.2, -1.115)..controls (0.2, -0.4) and (0.3, -0.2) ..(0.885, -0.2);
\draw[blue] (-1.115, -0.2)..controls (-0.4, -0.2) and (-0.2, -0.4)..(-0.2, -1.115);
\draw[blue] (-1.15, 0)--(0.85, 0);
\draw[blue, dashed] (0, 1.15) to (0, -0.85);
\draw[blue, dashed] (0, 1.15) to (0, -0.85);
\draw (-1, 0) [partial ellipse=90:270:0.15 and 0.3];
\draw[dashed] (-1, 0) [partial ellipse=90:-90:0.15 and 0.3];
\draw (0, 1) [partial ellipse=0:360:0.3 and 0.15];
\begin{scope}[shift={(2, 0)}]
\draw (-0.3, 1)..controls (-0.3, 0.5) and (-0.5, 0.3)..(-1, 0.3);
\draw (0.3, 1)..controls (0.3, 0.5) and (0.5, 0.3)..(1, 0.3);
\draw (0.3, -1)..controls (0.3, -0.5) and (0.5, -0.3)..(1, -0.3);
\draw (-0.3, -1)..controls (-0.3, -0.5) and (-0.5, -0.3)..(-1, -0.3);
\draw[blue] (-1.115, 0.2)..controls (-0.4, 0.2) and (-0.2, 0.3) ..(-0.2, 0.885);
\draw[blue] (0.2, 0.885)..controls (0.2, 0.3) and (0.3, 0.2)..(0.885, 0.2);
\draw[blue] (0.2, -1.115)..controls (0.2, -0.4) and (0.3, -0.2) ..(0.885, -0.2);
\draw[red] (-1.115, -0.2)..controls (-0.4, -0.2) and (-0.2, -0.4)..(-0.2, -1.115);
\draw[blue] (-1.15, 0)--(0.85, 0);
\draw[blue, dashed] (0, 1.15) to (0, -0.85);
\draw (0, 1) [partial ellipse=0:360:0.3 and 0.15];
\draw (1, 0) [partial ellipse=0:360:0.15 and 0.3];
\end{scope}
\begin{scope}[shift={(2, -2)}]
\draw (-0.3, 1)..controls (-0.3, 0.5) and (-0.5, 0.3)..(-1, 0.3);
\draw (0.3, 1)..controls (0.3, 0.5) and (0.5, 0.3)..(1, 0.3);
\draw (0.3, -1)..controls (0.3, -0.5) and (0.5, -0.3)..(1, -0.3);
\draw (-0.3, -1)..controls (-0.3, -0.5) and (-0.5, -0.3)..(-1, -0.3);
\draw[red] (-1.115, 0.2)..controls (-0.4, 0.2) and (-0.2, 0.3) ..(-0.2, 0.885);
\draw[blue] (0.2, 0.885)..controls (0.2, 0.3) and (0.3, 0.2)..(0.885, 0.2);
\draw[blue] (0.2, -1.115)..controls (0.2, -0.4) and (0.3, -0.2) ..(0.885, -0.2);
\draw[blue] (-1.115, -0.2)..controls (-0.4, -0.2) and (-0.2, -0.4)..(-0.2, -1.115);
\draw[blue] (-1.15, 0)--(0.85, 0);
\draw[blue, dashed] (0, 1.15) to (0, -0.85);
\draw (1, 0) [partial ellipse=0:360:0.15 and 0.3];
\draw[dashed] (0, -1) [partial ellipse=0:180:0.3 and 0.15];
\draw (0, -1) [partial ellipse=180:360:0.3 and 0.15];
\end{scope}
\begin{scope}[shift={(0, -2)}]
\draw (-0.3, 1)..controls (-0.3, 0.5) and (-0.5, 0.3)..(-1, 0.3);
\draw (0.3, 1)..controls (0.3, 0.5) and (0.5, 0.3)..(1, 0.3);
\draw (0.3, -1)..controls (0.3, -0.5) and (0.5, -0.3)..(1, -0.3);
\draw (-0.3, -1)..controls (-0.3, -0.5) and (-0.5, -0.3)..(-1, -0.3);
\draw[blue] (-1.115, 0.2)..controls (-0.4, 0.2) and (-0.2, 0.3) ..(-0.2, 0.885);
\draw[red] (0.2, 0.885)..controls (0.2, 0.3) and (0.3, 0.2)..(0.885, 0.2);
\draw[blue] (0.2, -1.115)..controls (0.2, -0.4) and (0.3, -0.2) ..(0.885, -0.2);
\draw[blue] (-1.115, -0.2)..controls (-0.4, -0.2) and (-0.2, -0.4)..(-0.2, -1.115);
\draw[blue] (-1.15, 0)--(0.85, 0);
\draw[blue, dashed] (0, 1.15) to (0, -0.85);
\draw (-1, 0) [partial ellipse=90:270:0.15 and 0.3];
\draw[dashed] (-1, 0) [partial ellipse=90:-90:0.15 and 0.3];
\draw (-1, 0) [partial ellipse=90:270:0.15 and 0.3];
\draw[dashed] (-1, 0) [partial ellipse=90:-90:0.15 and 0.3];
\draw[dashed] (0, -1) [partial ellipse=0:180:0.3 and 0.15];
\draw (0, -1) [partial ellipse=180:360:0.3 and 0.15];
\end{scope}
\end{tikzpicture}}}
\xrightarrow{\text{Move } 1}
\vcenter{\hbox{
\begin{tikzpicture}
\draw (-0.3, 1)..controls (-0.3, 0.5) and (-0.5, 0.3)..(-1, 0.3);
\draw (0.3, 1)..controls (0.3, 0.5) and (0.5, 0.3)..(1, 0.3);
\draw (0.3, -1)..controls (0.3, -0.5) and (0.5, -0.3)..(1, -0.3);
\draw (-0.3, -1)..controls (-0.3, -0.5) and (-0.5, -0.3)..(-1, -0.3);
\draw[blue] (-1.115, 0.2)..controls (-0.4, 0.2) and (-0.2, 0.3) ..(-0.2, 0.885);
\draw[blue] (0.2, 0.885)..controls (0.2, 0.3) and (0.3, 0.2)..(0.885, 0.2);
\draw[red] (0.2, -1.115)..controls (0.2, -0.4) and (0.3, -0.2) ..(0.885, -0.2);
\draw[blue] (-1.115, -0.2)..controls (-0.4, -0.2) and (-0.2, -0.4)..(-0.2, -1.115);
\draw[blue] (-1.15, 0)--(0.85, 0);
\draw[blue, dashed] (0, 1.15) to (0, -0.85);
\draw[blue, dashed] (0, 1.15) to (0, -0.85);
\draw (-1, 0) [partial ellipse=90:270:0.15 and 0.3];
\draw[dashed] (-1, 0) [partial ellipse=90:-90:0.15 and 0.3];
\draw (0, 1) [partial ellipse=0:360:0.3 and 0.15];
\begin{scope}[shift={(2, 0)}]
\draw (-0.3, 1)..controls (-0.3, 0.5) and (-0.5, 0.3)..(-1, 0.3);
\draw (0.3, 1)..controls (0.3, 0.5) and (0.5, 0.3)..(1, 0.3);
\draw (0.3, -1)..controls (0.3, -0.5) and (0.5, -0.3)..(1, -0.3);
\draw (-0.3, -1)..controls (-0.3, -0.5) and (-0.5, -0.3)..(-1, -0.3);
\draw[blue] (-1.115, 0.2)..controls (-0.4, 0.2) and (-0.2, 0.3) ..(-0.2, 0.885);
\draw[blue] (0.2, 0.885)..controls (0.2, 0.3) and (0.3, 0.2)..(0.885, 0.2);
\draw[blue] (0.2, -1.115)..controls (0.2, -0.4) and (0.3, -0.2) ..(0.885, -0.2);
\draw[red] (-1.115, -0.2)..controls (-0.4, -0.2) and (-0.2, -0.4)..(-0.2, -1.115);
\draw[blue] (-1.15, 0)--(0.85, 0);
\draw[blue, dashed] (0, 1.15) to (0, -0.85);
\draw (0, 1) [partial ellipse=0:360:0.3 and 0.15];
\draw (1, 0) [partial ellipse=0:360:0.15 and 0.3];
\end{scope}
\begin{scope}[shift={(2, -2)}]
\draw (-0.3, 1)..controls (-0.3, 0.5) and (-0.5, 0.3)..(-1, 0.3);
\draw (0.3, 1)..controls (0.3, 0.5) and (0.5, 0.3)..(1, 0.3);
\draw (0.3, -1)..controls (0.3, -0.5) and (0.5, -0.3)..(1, -0.3);
\draw (-0.3, -1)..controls (-0.3, -0.5) and (-0.5, -0.3)..(-1, -0.3);
\draw[red] (-1.115, 0.2)..controls (-0.4, 0.2) and (-0.2, 0.3) ..(-0.2, 0.885);
\draw[blue] (0.2, 0.885)..controls (0.2, 0.3) and (0.3, 0.2)..(0.885, 0.2);
\draw[blue] (0.2, -1.115)..controls (0.2, -0.4) and (0.3, -0.2) ..(0.885, -0.2);
\draw[blue] (-1.115, -0.2)..controls (-0.4, -0.2) and (-0.2, -0.4)..(-0.2, -1.115);
\draw[blue] (-1.15, 0)--(0.85, 0);
\draw[blue, dashed] (0, 1.15) to (0, -0.85);
\draw (1, 0) [partial ellipse=0:360:0.15 and 0.3];
\draw[dashed] (0, -1) [partial ellipse=0:180:0.3 and 0.15];
\draw (0, -1) [partial ellipse=180:360:0.3 and 0.15];
\end{scope}
\begin{scope}[shift={(0, -2)}]
\draw (-0.3, 1)..controls (-0.3, 0.5) and (-0.5, 0.3)..(-1, 0.3);
\draw (0.3, 1)..controls (0.3, 0.5) and (0.5, 0.3)..(1, 0.3);
\draw (0.3, -1)..controls (0.3, -0.5) and (0.5, -0.3)..(1, -0.3);
\draw (-0.3, -1)..controls (-0.3, -0.5) and (-0.5, -0.3)..(-1, -0.3);
\draw[blue] (-1.115, 0.2)..controls (-0.4, 0.2) and (-0.2, 0.3) ..(-0.2, 0.885);
\draw[red] (0.2, 0.885)..controls (0.2, 0.3) and (0.3, 0.2)..(0.885, 0.2);
\draw[blue] (0.2, -1.115)..controls (0.2, -0.4) and (0.3, -0.2) ..(0.885, -0.2);
\draw[blue] (-1.115, -0.2)..controls (-0.4, -0.2) and (-0.2, -0.4)..(-0.2, -1.115);
\draw[blue] (-1.15, 0)--(0.85, 0);
\draw[blue, dashed] (0, 1.15) to (0, -0.85);
\draw (-1, 0) [partial ellipse=90:270:0.15 and 0.3];
\draw[dashed] (-1, 0) [partial ellipse=90:-90:0.15 and 0.3];
\draw (-1, 0) [partial ellipse=90:270:0.15 and 0.3];
\draw[dashed] (-1, 0) [partial ellipse=90:-90:0.15 and 0.3];
\draw[dashed] (0, -1) [partial ellipse=0:180:0.3 and 0.15];
\draw (0, -1) [partial ellipse=180:360:0.3 and 0.15];
\end{scope}
\fill[white] (0.1, -0.1) rectangle (1.9, -1.9);
\end{tikzpicture}}}
$$
To check the flatness of the bi-invertible connections associated to a spherical fusion category with respect to the tensor unit, we need to verify the alterfold above with the following decoration been evaluated to $1$:
\begin{enumerate}
\item The four corners strands are transparent, i.e., colored by tensor unit.
\item  The opposite strands are labelled by dual objects and the opposite coupons are labelled by a pair of dual morphism.
\end{enumerate}
The condition (1) makes the tensor diagram disconnected, the vertical strands in the middle together with the horizontal strands on the boundary form one connected component and the rest of the tensor diagram form the other. Then the condition (2) ensures that each connected components evaluate to $1$.

Given a spherical Morita context generated by a braided spherical fusion category $\mathcal{C}$ and $\mathcal{D}=J\overline{J}$-bimodules. Following \cite{Kaw23}, The $\alpha$-induced bi-invertible connection in the alterfold is given by the following graph:
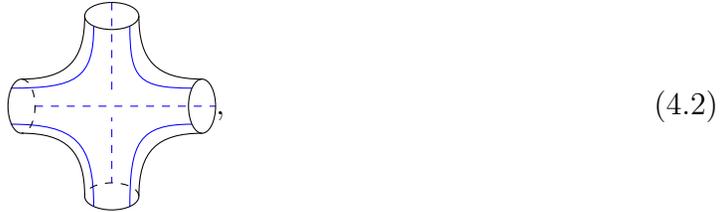
\begin{equation}\label{pic:alpha4tensor}
\vcenter{\hbox{\scalebox{1}{\begin{tikzpicture}[scale=1.2]
\draw (-0.3, 1)..controls (-0.3, 0.5) and (-0.5, 0.3)..(-1, 0.3) node [above] {};
\draw (0.3, 1)..controls (0.3, 0.5) and (0.5, 0.3)..(1, 0.3) node [above] {};
\draw (0.3, -1)..controls (0.3, -0.5) and (0.5, -0.3)..(1, -0.3) node [below] {};
\draw (-0.3, -1)..controls (-0.3, -0.5) and (-0.5, -0.3)..(-1, -0.3) node [below] {};
\draw[blue] (-1.115, 0.2)..controls (-0.4, 0.2) and (-0.2, 0.3) ..(-0.2, 0.885);
\draw[blue] (0.2, 0.885)..controls (0.2, 0.3) and (0.3, 0.2)..(0.885, 0.2);
\draw[blue] (0.2, -1.115)..controls (0.2, -0.4) and (0.3, -0.2) ..(0.885, -0.2);
\draw[blue] (-1.115, -0.2)..controls (-0.4, -0.2) and (-0.2, -0.4)..(-0.2, -1.115);
\draw[blue, dashed] (0, 1.15) to (0, -0.85);
\draw [line width=0.3cm, white] (-0.2, 0)--(0.2, 0);
\draw[blue, dashed] (-0.85, 0)--(1.15, 0);
\draw (-1, 0) [partial ellipse=90:270:0.15 and 0.3];
\draw[dashed] (-1, 0) [partial ellipse=90:-90:0.15 and 0.3];
\draw (0, 1) [partial ellipse=0:360:0.3 and 0.15];
\draw[dashed] (0, -1) [partial ellipse=0:180:0.3 and 0.15];
\draw (0, -1) [partial ellipse=180:360:0.3 and 0.15];
\draw (1, 0) [partial ellipse=0:360:0.15 and 0.3];
\end{tikzpicture}}}},
\end{equation}
where the inside is colored by $B$ and the outside is colored by $A$, and the strand on the corner(resp. on the back) are labelled by object in the $\mathcal{C}J$(resp. objects in $\mathcal{C}$)
Applying Move $2$ to connect adjacent corners then apply Move $1$ to cancel the hole in the middle. It is clear that \eqref{pic:alpha4tensor} is a bi-unitary connection.

\begin{remark}
Given a unitary spherical fusion category, the connection defined via \eqref{pic:alpha4tensor} coincides with the $\alpha$-induced bi-unitary connection introduced by Kawahigashi in \cite{Kaw23,Kaw24} (See Figure 28 in \cite{Kaw23}) if the dual basis coming from orthonormal basis and its dagger. Such identifications has been carefully explained in \cite{LMWW23b}.
\end{remark}

The Proposition \ref{prop:kaw1} and Proposition \ref{prop:comflat} with their proofs follows verbatim that in \cite{Kaw23}. We put it here for completeness.
\begin{proposition}\label{prop:kaw1}
The $\alpha$-induced bi-invertible connection is flat if and only if 
$$
\vcenter{\hbox{\begin{tikzpicture}[xscale=0.8, yscale=0.6]
\draw[brown] (0, 0)--(0, 1)--(1, 2);
\draw[blue, ->-=0.5] (1, 2)--(1, -2);
\draw (1.3, 0) node{\tiny $Y$};
\draw (-1.3, 0) node{\tiny $X$};
\draw[brown] (1, -2)..controls (0, -1) and (-1, -1)..(0, 0);
\draw[fill=white] (0.6, 1.5) rectangle (1.4, 2.3);
\draw[fill=white] (0.6, -1.5) rectangle (1.4, -2.3);
\draw (1, 1.9) node{\tiny$\varphi$};
\draw (1, -1.9) node{\tiny$\varphi^{'*}$};
\fill[white] (0, -1.2) circle(0.2);
\begin{scope}[xscale=-1]
\draw[brown] (0, 1)--(1, 2);
\draw[blue, ->-=0.5] (1, 2)--(1, -2);
\draw[brown] (1, -2)..controls (0, -1) and (-1, -1)..(0, 0);
\draw[fill=white] (0.6, 1.5) rectangle (1.4, 2.3);
\draw[fill=white] (0.6, -1.5) rectangle (1.4, -2.3);
\draw (1, 1.9) node{\tiny$\psi$};
\draw (1, -1.9) node{\tiny$\psi^{'*}$};
\end{scope}
\end{tikzpicture}}}=\frac{1}{d_J}\delta_{\varphi, \varphi'}\delta_{\psi, \psi'}$$
for any dual base $\varphi, \varphi^*$, $\psi, \psi^*$.
\end{proposition}

\begin{proof}

The tensor diagram is connected, shrinking all morphisms and strands on each edges, the flatness is equivalent to the following equality:
$$
\vcenter{\hbox{\begin{tikzpicture}
\draw[rounded corners, blue, ->-=0.3] (-2, 0) rectangle (2, 1.5);
\draw(0, 1.5) node[above]{\tiny$J$};
\draw[rounded corners, blue] (-1.5, 0)--(-1.5, -1)--(0.5, -1)--(0.5, 0);
\fill[white] (-0.6, -1.1) rectangle (-0.4, -0.9);
\draw[rounded corners, blue] (-.5, 0)--(-.5, -1.5)--(1.5, -1.5)--(1.5, 0);
\begin{scope}[shift={(-1.5, 0)}]
\draw[fill=white] (-0.4, -0.4) rectangle (0.4, 0.4);
\draw (0, 0) node{\tiny$\varphi$};
\end{scope}
\begin{scope}[shift={(-.5, 0)}]
\draw[fill=white] (-0.4, -0.4) rectangle (0.4, 0.4);
\draw (0, 0) node{\tiny$\psi$};
\end{scope}
\begin{scope}[shift={(.5, 0)}]
\draw[fill=white] (-0.4, -0.4) rectangle (0.4, 0.4);
\draw (0, 0) node{\tiny$\varphi^{'*}$};
\end{scope}
\begin{scope}[shift={(1.5, 0)}]
\draw[fill=white] (-0.4, -0.4) rectangle (0.4, 0.4);
\draw (0, 0) node{\tiny$\psi^{'*}$};
\end{scope}
\end{tikzpicture}}}=\frac{1}{d_J}\delta_{\varphi, \varphi'}\delta_{\psi, \psi'}.
$$
Here, it is understood that the horizontal edges connecting the pairs ($\varphi$, $\psi$), ($\psi$, $\varphi^*$), and ($\varphi^*$, $\psi^*$) are labeled by $J$. 
In this case, the tensor diagram is connected, shrinking all morphisms and strands on each edges, the flatness is equivalent to the following equality:
$$
\vcenter{\hbox{\begin{tikzpicture}[xscale=0.8, yscale=0.6]
\draw[brown] (0, 0)--(0, 1)--(1, 2);
\draw[blue, ->-=0.5] (1, 2)--(1, -2);
\draw[brown] (1, -2)..controls (0, -1) and (-1, -1)..(0, 0);
\draw[fill=white] (0.6, 1.5) rectangle (1.4, 2.3);
\draw[fill=white] (0.6, -1.5) rectangle (1.4, -2.3);
\draw (1, 1.9) node{\tiny$\varphi$};
\draw (1, -1.9) node{\tiny$\varphi^{'*}$};
\fill[white] (0, -1.2) circle(0.2);
\begin{scope}[xscale=-1]
\draw[brown] (0, 1)--(1, 2);
\draw[blue, ->-=0.5] (1, 2)--(1, -2);
\draw[brown] (1, -2)..controls (0, -1) and (-1, -1)..(0, 0);
\draw[fill=white] (0.6, 1.5) rectangle (1.4, 2.3);
\draw[fill=white] (0.6, -1.5) rectangle (1.4, -2.3);
\draw (1, 1.9) node{\tiny$\psi$};
\draw (1, -1.9) node{\tiny$\psi^{'*}$};
\end{scope}
\end{tikzpicture}}}=\frac{1}{d_J}\delta_{\varphi, \varphi'}\delta_{\psi, \psi'}.$$
This completes the proof of the proposition.
\end{proof}

\begin{proposition}\label{prop:comflat}
Suppose that the Frobenius algebra $Q=J\overline{J}$ is commutative.
Then the $\alpha$-induced bi-unitary connection is flat.
\end{proposition}

\begin{proof}
The Frobenius algebra is commutative immediately implies the following:
$$
\vcenter{\hbox{\begin{tikzpicture}[xscale=0.8, yscale=0.6]
\draw[brown] (0, 0)--(0, 1)--(1, 2);
\draw[blue, ->-=0.5] (1, 2)--(1, -2);
\draw[brown] (1, -2)..controls (0, -1) and (-1, -1)..(0, 0);
\draw[fill=white] (0.6, 1.5) rectangle (1.4, 2.3);
\draw[fill=white] (0.6, -1.5) rectangle (1.4, -2.3);
\draw (1, 1.9) node{\tiny$\varphi$};
\draw (1, -1.9) node{\tiny$\varphi^{*}$};
\fill[white] (0, -1.2) circle(0.2);
\begin{scope}[xscale=-1]
\draw[brown] (0, 1)--(1, 2);
\draw[blue, ->-=0.5] (1, 2)--(1, -2);
\draw[brown] (1, -2)..controls (0, -1) and (-1, -1)..(0, 0);
\draw[fill=white] (0.6, 1.5) rectangle (1.4, 2.3);
\draw[fill=white] (0.6, -1.5) rectangle (1.4, -2.3);
\draw (1, 1.9) node{\tiny$\psi$};
\draw (1, -1.9) node{\tiny$\psi^{'*}$};
\end{scope}
\end{tikzpicture}}}=\vcenter{\hbox{\begin{tikzpicture}[scale=0.8]
\draw[brown] (0, 0)--(0, 1)--(1, 2);
\draw[blue, ->-=0.5] (1, 2)--(1, -1);
\draw[brown] (1, -1)--(0, 0);
\draw[fill=white] (0.6, 1.5) rectangle (1.4, 2.3);
\draw[fill=white] (0.6, -1.3) rectangle (1.4, -0.5);
\draw (1, 1.9) node{\tiny$\varphi$};
\draw (1, -.9) node{\tiny$\varphi^{*}$};
\begin{scope}[xscale=-1]
\draw[brown] (0, 1)--(1, 2);
\draw[blue, ->-=0.5] (1, 2)--(1, -1);
\draw[brown] (1, -1)--(0, 0);
\draw[fill=white] (0.6, 1.5) rectangle (1.4, 2.3);
\draw[fill=white] (0.6, -1.3) rectangle (1.4, -0.5);
\draw (1, 1.9) node{\tiny$\psi$};
\draw (1, -.9) node{\tiny$\psi^{'*}$};
\end{scope}
\end{tikzpicture}}}
=
\vcenter{\hbox{\begin{tikzpicture}[xscale=0.8]
\begin{scope}[shift={(0.5, 0)}]
\draw[brown] (0, 0)--(1, 1);
\draw[blue, ->-=0.5] (1, 1)--(1, -1);
\draw[brown] (1, -1)--(0, 0);
\draw[fill=white] (0.6, .5) rectangle (1.4, 1.3);
\draw[fill=white] (0.6, -1.3) rectangle (1.4, -0.5);
\draw (1, .9) node{\tiny$\varphi$};
\draw (1, -.9) node{\tiny$\varphi^{'*}$};
\end{scope}
\begin{scope}[xscale=-1, shift={(0.5, 0)}]
\draw[brown] (0, 0)--(1, 1);
\draw[blue, ->-=0.5] (1, 1)--(1, -1);
\draw[brown] (1, -1)--(0, 0);
\draw[fill=white] (0.6, .5) rectangle (1.4, 1.3);
\draw[fill=white] (0.6, -1.3) rectangle (1.4, -0.5);
\draw (1, .9) node{\tiny$\psi$};
\draw (1, -.9) node{\tiny$\psi^{'*}$};
\end{scope}
\draw[brown] (-0.5, 0)--(0.5, 0);
\end{tikzpicture}}}
=\frac{1}{d_J}
\vcenter{\hbox{\begin{tikzpicture}[xscale=0.8]
\begin{scope}[shift={(0.5, 0)}]
\draw[brown] (0, 0)--(1, 1);
\draw[blue, ->-=0.5] (1, 1)--(1, -1);
\draw[brown] (1, -1)--(0, 0);
\draw[fill=white] (0.6, .5) rectangle (1.4, 1.3);
\draw[fill=white] (0.6, -1.3) rectangle (1.4, -0.5);
\draw (1, .9) node{\tiny$\varphi$};
\draw (1, -.9) node{\tiny$\varphi^{'*}$};
\end{scope}
\begin{scope}[xscale=-1, shift={(0.5, 0)}]
\draw[brown] (0, 0)--(1, 1);
\draw[blue, ->-=0.5] (1, 1)--(1, -1);
\draw[brown] (1, -1)--(0, 0);
\draw[fill=white] (0.6, .5) rectangle (1.4, 1.3);
\draw[fill=white] (0.6, -1.3) rectangle (1.4, -0.5);
\draw (1, .9) node{\tiny$\psi$};
\draw (1, -.9) node{\tiny$\psi^{'*}$};
\end{scope}
\end{tikzpicture}}}=\frac{1}{d_J}\delta_{\varphi, \varphi'}\delta_{\psi, \psi'}.
$$
The second equality comes from the definition of Frobenius algebra and the last equality comes from the connectivity of the Frobenius algebra.
\end{proof}

\begin{proposition}\label{prop:flatcom}
    Suppose the $\alpha$-induced bi-unitary connection is flat.
    Then \begin{align*}
\vcenter{\hbox{\scalebox{0.8}{
\begin{tikzpicture}[xscale=0.8, yscale=0.6]
\draw [line width=0.8cm] (0,0) [partial ellipse=0:180:2 and 1.5];
\draw [white, line width=0.77cm] (0,0) [partial ellipse=-0.1:180.1:2 and 1.5];
\draw [line width=0.8cm] (0,0) [partial ellipse=180:360:2 and 1.5];
\draw [white, line width=0.77cm] (0,0) [partial ellipse=178:362:2 and 1.5];
\draw [brown] (0,0) [partial ellipse=90:180:2 and 1.5];
\draw [brown] (0,0) [partial ellipse=180:270:2 and 1.5];
\draw [brown] (0,0) [partial ellipse=0:90:2 and 1.5];
\draw [brown] (0,0) [partial ellipse=270:360:2 and 1.5];
\begin{scope}[shift={(-2, 0)}]
\path [fill=brown] (0.05, -0.27) circle (0.08cm);
\draw [brown, dashed] (0,0) [partial ellipse=0:180:0.5 and 0.3];
\draw [brown] (0,0) [partial ellipse=180:360:0.5 and 0.3];
\end{scope}
\begin{scope}[shift={(2, 0)}]
\draw [red, dashed] (0,0) [partial ellipse=0:180:0.5 and 0.3];
\draw [red] (0,0) [partial ellipse=180:250:0.5 and 0.3];
\draw [red] (0,0) [partial ellipse=275:360:0.5 and 0.3];
\end{scope}
\end{tikzpicture}}}}=\vcenter{\hbox{\scalebox{0.8}{
\begin{tikzpicture}[xscale=0.8, yscale=0.6]
\draw [line width=0.8cm] (0,0) [partial ellipse=0:180:2 and 1.5];
\draw [white, line width=0.77cm] (0,0) [partial ellipse=-0.1:180.1:2 and 1.5];
\draw [line width=0.8cm] (0,0) [partial ellipse=180:360:2 and 1.5];
\draw [white, line width=0.77cm] (0,0) [partial ellipse=178:362:2 and 1.5];
\draw [brown] (0,0) [partial ellipse=90:180:2 and 1.5];
\draw [brown] (0,0) [partial ellipse=180:270:2 and 1.5];
\draw [brown] (0,0) [partial ellipse=0:90:2 and 1.5];
\draw [brown] (0,0) [partial ellipse=270:360:2 and 1.5];
\begin{scope}[shift={(-2, 0)}]
\path [fill=brown] (0.05, -0.27) circle (0.08cm);
\draw [brown, dashed] (0,0) [partial ellipse=0:180:0.5 and 0.3];
\draw [brown] (0,0) [partial ellipse=180:360:0.5 and 0.3];
\end{scope}
\begin{scope}[shift={(2, 0)}]
\draw [red, dashed] (0,0) [partial ellipse=0:180:0.5 and 0.3];
\draw [white, line width=0.2cm] (0,0) [partial ellipse=260:275:0.5 and 0.3];
\draw [red] (0,0) [partial ellipse=180:360:0.5 and 0.3];
\end{scope}
\end{tikzpicture}}}}= d_J^3\mu,
\end{align*}
where the inside is colored by $A$, and the outside is colored by $B$.
\end{proposition}
\begin{proof}
By the flatness, we have that 
\begin{align}\label{eq:flatness}
    \vcenter{\hbox{\scalebox{1}{\begin{tikzpicture}[scale=1.2]
\draw (-0.3, 1)..controls (-0.3, 0.5) and (-0.5, 0.3)..(-1, 0.3) node [above] {\tiny $J$};
\draw (0.3, 1)..controls (0.3, 0.5) and (0.5, 0.3)..(1, 0.3) node [above] {\tiny $J$};
\draw (0.3, -1)..controls (0.3, -0.5) and (0.5, -0.3)..(1, -0.3) node [below] {\tiny $\overline{J}$};
\draw (-0.3, -1)..controls (-0.3, -0.5) and (-0.5, -0.3)..(-1, -0.3) node [below] {\tiny $\overline{J}$};
\draw[blue] (-1.115, 0.2)..controls (-0.4, 0.2) and (-0.2, 0.3) ..(-0.2, 0.885);
\draw[blue] (0.2, 0.885)..controls (0.2, 0.3) and (0.3, 0.2)..(0.885, 0.2);
\draw[blue] (0.2, -1.115)..controls (0.2, -0.4) and (0.3, -0.2) ..(0.885, -0.2);
\draw[blue] (-1.115, -0.2)..controls (-0.4, -0.2) and (-0.2, -0.4)..(-0.2, -1.115);
\draw[blue, dashed] (0, 1.15) to (0, -0.85);
\draw [line width=0.3cm, white] (-0.2, 0)--(0.2, 0);
\draw[blue, dashed] (-0.85, 0)--(0.85, 0);
\draw (-1, 0) node {\tiny $\varphi$} [partial ellipse=90:270:0.15 and 0.3];
\draw[dashed] (-1, 0) [partial ellipse=90:-90:0.15 and 0.3];
\draw (0, 1) node {\tiny $\psi$}  [partial ellipse=0:360:0.3 and 0.15];
\draw[dashed] (0, -1) [partial ellipse=0:180:0.3 and 0.15];
\draw (0, -1) node {\tiny $\psi^*$}  [partial ellipse=180:360:0.3 and 0.15];
\draw (1, 0) node {\tiny $\varphi^*$}  [partial ellipse=0:360:0.15 and 0.3];
\end{tikzpicture}}}}=\frac{1}{d_J}.
\end{align}
By taking the sum of the basis $\{\psi\}$ of $\hom_{\mathcal{C}}(1, QX_j)$ and $j$ with coefficient $d_j$, we have that 
\begin{align*}
\vcenter{\hbox{\scalebox{0.7}{
\begin{tikzpicture}[xscale=0.8, yscale=0.6]
\draw [line width=0.8cm] (0,0) [partial ellipse=0:180:2 and 1.5];
\draw [white, line width=0.77cm] (0,0) [partial ellipse=-0.1:180.1:2 and 1.5];
\draw [line width=0.8cm] (0,0) [partial ellipse=180:360:2 and 1.5];
\draw [white, line width=0.77cm] (0,0) [partial ellipse=178:362:2 and 1.5];
\draw [red] (0,0) [partial ellipse=0:180:1.7 and 1.2];
\draw [red] (0,0) [partial ellipse=180:360:1.7 and 1.2];
\draw [brown] (0,0) [partial ellipse=0:180:2.2 and 1.7];
\draw [brown] (0,0) [partial ellipse=180:360:2.2 and 1.7];
\draw [white, line width=0.2cm]  (0,0) [partial ellipse=185:195:1.7 and 1.2];
\begin{scope}[shift={(-2, 0)}]
\path [fill=brown] (-0.2, -0.25) circle (0.08cm);
\draw [brown, dashed] (0,0) [partial ellipse=80:180:0.5 and 0.3];
\draw [blue, dashed] (0,0) [partial ellipse=0:80:0.5 and 0.3];
\draw [brown] (0,0) [partial ellipse=180:270:0.5 and 0.3];
\draw [blue] (0,0) [partial ellipse=270:360:0.5 and 0.3];
\end{scope}
\end{tikzpicture}}}}= d_J,
\end{align*}
where the inside is colored by $B$ and the outside is colored by $A$.
Now we switch the inside to the outside and obtain that 
\begin{align*}
\vcenter{\hbox{\scalebox{0.7}{
\begin{tikzpicture}[xscale=0.8, yscale=0.6]
\draw [line width=0.8cm] (0,0) [partial ellipse=0:180:2 and 1.5];
\draw [white, line width=0.77cm] (0,0) [partial ellipse=-0.1:180.1:2 and 1.5];
\draw [line width=0.8cm] (0,0) [partial ellipse=180:360:2 and 1.5];
\draw [white, line width=0.77cm] (0,0) [partial ellipse=178:362:2 and 1.5];
\draw [brown] (0,0) [partial ellipse=90:180:2 and 1.5];
\draw [brown] (0,0) [partial ellipse=180:270:2 and 1.5];
\draw [blue] (0,0) [partial ellipse=0:90:2 and 1.5];
\draw [blue] (0,0) [partial ellipse=270:360:2 and 1.5];
\begin{scope}[shift={(-2, 0)}]
\path [fill=brown] (0.05, -0.27) circle (0.08cm);
\draw [brown, dashed] (0,0) [partial ellipse=0:180:0.5 and 0.3];
\draw [brown] (0,0) [partial ellipse=180:360:0.5 and 0.3];
\end{scope}
\begin{scope}[shift={(2, 0)}]
\draw [red, dashed] (0,0) [partial ellipse=0:180:0.5 and 0.3];
\draw [red] (0,0) [partial ellipse=180:250:0.5 and 0.3];
\draw [red] (0,0) [partial ellipse=275:360:0.5 and 0.3];
\end{scope}
\end{tikzpicture}}}}=d_J\mu,
\end{align*}
where the inside is colored by $A$ and the outside is colored by $B$.
By taking the sum of the basis $\{\varphi\}$ of $\hom_{\mathcal{C}}(1, QX_j)$ and $j$ with coefficient $d_j$, we have that 
\begin{align*}
\vcenter{\hbox{\scalebox{0.7}{
\begin{tikzpicture}[xscale=0.8, yscale=0.6]
\draw [line width=0.8cm] (0,0) [partial ellipse=0:180:2 and 1.5];
\draw [white, line width=0.77cm] (0,0) [partial ellipse=-0.1:180.1:2 and 1.5];
\draw [line width=0.8cm] (0,0) [partial ellipse=180:360:2 and 1.5];
\draw [white, line width=0.77cm] (0,0) [partial ellipse=178:362:2 and 1.5];
\draw [brown] (0,0) [partial ellipse=90:180:2 and 1.5];
\draw [brown] (0,0) [partial ellipse=180:270:2 and 1.5];
\draw [brown] (0,0) [partial ellipse=0:90:2 and 1.5];
\draw [brown] (0,0) [partial ellipse=270:360:2 and 1.5];
\begin{scope}[shift={(-2, 0)}]
\path [fill=brown] (0.05, -0.27) circle (0.08cm);
\draw [brown, dashed] (0,0) [partial ellipse=0:180:0.5 and 0.3];
\draw [brown] (0,0) [partial ellipse=180:360:0.5 and 0.3];
\end{scope}
\begin{scope}[shift={(2, 0)}]
\draw [red, dashed] (0,0) [partial ellipse=0:180:0.5 and 0.3];
\draw [red] (0,0) [partial ellipse=180:250:0.5 and 0.3];
\draw [red] (0,0) [partial ellipse=275:360:0.5 and 0.3];
\end{scope}
\end{tikzpicture}}}}
=d_J^3\mu.
\end{align*}
Rotating Equation \eqref{eq:flatness} by 90$^\circ$, we have that
\begin{align*}
\vcenter{\hbox{\scalebox{0.7}{
\begin{tikzpicture}[xscale=0.8, yscale=0.6]
\draw [line width=0.8cm] (0,0) [partial ellipse=0:180:2 and 1.5];
\draw [white, line width=0.77cm] (0,0) [partial ellipse=-0.1:180.1:2 and 1.5];
\draw [line width=0.8cm] (0,0) [partial ellipse=180:360:2 and 1.5];
\draw [white, line width=0.77cm] (0,0) [partial ellipse=178:362:2 and 1.5];
\draw [brown] (0,0) [partial ellipse=90:180:2 and 1.5];
\draw [brown] (0,0) [partial ellipse=180:270:2 and 1.5];
\draw [brown] (0,0) [partial ellipse=0:90:2 and 1.5];
\draw [brown] (0,0) [partial ellipse=270:360:2 and 1.5];
\begin{scope}[shift={(-2, 0)}]
\path [fill=brown] (0.05, -0.27) circle (0.08cm);
\draw [brown, dashed] (0,0) [partial ellipse=0:180:0.5 and 0.3];
\draw [brown] (0,0) [partial ellipse=180:360:0.5 and 0.3];
\end{scope}
\begin{scope}[shift={(2, 0)}]
\draw [red, dashed] (0,0) [partial ellipse=0:180:0.5 and 0.3];
\draw [white, line width=0.2cm] (0,0) [partial ellipse=260:275:0.5 and 0.3];
\draw [red] (0,0) [partial ellipse=180:360:0.5 and 0.3];
\end{scope}
\end{tikzpicture}}}}=d_J^3\mu.
\end{align*}
This completes the proof of the proposition.
\end{proof}

In the following, we assume that the input Morita context is unitary hence the alterfold topological quantum field theory is unitary. We list equivalent statements for the flatness of the $\alpha$-induced bi-unitary connection. 
The equivalence of (1) and (5) are the main theorem of \cite{Kaw24}.
\begin{theorem}\label{thm:flatequiv}
   Suppose that the spherical fusion category $\mathcal{C}$ is unitary and $Q$ is a $Q$-system in $\mathcal{C}$.
   Then the following statements are equivalent:
   \begin{enumerate}[(1)]
   \item the Frobenius algebra $Q$ is commutative, ( See Equation \eqref{eq:comm});
   \item \begin{align}\label{eq:qtwist}
\vcenter{\hbox{\scalebox{1}{
\begin{tikzpicture}[scale=1.2]
\draw [blue, line width=0.2cm] (-0.3, 0.5)--(0.3, 0.5) ;
\draw [white, line width=0.17cm] (-0.31, 0.5)--(0.31, 0.5) ;
\draw [blue, line width=0.2cm] (0, -0.5) --(0, 1)  node [above] {\tiny $J\overline{J}$} (-0.3, 0)--(0.3, 0);
\draw [white, line width=0.17cm] (0, -0.51) --(0, 1.1) (-0.3, 0)--(0.3, 0);
\begin{scope}[shift={(-0.3, 0.25)}]
\draw [blue, line width=0.2cm] (0,0) [partial ellipse=90:270:0.25 and 0.25];
\draw [white, line width=0.17cm] (0,0) [partial ellipse=89:271:0.25 and 0.25];
\end{scope}
\begin{scope}[shift={(0.3, 0.25)}]
\draw [blue, line width=0.2cm] (0,0) [partial ellipse=90:-90:0.25 and 0.25];
\draw [white, line width=0.17cm] (0,0) [partial ellipse=92:-92:0.25 and 0.25];
\end{scope}
\end{tikzpicture}}}}
=
\vcenter{\hbox{\scalebox{1}{
\begin{tikzpicture}[scale=1.2]
\draw [blue, line width=0.2cm] (0, -0.5) --(0, 1)  node [above] {\tiny $J\overline{J}$} (-0.3, 0)--(0.3, 0);
\draw [white, line width=0.17cm] (0, -0.51) --(0, 1.1) (-0.3, 0)--(0.3, 0);
\begin{scope}[shift={(-0.3, 0.25)}]
\draw [blue, line width=0.2cm] (0,0) [partial ellipse=90:270:0.25 and 0.25];
\draw [white, line width=0.17cm] (0,0) [partial ellipse=89:271:0.25 and 0.25];
\end{scope}
\begin{scope}[shift={(0.3, 0.25)}]
\draw [blue, line width=0.2cm] (0,0) [partial ellipse=90:-90:0.25 and 0.25];
\draw [white, line width=0.17cm] (0,0) [partial ellipse=92:-92:0.25 and 0.25];
\end{scope}
\draw [white, line width=0.25cm] (-0.25, 0.5)--(0.25, 0.5) ;
\draw [blue, line width=0.2cm] (-0.3, 0.5)--(0.3, 0.5) ;
\draw [white, line width=0.17cm] (-0.31, 0.5)--(0.31, 0.5) ;
\end{tikzpicture}}}}
=d_J
\vcenter{\hbox{\scalebox{1}{
\begin{tikzpicture}[scale=1.2]
\draw [blue, line width=0.2cm] (0, -0.5) --(0, 0.5) node [above] {\tiny $J\overline{J}$};
\draw [white, line width=0.17cm] (0, -0.51) --(0, 0.51);
\end{tikzpicture}}}}.
\end{align}
\item \begin{align}\label{eq:tubetwist}
\vcenter{\hbox{\scalebox{0.7}{
\begin{tikzpicture}[xscale=0.8, yscale=0.6]
\draw [brown] (0,0) [partial ellipse=180:360:0.6 and 0.3];
\draw [brown, dashed] (0,0) [partial ellipse=0:180:0.6 and 0.3];
\path [fill=brown] (0, -0.3) circle (0.07cm);
\begin{scope}[shift={(0,2)}]
\draw (0,0) [partial ellipse=0:360:0.6 and 0.3];
\end{scope}
\draw (-0.6, 2)--(-0.6, 0) (0.6, 2)--(0.6, 0); 
\draw (-0.6, -2)--(-0.6, 0) (0.6, -2)--(0.6, 0);
\draw [brown] (0, -2.3) --(0, 1.7);
\begin{scope}[shift={(0,-2)}]
\draw [dashed](0,0) [partial ellipse=0:180:0.6 and 0.3];
\draw (0,0) [partial ellipse=180:360:0.6 and 0.3];
\end{scope}
\begin{scope}[shift={(0, 1)}]
\draw [red] (0,0) [partial ellipse=180:265:0.6 and 0.3];
\draw [red] (0,0) [partial ellipse=275:360:0.6 and 0.3];
\draw [red, dashed] (0,0) [partial ellipse=0:180:0.6 and 0.3];
\end{scope}
\end{tikzpicture}}}}
=d_J
\vcenter{\hbox{\scalebox{0.7}{
\begin{tikzpicture}[xscale=0.8, yscale=0.6]
\draw [red] (0,0) [partial ellipse=180:265:0.6 and 0.3];
\draw [red] (0,0) [partial ellipse=275:360:0.6 and 0.3];
\draw [red, dashed] (0,0) [partial ellipse=0:180:0.6 and 0.3];
\begin{scope}[shift={(0,2)}]
\draw (0,0) [partial ellipse=0:360:0.6 and 0.3];
\end{scope}
\draw (-0.6, 2)--(-0.6, 0) (0.6, 2)--(0.6, 0); 
\draw (-0.6, -2)--(-0.6, 0) (0.6, -2)--(0.6, 0);
\draw [brown] (0, -2.3) --(0, 1.7);
\begin{scope}[shift={(0,-2)}]
\draw [dashed](0,0) [partial ellipse=0:180:0.6 and 0.3];
\draw (0,0) [partial ellipse=180:360:0.6 and 0.3];
\end{scope}
\end{tikzpicture}}}},
\quad
\vcenter{\hbox{\scalebox{0.7}{
\begin{tikzpicture}[xscale=0.8, yscale=0.6]
\draw [brown] (0,0) [partial ellipse=180:360:0.6 and 0.3];
\draw [brown, dashed] (0,0) [partial ellipse=0:180:0.6 and 0.3];
\path [fill=brown] (0, -0.3) circle (0.07cm);
\begin{scope}[shift={(0,2)}]
\draw (0,0) [partial ellipse=0:360:0.6 and 0.3];
\end{scope}
\draw (-0.6, 2)--(-0.6, 0) (0.6, 2)--(0.6, 0); 
\draw (-0.6, -2)--(-0.6, 0) (0.6, -2)--(0.6, 0);
\draw [brown] (0, -2.3) --(0, 1.7);
\begin{scope}[shift={(0,-2)}]
\draw [dashed](0,0) [partial ellipse=0:180:0.6 and 0.3];
\draw (0,0) [partial ellipse=180:360:0.6 and 0.3];
\end{scope}
\begin{scope}[shift={(0, 1)}]
\draw [line width=0.2cm,white] (0, -0.18) --(0, -0.38);
\draw [red] (0,0) [partial ellipse=180:360:0.6 and 0.3];
\draw [red, dashed] (0,0) [partial ellipse=0:180:0.6 and 0.3];
\end{scope}
\end{tikzpicture}}}}
=d_J
\vcenter{\hbox{\scalebox{0.7}{
\begin{tikzpicture}[xscale=0.8, yscale=0.6]
\begin{scope}[shift={(0,2)}]
\draw (0,0) [partial ellipse=0:360:0.6 and 0.3];
\end{scope}
\draw (-0.6, 2)--(-0.6, 0) (0.6, 2)--(0.6, 0); 
\draw (-0.6, -2)--(-0.6, 0) (0.6, -2)--(0.6, 0);
\draw [brown] (0, -2.3) --(0, 1.7);
\draw [line width=0.2cm,white] (0, -0.18) --(0, -0.38);
\draw [red] (0,0) [partial ellipse=180:360:0.6 and 0.3];
\draw [red, dashed] (0,0) [partial ellipse=0:180:0.6 and 0.3];
\begin{scope}[shift={(0,-2)}]
\draw [dashed](0,0) [partial ellipse=0:180:0.6 and 0.3];
\draw (0,0) [partial ellipse=180:360:0.6 and 0.3];
\end{scope}
\end{tikzpicture}}}}.
\end{align}
\item locally, 
\begin{align*}
\vcenter{\hbox{\scalebox{0.7}{
\begin{tikzpicture}[xscale=0.8, yscale=0.6]
\draw [line width=0.8cm] (0,0) [partial ellipse=0:180:2 and 1.5];
\draw [white, line width=0.77cm] (0,0) [partial ellipse=-0.1:180.1:2 and 1.5];
\draw [line width=0.8cm] (0,0) [partial ellipse=180:360:2 and 1.5];
\draw [white, line width=0.77cm] (0,0) [partial ellipse=178:362:2 and 1.5];
\draw [brown] (0,0) [partial ellipse=90:180:2 and 1.5];
\draw [brown] (0,0) [partial ellipse=180:270:2 and 1.5];
\draw [brown] (0,0) [partial ellipse=0:90:2 and 1.5];
\draw [brown] (0,0) [partial ellipse=270:360:2 and 1.5];
\begin{scope}[shift={(-2, 0)}]
\path [fill=brown] (0.05, -0.27) circle (0.08cm);
\draw [brown, dashed] (0,0) [partial ellipse=0:180:0.5 and 0.3];
\draw [brown] (0,0) [partial ellipse=180:360:0.5 and 0.3];
\end{scope}
\begin{scope}[shift={(2, 0)}]
\draw [red, dashed] (0,0) [partial ellipse=0:180:0.5 and 0.3];
\draw [white, line width=0.2cm] (0,0) [partial ellipse=260:275:0.5 and 0.3];
\draw [red] (0,0) [partial ellipse=180:360:0.5 and 0.3];
\end{scope}
\end{tikzpicture}}}}
=&
d_J
\vcenter{\hbox{\scalebox{0.7}{
\begin{tikzpicture}[xscale=0.8, yscale=0.6]
\draw [line width=0.83cm] (0,0) [partial ellipse=-0.1:180.1:2 and 1.5];
\draw [white, line width=0.8cm] (0,0) [partial ellipse=-0.1:180.1:2 and 1.5];
\draw [brown] (0,0) [partial ellipse=-0.1:180.1:2 and 1.5];
\draw [line width=0.83cm] (0,0) [partial ellipse=180:360:2 and 1.5];
\draw [white, line width=0.8cm] (0,0) [partial ellipse=179:361:2 and 1.5];
\draw [brown] (0,0) [partial ellipse=192:361:2 and 1.5];
\begin{scope}[shift={(-2, 0)}]
\draw [red, dashed] (0,0) [partial ellipse=0:180:0.46 and 0.25];
\draw [red] (0,0) [partial ellipse=180:360:0.46 and 0.25];
\end{scope}
\end{tikzpicture}}}},\\
\vcenter{\hbox{\scalebox{0.7}{
\begin{tikzpicture}[xscale=0.8, yscale=0.6]
\draw [line width=0.8cm] (0,0) [partial ellipse=0:180:2 and 1.5];
\draw [white, line width=0.77cm] (0,0) [partial ellipse=-0.1:180.1:2 and 1.5];
\draw [line width=0.8cm] (0,0) [partial ellipse=180:360:2 and 1.5];
\draw [white, line width=0.77cm] (0,0) [partial ellipse=178:362:2 and 1.5];
\draw [brown] (0,0) [partial ellipse=90:180:2 and 1.5];
\draw [brown] (0,0) [partial ellipse=180:270:2 and 1.5];
\draw [brown] (0,0) [partial ellipse=0:90:2 and 1.5];
\draw [brown] (0,0) [partial ellipse=270:360:2 and 1.5];
\begin{scope}[shift={(-2, 0)}]
\path [fill=brown] (0.05, -0.27) circle (0.08cm);
\draw [brown, dashed] (0,0) [partial ellipse=0:180:0.5 and 0.3];
\draw [brown] (0,0) [partial ellipse=180:360:0.5 and 0.3];
\end{scope}
\begin{scope}[shift={(2, 0)}]
\draw [red, dashed] (0,0) [partial ellipse=0:180:0.5 and 0.3];
\draw [red] (0,0) [partial ellipse=180:260:0.5 and 0.3];
\draw [red] (0,0) [partial ellipse=280:360:0.5 and 0.3];
\end{scope}
\end{tikzpicture}}}}
=&
d_J
\vcenter{\hbox{\scalebox{0.7}{
\begin{tikzpicture}[xscale=0.8, yscale=0.6]
\draw [line width=0.83cm] (0,0) [partial ellipse=-0.1:180.1:2 and 1.5];
\draw [white, line width=0.8cm] (0,0) [partial ellipse=-0.1:180.1:2 and 1.5];
\draw [brown] (0,0) [partial ellipse=-0.1:180.1:2 and 1.5];
\draw [line width=0.83cm] (0,0) [partial ellipse=180:360:2 and 1.5];
\draw [white, line width=0.8cm] (0,0) [partial ellipse=179:361:2 and 1.5];
\draw [brown] (0,0) [partial ellipse=178:361:2 and 1.5];
\begin{scope}[shift={(-2, 0)}]
\draw [red, dashed] (0,0) [partial ellipse=0:180:0.46 and 0.25];
\draw [red] (0,0) [partial ellipse=180:265:0.46 and 0.25];
\draw [red] (0,0) [partial ellipse=285:360:0.46 and 0.25];
\end{scope}
\end{tikzpicture}}}}.
\end{align*}
   \item the $\alpha$-induced bi-unitary connection is flat.
   \end{enumerate}
\end{theorem}
\begin{proof}
 (1)$\Rightarrow$(2):
 By taking the cap from the left hand side, we see that (2) is true.

 (2)$\Rightarrow$(1):
 Let $R=\vcenter{\hbox{\scalebox{0.6}{
\begin{tikzpicture}[scale=0.9]
\begin{scope}[shift={(0,1)}]
\draw [blue] (1, 1.5)--(-0.5, 0) (1, 2)--(-1, 0);
\path [fill=white] (0, 0.75) circle (0.4cm);
\draw [blue] (-1, 1.5) node [left] {\tiny $J$}--(0.5, 0) (-1, 2) node [left] {\tiny $\overline{J}$} --(1, 0);
\end{scope}
\draw [blue](-1, -0.1)--(-1, 1) (1, -0.1)--(1, 1);
\begin{scope}[shift={(0, 1)}]
\draw [blue](0,0) [partial ellipse=180:360:0.5 and 0.7];
\end{scope}
\end{tikzpicture}}}}$ and $W= \vcenter{\hbox{\scalebox{0.6}{
\begin{tikzpicture}[scale=0.9]
\draw [blue](-1, -0.1)--(-1, 1) node [above] {\tiny $J$} (1, -0.1)--(1, 1);
\begin{scope}[shift={(0, 1)}]
\draw [blue](0,0) [partial ellipse=180:360:0.7 and 0.8] node [pos=0, above] {\tiny $\overline{J}$} ;
\end{scope}
\end{tikzpicture}}}}$. 
Then $RR^*=WW^*=d_J \vcenter{\hbox{\scalebox{1}{
\begin{tikzpicture}[scale=1.2]
\draw [blue, line width=0.2cm] (0, -0.5) --(0, 0.5) node [above] {\tiny $J\overline{J}$};
\draw [white, line width=0.17cm] (0, -0.51) --(0, 0.51);
\end{tikzpicture}}}}$ and $RW^*=\vcenter{\hbox{\scalebox{0.6}{
\begin{tikzpicture}[scale=0.9]
\begin{scope}[shift={(0,1)}]
\draw [blue] (1, 1.5)--(-0.5, 0) (0.5, 1.5)--(-1, 0);
\path [fill=white] (0, 0.75) circle (0.4cm);
\draw [blue] (-1, 1.5) node [left] {\tiny $J$}--(0.5, 0) (-0.5, 1.5) node [left] {\tiny $\overline{J}$} --(1, 0);
\end{scope}
\draw [blue](-1, -0.1)--(-1, 1) (1, -0.1)--(1, 1);
\begin{scope}[shift={(0, 1)}]
\draw [blue](0,0) [partial ellipse=180:360:0.5 and 0.7];
\end{scope}
\begin{scope}[shift={(0,2.5)}]
\draw [blue](-1, 0)--(-1, 1) (1, 0)--(1, 1);
\draw [blue](0,0) [partial ellipse=0:180:0.5 and 0.7];
\end{scope}
\end{tikzpicture}}}}$.
Now by the assumption, we have that 
\begin{align*}
    tr((R-W)(R^*-W^*))=& tr(RR^*)+tr(WW^*)-tr(RW^*)-tr(WR^*) \\
    =& d_J^3+d_J^3-d_J^3-d_J^3=0.
\end{align*}
This implies that (1) holds.

(2)$\Rightarrow$(3)
It is clear that 
\begin{align*}
P_1=\frac{1}{d_J}\vcenter{\hbox{\scalebox{0.7}{
\begin{tikzpicture}[xscale=0.8, yscale=0.6]
\draw [brown] (0,0) [partial ellipse=180:360:0.6 and 0.3];
\draw [brown, dashed] (0,0) [partial ellipse=0:180:0.6 and 0.3];
\path [fill=brown] (0, -0.3) circle (0.07cm);
\begin{scope}[shift={(0,2)}]
\draw (0,0) [partial ellipse=0:360:0.6 and 0.3];
\end{scope}
\draw (-0.6, 2)--(-0.6, 0) (0.6, 2)--(0.6, 0); 
\draw (-0.6, -2)--(-0.6, 0) (0.6, -2)--(0.6, 0);
\draw [brown] (0, -2.3) --(0, 1.7);
\begin{scope}[shift={(0,-2)}]
\draw [dashed](0,0) [partial ellipse=0:180:0.6 and 0.3];
\draw (0,0) [partial ellipse=180:360:0.6 and 0.3];
\end{scope}
\end{tikzpicture}}}}, \quad P_2=
\frac{1}{\mu}\vcenter{\hbox{\scalebox{0.7}{
\begin{tikzpicture}[xscale=0.8, yscale=0.6]
\draw [red] (0,0) [partial ellipse=180:265:0.6 and 0.3];
\draw [red] (0,0) [partial ellipse=275:360:0.6 and 0.3];
\draw [red, dashed] (0,0) [partial ellipse=0:180:0.6 and 0.3];
\begin{scope}[shift={(0,2)}]
\draw (0,0) [partial ellipse=0:360:0.6 and 0.3];
\end{scope}
\draw (-0.6, 2)--(-0.6, 0) (0.6, 2)--(0.6, 0); 
\draw (-0.6, -2)--(-0.6, 0) (0.6, -2)--(0.6, 0);
\draw [brown] (0, -2.3) --(0, 1.7);
\begin{scope}[shift={(0,-2)}]
\draw [dashed](0,0) [partial ellipse=0:180:0.6 and 0.3];
\draw (0,0) [partial ellipse=180:360:0.6 and 0.3];
\end{scope}
\end{tikzpicture}}}}
\end{align*}
are idempotents.
By handle slides and Equation \eqref{eq:qtwist}, we have that 
\begin{align}\label{eq:handletwist}
\vcenter{\hbox{\scalebox{0.7}{
\begin{tikzpicture}[xscale=0.8, yscale=0.6]
\draw [brown] (0,0) [partial ellipse=180:360:0.6 and 0.3];
\draw [brown, dashed] (0,0) [partial ellipse=0:180:0.6 and 0.3];
\path [fill=brown] (0, -0.3) circle (0.07cm);
\begin{scope}[shift={(0,2)}]
\draw (0,0) [partial ellipse=0:360:0.6 and 0.3];
\end{scope}
\draw (-0.6, 2)--(-0.6, 0) (0.6, 2)--(0.6, 0); 
\draw (-0.6, -2)--(-0.6, 0) (0.6, -2)--(0.6, 0);
\draw [brown] (0, -2.3) --(0, 1.7);
\begin{scope}[shift={(0,-2)}]
\draw [dashed](0,0) [partial ellipse=0:180:0.6 and 0.3];
\draw (0,0) [partial ellipse=180:360:0.6 and 0.3];
\end{scope}
\begin{scope}[shift={(0, 1)}]
\draw [red] (0,0) [partial ellipse=180:265:0.6 and 0.3];
\draw [red] (0,0) [partial ellipse=275:360:0.6 and 0.3];
\draw [red, dashed] (0,0) [partial ellipse=0:180:0.6 and 0.3];
\end{scope}
\end{tikzpicture}}}}
=
\vcenter{\hbox{\scalebox{0.7}{
\begin{tikzpicture}[xscale=0.8, yscale=0.6]
\begin{scope}[shift={(0, -0.5)}]
\draw [brown] (0,0) [partial ellipse=80:-90:0.4 and 0.3];
\draw [brown] (0,0) [partial ellipse=100:270:0.4 and 0.3];
\path [fill=brown] (0, -0.3) circle (0.07cm);
\end{scope}
\begin{scope}[shift={(0,2)}]
\draw (0,0) [partial ellipse=0:360:0.6 and 0.3];
\end{scope}
\draw (-0.6, 2)--(-0.6, 0) (0.6, 2)--(0.6, 0); 
\draw (-0.6, -2)--(-0.6, 0) (0.6, -2)--(0.6, 0);
\draw [brown] (0, -2.3) --(0, 1.7);
\begin{scope}[shift={(0,-2)}]
\draw [dashed](0,0) [partial ellipse=0:180:0.6 and 0.3];
\draw (0,0) [partial ellipse=180:360:0.6 and 0.3];
\end{scope}
\begin{scope}[shift={(0, 1)}]
\draw [red] (0,0) [partial ellipse=180:265:0.6 and 0.3];
\draw [red] (0,0) [partial ellipse=275:360:0.6 and 0.3];
\draw [red, dashed] (0,0) [partial ellipse=0:180:0.6 and 0.3];
\end{scope}
\end{tikzpicture}}}}
= d_J\vcenter{\hbox{\scalebox{0.7}{
\begin{tikzpicture}[xscale=0.8, yscale=0.6]
\draw [red] (0,0) [partial ellipse=180:265:0.6 and 0.3];
\draw [red] (0,0) [partial ellipse=275:360:0.6 and 0.3];
\draw [red, dashed] (0,0) [partial ellipse=0:180:0.6 and 0.3];
\begin{scope}[shift={(0,2)}]
\draw (0,0) [partial ellipse=0:360:0.6 and 0.3];
\end{scope}
\draw (-0.6, 2)--(-0.6, 0) (0.6, 2)--(0.6, 0); 
\draw (-0.6, -2)--(-0.6, 0) (0.6, -2)--(0.6, 0);
\draw [brown] (0, -2.3) --(0, 1.7);
\begin{scope}[shift={(0,-2)}]
\draw [dashed](0,0) [partial ellipse=0:180:0.6 and 0.3];
\draw (0,0) [partial ellipse=180:360:0.6 and 0.3];
\end{scope}
\end{tikzpicture}}}}.
\end{align}
This implies that $P_2$ is a sub-idmepotent of $P_1$.

(3)$\Rightarrow$(2):
By the assumption, we have Equation \eqref{eq:handletwist} and this implies (2) directly.

(3)$\Rightarrow$(4): It follows directly from the assumption.

(4)$\Rightarrow$(3): 
By the assumption, we see that $\vcenter{\hbox{\scalebox{0.8}{
\begin{tikzpicture}[xscale=0.8, yscale=0.6]
\draw [line width=0.8cm] (0,0) [partial ellipse=0:180:2 and 1.5];
\draw [white, line width=0.77cm] (0,0) [partial ellipse=-0.1:180.1:2 and 1.5];
\draw [line width=0.8cm] (0,0) [partial ellipse=180:360:2 and 1.5];
\draw [white, line width=0.77cm] (0,0) [partial ellipse=178:362:2 and 1.5];
\draw [brown] (0,0) [partial ellipse=90:180:2 and 1.5];
\draw [brown] (0,0) [partial ellipse=180:270:2 and 1.5];
\draw [brown] (0,0) [partial ellipse=0:90:2 and 1.5];
\draw [brown] (0,0) [partial ellipse=270:360:2 and 1.5];
\begin{scope}[shift={(-2, 0)}]
\path [fill=brown] (0.05, -0.27) circle (0.08cm);
\draw [brown, dashed] (0,0) [partial ellipse=0:180:0.5 and 0.3];
\draw [brown] (0,0) [partial ellipse=180:360:0.5 and 0.3];
\end{scope}
\begin{scope}[shift={(2, 0)}]
\draw [red, dashed] (0,0) [partial ellipse=0:180:0.5 and 0.3];
\draw [white, line width=0.2cm] (0,0) [partial ellipse=260:275:0.5 and 0.3];
\draw [red] (0,0) [partial ellipse=180:360:0.5 and 0.3];
\end{scope}
\end{tikzpicture}}}}=d_J^3\mu$.
Hence, we obtain that 
\begin{align*}
    tr((P_1P_2-P_2)^2)=& tr(P_2)-tr(P_1P_2)\\
    =& \frac{\mu d_J^2}{\mu}- \frac{1}{d_J\mu} d_J^3\mu=0.
\end{align*}
This implies that (3) holds.

(5)$\Rightarrow$(3): It follows from Proposition \ref{prop:flatcom}.

(1)$\Rightarrow$(5): It is proved in Proposition \ref{prop:comflat}.
Finally, we see that the statements are equivalent.
\end{proof}

\begin{remark}
Note that the picture in the statement (3) of Theorem \ref{thm:flatequiv} can be reformulated as
\begin{align*}
\frac{d_J}{\mu}\vcenter{\hbox{\scalebox{0.7}{
\begin{tikzpicture}[scale=0.7]
\begin{scope}[shift={(5, 0)}]
\draw [double distance=0.57cm] (0,0) [partial ellipse=-0.1:180.1:2 and 1.5];
\draw [red] (0,0) [partial ellipse=-0.1:180.1:2 and 1.5];
\end{scope} 
\begin{scope}[shift={(2.5, 0)}]
\draw [white, line width=0.63cm] (0,0) [partial ellipse=-0.1:180.1:2 and 1.5];
\draw [line width=0.58cm, brown!20!white] (0,0) [partial ellipse=-0.1:180.1:2 and 1.5];
\end{scope} 
\begin{scope}[shift={(2.5, 0)}]
\draw [white, line width=0.63cm] (0,0) [partial ellipse=180:360:2 and 1.5];
\draw [brown!20!white, line width=0.58cm] (0,0) [partial ellipse=178:362:2 and 1.5];
\end{scope}
\begin{scope}[shift={(5, 0)}]
\draw [line width=0.6cm] (0,0) [partial ellipse=180:360:2 and 1.5];
\draw [white, line width=0.57cm] (0,0) [partial ellipse=178:362:2 and 1.5];
\draw [red] (0,0) [partial ellipse=178:362:2 and 1.5];
\end{scope}
\begin{scope}[shift={(7, 0)}]
\draw [red, dashed] (0,0) [partial ellipse=0:180:0.42 and 0.25];
\draw [red] (0,0) [partial ellipse=180:260:0.42 and 0.25];
\draw [red] (0,0) [partial ellipse=280:360:0.42 and 0.25];
\end{scope}
\end{tikzpicture}}}}
= \vcenter{\hbox{\scalebox{0.7}{
\begin{tikzpicture}[xscale=0.8, yscale=0.6]
\draw [line width=0.8cm] (0,0) [partial ellipse=0:180:2 and 1.5];
\draw [white, line width=0.77cm] (0,0) [partial ellipse=-0.1:180.1:2 and 1.5];
\draw [line width=0.8cm] (0,0) [partial ellipse=180:360:2 and 1.5];
\draw [white, line width=0.77cm] (0,0) [partial ellipse=178:362:2 and 1.5];
\draw [brown] (0,0) [partial ellipse=90:180:2 and 1.5];
\draw [brown] (0,0) [partial ellipse=180:270:2 and 1.5];
\draw [brown] (0,0) [partial ellipse=0:90:2 and 1.5];
\draw [brown] (0,0) [partial ellipse=270:360:2 and 1.5];
\begin{scope}[shift={(-2, 0)}]
\path [fill=brown] (0.05, -0.27) circle (0.08cm);
\draw [brown, dashed] (0,0) [partial ellipse=0:180:0.5 and 0.3];
\draw [brown] (0,0) [partial ellipse=180:360:0.5 and 0.3];
\end{scope}
\begin{scope}[shift={(2, 0)}]
\draw [red, dashed] (0,0) [partial ellipse=0:180:0.5 and 0.3];
\draw [white, line width=0.2cm] (0,0) [partial ellipse=260:275:0.5 and 0.3];
\draw [red] (0,0) [partial ellipse=180:360:0.5 and 0.3];
\end{scope}
\end{tikzpicture}}}}
=\sum_{j} d_J d_j
\vcenter{\hbox{\scalebox{0.7}{
\begin{tikzpicture}[xscale=0.8, yscale=0.6]
\draw [brown!20!white,line width=0.83cm] (0,0) [partial ellipse=-0.1:180.1:2 and 1.5];
\draw [brown!20!white, line width=0.83cm] (0,0) [partial ellipse=179:361:2 and 1.5];
\begin{scope}[shift={(-2, 0)}]
\draw [blue, dashed] (0,0) [partial ellipse=0:180:0.46 and 0.25];
\draw [blue] (0,0) [partial ellipse=180:360:0.46 and 0.25];
\node [below] at (0, -0.25) {\tiny $\alpha_+(X_j)$};
\end{scope}
\end{tikzpicture}}}}
.
\end{align*}

\end{remark}

Now we show that the results of Kawahigashi are true for general spherical fusion categories in the context of bi-invertible connections.
\begin{theorem}\label{thm:flatequi}
 The $\alpha$-induced bi-invertible connection is flat if and only if the Frobenius algebra $Q=J\overline{J}$ is commutative.
\end{theorem}
\begin{proof}
It suffices to show the commutativity under the assumption that the connection is flat.
By the flatness, we have that 
\begin{align*}
    \vcenter{\hbox{\begin{tikzpicture}[xscale=0.8, yscale=0.6]
\draw[brown] (0, 0)--(0, 1)--(1, 2);
\draw[blue, ->-=0.5] (1, 2)--(1, -2);
\draw[brown] (1, -2)..controls (0, -1) and (-1, -1)..(0, 0);
\draw[fill=white] (0.6, 1.5) rectangle (1.4, 2.3);
\draw[fill=white] (0.6, -1.5) rectangle (1.4, -2.3);
\draw (1, 1.9) node{\tiny$\varphi$};
\draw (1, -1.9) node{\tiny$\varphi^{'*}$};
\fill[white] (0, -1.2) circle(0.2);
\begin{scope}[xscale=-1]
\draw[brown] (0, 1)--(1, 2);
\draw[blue, ->-=0.5] (1, 2)--(1, -2);
\draw[brown] (1, -2)..controls (0, -1) and (-1, -1)..(0, 0);
\draw[fill=white] (0.6, 1.5) rectangle (1.4, 2.3);
\draw[fill=white] (0.6, -1.5) rectangle (1.4, -2.3);
\draw (1, 1.9) node{\tiny$\psi$};
\draw (1, -1.9) node{\tiny$\psi^{*}$};
\end{scope}
\end{tikzpicture}}}
=\vcenter{\hbox{\begin{tikzpicture}[scale=0.8]
\draw[brown] (0, 0)--(0, 1)--(1, 2);
\draw[blue, ->-=0.5] (1, 2)--(1, -1);
\draw[brown] (1, -1)--(0, 0);
\draw[fill=white] (0.6, 1.5) rectangle (1.4, 2.3);
\draw[fill=white] (0.6, -1.3) rectangle (1.4, -0.5);
\draw (1, 1.9) node{\tiny$\varphi$};
\draw (1, -.9) node{\tiny$\varphi^{'*}$};
\begin{scope}[xscale=-1]
\draw[brown] (0, 1)--(1, 2);
\draw[blue, ->-=0.5] (1, 2)--(1, -1);
\draw[brown] (1, -1)--(0, 0);
\draw[fill=white] (0.6, 1.5) rectangle (1.4, 2.3);
\draw[fill=white] (0.6, -1.3) rectangle (1.4, -0.5);
\draw (1, 1.9) node{\tiny$\psi$};
\draw (1, -.9) node{\tiny$\psi^{*}$};
\end{scope}
\end{tikzpicture}}}
\end{align*}
Note that if $\varphi^{'*}(W\varphi)=0$ for all $\varphi, \varphi'$, then $W=0$.
We have that 
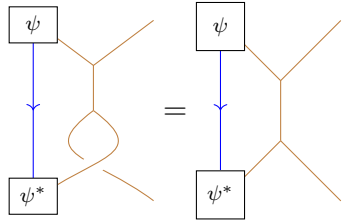
\begin{align*}
    \vcenter{\hbox{\begin{tikzpicture}[xscale=0.8, yscale=0.6]
\draw[brown] (0, 0)--(0, 1)--(1, 2);
\draw[brown] (1, -2)..controls (0, -1) and (-1, -1)..(0, 0);
\fill[white] (0, -1.2) circle(0.2);
\begin{scope}[xscale=-1]
\draw[brown] (0, 1)--(1, 2);
\draw[blue, ->-=0.5] (1, 2)--(1, -2);
\draw[brown] (1, -2)..controls (0, -1) and (-1, -1)..(0, 0);
\draw[fill=white] (0.6, 1.5) rectangle (1.4, 2.3);
\draw[fill=white] (0.6, -1.5) rectangle (1.4, -2.3);
\draw (1, 1.9) node{\tiny$\psi$};
\draw (1, -1.9) node{\tiny$\psi^{*}$};
\end{scope}
\end{tikzpicture}}}
=\vcenter{\hbox{\begin{tikzpicture}[scale=0.8]
\draw[brown] (0, 0)--(0, 1)--(1, 2);
\draw[brown] (1, -1)--(0, 0);
\begin{scope}[xscale=-1]
\draw[brown] (0, 1)--(1, 2);
\draw[blue, ->-=0.5] (1, 2)--(1, -1);
\draw[brown] (1, -1)--(0, 0);
\draw[fill=white] (0.6, 1.5) rectangle (1.4, 2.3);
\draw[fill=white] (0.6, -1.3) rectangle (1.4, -0.5);
\draw (1, 1.9) node{\tiny$\psi$};
\draw (1, -.9) node{\tiny$\psi^{*}$};
\end{scope}
\end{tikzpicture}}}
\end{align*}
By taking sum of the dual base, we have that Equation \eqref{eq:qtwist} is true.
Then we see that Equation \eqref{eq:tubetwist} is true.
Note that by Move 1 and Move 2, we have that
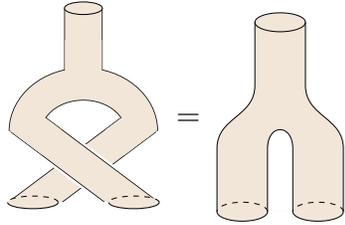
\begin{align*}
\vcenter{\hbox{\scalebox{0.7}{
\begin{tikzpicture}[xscale=0.8, yscale=0.6]
\draw [line width=0.73cm] (0, 4.5)--(0, 6.5);
\draw [brown!20!white, line width=0.7cm] (0, 4.5)--(0, 6.5);
\draw [line width=0.63cm] (-1.5, 0) --(1.5, 3);
\draw [brown!20!white, line width=0.6cm] (-1.5, 0) --(1.5, 3);
\draw [white, line width=0.72cm] (1.5, 0) --(-1.5, 3);
\draw [line width=0.63cm] (1.5, 0) --(-1.5, 3);
\draw [brown!20!white, line width=0.6cm] (1.5, 0) --(-1.5, 3);
\draw [line width=0.63cm] (0,2.6) [partial ellipse=0:180:1.38 and 1.5];
\draw [brown!20!white, line width=0.6cm] (0,2.6) [partial ellipse=0:180:1.38 and 1.5];
\draw [brown!20!white, line width=0.7cm] (0, 4.5)--(0, 5.5);
\draw [brown!20!white, line width=0.6cm] (0.75,2.25) --(1.5, 3);
\draw [brown!20!white, line width=0.6cm] (-0.75,2.25) --(-1.5, 3);
\begin{scope}[shift={(0,6.5)}]
\path [fill=brown!20!white] (0,0) [partial ellipse=0:360:0.45 and 0.2];
\draw (0,0) [partial ellipse=0:360:0.45 and 0.2];
\end{scope}
\path [fill=white] (-2, 0.35) rectangle (2,-1);
\begin{scope}[shift={(-1.2,0.35)}]
\path [fill=brown!20!white] (0,0) [partial ellipse=0:360:0.6 and 0.2];
\draw [dashed] (0,0) [partial ellipse=0:180:0.6 and 0.2];
\draw (0,0) [partial ellipse=180:360:0.6 and 0.2];
\end{scope}
\begin{scope}[shift={(1.2,0.35)}]
\path [fill=brown!20!white] (0,0) [partial ellipse=0:360:0.6 and 0.2];
\draw [dashed] (0,0) [partial ellipse=0:180:0.6 and 0.2];
\draw (0,0) [partial ellipse=180:360:0.6 and 0.2];
\end{scope}
\end{tikzpicture}}}}
=
\vcenter{\hbox{\scalebox{0.7}{
\begin{tikzpicture}[xscale=0.8, yscale=0.6]
\path [fill=brown!20!white] (-0.6, 2)--(-0.6, 0) .. controls +(0, -1) and +(0, 1).. (-1.5, -2)--(-1.5, -4) -- (1.5, -4)--(1.5, -2) .. controls +(0, 1) and +(0, -1).. (0.6, 0)--(0.6, 2)--(-0.6, 2) ;
\begin{scope}[shift={(0,2)}]
\path [fill=brown!20!white] (0,0) [partial ellipse=0:360:0.6 and 0.3];
\draw (0,0) [partial ellipse=0:360:0.6 and 0.3];
\end{scope}
\draw (-0.6, 2)--(-0.6, 0) .. controls +(0, -1) and +(0, 1).. (-1.5, -2)--(-1.5, -4);
\draw (0.6, 2)--(0.6, 0).. controls +(0, -1) and +(0, 1).. (1.5, -2)--(1.5, -4); 
\path [fill=white] (-0.3, -4)--(-0.3, -2).. controls +(0, 0.8) and +(0, 0.8) .. (0.3, -2)--(0.3, -4)--(-0.3, -4);
\draw (-0.3, -4)--(-0.3, -2).. controls +(0, 0.8) and +(0, 0.8) .. (0.3, -2)--(0.3, -4);
\begin{scope}[shift={(-0.9,-4)}]
\path [fill=brown!20!white] (0,0) [partial ellipse=180:360:0.6 and 0.3];
\draw [dashed](0,0) [partial ellipse=0:180:0.6 and 0.3];
\draw (0,0) [partial ellipse=180:360:0.6 and 0.3];
\end{scope}
\begin{scope}[shift={(0.9,-4)}]
\path [fill=brown!20!white] (0,0) [partial ellipse=180:360:0.6 and 0.3];
\draw [dashed](0,0) [partial ellipse=0:180:0.6 and 0.3];
\draw (0,0) [partial ellipse=180:360:0.6 and 0.3];
\end{scope}
\end{tikzpicture}}}}.
\end{align*}
By gluing $\vcenter{\hbox{\scalebox{0.7}{
\begin{tikzpicture}[xscale=0.8, yscale=0.6]
\path  [fill=brown!20!white] (-0.6, -2) rectangle (0.6, 1);
\path  [fill=white] (-0.6, -1) rectangle (0.6, 1);
\begin{scope}[shift={(0,-1)}]
\path  [fill=white] (0,0) [partial ellipse=180:360:0.6 and 0.3];
\end{scope}
\path [fill=brown!20!white] (-0.2, -1) rectangle (0.2, 0.8);
\begin{scope}[shift={(0,1)}]
\path  [fill=white] (0,0) [partial ellipse=0:360:0.6 and 0.3];
\draw (0,0) [partial ellipse=0:360:0.6 and 0.3];
\end{scope}
\draw (-0.6, 1)--(-0.6, 0) (0.6, 1)--(0.6, 0); 
\draw (-0.6, -2)--(-0.6, 0) (0.6, -2)--(0.6, 0);
\begin{scope}[shift={(0,-2)}]
\path  [fill=brown!20!white] (0,0) [partial ellipse=180:360:0.6 and 0.3];
\draw [dashed](0,0) [partial ellipse=0:180:0.6 and 0.3];
\draw (0,0) [partial ellipse=180:360:0.6 and 0.3];
\end{scope}
\begin{scope}[shift={(0,-1)}]
\draw [dashed, blue](0,0) [partial ellipse=0:180:0.6 and 0.3];
\draw [blue] (0,0) [partial ellipse=180:360:0.6 and 0.3];
\end{scope}
\path [fill=brown!20!white] (-0.2, -1.4) rectangle (0.2, 0);
\draw [blue, -<-=0.75] (-0.2, 0.7)--(-0.2, -1.3) (0.2, 0.7)--(0.2, -1.3) node[black, pos=0.5, right]{\tiny $J$};
\end{tikzpicture}}}}$ to each corner and create a circle labelled by $J$ on the surface, we obtain that
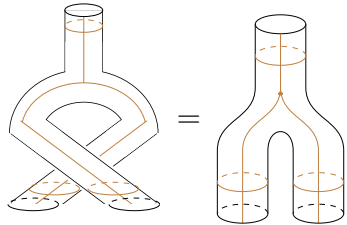
\begin{align*}
    \vcenter{\hbox{\scalebox{0.7}{
\begin{tikzpicture}[xscale=0.8, yscale=0.6]
\draw [line width=0.73cm] (0, 4.5)--(0, 6.5);
\draw [white, line width=0.7cm] (0, 4.5)--(0, 6.5);
\draw [line width=0.63cm] (-1.5, 0) --(1.5, 3);
\draw [white, line width=0.6cm] (-1.5, 0) --(1.5, 3);
\draw [brown] (-1.5, 0) --(1.5, 3);
\draw [white, line width=0.72cm] (1.5, 0) --(-1.5, 3);
\draw [line width=0.63cm] (1.5, 0) --(-1.5, 3);
\draw [white, line width=0.6cm] (1.5, 0) --(-1.5, 3);
\draw [line width=0.63cm] (0,2.6) [partial ellipse=0:180:1.38 and 1.5];
\draw [white, line width=0.6cm] (0,2.6) [partial ellipse=0:180:1.38 and 1.5];
\draw [white, line width=0.7cm] (0, 4.5)--(0, 5.5);
\draw [white, line width=0.6cm] (0.75,2.25) --(1.5, 3);
\draw [brown] (0.75,2.25) --(1.5, 3);
\draw [white, line width=0.6cm] (-0.75,2.25) --(-1.5, 3);
\draw [brown] (0, 4.2)--(0, 6.3);
\draw [brown] (1.5, 0) --(-1.5, 3);
\draw [brown] (0,2.7) [partial ellipse=10:170:1.48 and 1.5];
\begin{scope}[shift={(0,6.5)}]
\draw (0,0) [partial ellipse=0:360:0.45 and 0.2];
\end{scope}
\draw [brown,dashed] (0,6) [partial ellipse=0:180:0.45 and 0.2];
\draw [brown] (0,6) [partial ellipse=180:360:0.45 and 0.2];
\path [fill=white] (-2, 0.35) rectangle (2,-1);
\begin{scope}[shift={(-1.2,0.35)}]
\draw [dashed] (0,0) [partial ellipse=0:180:0.6 and 0.2];
\draw (0,0) [partial ellipse=180:360:0.6 and 0.2];
\end{scope}
\begin{scope}[shift={(1.2,0.35)}]
\draw [dashed] (0,0) [partial ellipse=0:180:0.6 and 0.2];
\draw (0,0) [partial ellipse=180:360:0.6 and 0.2];
\end{scope}
\begin{scope}[shift={(0.7,0.85)}]
\draw [brown, dashed] (0,0) [partial ellipse=0:180:0.6 and 0.2];
\draw [brown] (0,0) [partial ellipse=180:360:0.6 and 0.2];
\end{scope}
\begin{scope}[shift={(-0.7,0.85)}]
\draw [brown, dashed] (0,0) [partial ellipse=0:180:0.6 and 0.2];
\draw [brown] (0,0) [partial ellipse=180:360:0.6 and 0.2];
\end{scope}
\end{tikzpicture}}}}
=
\vcenter{\hbox{\scalebox{0.7}{
\begin{tikzpicture}[xscale=0.8, yscale=0.6]
\begin{scope}[shift={(0,2)}]
\draw (0,0) [partial ellipse=0:360:0.6 and 0.3];
\end{scope}
\draw (-0.6, 2)--(-0.6, 0) .. controls +(0, -1) and +(0, 1).. (-1.5, -2)--(-1.5, -4);
\draw (0.6, 2)--(0.6, 0).. controls +(0, -1) and +(0, 1).. (1.5, -2)--(1.5, -4); 
\draw (-0.3, -4)--(-0.3, -2).. controls +(0, 0.8) and +(0, 0.8) .. (0.3, -2)--(0.3, -4);
\draw [brown] (0, 1.7)--(0, 0) .. controls +(0, -1) and +(0, 1).. (-0.9, -2)--(-0.9, -4.3);
\draw [brown] (0, 1.7)--(0, 0) .. controls +(0, -1) and +(0, 1).. (0.9, -2)--(0.9, -4.3);
\path [fill=brown] (0, -0.2) circle (0.06cm);
\begin{scope}[shift={(-0.9,-4)}]
\draw [dashed](0,0) [partial ellipse=0:180:0.6 and 0.3];
\draw (0,0) [partial ellipse=180:360:0.6 and 0.3];
\end{scope}
\begin{scope}[shift={(0.9,-4)}]
\draw [dashed](0,0) [partial ellipse=0:180:0.6 and 0.3];
\draw (0,0) [partial ellipse=180:360:0.6 and 0.3];
\end{scope}
\begin{scope}[shift={(0.9,-3)}]
\draw [brown, dashed](0,0) [partial ellipse=0:180:0.6 and 0.3];
\draw [brown] (0,0) [partial ellipse=180:360:0.6 and 0.3];
\end{scope}
\begin{scope}[shift={(-0.9,-3)}]
\draw [brown, dashed](0,0) [partial ellipse=0:180:0.6 and 0.3];
\draw [brown] (0,0) [partial ellipse=180:360:0.6 and 0.3];
\end{scope}
\begin{scope}[shift={(0,1)}]
\draw [brown, dashed](0,0) [partial ellipse=0:180:0.6 and 0.3];
\draw [brown] (0,0) [partial ellipse=180:360:0.6 and 0.3];
\end{scope}
\end{tikzpicture}}}}
\end{align*}
Now we glue $\vcenter{\hbox{\scalebox{0.7}{
\begin{tikzpicture}[xscale=0.8, yscale=0.6]
\draw [red] (0,0) [partial ellipse=180:265:0.6 and 0.3];
\draw [red] (0,0) [partial ellipse=275:360:0.6 and 0.3];
\draw [red, dashed] (0,0) [partial ellipse=0:180:0.6 and 0.3];
\begin{scope}[shift={(0,2)}]
\draw (0,0) [partial ellipse=0:360:0.6 and 0.3];
\end{scope}
\draw (-0.6, 2)--(-0.6, 0) (0.6, 2)--(0.6, 0); 
\draw (-0.6, -2)--(-0.6, 0) (0.6, -2)--(0.6, 0);
\draw [brown] (0, -2.3) --(0, 1.7);
\begin{scope}[shift={(0,-2)}]
\draw [dashed](0,0) [partial ellipse=0:180:0.6 and 0.3];
\draw (0,0) [partial ellipse=180:360:0.6 and 0.3];
\end{scope}
\end{tikzpicture}}}}$ to each corner, apply Equation \eqref{eq:tubetwist}, we obtain that 
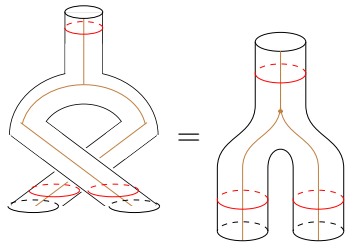
\begin{align*}
    \vcenter{\hbox{\scalebox{0.7}{
\begin{tikzpicture}[xscale=0.8, yscale=0.6]
\draw [line width=0.73cm] (0, 4.5)--(0, 6.5);
\draw [white, line width=0.7cm] (0, 4.5)--(0, 6.5);
\draw [line width=0.63cm] (-1.5, 0) --(1.5, 3);
\draw [white, line width=0.6cm] (-1.5, 0) --(1.5, 3);
\draw [brown] (-1.5, 0) --(1.5, 3);
\draw [white, line width=0.72cm] (1.5, 0) --(-1.5, 3);
\draw [line width=0.63cm] (1.5, 0) --(-1.5, 3);
\draw [white, line width=0.6cm] (1.5, 0) --(-1.5, 3);
\draw [line width=0.63cm] (0,2.6) [partial ellipse=0:180:1.38 and 1.5];
\draw [white, line width=0.6cm] (0,2.6) [partial ellipse=0:180:1.38 and 1.5];
\draw [white, line width=0.7cm] (0, 4.5)--(0, 5.5);
\draw [white, line width=0.6cm] (0.75,2.25) --(1.5, 3);
\draw [brown] (0.75,2.25) --(1.5, 3);
\draw [white, line width=0.6cm] (-0.75,2.25) --(-1.5, 3);
\draw [brown] (0, 4.2)--(0, 6.3);
\draw [brown] (1.5, 0) --(-1.5, 3);
\draw [brown] (0,2.7) [partial ellipse=10:170:1.48 and 1.5];
\begin{scope}[shift={(0,6.5)}]
\draw (0,0) [partial ellipse=0:360:0.45 and 0.2];
\end{scope}
\draw [red,dashed] (0,6) [partial ellipse=0:180:0.45 and 0.2];
\draw [red] (0,6) [partial ellipse=180:360:0.45 and 0.2];
\path [fill=white] (-2, 0.35) rectangle (2,-2);
\begin{scope}[shift={(-1.2,0.35)}]
\draw [dashed] (0,0) [partial ellipse=0:180:0.6 and 0.2];
\draw (0,0) [partial ellipse=180:360:0.6 and 0.2];
\end{scope}
\begin{scope}[shift={(1.2,0.35)}]
\draw [dashed] (0,0) [partial ellipse=0:180:0.6 and 0.2];
\draw (0,0) [partial ellipse=180:360:0.6 and 0.2];
\end{scope}
\begin{scope}[shift={(0.7,0.85)}]
\draw [red, dashed] (0,0) [partial ellipse=0:180:0.6 and 0.2];
\draw [red] (0,0) [partial ellipse=180:360:0.6 and 0.2];
\end{scope}
\begin{scope}[shift={(-0.7,0.85)}]
\draw [red, dashed] (0,0) [partial ellipse=0:180:0.6 and 0.2];
\draw [red] (0,0) [partial ellipse=180:360:0.6 and 0.2];
\end{scope}
\end{tikzpicture}}}}
=
\vcenter{\hbox{\scalebox{0.7}{
\begin{tikzpicture}[xscale=0.8, yscale=0.6]
\begin{scope}[shift={(0,2)}]
\draw (0,0) [partial ellipse=0:360:0.6 and 0.3];
\end{scope}
\draw (-0.6, 2)--(-0.6, 0) .. controls +(0, -1) and +(0, 1).. (-1.5, -2)--(-1.5, -4);
\draw (0.6, 2)--(0.6, 0).. controls +(0, -1) and +(0, 1).. (1.5, -2)--(1.5, -4); 
\draw (-0.3, -4)--(-0.3, -2).. controls +(0, 0.8) and +(0, 0.8) .. (0.3, -2)--(0.3, -4);
\draw [brown] (0, 1.7)--(0, 0) .. controls +(0, -1) and +(0, 1).. (-0.9, -2)--(-0.9, -4.3);
\draw [brown] (0, 1.7)--(0, 0) .. controls +(0, -1) and +(0, 1).. (0.9, -2)--(0.9, -4.3);
\path [fill=brown] (0, -0.2) circle (0.06cm);
\begin{scope}[shift={(-0.9,-4)}]
\draw [dashed](0,0) [partial ellipse=0:180:0.6 and 0.3];
\draw (0,0) [partial ellipse=180:360:0.6 and 0.3];
\end{scope}
\begin{scope}[shift={(0.9,-4)}]
\draw [dashed](0,0) [partial ellipse=0:180:0.6 and 0.3];
\draw (0,0) [partial ellipse=180:360:0.6 and 0.3];
\end{scope}
\begin{scope}[shift={(0.9,-3)}]
\draw [red, dashed](0,0) [partial ellipse=0:180:0.6 and 0.3];
\draw [red] (0,0) [partial ellipse=180:360:0.6 and 0.3];
\end{scope}
\begin{scope}[shift={(-0.9,-3)}]
\draw [red, dashed](0,0) [partial ellipse=0:180:0.6 and 0.3];
\draw [red] (0,0) [partial ellipse=180:360:0.6 and 0.3];
\end{scope}
\begin{scope}[shift={(0,1)}]
\draw [red, dashed](0,0) [partial ellipse=0:180:0.6 and 0.3];
\draw [red] (0,0) [partial ellipse=180:360:0.6 and 0.3];
\end{scope}
\end{tikzpicture}}}}
\end{align*}

Now we attach it to a plane, we see that Equation \eqref{eq:comm} is true, i.e. the Frobenius algebra is commutative.

\end{proof}

\section{Topological Full Center}
The full center studied in \cite{KR08} of a Frobenius algebra is nothing but a $\mathcal{D}$-color tube in the alterfold TQFT.
We shall describe the interpretation as follows.
Let $\cC$ be a modular fusion category and $Q$ Frobenius algebra in $\cC$, recall that by M\"uger's construction, we obtain a Morita context with the dual fusion category $\cD$ equivalent to the opposite category of the category of $Q$-$Q$-bimodules in $\cC$, which is denoted by $(_Q\cC_Q)^{op}$. 
In this case, the full center $Z(Q)$ of $Q$ is defined to be the object $\displaystyle Z(Q)=\bigoplus_{j,k=1}^{r} z_{kj} X_k \boxtimes X_j^*$ in $\mathcal{Z(C)} = \cC \boxtimes\cC^{op}$, where $(z_{kj})_{k,j}$ is the modular invariant matrix associated to $Q$ \cite{FFRS06}. 
In the alterfold TQFT, the full center $Z(Q)$ is depicted as
\begin{align}\label{eq:fullcenter}
 \sum_{j,k=1}^r   \frac{1}{\mu^2} \vcenter{\hbox{\scalebox{0.7}{
\begin{tikzpicture}[scale=0.7]
\draw [line width=0.6cm, brown!20!white] (0,0) [partial ellipse=-0.1:180.1:2 and 1.5];
\begin{scope}[shift={(2.5, 0)}]
\draw [line width=0.6cm] (0,0) [partial ellipse=0:180:2 and 1.5];
\draw [white, line width=0.57cm] (0,0) [partial ellipse=-0.1:180.1:2 and 1.5];
\draw [blue] (0,0) [partial ellipse=0:180:2.15 and 1.65];
\draw [blue] (0,0) [partial ellipse=0:180:1.85 and 1.35];
\end{scope} 
\begin{scope}[shift={(2.5, 0)}]
\draw [line width=0.6cm] (0,0) [partial ellipse=180:360:2 and 1.5];
\draw [white, line width=0.57cm] (0,0) [partial ellipse=178:362:2 and 1.5];
\draw [blue, -<-=0.5] (0,0) [partial ellipse=178:362:2.15 and 1.65] node[black, pos=0.7,below ] {\tiny $X_j$}; 
\draw [blue, ->-=0.5] (0,0) [partial ellipse=178:362:1.85 and 1.35] node[black, pos=0.7,above ] {\tiny $X_k$};
\end{scope}
\draw [line width=0.6cm, brown!20!white] (0,0) [partial ellipse=180:360:2 and 1.5];
\begin{scope}[shift={(4.5, 0)}]
\draw [red, dashed](0,0) [partial ellipse=0:180:0.4 and 0.25] ;
\draw [white, line width=4pt] (0,0) [partial ellipse=290:270:0.4 and 0.25];
\draw [red] (0,0) [partial ellipse=260:360:0.4 and 0.25];
\draw [red] (0,0) [partial ellipse=180:230:0.4 and 0.25];
\end{scope}
\end{tikzpicture}}}}
\frac{1}{\mu}
\vcenter{\hbox{\scalebox{0.7}{
\begin{tikzpicture}[xscale=0.8, yscale=0.6]
\begin{scope}[shift={(0,3)}]
\draw (0,0) [partial ellipse=0:360:0.6 and 0.3];
\end{scope}
\path [fill=white] (-0.6, 0) rectangle (0.6, 2.7);
\draw [blue, ->-=0.5] (-0.2, 2.8)--(-0.2, 0) node [left, pos=0.6] {\tiny $X_j$} ;
\draw [blue, -<-=0.5] (0.2, 2.8)--(0.2, 0) node [right, pos=0.6] {\tiny $X_k$}; 
\draw (-0.6, 3)--(-0.6, 0) (0.6, 3)--(0.6, 0); 
\draw [blue] (-0.2, -3.2)--(-0.2, 0) (0.2, -3.2)--(0.2, 0);
\draw (-0.6, -3)--(-0.6, 0) (0.6, -3)--(0.6, 0);
\begin{scope}[shift={(0,-3)}]
\draw [dashed](0,0) [partial ellipse=0:180:0.6 and 0.3];
\draw (0,0) [partial ellipse=180:360:0.6 and 0.3];
\end{scope}
\draw [red, dashed](0,0) [partial ellipse=0:180:0.6 and 0.3]; 
\draw [white, line width=4pt]  (0,0) [partial ellipse=220:270:0.6 and 0.3];
\draw [red]  (0,0) [partial ellipse=180:280:0.6 and 0.3];
\draw [red]  (0,0) [partial ellipse=300:360:0.6 and 0.3];
\end{tikzpicture}}}}
\end{align}
Applying the Kirby color, we see the full center is isomorphic to the following morphism
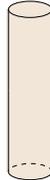
\begin{align*}
\vcenter{\hbox{\scalebox{0.6}{
\begin{tikzpicture}[xscale=0.8, yscale=0.6]
\path [fill=brown!20!white](-0.65, -3) rectangle (0.65, 3);
\begin{scope}[shift={(0,3)}]
\path [fill=brown!20!white] (0,0) [partial ellipse=0:180:0.6 and 0.3];
\draw (0,0) [partial ellipse=0:360:0.6 and 0.3];
\end{scope}
\draw (-0.6, 3)--(-0.6, 0) (0.6, 3)--(0.6, 0); 
\draw (-0.6, -3)--(-0.6, 0) (0.6, -3)--(0.6, 0);
\begin{scope}[shift={(0,-3)}]
\path [fill=brown!20!white] (0,0) [partial ellipse=180:360:0.6 and 0.3];
\draw [dashed](0,0) [partial ellipse=0:180:0.6 and 0.3];
\draw (0,0) [partial ellipse=180:360:0.6 and 0.3];
\end{scope}
\end{tikzpicture}}}}
\end{align*}
via $\vcenter{\hbox{\scalebox{0.7}{
\begin{tikzpicture}[xscale=0.8, yscale=0.6]
\path  [fill=brown!20!white] (-0.6, -2) rectangle (0.6, 1);
\path  [fill=white] (-0.6, -1) rectangle (0.6, 1);
\begin{scope}[shift={(0,-1)}]
\path  [fill=white] (0,0) [partial ellipse=180:360:0.6 and 0.3];
\end{scope}
\path [fill=brown!20!white] (-0.2, -1) rectangle (0.2, 0.8);
\begin{scope}[shift={(0,1)}]
\path  [fill=white] (0,0) [partial ellipse=0:360:0.6 and 0.3];
\draw (0,0) [partial ellipse=0:360:0.6 and 0.3];
\end{scope}
\draw (-0.6, 1)--(-0.6, 0) (0.6, 1)--(0.6, 0); 
\draw (-0.6, -2)--(-0.6, 0) (0.6, -2)--(0.6, 0);
\begin{scope}[shift={(0,-2)}]
\path  [fill=brown!20!white] (0,0) [partial ellipse=180:360:0.6 and 0.3];
\draw [dashed](0,0) [partial ellipse=0:180:0.6 and 0.3];
\draw (0,0) [partial ellipse=180:360:0.6 and 0.3];
\end{scope}
\begin{scope}[shift={(0,-1)}]
\draw [dashed, blue](0,0) [partial ellipse=0:180:0.6 and 0.3];
\draw [blue] (0,0) [partial ellipse=180:360:0.6 and 0.3];
\end{scope}
\path [fill=brown!20!white] (-0.2, -1.4) rectangle (0.2, 0);
\draw [blue, -<-=0.75] (-0.2, 0.7)--(-0.2, -1.3) (0.2, 0.7)--(0.2, -1.3) node[black, pos=0.5, right]{\tiny $J$};
\end{tikzpicture}}}}$.
We shall call it as the topological full center of $Q$ in the alterfold TQFT.
Note that the definition of a Frobenius algebra of a modular fusion category in this article is the same as that in \cite{Mug03b}, which are automatically nondegenerate in the sense of \cite[Definition 2.1]{KR08}. 

Two algebras $Q$ and $Q'$ in a monoidal category $\cC$ are Morita equivalent if there exist an $Q$-$Q'$ bimodule $M \in {}_Q\cC_{Q'}$ and an $Q'$-$Q$-bimodule $N\in {}_{Q'}\cC_Q$, such that $Q\otimes_{Q'} N\cong Q$ as $Q$-$Q$-bimodules and $N\otimes_{Q} M\cong Q'$ as $Q'$-$Q'$-bimodules. 
In \cite{KR08}, Kong and Runkel showed that if $\cC$ is a modular fusion category and $Q$, $Q'$ are Frobenius algebras in $\cC$, then $Q$ and $Q'$ are Morita equivalent if and only if their full centers $Z(Q)$ and $Z(Q')$ are isomorphic as algebras in $\mathcal{Z(C)}$. 
In this section, we shall present a simple topological proof in the alterfold TQFT.

The Morita equivalence of algebras above admits the following graphical calculus in alterfold TQFT. 
Let $Q$, $Q'$ be Morita equivalent Frobenius algebras in a modular fusion category $\cC$ with bimodules $M \in {}_Q\cC_{Q'}$ and $N\in {}_{Q'}\cC_Q$ as above. 
Then in the multi-colored tube category, we color surfaces by $\mathcal{D} = {}_{Q}\cC_{Q}$ and $\mathcal{E}={}_{Q'}\cC_{Q'}$, and place tensor diagrams of $\cD$ and $\cE$ in the corresponding colored region. 
The bimodule $M$ is drawn as a domain wall from the $\mathcal{D}$-colored region to the $\mathcal{E}$-colored region when reading from top to bottom; and similarly, $N$ is drawn as a domain wall from the $\cE$-colored region to $\cD$-colored region. The condition that $M\otimes_{Q'} N\cong Q$ and $N \otimes_Q M \cong Q'$ implies the following equalities of alterfolds:
\begin{equation}\label{eq:Mrt}
\vcenter{\hbox{\scalebox{0.5}{
\begin{tikzpicture}[yscale=0.7 ]
\path [fill=brown!20!white](-0.65, -3) rectangle (0.65, 3);
\begin{scope}[shift={(0,3)}]
\path [fill=brown!20!white] (0,0) [partial ellipse=0:180:0.6 and 0.3];
\draw (0,0) [partial ellipse=0:360:0.6 and 0.3];
\end{scope}
\draw (-0.6, 3)--(-0.6, 0) (0.6, 3)--(0.6, 0); 
\draw (-0.6, -3)--(-0.6, 0) (0.6, -3)--(0.6, 0);
\begin{scope}[shift={(0,-3)}]
\path [fill=brown!20!white] (0,0) [partial ellipse=180:360:0.6 and 0.3];
\draw [dashed](0,0) [partial ellipse=0:180:0.6 and 0.3];
\draw (0,0) [partial ellipse=180:360:0.6 and 0.3];
\end{scope}
\draw[blue, ->-=0.75] (0, 0.8) [partial ellipse=0:360:0.6 and 0.3];
\draw[blue, ->-=0.75] (0, -0.8) [partial ellipse=0:360:0.6 and 0.3];
\draw (0.8, 0.8) node{\tiny$M$};
\draw (0.8, -0.8) node{\tiny$N$};
\draw (0, 1.6) node{\tiny $\mathcal{D}$};
\draw (0, -0.1) node{\tiny $\mathcal{E}$};
\draw (0, -2) node{\tiny $\mathcal{D}$};
\end{tikzpicture}}}}
=
\vcenter{\hbox{\scalebox{0.6}{
\begin{tikzpicture}[yscale=0.7 ]
\path [fill=brown!20!white](-0.65, -3) rectangle (0.65, 3);
\begin{scope}[shift={(0,3)}]
\path [fill=brown!20!white] (0,0) [partial ellipse=0:180:0.6 and 0.3];
\draw (0,0) [partial ellipse=0:360:0.6 and 0.3];
\end{scope}
\draw (-0.6, 3)--(-0.6, 0) (0.6, 3)--(0.6, 0); 
\draw (-0.6, -3)--(-0.6, 0) (0.6, -3)--(0.6, 0);
\begin{scope}[shift={(0,-3)}]
\path [fill=brown!20!white] (0,0) [partial ellipse=180:360:0.6 and 0.3];
\draw [dashed](0,0) [partial ellipse=0:180:0.6 and 0.3];
\draw (0,0) [partial ellipse=180:360:0.6 and 0.3];
\end{scope}
\draw (0, 0) node{\tiny $\mathcal{D}$};
\end{tikzpicture}}}}\,,\quad
\vcenter{\hbox{\scalebox{0.6}{
\begin{tikzpicture}[yscale=0.7 ]
\path [fill=brown!20!white](-0.65, -3) rectangle (0.65, 3);
\begin{scope}[shift={(0,3)}]
\path [fill=brown!20!white] (0,0) [partial ellipse=0:180:0.6 and 0.3];
\draw (0,0) [partial ellipse=0:360:0.6 and 0.3];
\end{scope}
\draw (-0.6, 3)--(-0.6, 0) (0.6, 3)--(0.6, 0); 
\draw (-0.6, -3)--(-0.6, 0) (0.6, -3)--(0.6, 0);
\begin{scope}[shift={(0,-3)}]
\path [fill=brown!20!white] (0,0) [partial ellipse=180:360:0.6 and 0.3];
\draw [dashed](0,0) [partial ellipse=0:180:0.6 and 0.3];
\draw (0,0) [partial ellipse=180:360:0.6 and 0.3];
\end{scope}
\draw[blue, ->-=0.75] (0, 0.8) [partial ellipse=0:360:0.6 and 0.3];
\draw[blue, ->-=0.75] (0, -0.8) [partial ellipse=0:360:0.6 and 0.3];
\draw (0.8, 0.8) node{\tiny$N$};
\draw (0.8, -0.8) node{\tiny$M$};
\draw (0, 1.6) node{\tiny $\mathcal{E}$};
\draw (0, -0.1) node{\tiny $\mathcal{D}$};
\draw (0, -2) node{\tiny $\mathcal{E}$};
\end{tikzpicture}}}}
=
\vcenter{\hbox{\scalebox{0.6}{
\begin{tikzpicture}[yscale=0.7 ]
\path [fill=brown!20!white](-0.65, -3) rectangle (0.65, 3);
\begin{scope}[shift={(0,3)}]
\path [fill=brown!20!white] (0,0) [partial ellipse=0:180:0.6 and 0.3];
\draw (0,0) [partial ellipse=0:360:0.6 and 0.3];
\end{scope}
\draw (-0.6, 3)--(-0.6, 0) (0.6, 3)--(0.6, 0); 
\draw (-0.6, -3)--(-0.6, 0) (0.6, -3)--(0.6, 0);
\begin{scope}[shift={(0,-3)}]
\path [fill=brown!20!white] (0,0) [partial ellipse=180:360:0.6 and 0.3];
\draw [dashed](0,0) [partial ellipse=0:180:0.6 and 0.3];
\draw (0,0) [partial ellipse=180:360:0.6 and 0.3];
\end{scope}
\draw (0, 0) node{\tiny $\mathcal{E}$};
\end{tikzpicture}}}}\,,
\end{equation}
which simply means that the morphism 
$\vcenter{\hbox{\scalebox{0.5}{
\begin{tikzpicture}[yscale=0.7 ]
\path [fill=brown!20!white](-0.65, -3) rectangle (0.65, 3);
\begin{scope}[shift={(0,3)}]
\path [fill=brown!20!white] (0,0) [partial ellipse=0:180:0.6 and 0.3];
\draw (0,0) [partial ellipse=0:360:0.6 and 0.3];
\end{scope}
\draw (-0.6, 3)--(-0.6, 0) (0.6, 3)--(0.6, 0); 
\draw (-0.6, -3)--(-0.6, 0) (0.6, -3)--(0.6, 0);
\begin{scope}[shift={(0,-3)}]
\path [fill=brown!20!white] (0,0) [partial ellipse=180:360:0.6 and 0.3];
\draw [dashed](0,0) [partial ellipse=0:180:0.6 and 0.3];
\draw (0,0) [partial ellipse=180:360:0.6 and 0.3];
\end{scope}
\draw[blue, ->-=0.75] (0, 0) [partial ellipse=0:360:0.6 and 0.3];
\draw (0, -0.6) node{\tiny$M$};
\draw (0, 1.6) node{\tiny $\mathcal{D}$};
\draw (0, -1.6) node{\tiny $\mathcal{E}$};
\end{tikzpicture}}}}$
is an isomorphism between the full centers $Z(Q)$ and $Z(Q')$ as objects. 
To verify that it is an algebra homomorphism, note that Equation \eqref{eq:Mrt} implies that the Frobenius-Perron dimensions of $M$ and $N$ are both $1$, and this invertibility can be easily translated into the equality 
$$\vcenter{\hbox{\scalebox{0.6}{
\begin{tikzpicture}[xscale=0.8, yscale=-0.6]
\path [fill=brown!20!white] (-0.6, 2)--(-0.6, 0) .. controls +(0, -1) and +(0, 1).. (-1.5, -2)--(-1.5, -4) -- (1.5, -4)--(1.5, -2) .. controls +(0, 1) and +(0, -1).. (0.6, 0)--(0.6, 2)--(-0.6, 2) ;
\begin{scope}[shift={(0,2)}]
\path [fill=brown!20!white] (0,0) [partial ellipse=0:360:0.6 and 0.3];
\draw [dashed](0,0) [partial ellipse=0:360:0.6 and 0.3];
\draw (0,0) [partial ellipse=0:180:0.6 and 0.3];
\end{scope}
\draw (-0.6, 2)--(-0.6, 0) .. controls +(0, -1) and +(0, 1).. (-1.5, -2)--(-1.5, -4);
\draw (0.6, 2)--(0.6, 0).. controls +(0, -1) and +(0, 1).. (1.5, -2)--(1.5, -4); 
\path [fill=white] (-0.3, -4)--(-0.3, -2).. controls +(0, 0.8) and +(0, 0.8) .. (0.3, -2)--(0.3, -4)--(-0.3, -4);
\draw (-0.3, -4)--(-0.3, -2).. controls +(0, 0.8) and +(0, 0.8) .. (0.3, -2)--(0.3, -4);
\begin{scope}[shift={(-0.9,-4)}]
\path [fill=brown!20!white] (0,0) [partial ellipse=180:360:0.6 and 0.3];
\draw (0,0) [partial ellipse=0:180:0.6 and 0.3];
\draw (0,0) [partial ellipse=180:360:0.6 and 0.3];
\end{scope}
\begin{scope}[shift={(0.9,-4)}]
\path [fill=brown!20!white] (0,0) [partial ellipse=180:360:0.6 and 0.3];
\draw (0,0) [partial ellipse=0:180:0.6 and 0.3];
\draw (0,0) [partial ellipse=180:360:0.6 and 0.3];
\end{scope}
\draw[blue, ->-=0.75] (0, 0) [partial ellipse=360:0:0.6 and 0.3];
\node at (0,-0.9) {$\cD$};\node at (0,1) {$\cE$};
\node at (1,0.2) {$M$};
\end{tikzpicture}}}}
=
\vcenter{\hbox{\scalebox{0.6}{
\begin{tikzpicture}[xscale=0.8, yscale=-0.6]
\path [fill=brown!20!white] (-0.6, 2)--(-0.6, 0) .. controls +(0, -1) and +(0, 1).. (-1.5, -2)--(-1.5, -4) -- (1.5, -4)--(1.5, -2) .. controls +(0, 1) and +(0, -1).. (0.6, 0)--(0.6, 2)--(-0.6, 2) ;
\begin{scope}[shift={(0,2)}]
\path [fill=brown!20!white] (0,0) [partial ellipse=0:360:0.6 and 0.3];
\draw [dashed](0,0) [partial ellipse=0:360:0.6 and 0.3];
\draw (0,0) [partial ellipse=0:180:0.6 and 0.3];
\end{scope}
\draw (-0.6, 2)--(-0.6, 0) .. controls +(0, -1) and +(0, 1).. (-1.5, -2)--(-1.5, -4);
\draw (0.6, 2)--(0.6, 0).. controls +(0, -1) and +(0, 1).. (1.5, -2)--(1.5, -4); 
\path [fill=white] (-0.3, -4)--(-0.3, -2).. controls +(0, 0.8) and +(0, 0.8) .. (0.3, -2)--(0.3, -4)--(-0.3, -4);
\draw (-0.3, -4)--(-0.3, -2).. controls +(0, 0.8) and +(0, 0.8) .. (0.3, -2)--(0.3, -4);
\begin{scope}[shift={(-0.9,-4)}]
\path [fill=brown!20!white] (0,0) [partial ellipse=180:360:0.6 and 0.3];
\draw (0,0) [partial ellipse=0:180:0.6 and 0.3];
\draw (0,0) [partial ellipse=180:360:0.6 and 0.3];
\end{scope}
\begin{scope}[shift={(0.9,-4)}]
\path [fill=brown!20!white] (0,0) [partial ellipse=180:360:0.6 and 0.3];
\draw (0,0) [partial ellipse=0:180:0.6 and 0.3];
\draw (0,0) [partial ellipse=180:360:0.6 and 0.3];
\end{scope}
\draw[blue, ->-=0.75] (-0.9,-2.5) [partial ellipse=360:0:0.6 and 0.3];
\draw[blue, ->-=0.75] (0.9,-2.5) [partial ellipse=360:0:0.6 and 0.3];
\node at (-0.9,-3.2) {$\cD$};\node at (0.9,-3.2) {$\cD$};
\node at (0,0) {$\cE$}; 
\node at (-2,-2.5) {$Q$};\node at (2,-2.5) {$Q$};
\end{tikzpicture}}}}\,.
$$
Therefore, two Morita equivalent Frobenius algebras have isomorphic full centers.

Conversely, assume that $\phi: Z(Q) \xrightarrow{\cong} Z(Q')$ as algebras in $\mathcal{Z(C)}$, we argue that $Q$ and $Q'$ are Morita equivalent. 
Note that the isomorphism $\phi$ induces a functor $\cD = \mathcal{Z(C)}_{Z(Q)} \to \mathcal{Z(C)}_{Z(Q')} = \cE$, which, by the Eilenberg-Watts theorem (see \cite[Proposition 7.11.1]{EGNO15}), can be identified as an object $\phi \in {}_{Q}\cC_{Q'}$ by an abuse of notations. 
Since ${}_{Q}\cC_{Q'}$ is semisimple, we can write $\displaystyle \phi = \bigoplus_{i} a_i M_i$ for some nonnegative integers $a_i$, where the sum is taken over the (isomorphism classes) of simple objects of ${}_{Q}\cC_{Q'}$. 
Then $\phi$ is depicted in the alterfold TQFT as follows:
\begin{equation}\label{eq:phi}
\phi=\sum_{i}a_i
\vcenter{\hbox{\scalebox{0.6}{
\begin{tikzpicture}[yscale=0.7 ]
\path [fill=brown!20!white](-0.65, -3) rectangle (0.65, 3);
\begin{scope}[shift={(0,3)}]
\path [fill=brown!20!white] (0,0) [partial ellipse=0:180:0.6 and 0.3];
\draw (0,0) [partial ellipse=0:360:0.6 and 0.3];
\end{scope}
\draw (-0.6, 3)--(-0.6, 0) (0.6, 3)--(0.6, 0); 
\draw (-0.6, -3)--(-0.6, 0) (0.6, -3)--(0.6, 0);
\begin{scope}[shift={(0,-3)}]
\path [fill=brown!20!white] (0,0) [partial ellipse=180:360:0.6 and 0.3];
\draw [dashed](0,0) [partial ellipse=0:180:0.6 and 0.3];
\draw (0,0) [partial ellipse=180:360:0.6 and 0.3];
\end{scope}
\draw[blue, ->-=0.75] (0, 0) [partial ellipse=0:360:0.6 and 0.3];
\draw (0, -0.6) node{\tiny$M_i$};
\draw (0, 1.6) node{\tiny $\mathcal{D}$};
\draw (0, -1.6) node{\tiny $\mathcal{E}$};
\end{tikzpicture}}}}\,,
\end{equation}
where we again identified the object in ${}_{Q}\cC_{Q'}$ with morphisms from $Z(Q)$ to $Z(Q')$. 
Now the assumption that $\phi$ is an algebra homeomorphism is depicted as:
\[\vcenter{\hbox{\scalebox{0.6}{
\begin{tikzpicture}[xscale=0.8, yscale=-0.6]
\path [fill=brown!20!white] (-0.6, 2)--(-0.6, 0) .. controls +(0, -1) and +(0, 1).. (-1.5, -2)--(-1.5, -4) -- (1.5, -4)--(1.5, -2) .. controls +(0, 1) and +(0, -1).. (0.6, 0)--(0.6, 2)--(-0.6, 2) ;
\begin{scope}[shift={(0,2)}]
\path [fill=brown!20!white] (0,0) [partial ellipse=0:360:0.6 and 0.3];
\draw [dashed](0,0) [partial ellipse=0:360:0.6 and 0.3];
\draw (0,0) [partial ellipse=0:180:0.6 and 0.3];
\end{scope}
\draw (-0.6, 2)--(-0.6, 0) .. controls +(0, -1) and +(0, 1).. (-1.5, -2)--(-1.5, -4);
\draw (0.6, 2)--(0.6, 0).. controls +(0, -1) and +(0, 1).. (1.5, -2)--(1.5, -4); 
\path [fill=white] (-0.3, -4)--(-0.3, -2).. controls +(0, 0.8) and +(0, 0.8) .. (0.3, -2)--(0.3, -4)--(-0.3, -4);
\draw (-0.3, -4)--(-0.3, -2).. controls +(0, 0.8) and +(0, 0.8) .. (0.3, -2)--(0.3, -4);
\begin{scope}[shift={(-0.9,-4)}]
\path [fill=brown!20!white] (0,0) [partial ellipse=180:360:0.6 and 0.3];
\draw (0,0) [partial ellipse=0:180:0.6 and 0.3];
\draw (0,0) [partial ellipse=180:360:0.6 and 0.3];
\end{scope}
\begin{scope}[shift={(0.9,-4)}]
\path [fill=brown!20!white] (0,0) [partial ellipse=180:360:0.6 and 0.3];
\draw (0,0) [partial ellipse=0:180:0.6 and 0.3];
\draw (0,0) [partial ellipse=180:360:0.6 and 0.3];
\end{scope}
\draw[blue, ->-=0.75] (0, 0) [partial ellipse=360:0:0.6 and 0.3];
\node at (0,-0.9) {$\cD$};\node at (0,1) {$\cE$};
\node at (1,0.2) {$\phi$};
\end{tikzpicture}}}}
=
\vcenter{\hbox{\scalebox{0.6}{
\begin{tikzpicture}[xscale=0.8, yscale=-0.6]
\path [fill=brown!20!white] (-0.6, 2)--(-0.6, 0) .. controls +(0, -1) and +(0, 1).. (-1.5, -2)--(-1.5, -4) -- (1.5, -4)--(1.5, -2) .. controls +(0, 1) and +(0, -1).. (0.6, 0)--(0.6, 2)--(-0.6, 2) ;
\begin{scope}[shift={(0,2)}]
\path [fill=brown!20!white] (0,0) [partial ellipse=0:360:0.6 and 0.3];
\draw [dashed](0,0) [partial ellipse=0:360:0.6 and 0.3];
\draw (0,0) [partial ellipse=0:180:0.6 and 0.3];
\end{scope}
\draw (-0.6, 2)--(-0.6, 0) .. controls +(0, -1) and +(0, 1).. (-1.5, -2)--(-1.5, -4);
\draw (0.6, 2)--(0.6, 0).. controls +(0, -1) and +(0, 1).. (1.5, -2)--(1.5, -4); 
\path [fill=white] (-0.3, -4)--(-0.3, -2).. controls +(0, 0.8) and +(0, 0.8) .. (0.3, -2)--(0.3, -4)--(-0.3, -4);
\draw (-0.3, -4)--(-0.3, -2).. controls +(0, 0.8) and +(0, 0.8) .. (0.3, -2)--(0.3, -4);
\begin{scope}[shift={(-0.9,-4)}]
\path [fill=brown!20!white] (0,0) [partial ellipse=180:360:0.6 and 0.3];
\draw (0,0) [partial ellipse=0:180:0.6 and 0.3];
\draw (0,0) [partial ellipse=180:360:0.6 and 0.3];
\end{scope}
\begin{scope}[shift={(0.9,-4)}]
\path [fill=brown!20!white] (0,0) [partial ellipse=180:360:0.6 and 0.3];
\draw (0,0) [partial ellipse=0:180:0.6 and 0.3];
\draw (0,0) [partial ellipse=180:360:0.6 and 0.3];
\end{scope}
\draw[blue, ->-=0.75] (-0.9,-2.5) [partial ellipse=360:0:0.6 and 0.3];
\draw[blue, ->-=0.75] (0.9,-2.5) [partial ellipse=360:0:0.6 and 0.3];
\node at (-0.9,-3.2) {$\cD$};\node at (0.9,-3.2) {$\cD$};
\node at (0,0) {$\cE$}; 
\node at (-2,-2.5) {$\phi$};\node at (2,-2.5) {$\phi$};
\end{tikzpicture}}}}\,,
\]
which, combined with Equation \eqref{eq:phi} and projected to the plane that contains the disc at the bottom (imagine we stretch the disc at the bottom on both sides and push the top tubes down to the plane that contains the bottom disc), results in the following equality:
\begin{equation}\label{eq:phi2}
\sum_{k} a_k\cdot \vcenter{\hbox{\scalebox{0.6}{\begin{tikzpicture}
\draw[fill=brown!20!white](0, 0) circle (2);
\draw[fill=white](-0.8, 0) circle (0.4);
\draw[fill=white](0.8, 0) circle (0.4);
\draw[blue, ->-=0.5] (0, 0) circle (1.6);
\draw (0, 1.3) node{\tiny$M_k$};
\end{tikzpicture}}}}
=
\sum_{i, j} a_ia_j \cdot \vcenter{\hbox{\scalebox{0.6}{\begin{tikzpicture}
\draw[fill=brown!20!white](0, 0) circle (2);
\draw[fill=white](-0.8, 0) circle (0.4);
\draw[fill=white](0.8, 0) circle (0.4);
\draw[blue, ->-=0.5] (-0.8, 0) circle (0.6);
\draw (1.7, 0) node{\tiny$M_i$};
\draw (-1.7, 0) node{\tiny$M_j$};
\draw[blue, ->-=0.5] (0.8, 0) circle (0.6);
\end{tikzpicture}}}}\,,
\end{equation}
where we view both sides, via the alterfold TQFT in \cite{LMWW23b}, as vectors in the vector space associated to a 2-alterfold whose $B$-colored region is a pair of pants. 

Applying fusion on the right hand side of Equation \eqref{eq:phi2}, we have
$$
\sum_{k} a_k\cdot \vcenter{\hbox{\scalebox{0.6}{\begin{tikzpicture}
\draw[fill=brown!20!white](0, 0) circle (2);
\draw[fill=white](-0.8, 0) circle (0.4);
\draw[fill=white](0.8, 0) circle (0.4);
\draw[blue, ->-=0.5] (0, 0) circle (1.6);
\draw (0, 1.3) node{\tiny$M_k$};
\end{tikzpicture}}}}
=
\sum_{i, j,\ell,\alpha} a_ia_j \cdot \vcenter{\hbox{\scalebox{0.6}{\begin{tikzpicture}
\draw[fill=brown!20!white](0, 0) circle (2);
\draw[fill=white](-0.8, 0) circle (0.4);
\draw[fill=white](0.8, 0) circle (0.4);
\draw[blue, ->-=0, ->-=0.5] (0, 0) [partial ellipse=0:360:1.4 and 0.8];
\draw[blue, ->-=0.5] (0, 0.8) to (0, -0.8);
\draw (0.2, -0.2) node{\tiny$X_\ell$};
\draw (1.7, 0) node{\tiny$M_i$};
\draw (-1.7, 0) node{\tiny$M_j$};
\draw[fill=white] (-0.1, 0.7) rectangle (0.1, 0.9);
\node[above] at (0,0.9) {\tiny$\alpha$};
\draw[fill=white] (-0.1, -0.7) rectangle (0.1, -0.9);
\node[below] at (0,-0.9) {\tiny$\alpha^*$};
\end{tikzpicture}}}}$$
where $\{\alpha, \alpha^*\}$ is the dual basis in the hom-space $\hom_{\mathcal{C}}(M_i\otimes M_{j^*}, X_\ell)$ afforded by the semisimplicity of ${}_Q\cC_{Q'}$. 
By \cite[Theorem 4.29]{LMWW23b}, as vectors in the vector space associated to a 2-alterfold, all summands on the left hand side of the above equation are linearly independent, and the same is true for summands on the right hand side. 
Therefore, it is easy to see that to make the equality hold, the right hand side only allows one summand, and the corresponding object $M_i$ has to be an invertible object. 
In the other words, we have shown that $\phi$ corresponds to an invertible $Q$-$Q'$-bimodule, which implies the Morita equivalence of $Q$ and $Q'$.

The above arguments imply the following result:
\begin{proposition}
    Suppose $Q,Q'$ are Frobenius algebras.
    \begin{align*}
       \dim \hom_{\mathcal{Z}(C)}(Z(Q), Z(Q')) =|\ _Q\mathcal{C}_{Q'}|,
    \end{align*}
    where $|\  _Q\mathcal{C}_{Q'}|$ is the number of irreducible $Q$-$Q'$-bimodules.
\end{proposition}
When $Q$, $Q'$ run through all Frobenius algebras, we have that
$\displaystyle \bigoplus_{Q, Q'}\hom_{\mathcal{Z}(C)}(Z(Q), Z(Q'))$ is isomorphic to the fusion bi-algebra of $\displaystyle \bigoplus_{Q, Q'}\  _Q\mathcal{C}_{Q'}$.

\section{Topological Characters}
By previous sections, we see that the $\mathcal{D}$-colored surface plays important role in the study of the alterfold TQFT for Morita equivalence.
In this section, inspired by topological interpretation of the character $\chi_j$ in Equation \eqref{eq:charactermodular}, we change the color of a tube to be $\mathcal{D}$-color and obtain a topological character in the alterfold TQFT for Morita equivalence.
This is also known as the first torus action in the introduction.
We assume that $\mathcal{C}$ is a braided spherical fusion category and $\mathcal{D}$ is a spherical fusion category.

\begin{definition}[Character]
 Suppose $\mathcal{C}$ is a braided spherical fusion category and $X\in \Irr(\mathcal{C})$.
 We define the character $\chi_X^{\pm}$ to be the following graph:
 \begin{align*}
\chi_X^+= \frac{1}{\mu^2}
\vcenter{\hbox{\scalebox{0.7}{
\begin{tikzpicture}[xscale=0.8, yscale=0.6]
\draw [line width=0.83cm] (0,0) [partial ellipse=-0.1:180.1:2 and 1.5];
\draw [white, line width=0.8cm] (0,0) [partial ellipse=-0.1:180.1:2 and 1.5];
\draw [red] (0,0) [partial ellipse=-0.1:180.1:2 and 1.5];
\path [fill=brown!20!white](-0.65, -3) rectangle (0.65, 3);
\begin{scope}[shift={(0,3)}]
\path [fill=brown!20!white] (0,0) [partial ellipse=0:180:0.6 and 0.3];
\draw (0,0) [partial ellipse=0:360:0.6 and 0.3];
\end{scope}
\draw (-0.6, 3)--(-0.6, 0) (0.6, 3)--(0.6, 0); 
\draw (-0.6, -3)--(-0.6, 0) (0.6, -3)--(0.6, 0);
\draw [line width=0.83cm] (0,0) [partial ellipse=179:361:2 and 1.5];
\draw [white, line width=0.8cm] (0,0) [partial ellipse=177:363:2 and 1.5];
\draw [red] (0,0) [partial ellipse=192:361:2 and 1.5];
\begin{scope}[shift={(0,-3)}]
\path [fill=brown!20!white] (0,0) [partial ellipse=180:360:0.6 and 0.3];
\draw [dashed](0,0) [partial ellipse=0:180:0.6 and 0.3];
\draw (0,0) [partial ellipse=180:360:0.6 and 0.3];
\end{scope}
\begin{scope}[shift={(-2, 0)}]
\draw [blue, dashed] (0,0) [partial ellipse=0:180:0.46 and 0.25];
\draw [blue, ->-=0.5] (0,0) [partial ellipse=180:360:0.46 and 0.25] node[black, pos=0.5, below] {\tiny $X$};
\end{scope}
\end{tikzpicture}}}}, \quad 
\chi_X^-= \frac{1}{\mu^2}
\vcenter{\hbox{\scalebox{0.7}{
\begin{tikzpicture}[xscale=0.8, yscale=0.6]
\draw [line width=0.83cm] (0,0) [partial ellipse=-0.1:180.1:2 and 1.5];
\draw [white, line width=0.8cm] (0,0) [partial ellipse=-0.1:180.1:2 and 1.5];
\draw [red] (0,0) [partial ellipse=-0.1:180.1:2 and 1.5];
\path [fill=brown!20!white](-0.65, -3) rectangle (0.65, 3);
\begin{scope}[shift={(0,3)}]
\path [fill=brown!20!white] (0,0) [partial ellipse=0:180:0.6 and 0.3];
\draw (0,0) [partial ellipse=0:360:0.6 and 0.3];
\end{scope}
\draw (-0.6, 3)--(-0.6, 0) (0.6, 3)--(0.6, 0); 
\draw (-0.6, -3)--(-0.6, 0) (0.6, -3)--(0.6, 0);
\draw [line width=0.83cm] (0,0) [partial ellipse=179:361:2 and 1.5];
\draw [white, line width=0.8cm] (0,0) [partial ellipse=177:363:2 and 1.5];
\draw [red] (0,0) [partial ellipse=180:361:2 and 1.5];
\begin{scope}[shift={(0,-3)}]
\path [fill=brown!20!white] (0,0) [partial ellipse=180:360:0.6 and 0.3];
\draw [dashed](0,0) [partial ellipse=0:180:0.6 and 0.3];
\draw (0,0) [partial ellipse=180:360:0.6 and 0.3];
\end{scope}
\begin{scope}[shift={(-2, 0)}]
\draw [blue, dashed] (0,0) [partial ellipse=0:180:0.46 and 0.25];
\draw [blue, ->-=0.8] (0,0) [partial ellipse=180:270:0.46 and 0.25] node[black, pos=0.5, below right] {\tiny $X$};
\draw [blue] (0,0) [partial ellipse=290:360:0.46 and 0.25];
\end{scope}
\end{tikzpicture}}}}, 
\end{align*}
 \end{definition}
 
 \begin{remark}
 Suppose $Y\in \Irr(\mathcal{D})$.
 We can define $\chi_X^\pm(Y)$ as follows:
 \begin{align*}
\chi_X^+(Y)=\frac{1}{\mu^2}\vcenter{\hbox{\scalebox{0.6}{
\begin{tikzpicture}[scale=0.9]
\draw [line width=0.77cm, brown!20!white] (0,0) [partial ellipse=-0.1:180.1:2 and 1.5];
\begin{scope}[shift={(2.5, 0)}]
\draw [line width=0.8cm] (0,0) [partial ellipse=0:180:2 and 1.5];
\draw [white, line width=0.77cm] (0,0) [partial ellipse=-0.1:180.1:2 and 1.5];
\draw [red] (0,0) [partial ellipse=-0.1:180.1:2 and 1.5];
\end{scope} 
\begin{scope}[shift={(2.5, 0)}]
\draw [line width=0.8cm] (0,0) [partial ellipse=180:360:2 and 1.5];
\draw [white, line width=0.77cm] (0,0) [partial ellipse=178:362:2 and 1.5];
\draw [red] (0,0) [partial ellipse=195:362:2 and 1.5];
\end{scope}
\draw [line width=0.77cm, brown!20!white] (0,0) [partial ellipse=180:360:2 and 1.5];
\begin{scope}[shift={(-2, 0)}]
\draw [blue, dashed](0,0) [partial ellipse=0:180:0.4 and 0.25] ;
\draw [blue, ->-=0.5] (0,0) [partial ellipse=180:360:0.4 and 0.25] node[black, pos=0.5, below] {\tiny $Y$};
\end{scope}
\begin{scope}[shift={(0.5, 0)}]
\draw [blue, dashed] (0,0) [partial ellipse=0:180:0.42 and 0.25];
\draw [blue, ->-=0.5] (0,0) [partial ellipse=180:360:0.42 and 0.25] node[black, pos=0.5, below] {\tiny $X$};
\end{scope}
\end{tikzpicture}}}}, \quad 
\chi_X^-(Y)=\frac{1}{\mu^2}\vcenter{\hbox{\scalebox{0.6}{
\begin{tikzpicture}[scale=0.9]
\draw [line width=0.77cm, brown!20!white] (0,0) [partial ellipse=-0.1:180.1:2 and 1.5];
\begin{scope}[shift={(2.5, 0)}]
\draw [line width=0.8cm] (0,0) [partial ellipse=0:180:2 and 1.5];
\draw [white, line width=0.77cm] (0,0) [partial ellipse=-0.1:180.1:2 and 1.5];
\draw [red] (0,0) [partial ellipse=-0.1:180.1:2 and 1.5];
\end{scope} 
\begin{scope}[shift={(2.5, 0)}]
\draw [line width=0.8cm] (0,0) [partial ellipse=180:360:2 and 1.5];
\draw [white, line width=0.77cm] (0,0) [partial ellipse=178:362:2 and 1.5];
\draw [red] (0,0) [partial ellipse=180:362:2 and 1.5];
\end{scope}
\draw [line width=0.77cm, brown!20!white] (0,0) [partial ellipse=180:360:2 and 1.5];
\begin{scope}[shift={(-2, 0)}]
\draw [blue, dashed](0,0) [partial ellipse=0:180:0.4 and 0.25] ;
\draw [blue, ->-=0.5] (0,0) [partial ellipse=180:360:0.4 and 0.25] node[black, pos=0.5, below] {\tiny $Y$};
\end{scope}
\begin{scope}[shift={(0.5, 0)}]
\draw [blue, dashed] (0,0) [partial ellipse=0:180:0.4 and 0.25];
\draw [blue, ->-=0.5] (0,0) [partial ellipse=180:270:0.4 and 0.25] node[black, pos=0.3, below right] {\tiny $X$};
\draw [blue] (0,0) [partial ellipse=290:360:0.4 and 0.25];
\end{scope}
\end{tikzpicture}}}} .
\end{align*}
By the fact that the $S$-coefficients are algebraic integers, we have that $\mu\chi_X^\pm(Y)$ is an algebraic integer.
We also have that $\displaystyle \frac{1}{\mu}\vcenter{\hbox{\scalebox{0.6}{
\begin{tikzpicture}[scale=0.9]
\draw [line width=0.77cm, brown!20!white] (0,0) [partial ellipse=-0.1:180.1:2 and 1.5];
\begin{scope}[shift={(2.5, 0)}]
\draw [line width=0.8cm] (0,0) [partial ellipse=0:180:2 and 1.5];
\draw [white, line width=0.77cm] (0,0) [partial ellipse=-0.1:180.1:2 and 1.5];
\end{scope} 
\begin{scope}[shift={(2.5, 0)}]
\draw [line width=0.8cm] (0,0) [partial ellipse=180:360:2 and 1.5];
\draw [white, line width=0.77cm] (0,0) [partial ellipse=178:362:2 and 1.5];
\end{scope}
\draw [line width=0.77cm, brown!20!white] (0,0) [partial ellipse=180:360:2 and 1.5];
\begin{scope}[shift={(-2, 0)}]
\draw [blue, dashed](0,0) [partial ellipse=0:180:0.4 and 0.25] ;
\draw [blue, ->-=0.5] (0,0) [partial ellipse=180:360:0.4 and 0.25] node[black, pos=0.5, below] {\tiny $Y$};
\end{scope}
\begin{scope}[shift={(0.5, 0)}]
\draw [blue, dashed] (0,0) [partial ellipse=0:180:0.4 and 0.25];
\draw [blue, ->-=0.5] (0,0) [partial ellipse=180:360:0.4 and 0.25] node[black, pos=0.3, below right] {\tiny $X$};
\end{scope}
\end{tikzpicture}}}} $ is an algebraic integer.
 \end{remark}
 
 \begin{remark}
 Suppose $Y, Y'\in \Irr(\mathcal{D})$.
 By flipping the graph vertically, we have that 
 \begin{align*}
 \overline{\chi_X^\pm(Y)}=\chi_{X^*}^{\pm}(Y)=\chi_X^{\pm}(Y^*).
 \end{align*}
 By moving objects along the torus, we have that 
 \begin{align*}
 \chi_X^\pm(YY')=\chi_X^{\pm}(Y'Y).
 \end{align*}
 Note that the entries of $S$-matrix are algebraic integers.
 We have that $\mu\chi_X(Y)$ is an algebraic integer.
 \end{remark}
 
 \begin{remark}
 If $\mathcal{C}$ is symmetric, we then have that $\chi_X^+=\chi_X^-$ and we will denote it by $\chi_X$.
 \end{remark}

 \begin{example}
 Suppose $G$ is a finite group and $\rho$ is a representation of $G$.
 We have that 
 \begin{align*}
 \frac{1}{|G|}\vcenter{\hbox{\scalebox{0.6}{
\begin{tikzpicture}[scale=0.9]
\draw [line width=0.77cm, brown!20!white] (0,0) [partial ellipse=-0.1:180.1:2 and 1.5];
\begin{scope}[shift={(2.5, 0)}]
\draw [line width=0.8cm] (0,0) [partial ellipse=0:180:2 and 1.5];
\draw [white, line width=0.77cm] (0,0) [partial ellipse=-0.1:180.1:2 and 1.5];
\end{scope} 
\begin{scope}[shift={(2.5, 0)}]
\draw [line width=0.8cm] (0,0) [partial ellipse=180:360:2 and 1.5];
\draw [white, line width=0.77cm] (0,0) [partial ellipse=178:362:2 and 1.5];
\end{scope}
\draw [line width=0.77cm, brown!20!white] (0,0) [partial ellipse=180:360:2 and 1.5];
\begin{scope}[shift={(4.5, 0)}]
\draw [white, line width=4pt] (0,0) [partial ellipse=290:270:0.42 and 0.25];
\end{scope}
\begin{scope}[shift={(-2, 0)}]
\draw [blue, dashed](0,0) [partial ellipse=0:180:0.4 and 0.25] ;
\draw [blue, ->-=0.5] (0,0) [partial ellipse=180:360:0.4 and 0.25] node[black, pos=0.5, below] {\tiny $\bC_g$};
\end{scope}
\begin{scope}[shift={(4.5, 0)}]
\draw [blue, dashed] (0,0) [partial ellipse=0:180:0.42 and 0.25];
\draw [blue, ->-=0.5] (0,0) [partial ellipse=180:360:0.42 and 0.25] node[black, pos=0.5, below] {\tiny $\rho$};
\end{scope}
\end{tikzpicture}}}}    = & Tr(\rho(g))=\chi_\rho(g),
\end{align*}
where $\chi_\rho$ is the character in the group theory.
Note that for $\rho\in \Irr(\rep(G))$, we have
\begin{align*}
\frac{1}{|G|}
\vcenter{\hbox{\scalebox{0.6}{
\begin{tikzpicture}[xscale=0.8, yscale=0.6]
\draw [line width=0.83cm] (0,0) [partial ellipse=-0.1:180.1:2 and 1.5];
\draw [white, line width=0.8cm] (0,0) [partial ellipse=-0.1:180.1:2 and 1.5];
\draw [red] (0,0) [partial ellipse=-0.1:180.1:2 and 1.5];
\path [fill=brown!20!white](-0.65, -3) rectangle (0.65, 3);
\begin{scope}[shift={(0,3)}]
\path [fill=brown!20!white] (0,0) [partial ellipse=0:180:0.6 and 0.3];
\draw (0,0) [partial ellipse=0:360:0.6 and 0.3];
\end{scope}
\draw (-0.6, 3)--(-0.6, 0) (0.6, 3)--(0.6, 0); 
\draw (-0.6, -3)--(-0.6, 0) (0.6, -3)--(0.6, 0);
\draw [line width=0.83cm] (0,0) [partial ellipse=180:360:2 and 1.5];
\draw [white, line width=0.8cm] (0,0) [partial ellipse=179:361:2 and 1.5];
\draw [red] (0,0) [partial ellipse=192:361:2 and 1.5];
\begin{scope}[shift={(0,-3)}]
\path [fill=brown!20!white] (0,0) [partial ellipse=180:360:0.6 and 0.3];
\draw [dashed](0,0) [partial ellipse=0:180:0.6 and 0.3];
\draw (0,0) [partial ellipse=180:360:0.6 and 0.3];
\end{scope}
\begin{scope}[shift={(-2, 0)}]
\draw [blue, dashed] (0,0) [partial ellipse=0:180:0.46 and 0.25];
\draw [blue, ->-=0.5] (0,0) [partial ellipse=180:360:0.46 and 0.25] node[black, pos=0.5, below] {\tiny $\rho$};
\end{scope}
\end{tikzpicture}}}}=
\vcenter{\hbox{\scalebox{0.6}{
\begin{tikzpicture}[xscale=0.8, yscale=0.6]
\draw [line width=0.83cm] (0,0) [partial ellipse=-0.1:180.1:2 and 1.5];
\draw [white, line width=0.8cm] (0,0) [partial ellipse=-0.1:180.1:2 and 1.5];
\path [fill=brown!20!white](-0.65, -3) rectangle (0.65, 3);
\begin{scope}[shift={(0,3)}]
\path [fill=brown!20!white] (0,0) [partial ellipse=0:180:0.6 and 0.3];
\draw (0,0) [partial ellipse=0:360:0.6 and 0.3];
\end{scope}
\draw (-0.6, 3)--(-0.6, 0) (0.6, 3)--(0.6, 0); 
\draw (-0.6, -3)--(-0.6, 0) (0.6, -3)--(0.6, 0);
\draw [line width=0.83cm] (0,0) [partial ellipse=180:360:2 and 1.5];
\draw [white, line width=0.8cm] (0,0) [partial ellipse=179:361:2 and 1.5];
\begin{scope}[shift={(0,-3)}]
\path [fill=brown!20!white] (0,0) [partial ellipse=180:360:0.6 and 0.3];
\draw [dashed](0,0) [partial ellipse=0:180:0.6 and 0.3];
\draw (0,0) [partial ellipse=180:360:0.6 and 0.3];
\end{scope}
\begin{scope}[shift={(-2, 0)}]
\draw [blue, dashed] (0,0) [partial ellipse=0:180:0.46 and 0.25];
\draw [blue, ->-=0.5] (0,0) [partial ellipse=180:360:0.46 and 0.25] node[black, pos=0.5, below] {\tiny $\rho$};
\end{scope}
\end{tikzpicture}}}}.
\end{align*}
Our characters coincide with the characters in the group theory.
 \end{example}

 \begin{definition}[$S$-Character]
 Suppose $\mathcal{C}$ is a braided spherical fusion category and $X\in \Irr(\mathcal{C})$. 
 We define a transformation $\mathfrak{S}$ in the following way:
 \begin{align*}
 \mathfrak{S}\vcenter{\hbox{\scalebox{0.6}{
\begin{tikzpicture}[xscale=0.8, yscale=0.6]
\draw [line width=0.83cm] (0,0) [partial ellipse=-0.1:180.1:2 and 1.5];
\draw [white, line width=0.8cm] (0,0) [partial ellipse=-0.1:180.1:2 and 1.5];
\draw [blue] (0,0) [partial ellipse=-0.1:180.1:2 and 1.5];
\path [fill=brown!20!white](-0.65, -3) rectangle (0.65, 3);
\begin{scope}[shift={(0,3)}]
\path [fill=brown!20!white] (0,0) [partial ellipse=0:180:0.6 and 0.3];
\draw (0,0) [partial ellipse=0:360:0.6 and 0.3];
\end{scope}
\draw (-0.6, 3)--(-0.6, 0) (0.6, 3)--(0.6, 0); 
\draw (-0.6, -3)--(-0.6, 0) (0.6, -3)--(0.6, 0);
\draw [line width=0.83cm] (0,0) [partial ellipse=179:361:2 and 1.5];
\draw [white, line width=0.8cm] (0,0) [partial ellipse=177:363:2 and 1.5];
\draw [blue, ->-=0.5] (0,0) [partial ellipse=192:361:2 and 1.5];
\begin{scope}[shift={(0,-3)}]
\path [fill=brown!20!white] (0,0) [partial ellipse=180:360:0.6 and 0.3];
\draw [dashed](0,0) [partial ellipse=0:180:0.6 and 0.3];
\draw (0,0) [partial ellipse=180:360:0.6 and 0.3];
\end{scope}
\begin{scope}[shift={(-2, 0)}]
\draw [purple, dashed] (0,0) [partial ellipse=0:180:0.46 and 0.25];
\draw [purple] (0,0) [partial ellipse=180:360:0.46 and 0.25];
\begin{scope}[shift={(0, -0.2)}]
\draw [fill=white] (-0.25, -0.25) rectangle (0.25, 0.25);
\node at (0, 0) {\tiny $f$};
\end{scope}
\end{scope}
\end{tikzpicture}}}}
=
\vcenter{\hbox{\scalebox{0.6}{
\begin{tikzpicture}[xscale=0.8, yscale=0.6]
\draw [line width=0.83cm] (0,0) [partial ellipse=-0.1:180.1:2 and 1.5];
\draw [white, line width=0.8cm] (0,0) [partial ellipse=-0.1:180.1:2 and 1.5];
\draw [purple] (0,0) [partial ellipse=-0.1:180.1:2 and 1.5];
\path [fill=brown!20!white](-0.65, -3) rectangle (0.65, 3);
\begin{scope}[shift={(0,3)}]
\path [fill=brown!20!white] (0,0) [partial ellipse=0:180:0.6 and 0.3];
\draw (0,0) [partial ellipse=0:360:0.6 and 0.3];
\end{scope}
\draw (-0.6, 3)--(-0.6, 0) (0.6, 3)--(0.6, 0); 
\draw (-0.6, -3)--(-0.6, 0) (0.6, -3)--(0.6, 0);
\draw [line width=0.83cm] (0,0) [partial ellipse=179:361:2 and 1.5];
\draw [white, line width=0.8cm] (0,0) [partial ellipse=177:363:2 and 1.5];
\draw [purple] (0,0) [partial ellipse=192:361:2 and 1.5];
\begin{scope}[shift={(0,-3)}]
\path [fill=brown!20!white] (0,0) [partial ellipse=180:360:0.6 and 0.3];
\draw [dashed](0,0) [partial ellipse=0:180:0.6 and 0.3];
\draw (0,0) [partial ellipse=180:360:0.6 and 0.3];
\end{scope}
\begin{scope}[shift={(-2, 0)}]
\draw [blue, dashed] (0,0) [partial ellipse=0:180:0.46 and 0.25];
\draw [blue, ->-=0.9] (0,0) [partial ellipse=180:360:0.46 and 0.25];
\begin{scope}[shift={(0, -0.2)}]
\draw [fill=white] (-0.25, -0.25) rectangle (0.25, 0.25);
\node at (0, 0) {\tiny $f$};
\end{scope}
\end{scope}
\end{tikzpicture}}}}
 \end{align*}
 by switching the longitude and median on the tube.
 The $S$-character $\Schi_X^\pm$ is defined by the following graph:
\begin{align*}
\Schi_X^+= \frac{1}{\mu^2}
\vcenter{\hbox{\scalebox{0.6}{
\begin{tikzpicture}[xscale=0.8, yscale=0.6]
\draw [line width=0.83cm] (0,0) [partial ellipse=-0.1:180.1:2 and 1.5];
\draw [white, line width=0.8cm] (0,0) [partial ellipse=-0.1:180.1:2 and 1.5];
\draw [blue] (0,0) [partial ellipse=-0.1:180.1:2 and 1.5];
\path [fill=brown!20!white](-0.65, -3) rectangle (0.65, 3);
\begin{scope}[shift={(0,3)}]
\path [fill=brown!20!white] (0,0) [partial ellipse=0:180:0.6 and 0.3];
\draw (0,0) [partial ellipse=0:360:0.6 and 0.3];
\end{scope}
\draw (-0.6, 3)--(-0.6, 0) (0.6, 3)--(0.6, 0); 
\draw (-0.6, -3)--(-0.6, 0) (0.6, -3)--(0.6, 0);
\draw [line width=0.83cm] (0,0) [partial ellipse=179:361:2 and 1.5];
\draw [white, line width=0.8cm] (0,0) [partial ellipse=177:363:2 and 1.5];
\draw [blue, ->-=0.5] (0,0) [partial ellipse=180:361:2 and 1.5] node[black, pos=0.5, below right] {\tiny $X$};
\begin{scope}[shift={(0,-3)}]
\path [fill=brown!20!white] (0,0) [partial ellipse=180:360:0.6 and 0.3];
\draw [dashed](0,0) [partial ellipse=0:180:0.6 and 0.3];
\draw (0,0) [partial ellipse=180:360:0.6 and 0.3];
\end{scope}
\begin{scope}[shift={(-2, 0)}]
\draw [red, dashed] (0,0) [partial ellipse=0:180:0.46 and 0.25];
\draw [red] (0,0) [partial ellipse=180:270:0.46 and 0.25] ;
\draw [red] (0,0) [partial ellipse=290:360:0.46 and 0.25];
\end{scope}
\end{tikzpicture}}}}, 
\quad
\Schi_X^-= \frac{1}{\mu^2}
\vcenter{\hbox{\scalebox{0.6}{
\begin{tikzpicture}[xscale=0.8, yscale=0.6]
\draw [line width=0.83cm] (0,0) [partial ellipse=-0.1:180.1:2 and 1.5];
\draw [white, line width=0.8cm] (0,0) [partial ellipse=-0.1:180.1:2 and 1.5];
\draw [blue] (0,0) [partial ellipse=-0.1:180.1:2 and 1.5];
\path [fill=brown!20!white](-0.65, -3) rectangle (0.65, 3);
\begin{scope}[shift={(0,3)}]
\path [fill=brown!20!white] (0,0) [partial ellipse=0:180:0.6 and 0.3];
\draw (0,0) [partial ellipse=0:360:0.6 and 0.3];
\end{scope}
\draw (-0.6, 3)--(-0.6, 0) (0.6, 3)--(0.6, 0); 
\draw (-0.6, -3)--(-0.6, 0) (0.6, -3)--(0.6, 0);
\draw [line width=0.83cm] (0,0) [partial ellipse=179:361:2 and 1.5];
\draw [white, line width=0.8cm] (0,0) [partial ellipse=177:363:2 and 1.5];
\draw [blue, ->-=0.5] (0,0) [partial ellipse=192:361:2 and 1.5] node[black, pos=0.5, below] {\tiny $X$};
\begin{scope}[shift={(0,-3)}]
\path [fill=brown!20!white] (0,0) [partial ellipse=180:360:0.6 and 0.3];
\draw [dashed](0,0) [partial ellipse=0:180:0.6 and 0.3];
\draw (0,0) [partial ellipse=180:360:0.6 and 0.3];
\end{scope}
\begin{scope}[shift={(-2, 0)}]
\draw [red, dashed] (0,0) [partial ellipse=0:180:0.46 and 0.25];
\draw [red] (0,0) [partial ellipse=180:360:0.46 and 0.25] ;
\end{scope}
\end{tikzpicture}}}}.
\end{align*}
\end{definition}

\begin{remark}
    In \cite{BEK99}, the $S$-character is also called chiral generator.
    By Equation \eqref{eq:alphatube}, we see that the $S$-character encodes the $\alpha$-induction.
\end{remark}

The vertical multiplication is denoted by $\circ$. We have that
 \begin{align*}
 \Schi_j^+\circ \Schi_k^-=
\frac{1}{\mu^3}
\vcenter{\hbox{\scalebox{0.7}{
\begin{tikzpicture}[xscale=0.8, yscale=0.6]
\draw [line width=0.8cm] (0,0) [partial ellipse=0:180:2 and 1.5];
\draw [white, line width=0.77cm] (0,0) [partial ellipse=-0.1:180.1:2 and 1.5];
\draw [blue] (0,0) [partial ellipse=0:180:2.15 and 1.65];
\draw [blue] (0,0) [partial ellipse=0:180:1.85 and 1.35];
\path [fill=brown!20!white](-0.65, -3) rectangle (0.65, 3);
\begin{scope}[shift={(0,3)}]
\path [fill=brown!20!white] (0,0) [partial ellipse=0:180:0.6 and 0.3];
\draw (0,0) [partial ellipse=0:360:0.6 and 0.3];
\end{scope}
\draw [line width=0.8cm] (0,0) [partial ellipse=180:360:2 and 1.5];
\draw (-0.6, 3)--(-0.6, 0) (0.6, 3)--(0.6, 0); 
\draw (-0.6, -3)--(-0.6, 0) (0.6, -3)--(0.6, 0);
\draw [white, line width=0.77cm] (0,0) [partial ellipse=178:362:2 and 1.5];
\draw [blue, ->-=0.5] (0,0) [partial ellipse=178:362:2.15 and 1.65] node[black, pos=0.7,below ] {\tiny $X_k$}; 
\draw [blue, ->-=0.5] (0,0) [partial ellipse=178:362:1.85 and 1.35] node[black, pos=0.7,above ] {\tiny $X_j$};
\begin{scope}[shift={(2, 0)}]
\draw [red, dashed](0,0) [partial ellipse=0:180:0.44 and 0.25] ;
\draw [white, line width=4pt] (0,0) [partial ellipse=290:270:0.44 and 0.25];
\draw [red] (0,0) [partial ellipse=260:360:0.44 and 0.25];
\draw [red] (0,0) [partial ellipse=180:230:0.44 and 0.25];
\end{scope}
\begin{scope}[shift={(0,-3)}]
\path [fill=brown!20!white] (0,0) [partial ellipse=180:360:0.6 and 0.3];
\draw [dashed](0,0) [partial ellipse=0:180:0.6 and 0.3];
\draw (0,0) [partial ellipse=180:360:0.6 and 0.3];
\end{scope}
\end{tikzpicture}}}}, 
\quad
\mathfrak{S}(\Schi_j^+\circ \Schi_k^-)=
\frac{1}{\mu^3}
\vcenter{\hbox{\scalebox{0.7}{
\begin{tikzpicture}[xscale=0.8, yscale=0.6]
\draw [line width=0.8cm] (0,0) [partial ellipse=0:180:2 and 1.5];
\draw [white, line width=0.77cm] (0,0) [partial ellipse=-0.1:180.1:2 and 1.5];
\draw [red] (0,0) [partial ellipse=0:180:2 and 1.5];
\path [fill=brown!20!white](-0.65, -3) rectangle (0.65, 3);
\begin{scope}[shift={(0,3)}]
\path [fill=brown!20!white] (0,0) [partial ellipse=0:180:0.6 and 0.3];
\draw (0,0) [partial ellipse=0:360:0.6 and 0.3];
\end{scope}
\draw [line width=0.8cm] (0,0) [partial ellipse=180:360:2 and 1.5];
\draw (-0.6, 3)--(-0.6, 0) (0.6, 3)--(0.6, 0); 
\draw (-0.6, -3)--(-0.6, 0) (0.6, -3)--(0.6, 0);
\draw [white, line width=0.77cm] (0,0) [partial ellipse=178:362:2 and 1.5];
\draw [red] (0,0) [partial ellipse=178:362:2 and 1.5]; 
\begin{scope}[shift={(-1.95, 0.3)}]
\draw [blue, ->-=0.4, dashed](0,0) [partial ellipse=0:180:0.47 and 0.25] node[black, pos=0.4, above] {\tiny $X_k$};
\draw [white, line width=4pt] (0,0) [partial ellipse=290:270:0.47 and 0.25];
\draw [blue] (0,0) [partial ellipse=280:360:0.47 and 0.25];
\draw [blue] (0,0) [partial ellipse=180:260:0.47 and 0.25];
\end{scope}
\begin{scope}[shift={(-1.95, -0.3)}]
\draw [blue, dashed](0,0) [partial ellipse=0:180:0.47 and 0.25] ;
\draw [white, line width=4pt] (0,0) [partial ellipse=290:270:0.47 and 0.25];
\draw [blue, ->-=0.2] (0,0) [partial ellipse=180:360:0.47 and 0.25] node[black, pos=0.2, below] {\tiny $X_j$};
\end{scope}
\begin{scope}[shift={(0,-3)}]
\path [fill=brown!20!white] (0,0) [partial ellipse=180:360:0.6 and 0.3];
\draw [dashed](0,0) [partial ellipse=0:180:0.6 and 0.3];
\draw (0,0) [partial ellipse=180:360:0.6 and 0.3];
\end{scope}
\end{tikzpicture}}}}.
\end{align*}
 
 \begin{remark}
 We have that 
\begin{align*}
 \mu^2( \Schi_j^+\circ \Schi_k^-)(\mathbbm{1})=\dim\hom_{\mathcal{D}}(\alpha_+(X_j), \alpha_-(X_k)).
 \end{align*}
\end{remark}

\begin{remark}
    Suppose $\mathcal{C},\mathcal{D}$ are unitary fusion categories.
    Then we have that 
    \begin{align*}
    \vcenter{\hbox{\scalebox{0.6}{
\begin{tikzpicture}[scale=0.9]
\draw [line width=0.77cm, brown!20!white] (0,0) [partial ellipse=-0.1:180.1:2 and 1.5];
\draw [line width=0.77cm, brown!20!white] (0,0) [partial ellipse=178:362:2 and 1.5];
\end{tikzpicture}}}}
\leq 
\vcenter{\hbox{\scalebox{0.6}{
\begin{tikzpicture}[scale=0.9]
\draw [line width=0.77cm, brown!20!white] (0,0) [partial ellipse=0:180.1:2 and 1.5];
\draw [line width=0.77cm, brown!20!white] (0,0) [partial ellipse=178:362:2 and 1.5];
\begin{scope}[shift={(-2, 0)}]
\path [fill=white] (-0.5, -0.3) rectangle (0.5, 0.3);
\end{scope}
\begin{scope}[shift={(-1.95, 0.3)}]
\path [fill=white](0,0) [partial ellipse=0:180:0.47 and 0.25];
\draw [blue, dashed](0,0) [partial ellipse=0:180:0.47 and 0.25];
\draw [blue, ->-=0.5] (0,0) [partial ellipse=180:360:0.47 and 0.25] node [below] {\tiny $J$};
\end{scope}
\begin{scope}[shift={(-1.95, -0.3)}]
\draw [blue, dashed](0,0) [partial ellipse=0:180:0.47 and 0.25] ;
\path [fill=white] (0,0) [partial ellipse=180:360:0.47 and 0.25];
\draw [blue, ->-=0.5] (0,0) [partial ellipse=180:360:0.47 and 0.25];
\end{scope}
\end{tikzpicture}}}}
\end{align*}
which follows from a simple fact that 
\begin{align*}
 \frac{1}{d_J}   \vcenter{\hbox{\scalebox{0.6}{
\begin{tikzpicture}[scale=0.9]
\begin{scope}[shift={(0, -1)}]
\draw [blue](0,0) [partial ellipse=0:180:0.7 and 0.8];
\end{scope}
\begin{scope}[shift={(0, 1)}]
\draw [blue](0,0) [partial ellipse=180:360:0.7 and 0.8];
\end{scope}
\end{tikzpicture}}}}
\leq 
 \vcenter{\hbox{\scalebox{0.6}{
\begin{tikzpicture}[scale=0.9]
\draw [blue](-0.5, -1)--(-0.5, 1) (0.5, -1)--(0.5, 1);
\end{tikzpicture}}}}.
\end{align*}
This implies that 
\begin{align*}
   \dim\hom_{\mathcal{D}}(\alpha_{\pm}(X_j), \alpha_{\pm}(X_k)) \leq \dim\hom_{\mathcal{C}}(J\overline{J}X_j, X_k).
\end{align*}
If the Frobenius algebra $Q$ satisfies the chiral relation, i.e. $Q=J\overline{J}$ is a commutative Frobenius algebra, then 
\begin{align*}
   \dim\hom_{\mathcal{D}}(\alpha_{\pm}(X_j), \alpha_{\pm}(X_k)) = \dim\hom_{\mathcal{C}}(J\overline{J}X_j, X_k).
\end{align*}
\end{remark}

 \begin{proposition}
 Suppose that $\mathcal{C}$ is a modular fusion category.
 Then we have that 
 \begin{align*}
 \sum_{j,k=1}^r d_jd_k \Schi_j^+\circ \Schi_k^-=
  \mu^2\vcenter{\hbox{\scalebox{0.6}{
\begin{tikzpicture}[xscale=0.8, yscale=0.6]
\begin{scope}[shift={(0,2)}]
\path [fill=brown!20!white] (0,0) [partial ellipse=0:180:1 and 0.6];
\path [fill=brown!20!white]  (0,0) [partial ellipse=180:360:1 and 1.2];
\draw (0,0) [partial ellipse=0:360:1 and 0.6];
\draw (0,0) [partial ellipse=180:360:1 and 1.2];
\end{scope}
\begin{scope}[shift={(0,-2)}]
\path [fill=brown!20!white] (0,0) [partial ellipse=180:360:1 and 0.6];
\path [fill=brown!20!white]  (0,0) [partial ellipse=0:180:1 and 1.2];
\draw [dashed](0,0) [partial ellipse=0:180:1 and 0.6];
\draw (0,0) [partial ellipse=180:360:1 and 0.6];
\draw (0,0) [partial ellipse=0:180:1 and 1.2];
\end{scope}
\end{tikzpicture}}}},
 \quad
  \sum_{j,k=1}^r d_jd_k\mathfrak{S}( \Schi_j^+\circ \Schi_k^- )= \mu^2
\vcenter{\hbox{\scalebox{0.6}{
\begin{tikzpicture}[xscale=0.8, yscale=0.6]
\path [fill=brown!20!white](-0.65, -3) rectangle (0.65, 3);
\begin{scope}[shift={(0,3)}]
\path [fill=brown!20!white] (0,0) [partial ellipse=0:180:0.6 and 0.3];
\draw (0,0) [partial ellipse=0:360:0.6 and 0.3];
\end{scope}
\draw (-0.6, 3)--(-0.6, 0) (0.6, 3)--(0.6, 0); 
\draw (-0.6, -3)--(-0.6, 0) (0.6, -3)--(0.6, 0);
\begin{scope}[shift={(0,-3)}]
\path [fill=brown!20!white] (0,0) [partial ellipse=180:360:0.6 and 0.3];
\draw [dashed](0,0) [partial ellipse=0:180:0.6 and 0.3];
\draw (0,0) [partial ellipse=180:360:0.6 and 0.3];
\end{scope}
\end{tikzpicture}}}}.
 \end{align*}
 \end{proposition}
 \begin{proof}
 It follows from Lemma \ref{lem:decomposition1}.
 \end{proof}

 \begin{theorem}[First Orthogonality]\label{thm:ortho1}
 Suppose that $\mathcal{C}$ is a braided fusion category and $\widetilde{Y}\in \mathcal{D}$. We have that 
 \begin{align*}
 \sum_{Y\in \Irr(\mathcal{D})}\chi_j^{\pm}(\widetilde{Y}Y)\chi_k^{\pm}(Y^*)=\frac{\mu}{d_j}\eta_{j}\delta_{j,k}\chi_j^\pm(\widetilde{Y}),
 \end{align*}
 where $\eta_{j}=1$ if $X_{j}$ is in the M\"{u}ger center and $\eta_{j}=0$ otherwise.
 \end{theorem} 
  \begin{proof}
 We consider the following graph:
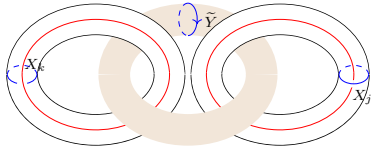
\begin{align*}
\vcenter{\hbox{\scalebox{0.7}{
\begin{tikzpicture}[scale=0.7]
\draw [line width=0.6cm, brown!20!white] (0,0) [partial ellipse=-0.1:180.1:2 and 1.5];
\draw [blue, -<-=0.5] (0,1.5) [partial ellipse=-90:90:0.25 and 0.44] node[black, pos=0.5, right] {\tiny $\widetilde{Y}$};
\draw [blue, dashed] (0,1.5) [partial ellipse=90:270:0.25 and 0.44];
\begin{scope}[shift={(2.5, 0)}]
\draw [double distance=0.57cm] (0,0) [partial ellipse=-0.1:180.1:2 and 1.5];
\draw [red] (0,0) [partial ellipse=0:180:2 and 1.5];
\end{scope} 
\begin{scope}[shift={(2.5, 0)}]
\draw [line width=0.6cm] (0,0) [partial ellipse=180:360:2 and 1.5];
\draw [white, line width=0.57cm] (0,0) [partial ellipse=178:362:2 and 1.5];
\draw [red] (0,0) [partial ellipse=178:362:2 and 1.5] ;
\end{scope}
\begin{scope}[shift={(-2.5, 0)}]
\draw [double distance=0.57cm] (0,0) [partial ellipse=-0.1:180.1:2 and 1.5];
\draw [red] (0,0) [partial ellipse=0:180:2 and 1.5] ;
\end{scope} 
\begin{scope}[shift={(-2.5, 0)}]
\draw [line width=0.6cm] (0,0) [partial ellipse=180:360:2 and 1.5];
\draw [white, line width=0.57cm] (0,0) [partial ellipse=178:362:2 and 1.5];
\draw [red] (0,0) [partial ellipse=178:362:2 and 1.5];
\end{scope}
\draw [line width=0.6cm, brown!20!white] (0,0) [partial ellipse=180:360:2 and 1.5];
\begin{scope}[shift={(4.5, 0)}]
\draw [blue, dashed](0,0) [partial ellipse=0:180:0.4 and 0.25] ;
\draw [white, line width=4pt] (0,0) [partial ellipse=270:250:0.4 and 0.25];
\draw [blue] (0,0) [partial ellipse=180:360:0.4 and 0.25] node[black, pos=0.7,below ] {\tiny $X_j$};
\end{scope}
\begin{scope}[shift={(-4.5, 0)}]
\draw [blue, dashed] (0,0) [partial ellipse=0:180:0.4 and 0.25];
\draw [blue] (0,0) [partial ellipse=180:250:0.4 and 0.25];
\draw [blue] (0,0) [partial ellipse=290:360:0.4 and 0.25] node[black, pos=0.6, above] {\tiny $X_k$};
\end{scope}
\end{tikzpicture}}}}
\end{align*}
By applying the Kirby color to the middle torus and Lemma \ref{lem:decomposition2}, we have that
\begin{align*}
 \delta_{\pm, \pm} \delta_{j,k} \frac{\mu^2}{d_j} \mu^2\chi_j^\pm(\widetilde{Y}) = \frac{1}{\mu^2} \mu \mu^4 \sum_{Y\in \Irr(\mathcal{D})}\chi_j^{\pm}(\widetilde{Y}Y)\chi_k^{\pm}(Y^*).
\end{align*}
This completes the proof.
 \end{proof}

\begin{remark}
For $\tilde{Y}=\1$, Theorem \ref{thm:ortho1} appeared in the proof of Theorem 5.5 in \cite{ENO21}.
\end{remark}

\begin{remark}
Theorem \ref{thm:ortho1} generalizes the orthogonality for the characters of finite groups. To be specific, taking $\tilde{Y}\in \text{Vec}(G)$ and $X_j, X_k\in \text{Rep}(G)$ be simple objects, the formula specialize to
$$\sum_{g\in G}\mathrm{Tr}\rho_{k}(gh)\overline{\mathrm{Tr}\rho_{j}(g)}=\frac{1}{\dim(\rho_{j})}\delta_{j, k}\text{Tr}\rho_{j}(g).$$
\end{remark}

\begin{theorem}[Second Orthogonality] \label{thm:ortho2}
Suppose that $\mathcal{C}$ is a braided spherical fusion category and $Y, Y'\in \mathcal{D}$. We have that 
\begin{align*}
 \sum_{j=1}^r \mu \chi_j^\pm(Y)\chi_j^\pm(Y')= &\sum_{k=1}^{r'}\sum_{j=1}^r \chi_j^\pm(YY_kY'Y_k^*) \\
=& \sum_{k=1}^{r'} \sum_{j=1}^r  d_j  \dim\hom_{\mathcal{D}}(YY_kY'Y_k^*, \alpha_{\mp}(X_j)).
\end{align*}
\end{theorem}
\begin{proof}
We consider the following graph:
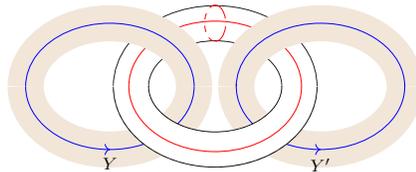
\begin{align*}
    \vcenter{\hbox{\scalebox{0.8}{
\begin{tikzpicture}[scale=0.7]
\draw [double distance=0.57cm] (0,0) [partial ellipse=-0.1:180.1:2 and 1.5];
\draw [red] (0,0) [partial ellipse=0:180:2.05 and 1.55];
\draw [red] (0,1.5) [partial ellipse=-90:-10:0.25 and 0.43];
\draw [red] (0,1.5) [partial ellipse=15:90:0.25 and 0.43];
\draw [red, dashed] (0,1.5) [partial ellipse=90:270:0.25 and 0.43];
\begin{scope}[shift={(2.5, 0)}]
\draw [line width=0.6cm, brown!20!white] (0,0) [partial ellipse=-0.1:180.1:2 and 1.5];
\draw [blue] (0,0) [partial ellipse=-0.1:180.1:2 and 1.5];
\end{scope} 
\begin{scope}[shift={(2.5, 0)}]
\draw [line width=0.6cm, brown!20!white] (0,0) [partial ellipse=180:360:2 and 1.5];
\draw [blue,  ->-=0.5] (0,0) [partial ellipse=178:362:2 and 1.5] node[black, pos=0.5, below] {\tiny $Y'$};
\end{scope}
\begin{scope}[shift={(-2.5, 0)}]
\draw [line width=0.6cm, brown!20!white] (0,0) [partial ellipse=-0.1:180.1:2 and 1.5];
\draw [blue] (0,0) [partial ellipse=-0.1:180.1:2 and 1.5];
\end{scope} 
\begin{scope}[shift={(-2.5, 0)}]
\draw [line width=0.6cm, brown!20!white] (0,0) [partial ellipse=180:360:2 and 1.5];
\draw [blue, ->-=0.5] (0,0) [partial ellipse=180:360:2 and 1.5] node[black, pos=0.5, below] {\tiny $Y$};
\end{scope}
\draw [line width=0.6cm] (0,0) [partial ellipse=180:360:2 and 1.5];
\draw [white, line width=0.57cm] (0,0) [partial ellipse=178:362:2 and 1.5];
\draw [red] (0,0) [partial ellipse=178:362:2.05 and 1.55];
\end{tikzpicture}}}}
\end{align*}
Applying the Kirby color to the middle torus and Lemma \ref{lem:decomposition2}, we have that 
\begin{align*}
& \sum_{k=1}^{r'}    \vcenter{\hbox{\scalebox{0.8}{
\begin{tikzpicture}[scale=0.7]
\draw [double distance=0.57cm] (0,0) [partial ellipse=-0.1:180.1:2 and 1.5];
\draw [red] (0,0) [partial ellipse=0:180:2.05 and 1.55];
\draw [red] (0,1.5) [partial ellipse=-90:-10:0.25 and 0.43];
\draw [red] (0,1.5) [partial ellipse=15:90:0.25 and 0.43];
\draw [red, dashed] (0,1.5) [partial ellipse=90:270:0.25 and 0.43];
\begin{scope}[shift={(-2.5, 0)}]
\draw [line width=0.75cm, brown!20!white] (0,0) [partial ellipse=-0.1:180.1:2 and 1.5];
\draw [blue] (0,0) [partial ellipse=-0.1:180.1:1.6 and 1.1];
\draw [blue] (0,0) [partial ellipse=-0.1:180.1:1.8 and 1.3];
\draw [blue] (0,0) [partial ellipse=-0.1:180.1:2.2 and 1.7];
\draw [blue] (0,0) [partial ellipse=-0.1:180.1:2 and 1.5];
\end{scope} 
\begin{scope}[shift={(-2.5, 0)}]
\draw [line width=0.75cm, brown!20!white] (0,0) [partial ellipse=180:360:2 and 1.5];
\draw [blue, ->-=0.5] (0,0) [partial ellipse=180:360:1.6 and 1.1] node[black, pos=0.5, above] {\tiny $Y'$};
\draw [blue] (0,0) [partial ellipse=180:360:1.8 and 1.3] node[black, pos=0.3, above] {\tiny $Y_k^*$};
\draw [blue ] (0,0) [partial ellipse=180:360:2.2 and 1.7] node[black, pos=0.3, below] {\tiny $Y_k$};
\draw [blue, ->-=0.5] (0,0) [partial ellipse=180:360:2 and 1.5] node[black, pos=0.5, below] {\tiny $Y$};
\end{scope}
\draw [line width=0.6cm] (0,0) [partial ellipse=180:360:2 and 1.5];
\draw [white, line width=0.57cm] (0,0) [partial ellipse=178:362:2 and 1.5];
\draw [red] (0,0) [partial ellipse=178:362:2.05 and 1.55];
\end{tikzpicture}}}}\\
=&\sum_{j=1}^r
    \vcenter{\hbox{\scalebox{0.8}{
\begin{tikzpicture}[scale=0.7]
\draw [double distance=0.57cm] (0,0) [partial ellipse=-0.1:180.1:2 and 1.5];
\draw [blue] (0,0) [partial ellipse=0:180:2.05 and 1.55];
\draw [red] (0,1.5) [partial ellipse=-90:-10:0.25 and 0.43];
\draw [red] (0,1.5) [partial ellipse=15:90:0.25 and 0.43];
\draw [red, dashed] (0,1.5) [partial ellipse=90:270:0.25 and 0.43];
\begin{scope}[shift={(-2.5, 0)}]
\draw [line width=0.6cm, brown!20!white] (0,0) [partial ellipse=-0.1:180.1:2 and 1.5];
\draw [blue] (0,0) [partial ellipse=-0.1:180.1:2 and 1.5];
\end{scope} 
\begin{scope}[shift={(-2.5, 0)}]
\draw [line width=0.6cm, brown!20!white] (0,0) [partial ellipse=180:360:2 and 1.5];
\draw [blue, ->-=0.5] (0,0) [partial ellipse=180:360:2 and 1.5] node[black, pos=0.5, below] {\tiny $Y$};
\end{scope}
\draw [line width=0.6cm] (0,0) [partial ellipse=180:360:2 and 1.5];
\draw [white, line width=0.57cm] (0,0) [partial ellipse=178:362:2 and 1.5];
\draw [blue, -<-=0.5] (0,0) [partial ellipse=178:362:2.05 and 1.55] node[black, pos=0.5, below] {\tiny $X_j$};
\end{tikzpicture}}}} 
 \vcenter{\hbox{\scalebox{0.8}{
\begin{tikzpicture}[scale=0.7]
\draw [double distance=0.57cm] (0,0) [partial ellipse=-0.1:180.1:2 and 1.5];
\draw [blue] (0,0) [partial ellipse=0:180:2.05 and 1.55];
\draw [red] (0,1.5) [partial ellipse=-90:-10:0.25 and 0.43];
\draw [red] (0,1.5) [partial ellipse=15:90:0.25 and 0.43];
\draw [red, dashed] (0,1.5) [partial ellipse=90:270:0.25 and 0.43];
\begin{scope}[shift={(2.5, 0)}]
\draw [line width=0.6cm, brown!20!white] (0,0) [partial ellipse=-0.1:180.1:2 and 1.5];
\draw [blue] (0,0) [partial ellipse=-0.1:180.1:2 and 1.5];
\end{scope} 
\begin{scope}[shift={(2.5, 0)}]
\draw [line width=0.6cm, brown!20!white] (0,0) [partial ellipse=180:360:2 and 1.5];
\draw [blue,  ->-=0.5] (0,0) [partial ellipse=178:362:2 and 1.5] node[black, pos=0.5, below] {\tiny $Y'$};
\end{scope}
\draw [line width=0.6cm] (0,0) [partial ellipse=180:360:2 and 1.5];
\draw [white, line width=0.57cm] (0,0) [partial ellipse=178:362:2 and 1.5];
\draw [blue, ->-=0.5] (0,0) [partial ellipse=178:362:2.05 and 1.55] node[black, pos=0.5, below] {\tiny $X_j$};
\end{tikzpicture}}}} \\ 
=& \sum_{j=1}^r \mu^3\chi_j^+(Y)\chi_j^+(Y').
\end{align*}
Note that 
\begin{align*}
\sum_{k=1}^{r'}    \vcenter{\hbox{\scalebox{0.8}{
\begin{tikzpicture}[scale=0.7]
\draw [double distance=0.57cm] (0,0) [partial ellipse=-0.1:180.1:2 and 1.5];
\draw [red] (0,0) [partial ellipse=0:180:2.05 and 1.55];
\draw [red] (0,1.5) [partial ellipse=-90:-10:0.25 and 0.43];
\draw [red] (0,1.5) [partial ellipse=15:90:0.25 and 0.43];
\draw [red, dashed] (0,1.5) [partial ellipse=90:270:0.25 and 0.43];
\begin{scope}[shift={(-2.5, 0)}]
\draw [line width=0.75cm, brown!20!white] (0,0) [partial ellipse=-0.1:180.1:2 and 1.5];
\draw [blue] (0,0) [partial ellipse=-0.1:180.1:1.6 and 1.1];
\draw [blue] (0,0) [partial ellipse=-0.1:180.1:1.8 and 1.3];
\draw [blue] (0,0) [partial ellipse=-0.1:180.1:2.2 and 1.7];
\draw [blue] (0,0) [partial ellipse=-0.1:180.1:2 and 1.5];
\end{scope} 
\begin{scope}[shift={(-2.5, 0)}]
\draw [line width=0.75cm, brown!20!white] (0,0) [partial ellipse=180:360:2 and 1.5];
\draw [blue, ->-=0.5] (0,0) [partial ellipse=180:360:1.6 and 1.1] node[black, pos=0.5, above] {\tiny $Y'$};
\draw [blue] (0,0) [partial ellipse=180:360:1.8 and 1.3] node[black, pos=0.3, above] {\tiny $Y_k^*$};
\draw [blue ] (0,0) [partial ellipse=180:360:2.2 and 1.7] node[black, pos=0.3, below] {\tiny $Y_k$};
\draw [blue, ->-=0.5] (0,0) [partial ellipse=180:360:2 and 1.5] node[black, pos=0.5, below] {\tiny $Y$};
\end{scope}
\draw [line width=0.6cm] (0,0) [partial ellipse=180:360:2 and 1.5];
\draw [white, line width=0.57cm] (0,0) [partial ellipse=178:362:2 and 1.5];
\draw [red] (0,0) [partial ellipse=178:362:2.05 and 1.55];
\end{tikzpicture}}}}
=&
\sum_{k=1}^{r'}    \vcenter{\hbox{\scalebox{0.8}{
\begin{tikzpicture}[scale=0.7]
\draw [double distance=0.57cm] (0,0) [partial ellipse=-0.1:180.1:2 and 1.5];
\draw [red] (0,0) [partial ellipse=0:80:2.05 and 1.55];
\draw [red] (0,0) [partial ellipse=90:180:2.05 and 1.55];
\draw [red] (0,1.5) [partial ellipse=-90:90:0.25 and 0.43];
\draw [red, dashed] (0,1.5) [partial ellipse=90:270:0.25 and 0.43];
\begin{scope}[shift={(-2.5, 0)}]
\draw [line width=0.6cm, brown!20!white] (0,0) [partial ellipse=-0.1:180.1:2 and 1.5];
\end{scope} 
\begin{scope}[shift={(-2.5, 0)}]
\draw [line width=0.6cm, brown!20!white] (0,0) [partial ellipse=180:360:2 and 1.5];
\end{scope}
\draw [line width=0.6cm] (0,0) [partial ellipse=180:360:2 and 1.5];
\draw [white, line width=0.57cm] (0,0) [partial ellipse=178:362:2 and 1.5];
\draw [red] (0,0) [partial ellipse=178:362:2.05 and 1.55];
\begin{scope}[shift={(-4.5, 0)}]
\draw [blue, dashed] (0,0) [partial ellipse=0:180:0.4 and 0.25];
\draw [blue] (0,0) [partial ellipse=180:360:0.4 and 0.25] node[black, pos=0.6, below] {\tiny $YY_kY'Y_k^*$};
\end{scope}
\end{tikzpicture}}}}\\
=&\sum_{k=1}^{r'} \sum_{j=1}^r d_j \mu^2 \dim\hom_{\mathcal{D}}(YY_kY'Y_k^*, \alpha_-(X_j)).
\end{align*}
Alternatively, we have that 
\begin{align*}
\sum_{k=1}^{r'}    \vcenter{\hbox{\scalebox{0.8}{
\begin{tikzpicture}[scale=0.7]
\draw [double distance=0.57cm] (0,0) [partial ellipse=-0.1:180.1:2 and 1.5];
\draw [red] (0,0) [partial ellipse=0:80:2.05 and 1.55];
\draw [red] (0,0) [partial ellipse=90:180:2.05 and 1.55];
\draw [red] (0,1.5) [partial ellipse=-90:90:0.25 and 0.43];
\draw [red, dashed] (0,1.5) [partial ellipse=90:270:0.25 and 0.43];
\begin{scope}[shift={(-2.5, 0)}]
\draw [line width=0.6cm, brown!20!white] (0,0) [partial ellipse=-0.1:180.1:2 and 1.5];
\end{scope} 
\begin{scope}[shift={(-2.5, 0)}]
\draw [line width=0.6cm, brown!20!white] (0,0) [partial ellipse=180:360:2 and 1.5];
\end{scope}
\draw [line width=0.6cm] (0,0) [partial ellipse=180:360:2 and 1.5];
\draw [white, line width=0.57cm] (0,0) [partial ellipse=178:362:2 and 1.5];
\draw [red] (0,0) [partial ellipse=178:362:2.05 and 1.55];
\begin{scope}[shift={(-4.5, 0)}]
\draw [blue, dashed] (0,0) [partial ellipse=0:180:0.4 and 0.25];
\draw [blue] (0,0) [partial ellipse=180:360:0.4 and 0.25] node[black, pos=0.6, below] {\tiny $YY_kY'Y_k^*$};
\end{scope}
\end{tikzpicture}}}}
= \sum_{k=1}^{r'}\sum_{j=1}^r \mu^2\chi_j^+(YY_kY'Y_k^*).
\end{align*}
The other identities follows from a similar argument.
\end{proof}

\begin{remark}
Theorem \ref{thm:ortho2} generalizes the orthogonality for the characters of finite groups by noting that the $\alpha$-induction is trivial.
\end{remark}

The second orthogonality is also true for multiple $2D$-colors. We denote the $\alpha$-induction with respect to $\mathcal{D}$ by $\alpha^{\mathcal{D}}_{\pm}$ and the character by $\chi_{j, \mathcal{D}}^\pm$.
\begin{theorem}[Second Orthogonality] \label{thm:ortho21}
Suppose that $\mathcal{C}$ is a braided spherical fusion category and $Y\in \mathcal{D}$, $Y'\in \mathcal{E}$. We have that 
\begin{align*}
 \sum_{j=1}^r \mu \chi_{j, \mathcal{D}}^\pm(Y)\chi_{j, \mathcal{E}}^\pm(Y')
=& \sum_{k=1}^{r'} \sum_{j=1}^r  d_j  \dim\hom(YM_kY'M_k^*, \alpha_{\mp}^{\mathcal{D}}(X_j)),
\end{align*}
where $M_k$ are simple objects in the module category $ _{\mathcal{D}}\mathcal{M}_{\mathcal{E}}$.
\end{theorem}
\begin{proof}
    The proof is similar to the one of Theorem \ref{thm:ortho2}.
\end{proof}

\begin{remark}
We can also consider the following graph:
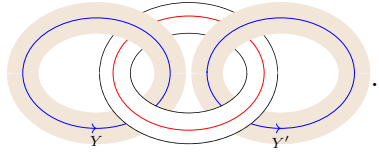
\begin{align*}
    \vcenter{\hbox{\scalebox{0.7}{
}}}.
\end{align*}
Applying the Kirby color to the middle torus and Lemma \ref{lem:decomposition2}, we have that 
\begin{align*}
& \mu^2\sum_{k=1}^{r'} \dim\hom(Y^*, Y_kY'Y_k^*)
= \sum_{k=1}^{r'}    \vcenter{\hbox{\scalebox{0.7}{
%
}}}.
\end{align*}
\end{remark}

\begin{theorem}[Third Orthogonality]\label{thm:ortho3}
Suppose that $\mathcal{C}$ is a braided spherical fusion category, $X, X'\in \Irr(\mathcal{C})$ and $Y\in \Irr(\mathcal{D})$.
We have that 
\begin{align*}
    &\chi_{X\otimes X'}^{+}(Y)\\
    = &\sum_{V\in \Irr(\mathcal{Z(D)})} \frac{\dim\hom_{\mathcal{Z(D)}}(I(Y), V)}{d_V\mu^3}
\vcenter{\hbox{\scalebox{0.7}{
%
}}}.
\end{align*}
A similar equality holds for $\chi_{X\otimes X'}^-(Y)$.
\end{theorem}

\begin{proof}
We consider the following graph:
    \begin{align*}
&\vcenter{\hbox{\scalebox{0.7}{
%
}}}
\end{align*}
Applying the Kirby color to the middle torus, we obtain that 
\begin{align*}
 \sum_{V\in \Irr(\mathcal{Z(D)})} \frac{1}{d_V}\dim\hom_{\mathcal{Z(D)}}(I(Y), V)
\vcenter{\hbox{\scalebox{0.7}{
%
}}}.
\end{align*}
On the other hand, we combine the two torus colored in $\mathcal{C}$ and obtain that
\begin{align*}
\vcenter{\hbox{\scalebox{0.7}{
%
}}}.
\end{align*}
This completes the proof of the theorem.
\end{proof}

\begin{remark}
    Suppose $\mathcal{C}=\rep(G)$ and $\mathcal{D}=\Vc(G)$, where $G$ is a finite group. 
    We obtain $\chi_{\rho\otimes \pi}(g)=\chi_\rho(g)\chi_\pi(g)$ from Theorem \ref{thm:ortho3},  where $g\in G$ and $\rho, \pi\in \rep(G)$.
\end{remark}

\begin{proposition}\label{prop:orthos}
Suppose $\mathcal{C}$ is a modular fusion category and $Y, Y'\in \mathcal{D}$.
We have that 
\begin{align*}
\sum_{k=1}^{r'}    \vcenter{\hbox{\scalebox{0.7}{
%
}}}
= \sum_{j=1}^r \frac{\mu^3}{d_j^2} \mathfrak{S}(\Schi_j^+\circ \Schi_j^-)(Y) \overline{\mathfrak{S}(\Schi_j^+\circ \Schi_j^-)(Y')}.
\end{align*}
\end{proposition}
\begin{proof}
We consider the following graph:
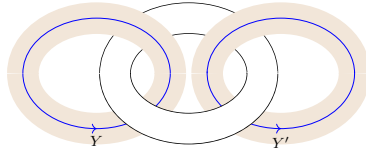
\begin{align*}
 \vcenter{\hbox{\scalebox{0.7}{
%
}}}
\end{align*}
Applying the Kirby color to the middle torus and Lemma \ref{lem:decomposition3}, we have that 
\begin{align*}
& \sum_{k=1}^{r'}    \vcenter{\hbox{\scalebox{0.7}{
%
}}} \\
=& \sum_{j=1}^r \frac{\mu^3}{d_j^2} \mathfrak{S}(\Schi_j^+\circ \Schi_j^-)(Y) \overline{\mathfrak{S}(\Schi_j^+\circ \Schi_j^-)(Y')}.
\end{align*}
By the $S$-matrix, we have that
\begin{align*}
 \sum_{k=1}^{r'}    \vcenter{\hbox{\scalebox{0.7}{
%
}}}.
\end{align*}
We then see the proposition is true.
\end{proof}

\begin{remark}
In Proposition \ref{prop:orthos}, if $\mathcal{C}$ is orthogonal to $\mathcal{D}$, we have
\begin{align*}
\sum_{j=1}^r \frac{\mu^2}{d_j^2} \mathfrak{S}(\Schi_j^+\circ \Schi_j^-)(Y) \overline{\mathfrak{S}(\Schi_j^+\circ \Schi_j^-)(Y')}=d_Yd_{Y'}.
\end{align*}
\end{remark}

Based on the $\alpha$-induction, we have the following trace formula.
\begin{proposition}[Trace Formula]\label{prop:mtrace}
Suppose that $\mathcal{C}$ is a modular fusion category.
Then we have that 
\begin{align*}
 \dim \hom_{\mathcal{D}} (\alpha_\pm(X_j), Y) =  \mu \sum_{k, \ell=1}^r S_{j\ell} d_k \mathfrak{S} (\Schi_\ell^\pm\circ \Schi_k^\mp)(Y).
\end{align*}
\end{proposition}
\begin{proof}
We have that 
\begin{align*}
\mu \dim \hom_{\mathcal{D}} (\alpha_+(X_j), Y)= & \vcenter{\hbox{\scalebox{0.7}{
\begin{tikzpicture}[scale=0.7]
\begin{scope}[shift={(5, 0)}]
\draw [double distance=0.57cm] (0,0) [partial ellipse=-0.1:180.1:2 and 1.5];
\draw [blue] (0,0) [partial ellipse=-0.1:180.1:2 and 1.5];
\end{scope} 
\begin{scope}[shift={(2.5, 0)}]
\draw [white, line width=0.63cm] (0,0) [partial ellipse=-0.1:180.1:2 and 1.5];
\draw [line width=0.58cm, brown!20!white] (0,0) [partial ellipse=-0.1:180.1:2 and 1.5];
\end{scope} 
\begin{scope}[shift={(2.5, 0)}]
\draw [white, line width=0.63cm] (0,0) [partial ellipse=180:360:2 and 1.5];
\draw [brown!20!white, line width=0.58cm] (0,0) [partial ellipse=178:362:2 and 1.5];
\end{scope}
\begin{scope}[shift={(5, 0)}]
\draw [line width=0.6cm] (0,0) [partial ellipse=180:360:2 and 1.5];
\draw [white, line width=0.57cm] (0,0) [partial ellipse=178:362:2 and 1.5];
\draw [blue, ->-=0.5] (0,0) [partial ellipse=178:362:2 and 1.5] node[black, pos=0.5, below] {\tiny $X_j$};
\end{scope}
\begin{scope}[shift={(7, 0)}]
\draw [red, dashed] (0,0) [partial ellipse=0:180:0.42 and 0.25];
\draw [red] (0,0) [partial ellipse=180:260:0.42 and 0.25];
\draw [red] (0,0) [partial ellipse=280:360:0.42 and 0.25];
\end{scope}
\begin{scope}[shift={(0.5, 0)}]
\draw [blue, dashed](0,0) [partial ellipse=0:180:0.4 and 0.25] ;
\draw [blue, ->-=0.5] (0,0) [partial ellipse=180:360:0.4 and 0.25] node[black, pos=0.5, below] {\tiny $Y$};
\end{scope}
\end{tikzpicture}}}}\\
=& \sum_{k, \ell=1}^r  \frac{S_{j\ell} d_k }{\mu}
\vcenter{\hbox{\scalebox{0.7}{
\begin{tikzpicture}[scale=0.7]
\begin{scope}[shift={(5, 0)}]
\draw [double distance=0.57cm] (0,0) [partial ellipse=-0.1:180.1:2 and 1.5];
\draw [red] (0,0) [partial ellipse=-0.1:180.1:2 and 1.5];
\end{scope} 
\begin{scope}[shift={(2.5, 0)}]
\draw [white, line width=0.63cm] (0,0) [partial ellipse=-0.1:180.1:2 and 1.5];
\draw [line width=0.58cm, brown!20!white] (0,0) [partial ellipse=-0.1:180.1:2 and 1.5];
\end{scope} 
\begin{scope}[shift={(2.5, 0)}]
\draw [white, line width=0.63cm] (0,0) [partial ellipse=180:360:2 and 1.5];
\draw [brown!20!white, line width=0.58cm] (0,0) [partial ellipse=178:362:2 and 1.5];
\end{scope}
\begin{scope}[shift={(5, 0)}]
\draw [line width=0.6cm] (0,0) [partial ellipse=180:360:2 and 1.5];
\draw [white, line width=0.57cm] (0,0) [partial ellipse=178:362:2 and 1.5];
\draw [red] (0,0) [partial ellipse=178:362:2 and 1.5];
\end{scope}
\begin{scope}[shift={(0.5, 0)}]
\draw [blue, dashed](0,0) [partial ellipse=0:180:0.4 and 0.25] ;
\draw [blue, ->-=0.5] (0,0) [partial ellipse=180:360:0.4 and 0.25] node[black, pos=0.5, below] {\tiny $Y$};
\end{scope}
\begin{scope}[shift={(3.05, 0.3)}]
\draw [blue, -<-=0.4, dashed](0,0) [partial ellipse=0:180:0.47 and 0.25] node[black, pos=0.4, above] {\tiny $X_k$};
\draw [white, line width=4pt] (0,0) [partial ellipse=290:270:0.47 and 0.25];
\draw [blue] (0,0) [partial ellipse=280:360:0.47 and 0.25];
\draw [blue] (0,0) [partial ellipse=180:260:0.47 and 0.25];
\end{scope}
\begin{scope}[shift={(3.05, -0.3)}]
\draw [blue, dashed](0,0) [partial ellipse=0:180:0.47 and 0.25] ;
\draw [white, line width=4pt] (0,0) [partial ellipse=290:270:0.47 and 0.25];
\draw [blue, -<-=0.2] (0,0) [partial ellipse=180:360:0.47 and 0.25] node[black, pos=0.2, below] {\tiny $X_\ell$};
\end{scope}
\end{tikzpicture}}}}\\
=& \mu^2\sum_{k, \ell=1}^r S_{j\ell} d_k \mathfrak{S} (\Schi_\ell^+\circ \Schi_k^-)(Y).
\end{align*}
This completes the proof.
\end{proof}

\begin{remark}
Suppose $G$ is a finite group and $H$ is a subgroup.
Let $\rho$ be the induced representation associated to the trivial representation of $H$.
The Poisson summation formula is 
\begin{align*}
    \sum_{\pi \in \text{Rep}(G)}\dim\hom(\rho, \pi) \hat{f}(\pi) = \frac{1}{|H|} \sum_{x\in G}\sum_{h\in H} f(xhx^{-1}),
\end{align*}
where $f$ is a function on $G$.
This formula can be interpreted in the alterfold TQFT as
\begin{align*}
\vcenter{\hbox{\scalebox{0.5}{
\begin{tikzpicture}[xscale=0.8, yscale=0.6]
\draw [line width=0.83cm] (0,0) [partial ellipse=-0.1:180.1:2 and 1.5];
\draw [white, line width=0.8cm] (0,0) [partial ellipse=-0.1:180.1:2 and 1.5];
\path [fill=brown!20!white](-0.65, -3) rectangle (0.65, 3);
\begin{scope}[shift={(0,3)}]
\path [fill=brown!20!white] (0,0) [partial ellipse=0:180:0.6 and 0.3];
\draw (0,0) [partial ellipse=0:360:0.6 and 0.3];
\end{scope}
\draw (-0.6, 3)--(-0.6, 0) (0.6, 3)--(0.6, 0); 
\draw (-0.6, -3)--(-0.6, 0) (0.6, -3)--(0.6, 0);
\draw [line width=0.83cm] (0,0) [partial ellipse=180:360:2 and 1.5];
\draw [white, line width=0.8cm] (0,0) [partial ellipse=179:361:2 and 1.5];
\begin{scope}[shift={(0,-3)}]
\path [fill=brown!20!white] (0,0) [partial ellipse=180:360:0.6 and 0.3];
\draw [dashed](0,0) [partial ellipse=0:180:0.6 and 0.3];
\draw (0,0) [partial ellipse=180:360:0.6 and 0.3];
\end{scope}
\begin{scope}[shift={(2, 0)}]
\draw [blue, dashed] (0,0) [partial ellipse=0:180:0.46 and 0.25];
\draw [blue, ->-=0.5] (0,0) [partial ellipse=180:360:0.46 and 0.25] node[black, pos=0.5, below] {\tiny $\rho$};
\end{scope}
\end{tikzpicture}}}}
=\frac{1}{|H|}
\sum_{h\in H}\vcenter{\hbox{\scalebox{0.5}{
\begin{tikzpicture}[xscale=0.8, yscale=0.6]
\draw [brown!20!white, line width=0.83cm] (0,0) [partial ellipse=-0.1:180.1:2 and 1.5];
\draw [blue] (0,0) [partial ellipse=-0.1:180.1:2 and 1.5];
\path [fill=brown!20!white](-0.65, -3) rectangle (0.65, 3);
\begin{scope}[shift={(0,3)}]
\path [fill=brown!20!white] (0,0) [partial ellipse=0:180:0.6 and 0.3];
\draw (0,0) [partial ellipse=0:360:0.6 and 0.3];
\end{scope}
\draw (-0.6, 3)--(-0.6, 0) (0.6, 3)--(0.6, 0); 
\draw (-0.6, -3)--(-0.6, 0) (0.6, -3)--(0.6, 0);
\draw [white, line width=0.88cm] (0,0) [partial ellipse=179:361:2 and 1.5];
\draw [brown!20!white, line width=0.83cm] (0,0) [partial ellipse=179:361:2 and 1.5];
\draw [blue, ->-=0.5] (0,0) [partial ellipse=179:361:2 and 1.5] node[black, pos=0.5, below] {\tiny $\bC_h$};
\begin{scope}[shift={(0,-3)}]
\path [fill=brown!20!white] (0,0) [partial ellipse=180:360:0.6 and 0.3];
\draw [dashed](0,0) [partial ellipse=0:180:0.6 and 0.3];
\draw (0,0) [partial ellipse=180:360:0.6 and 0.3];
\end{scope}
\end{tikzpicture}}}}.
\end{align*}
If $\rho$ is an induced representation associated to a representation $\pi_H$ of $H$, then 
\begin{align*}
\vcenter{\hbox{\scalebox{0.5}{
\begin{tikzpicture}[xscale=0.8, yscale=0.6]
\draw [line width=0.83cm] (0,0) [partial ellipse=-0.1:180.1:2 and 1.5];
\draw [white, line width=0.8cm] (0,0) [partial ellipse=-0.1:180.1:2 and 1.5];
\path [fill=brown!20!white](-0.65, -3) rectangle (0.65, 3);
\begin{scope}[shift={(0,3)}]
\path [fill=brown!20!white] (0,0) [partial ellipse=0:180:0.6 and 0.3];
\draw (0,0) [partial ellipse=0:360:0.6 and 0.3];
\end{scope}
\draw (-0.6, 3)--(-0.6, 0) (0.6, 3)--(0.6, 0); 
\draw (-0.6, -3)--(-0.6, 0) (0.6, -3)--(0.6, 0);
\draw [line width=0.83cm] (0,0) [partial ellipse=180:360:2 and 1.5];
\draw [white, line width=0.8cm] (0,0) [partial ellipse=179:361:2 and 1.5];
\begin{scope}[shift={(0,-3)}]
\path [fill=brown!20!white] (0,0) [partial ellipse=180:360:0.6 and 0.3];
\draw [dashed](0,0) [partial ellipse=0:180:0.6 and 0.3];
\draw (0,0) [partial ellipse=180:360:0.6 and 0.3];
\end{scope}
\begin{scope}[shift={(2, 0)}]
\draw [blue, dashed] (0,0) [partial ellipse=0:180:0.46 and 0.25];
\draw [blue, ->-=0.5] (0,0) [partial ellipse=180:360:0.46 and 0.25] node[black, pos=0.5, below] {\tiny $\rho$};
\end{scope}
\end{tikzpicture}}}}
=\frac{1}{|H|}
\sum_{h\in H} \chi_{\pi_H}(h)\vcenter{\hbox{\scalebox{0.5}{
\begin{tikzpicture}[xscale=0.8, yscale=0.6]
\draw [brown!20!white, line width=0.83cm] (0,0) [partial ellipse=-0.1:180.1:2 and 1.5];
\draw [blue] (0,0) [partial ellipse=-0.1:180.1:2 and 1.5];
\path [fill=brown!20!white](-0.65, -3) rectangle (0.65, 3);
\begin{scope}[shift={(0,3)}]
\path [fill=brown!20!white] (0,0) [partial ellipse=0:180:0.6 and 0.3];
\draw (0,0) [partial ellipse=0:360:0.6 and 0.3];
\end{scope}
\draw (-0.6, 3)--(-0.6, 0) (0.6, 3)--(0.6, 0); 
\draw (-0.6, -3)--(-0.6, 0) (0.6, -3)--(0.6, 0);
\draw [white, line width=0.88cm] (0,0) [partial ellipse=179:361:2 and 1.5];
\draw [brown!20!white, line width=0.83cm] (0,0) [partial ellipse=179:361:2 and 1.5];
\draw [blue, ->-=0.5] (0,0) [partial ellipse=179:361:2 and 1.5] node[black, pos=0.5, below] {\tiny $\bC_h$};
\begin{scope}[shift={(0,-3)}]
\path [fill=brown!20!white] (0,0) [partial ellipse=180:360:0.6 and 0.3];
\draw [dashed](0,0) [partial ellipse=0:180:0.6 and 0.3];
\draw (0,0) [partial ellipse=180:360:0.6 and 0.3];
\end{scope}
\end{tikzpicture}}}}.
\end{align*}
The trace formula in Proposition \ref{prop:mtrace} does not imply the Poisson summation formula.
\end{remark}

\section{Identities for Modular Invariance}

In this section, we establish several identities for the modular invariants associated with multiple Morita contexts. 
Not all modular invariants arise from Morita contexts. 
These identities serve as potential obstructions for determining whether a given modular invariant originates from a Morita context.

The topological interpretation of the modular invariant inspires us to study the spectral theory of the second torus action.
For a modular fusion category $\mathcal{C}$ and Morita context $\mathcal{D}$, we denote by $(z_{jk}^{\mathcal{D}})_{jk}$ the modular invariant.
\begin{lemma}\label{lem:zcoefficient}
Suppose $\mathcal{C}$ is a modular fusion category. Then we have that 
\begin{align*}
\vcenter{\hbox{\scalebox{0.7}{
\begin{tikzpicture}[xscale=0.8, yscale=0.6]
\draw [brown!20!white, line width=0.7cm] (0,0) [partial ellipse=0:180:2 and 1.5];
\node at (-2, 0.5) {$\mathcal{D}$};
\begin{scope}[shift={(0,3)}]
\draw (0,0) [partial ellipse=0:360:0.6 and 0.3];
\end{scope}
\path [fill=white] (-0.6, 0) rectangle (0.6, 2.7);
\draw [blue, ->-=0.5] (-0.2, 2.8)--(-0.2, 0) node [left, pos=0.6] {\tiny $X_k$} ;
\draw [blue, -<-=0.5] (0.2, 2.8)--(0.2, 0) node [right, pos=0.6] {\tiny $X_j$}; 
\draw (-0.6, 3)--(-0.6, 0) (0.6, 3)--(0.6, 0); 
\draw [blue] (-0.2, -3.2)--(-0.2, 0) (0.2, -3.2)--(0.2, 0);
\draw (-0.6, -3)--(-0.6, 0) (0.6, -3)--(0.6, 0);
\draw [brown!20!white, line width=0.7cm] (0,0) [partial ellipse=180:360:2 and 1.5];
\begin{scope}[shift={(0,-3)}]
\draw [dashed](0,0) [partial ellipse=0:180:0.6 and 0.3];
\draw (0,0) [partial ellipse=180:360:0.6 and 0.3];
\end{scope}
\draw [red, dashed](0,0) [partial ellipse=0:180:0.6 and 0.3]; 
\draw [white, line width=4pt]  (0,0) [partial ellipse=220:270:0.6 and 0.3];
\draw [red]  (0,0) [partial ellipse=180:280:0.6 and 0.3];
\draw [red]  (0,0) [partial ellipse=300:360:0.6 and 0.3];
\end{tikzpicture}}}}
=\frac{z_{jk}^{\mathcal{D}}\mu}{d_jd_k}
\vcenter{\hbox{\scalebox{0.7}{
\begin{tikzpicture}[xscale=0.8, yscale=0.6]
\begin{scope}[shift={(0,3)}]
\draw (0,0) [partial ellipse=0:360:0.6 and 0.3];
\end{scope}
\path [fill=white] (-0.6, 0) rectangle (0.6, 2.7);
\draw [blue, ->-=0.5] (-0.2, 2.8)--(-0.2, 0) node [left, pos=0.6] {\tiny $X_k$} ;
\draw [blue, -<-=0.5] (0.2, 2.8)--(0.2, 0) node [right, pos=0.6] {\tiny $X_j$}; 
\draw (-0.6, 3)--(-0.6, 0) (0.6, 3)--(0.6, 0); 
\draw [blue] (-0.2, -3.2)--(-0.2, 0) (0.2, -3.2)--(0.2, 0);
\draw (-0.6, -3)--(-0.6, 0) (0.6, -3)--(0.6, 0);
\begin{scope}[shift={(0,-3)}]
\draw [dashed](0,0) [partial ellipse=0:180:0.6 and 0.3];
\draw (0,0) [partial ellipse=180:360:0.6 and 0.3];
\end{scope}
\draw [red, dashed](0,0) [partial ellipse=0:180:0.6 and 0.3]; 
\draw [white, line width=4pt]  (0,0) [partial ellipse=220:270:0.6 and 0.3];
\draw [red]  (0,0) [partial ellipse=180:280:0.6 and 0.3];
\draw [red]  (0,0) [partial ellipse=300:360:0.6 and 0.3];
\end{tikzpicture}}}}.
\end{align*}
\end{lemma}

\begin{proof}
The inner tube is a simple idempotent in $\mathcal{Z(C)}$, hence the left hand side equals a multiple of identity. 
Hence we only need to verify the equality after taking the trace. 
The left hand side equals to $z_{jk}^{\mathcal{D}}$. 
The right hand side equals to 
$$\frac{z_{jk}^{\mathcal{D}}}{d_kd_j}\text{tr}\left(\vcenter{\hbox{\scalebox{0.7}{
\begin{tikzpicture}[xscale=0.8, yscale=0.6]
\begin{scope}[shift={(0,3)}]
\draw (0,0) [partial ellipse=0:360:0.6 and 0.3];
\end{scope}
\path [fill=white] (-0.6, 0) rectangle (0.6, 2.7);
\draw [blue, ->-=0.5] (-0.2, 2.8)--(-0.2, 0) node [left, pos=0.6] {\tiny $X_k$} ;
\draw [blue, -<-=0.5] (0.2, 2.8)--(0.2, 0) node [right, pos=0.6] {\tiny $X_j$}; 
\draw (-0.6, 3)--(-0.6, 0) (0.6, 3)--(0.6, 0); 
\draw [blue] (-0.2, -3.2)--(-0.2, 0) (0.2, -3.2)--(0.2, 0);
\draw (-0.6, -3)--(-0.6, 0) (0.6, -3)--(0.6, 0);
\begin{scope}[shift={(0,-3)}]
\draw [dashed](0,0) [partial ellipse=0:180:0.6 and 0.3];
\draw (0,0) [partial ellipse=180:360:0.6 and 0.3];
\end{scope}
\draw [red, dashed](0,0) [partial ellipse=0:180:0.6 and 0.3]; 
\draw [white, line width=4pt]  (0,0) [partial ellipse=220:270:0.6 and 0.3];
\draw [red]  (0,0) [partial ellipse=180:280:0.6 and 0.3];
\draw [red]  (0,0) [partial ellipse=300:360:0.6 and 0.3];
\end{tikzpicture}}}}\right)=\frac{z_{jk}^{\mathcal{D}}}{d_kd_j}\mu d_kd_j=z_{jk}^{\mathcal{D}}.$$
This completes the proof of the Lemma.
\end{proof}

Now let us consider multiple of $2D$-colors: $\mathcal{D}_1$, $\mathcal{D}_2$, \ldots, $\mathcal{D}_n$.
The corresponding $Z$-matrices are denoted by $\displaystyle \left(z^{\mathcal{D}_j}_{jk}\right)_{j,k=1}^r$. 
Recall that $I$ is the induction functor, and
\begin{align}\label{eq:G}
G^+\left(
\vcenter{\hbox{\scalebox{0.6}{
\begin{tikzpicture}[scale=0.35]
\draw[blue, ->-=0.5] (0, 4.2) node[above, black]{\tiny{${(X, e_X)}$}}->(0, 0.5);
\draw[blue, ->-=0.5] (0, -0.5)--(0, -4.8) node[below, black]{\tiny{${(Y, e_Y)}$}};
\node [draw, fill=white] (0, 0){\tiny $f$};
\end{tikzpicture}}}}\right)
=\frac{1}{\mu}
\vcenter{\hbox{\scalebox{0.6}{
\begin{tikzpicture}[scale=0.35]
\draw (0,5) [partial ellipse=0:360:2 and 0.8];
\draw (-2, 5)--(-2, -4);
\draw (2, 5)--(2, -4);
\draw[dashed] (0,-4) [partial ellipse=0:180:2 and 0.8];
\draw (0,-4) [partial ellipse=180:360:2 and 0.8];
\draw[blue, ->-=0.5] (0, 4.2) node[above, black]{\tiny{${X}$}}->(0, 0.5);
\draw[blue, ->-=0.5] (0, -0.5)--(0, -4.8)node[below, black]{\tiny{${Y}$}};
\draw [red](0, 3) [partial ellipse=-75:0:2 and 0.8];
\draw [red, dashed](0, 3) [partial ellipse=0:180:2 and 0.8];
\draw [red](0, 3) [partial ellipse=180:255:2 and 0.8];
\node [draw, fill=white] (0, 0){\tiny $f$};
\end{tikzpicture}}}},
\quad 
G^-\left(
\vcenter{\hbox{\scalebox{0.6}{
\begin{tikzpicture}[scale=0.35]
\draw[blue, ->-=0.5] (0, 4.2) node[above, black]{\tiny{${(X, e_X)}$}}->(0, 0.5);
\draw[blue, ->-=0.5] (0, -0.5)--(0, -4.8) node[below, black]{\tiny{${(Y, e_Y)}$}};
\node [draw, fill=white] (0, 0){\tiny $f$};
\end{tikzpicture}}}}\right)
=\frac{1}{\mu}
\vcenter{\hbox{\scalebox{0.6}{
\begin{tikzpicture}[scale=0.35]
\draw (0,5) [partial ellipse=0:360:2 and 0.8];
\draw (-2, 5)--(-2, -4);
\draw (2, 5)--(2, -4);
\draw[dashed] (0,-4) [partial ellipse=0:180:2 and 0.8];
\draw (0,-4) [partial ellipse=180:360:2 and 0.8];
\draw[blue, ->-=0.5] (0, 4.2) node[above, black]{\tiny{${X}$}}->(0, 0.5);
\draw[blue, ->-=0.5] (0, -0.5)--(0, -4.8)node[below, black]{\tiny{${Y}$}};
\path [fill=white] (0, 2.2) circle (0.2cm);
\draw [red](0, 3) [partial ellipse=-90:0:2 and 0.8];
\draw [red, dashed](0, 3) [partial ellipse=0:180:2 and 0.8];
\draw [red](0, 3) [partial ellipse=180:272:2 and 0.8];
\node [draw, fill=white] (0, 0){\tiny $f$};
\end{tikzpicture}}}}.
\end{align}

\begin{theorem}\label{thm:physical}
Suppose $\mathcal{C}$ is a modular fusion category and $\mathcal{D}_s$, $s=1, \ldots, n$ are fusion categories Morita equivalent to $\mathcal{C}$.
Then we have that 
\begin{enumerate}
\item $\displaystyle \sum_{j, k=1}^r \prod_{s=1}^nz_{jk}^{\mathcal{D}_s}\frac{\mu^{n-2}}{d_k^{n-2}d_j^{n-2}}=\dim\hom_{\mathcal{Z(C)}}(\1_{\mathcal{Z(C)}}, I(\1_{\mathcal{D}_1})\otimes \cdots \otimes I(\1_{\mathcal{D}_n}))$.
\item  $\displaystyle \sum_{j=1}^r\prod_{s=1}^nz_{jj}^{\mathcal{D}_s}\frac{\mu^{n-2}}{d_j^{2n-4}}= \dim\hom_{\mathcal{Z(C)}}(I(\1_{\mathcal{C}}), I(\1_{\mathcal{D}_1})\otimes \cdots \otimes I(\1_{\mathcal{D}_n}))$.
\item $\displaystyle \sum_{j=1}^r \prod_{s=1}^nz_{j1}^{\mathcal{D}_s}\frac{\mu^{n-2}}{d_j^{n-2}}=\dim\hom_{\mathcal{Z(C)}}(\sum_j G^+(X_j), I(\1_{\mathcal{D}_1})\otimes \cdots \otimes I(\1_{\mathcal{D}_n}))$.
\item $\displaystyle \sum_{j=1}^r \prod_{s=1}^nz_{1j}^{\mathcal{D}_s}\frac{\mu^{n-2}}{d_j^{n-2}} =\dim\hom_{\mathcal{Z(C)}}(\sum_j G^-(X_j), I(\1_{\mathcal{D}_1})\otimes \cdots \otimes I(\1_{\mathcal{D}_n}))$.
\end{enumerate}
\end{theorem}

\begin{proof}
We consider the following summation:
\begin{align*}
\mathcal{I}=\sum_{j,k=1}^r \frac{d_jd_k}{\mu}tr\left(
\vcenter{\hbox{\scalebox{0.7}{
\begin{tikzpicture}[xscale=0.8, yscale=0.6]
\begin{scope}[shift={(0, 2)}]
\draw [brown!20!white, line width=0.7cm] (0,0) [partial ellipse=0:180:2 and 1.2];
\node at (-2, 0.5) {$\mathcal{D}_1$};
\end{scope}
\node at (-2, 0.5) {$\vdots$};
\draw [decorate, decoration={brace}] (-2.5, -2) --(-2.5, 2) node [pos=0.5, left] {$n$};
\begin{scope}[shift={(0, -2)}]
\draw [brown!20!white, line width=0.7cm] (0,0) [partial ellipse=0:180:2 and 1.2];
\node at (-2, 0.5) {$\mathcal{D}_n$};
\end{scope}
\begin{scope}[shift={(0,5)}]
\draw (0,0) [partial ellipse=0:360:0.6 and 0.3];
\end{scope}
\path [fill=white] (-0.6, -2) rectangle (0.6, 4.7);
\draw [blue, ->-=0.5] (-0.2, 4.8)--(-0.2, 0) node [left, pos=0.6] {\tiny $X_k$} ;
\draw [blue, -<-=0.5] (0.2, 4.8)--(0.2, 0) node [right, pos=0.6] {\tiny $X_j$}; 
\draw (-0.6, 5)--(-0.6, 0) (0.6, 5)--(0.6, 0); 
\draw [blue] (-0.2, -5.2)--(-0.2, 0) (0.2, -5.2)--(0.2, 0);
\draw (-0.6, -5)--(-0.6, 0) (0.6, -5)--(0.6, 0);
\begin{scope}[shift={(0, 2)}]
\draw [brown!20!white, line width=0.7cm] (0,0) [partial ellipse=180:360:2 and 1.2];
\end{scope}
\begin{scope}[shift={(0, -2)}]
\draw [brown!20!white, line width=0.7cm] (0,0) [partial ellipse=180:360:2 and 1.2];
\end{scope}
\begin{scope}[shift={(0,-5)}]
\draw [dashed](0,0) [partial ellipse=0:180:0.6 and 0.3];
\draw (0,0) [partial ellipse=180:360:0.6 and 0.3];
\end{scope}
\begin{scope}[shift={(0, -0.5)}]
\draw [red, dashed](0,0) [partial ellipse=0:180:0.6 and 0.3]; 
\draw [white, line width=4pt]  (0,0) [partial ellipse=220:270:0.6 and 0.3];
\draw [red]  (0,0) [partial ellipse=180:280:0.6 and 0.3];
\draw [red]  (0,0) [partial ellipse=300:360:0.6 and 0.3];
\end{scope}
\end{tikzpicture}}}}\right).
\end{align*}
By Lemma \ref{lem:zcoefficient}, we have that
\begin{align*}
\mathcal{I}=\sum_{j,k=1}^r \frac{d_jd_k}{\mu}\prod_{s=1}^nz_{jk}^{\mathcal{D}_s}\frac{\mu^{n+1}}{d_k^{n-1}d_j^{n-1}}.
\end{align*}
On the other hand, by summing the minimal central idempotents, we have that 
\begin{align*}
\mathcal{I}=\vcenter{\hbox{\scalebox{0.8}{
\begin{tikzpicture}[scale=0.7]
\draw [double distance=0.57cm] (0,0) [partial ellipse=-0.1:180.1:2 and 1.5];
\draw [red] (0,0) [partial ellipse=0:180:2.05 and 1.55];
\begin{scope}[shift={(2.5, 0)}]
\draw [line width=0.6cm, brown!20!white] (0,0) [partial ellipse=-0.1:180.1:2 and 1.5];
\end{scope} 
\begin{scope}[shift={(2.5, 0)}]
\draw [line width=0.6cm, brown!20!white] (0,0) [partial ellipse=180:360:2 and 1.5];
\node at (-2, 0.5) {$\mathcal{D}_n$};
\end{scope}
\begin{scope}[shift={(-2.5, 0)}]
\draw [line width=0.6cm, brown!20!white] (0,0) [partial ellipse=-0.1:180.1:2 and 1.5];
\end{scope} 
\begin{scope}[shift={(-2.5, 0)}]
\draw [line width=0.6cm, brown!20!white] (0,0) [partial ellipse=180:360:2 and 1.5];
\node at (-2, 0.5) {$\mathcal{D}_1$};
\end{scope}
\draw [line width=0.6cm] (0,0) [partial ellipse=180:360:2 and 1.5];
\draw [white, line width=0.57cm] (0,0) [partial ellipse=178:362:2 and 1.5];
\draw [red] (0,0) [partial ellipse=178:362:2.05 and 1.55];
\node at (0, 2.2) {$\ldots\ldots$};
\draw [decorate, decoration={brace}] (-2.5, 2.5)--(2.5, 2.5);
\node [above] at (0, 2.6) {$n$};
\end{tikzpicture}}}}
=\dim\hom_{\mathcal{Z(C)}}(\1_{\mathcal{Z(C)}}, I(\1_{\mathcal{D}_1})\otimes \cdots \otimes I(\1_{\mathcal{D}_n})) \mu^2.
\end{align*}

Replacing the middle torus by the following torus respectively:
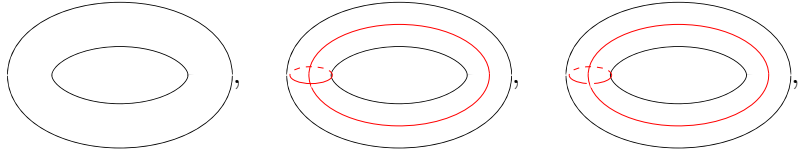
\begin{align*}
\vcenter{\hbox{\scalebox{0.75}{
\begin{tikzpicture}[xscale=0.8, yscale=0.6]
\draw [line width=0.8cm] (0,0) [partial ellipse=0:180:2 and 1.5];
\draw [white, line width=0.77cm] (0,0) [partial ellipse=-0.1:180.1:2 and 1.5];
\draw [line width=0.8cm] (0,0) [partial ellipse=180:360:2 and 1.5];
\draw [white, line width=0.77cm] (0,0) [partial ellipse=178:362:2 and 1.5];
\end{tikzpicture}}}},
\quad
\vcenter{\hbox{\scalebox{0.75}{
\begin{tikzpicture}[xscale=0.8, yscale=0.6]
\draw [line width=0.8cm] (0,0) [partial ellipse=0:180:2 and 1.5];
\draw [white, line width=0.77cm] (0,0) [partial ellipse=-0.1:180.1:2 and 1.5];
\draw [red] (0,0) [partial ellipse=0:180:2 and 1.5];
\draw [line width=0.8cm] (0,0) [partial ellipse=180:360:2 and 1.5];
\draw [white, line width=0.77cm] (0,0) [partial ellipse=178:362:2 and 1.5];
\draw [red] (0,0) [partial ellipse=178:362:2 and 1.5]; 
\begin{scope}[shift={(-1.95, 0)}]
\draw [red, dashed](0,0) [partial ellipse=0:180:0.47 and 0.25] ;
\draw [white, line width=4pt] (0,0) [partial ellipse=290:270:0.47 and 0.25];
\draw [red] (0,0) [partial ellipse=180:360:0.47 and 0.25];
\end{scope}
\end{tikzpicture}}}},
\quad 
\vcenter{\hbox{\scalebox{0.75}{
\begin{tikzpicture}[xscale=0.8, yscale=0.6]
\draw [line width=0.8cm] (0,0) [partial ellipse=0:180:2 and 1.5];
\draw [white, line width=0.77cm] (0,0) [partial ellipse=-0.1:180.1:2 and 1.5];
\draw [red] (0,0) [partial ellipse=0:180:2 and 1.5];
\draw [line width=0.8cm] (0,0) [partial ellipse=180:360:2 and 1.5];
\draw [white, line width=0.77cm] (0,0) [partial ellipse=178:362:2 and 1.5];
\draw [red] (0,0) [partial ellipse=178:362:2 and 1.5]; 
\begin{scope}[shift={(-1.95, 0)}]
\draw [red, dashed](0,0) [partial ellipse=0:180:0.47 and 0.25];
\draw [white, line width=4pt] (0,0) [partial ellipse=290:270:0.47 and 0.25];
\draw [red] (0,0) [partial ellipse=280:360:0.47 and 0.25];
\draw [red] (0,0) [partial ellipse=180:260:0.47 and 0.25];
\end{scope}
\end{tikzpicture}}}},
\end{align*}
we obtain the rest statements.
\end{proof}

Note that when $n=1$, we obtain $\displaystyle  \sum_{j,k=1}^r z_{jk} \frac{d_jd_k}{\mu}=1$. 
When $n=2$, we obtain the new identity $\displaystyle \sum_{j,k=1}^r z_{jk}^{\mathcal{D}_1}z_{kj}^{\mathcal{D}_2}=|\Irr(\ _{\mathcal{D}_1}\mathcal{M}_{\mathcal{D}_2})|$.
When $\mathcal{D}_1=\mathcal{D}_2$, it reduces to Corollary 6.10 in \cite{BEK00}.
When $\mathcal{D}_2=\mathcal{C}$, it reduces to Corollary 6.13 in \cite{BEK00}.

\begin{remark}
We checked the identities in Theorem \ref{thm:physical} for several modular invariant $\mathcal{P}_3$ in \cite[Section 4]{FSS95} not from the Morita context and have not found a counterexample yet.
\end{remark}

\begin{theorem}\label{thm:series}
 Suppose $\mathcal{C}$ is modular fusion category and $\lambda$ is a complex variable.
 We have that 
 \begin{enumerate}
     \item $\displaystyle \sum_{n \geq 0} \dim\hom_{\mathcal{Z(C)}}(\1_{\mathcal{Z(C)}}, I(\1_{\mathcal{D}})^{\otimes n}) \lambda^n =\sum_{j,k=1}^r\frac{1}{\mu^2}\frac{d_j^3d_k^3}{d_jd_k-\mu z_{jk}\lambda}.$
     \item $\displaystyle \sum_{n \geq 0} \dim\hom_{\mathcal{Z(C)}}(I(\1_{\mathcal{C}}), I(\1_{\mathcal{D}})^{\otimes n}) \lambda^n =\sum_{j=1}^r\frac{1}{\mu^2}\frac{d_j^6}{d_j^2-\mu z_{jj}\lambda}.$
     \item 
    $\displaystyle\sum_{n \geq 0} \dim\hom_{\mathcal{Z(C)}}(\sum_j G^+(X_j), I(\1_{\mathcal{D}})^{\otimes n})\lambda^n=\sum_{j=1}^r\frac{1}{\mu^2}\frac{d_j^3}{d_j-\mu z_{j1}\lambda}$.
    \item 
     $\displaystyle\sum_{n \geq 0} \dim\hom_{\mathcal{Z(C)}}(\sum_j G^-(X_j), I(\1_{\mathcal{D}})^{\otimes n}))\lambda^n=\sum_{j=1}^r\frac{1}{\mu^2}\frac{d_j^3}{d_j-\mu z_{1j}\lambda}$.
 \end{enumerate}
\end{theorem}
\begin{proof}
It follows from Theorem \ref{thm:physical} directly.    
\end{proof}

\section{Double \texorpdfstring{$\alpha$}{}-Induction}\label{sec:alpha}

In this section, we introduce the double $\alpha$-induction, which generalize the usual $\alpha$-induction.
We also obtain the equivariance of the double $\alpha$-induction thanks to the topological definition.
Consequently, we quickly read the properties of the $Z$-matrix induced by the double $\alpha$-induction.

Suppose $\Sigma$ is an oriented surface of genus $g$, and let $\RT_{\mathcal{Z(C)}}(\Sigma\times \{0, 1\})$ be the vector space associated to the 3-alterfolds with time boundaries $\Sigma\times \{0, 1\}$ in the alterfold TQFT (see \cite{LMWW23b}). 
In particular, as is shown in \cite{LMWW23b}, by topological moves, each vector in $\RT_{\mathcal{Z(C)}}(\Sigma\times \{0, 1\})$ can be obtained by putting inside $\Sigma\times [0,1]$ 3-manifolds with space boundary (whose interiors are colored by $A$, 
and whose space boundaries are decorated by $\mathcal{C}$-tensor diagrams), 
and coloring the rest of the interior $\Sigma\times (0,1)$ of $\Sigma\times [0,1]$ by $B$. 
Technically speaking, by the universal construction in \cite{LMWW23b}, elements in $\RT_{\mathcal{Z(C)}}(\Sigma\times \{0, 1\})$ are equivalence classes of the 3-alterfolds with time boundaries described above modulo the kernel of the partition function of closed 3-alterfolds, 
but for simplicity, in our alterfold graphical calculus, we will simply depict an element in $\RT_{\mathcal{Z(C)}}(\Sigma\times \{0, 1\})$ by choosing a representative of it.

For example, a vector in $\RT_{\mathcal{Z(C)}}(\Sigma\times \{0, 1\})$ can have the following local picture, where the cross in the middle represents a part of $A$-colored 3-manifold with space boundary, which possibly is decorated by $\mathcal{C}$-tensor diagrams:
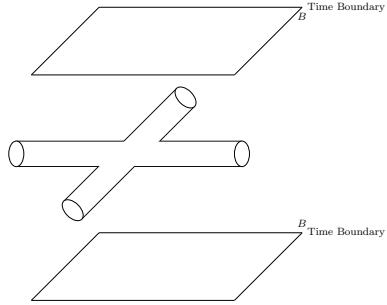
\[\vcenter{\hbox{\scalebox{0.5}{
\begin{tikzpicture}
\begin{scope}[shift={(1, 3)}, scale=1.8]
\draw (-1,0.5)--(2, 0.5) (-2,-0.5)--(1, -0.5);
\draw (-2,-0.5)--(-1, 0.5) (1,-0.5)--(2, 0.5) node [right] {\tiny Time Boundary}  node [below] {\tiny $B$};
\end{scope}
\draw [line width=0.7cm]  (-3,0)--(3, 0);
\draw [line width=0.7cm] (-1.5, -1.5)--(1.5, 1.5);
\draw [white, line width=0.65cm] (-1.75, -1.75)--(1.75, 1.75);
\draw [white, line width=0.65cm]  (-3,0)--(3, 0);
\begin{scope}[shift={(3, 0)}]
\draw [fill=white] (0, 0) [partial ellipse=0:360:0.2 and 0.34];
\end{scope}
\begin{scope}[shift={(-3, 0)}]
\draw [fill=white] (0, 0) [partial ellipse=0:360:0.2 and 0.34];
\end{scope}
\begin{scope}[shift={(-1.5, -1.5)},rotate=45 ]
\draw [fill=white] (0, 0) [partial ellipse=0:360:0.2 and 0.34];
\end{scope}
\begin{scope}[shift={(1.5, 1.5)},rotate=45 ]
\draw [fill=white] (0, 0) [partial ellipse=0:360:0.2 and 0.34];
\end{scope}
\begin{scope}[shift={(1, 1)},rotate=45 ]
\end{scope}
\begin{scope}[shift={(1, -3)}, scale=1.8]
\draw (-1,0.5)--(2, 0.5) (-2,-0.5)--(1, -0.5);
\draw (-2,-0.5)--(-1, 0.5) (1,-0.5)--(2, 0.5) node [right] {\tiny Time Boundary} node [above] {\tiny $B$} ;
\end{scope}
\end{tikzpicture}}}}\]
Now for any $v \in \RT_{\mathcal{Z(C)}}(\Sigma\times \{0, 1\})$, we color its time boundary $\Sigma\times \{0\}$ by $\mathcal{D}$ and $\Sigma\times \{1\}$ by $\mathcal{E}$, both of which contains empty tensor diagrams. Then embed $v$ in $\mathbb{S}^3$ and color the exterior of $v$ by $A$, we get a closed 3-alterfold which is decorated by $\mathcal{C}$, $\mathcal{D}$ and $\mathcal{E}$. We denote the resulting closed 3-alterfold by $\mathfrak{M}_{\cE}^{\cD}(v)$, whose partition function (defined in \cite{LMWW23b}) is denoted by $\langle \mathfrak{M}^{\cD}_{\cE}(v) \rangle$ (note that by \cite{LMWW23}, this value does not depend on the embedding, as the topology of the $A$-colored part does not effect the value of the partition function).

\begin{definition}[Double $\alpha$-Induction functional]\label{def:alpha1}
Using the notations above, for any closed oriented surface $\Sigma$ of genus $g$, we define the double $\alpha$-induction functional to be a linear functional 
\[\alpha_{g, \cD, \cE}: \RT_{\mathcal{Z(C)}}(\Sigma\times \{0,1\}) \to \mathbb{C}\,, v \mapsto\frac{1}{\mu}\langle \fM^{\cD}_{\cE}(v) \rangle\,.\]
\end{definition}

Using local pictures, the double $\alpha$-induction functional can be depicted as follows
\begin{align*}
\alpha_{g, \mathcal{D}, \mathcal{E}}\left(\vcenter{\hbox{\scalebox{0.5}{
\begin{tikzpicture}
\begin{scope}[shift={(1, 3)}, scale=1.8]
\draw (-1,0.5)--(2, 0.5) (-2,-0.5)--(1, -0.5);
\draw (-2,-0.5)--(-1, 0.5) (1,-0.5)--(2, 0.5) node [right] {\tiny Time Boundary}  node [below] {\tiny $B$};
\end{scope}
\draw [line width=0.7cm]  (-3,0)--(3, 0);
\draw [line width=0.7cm] (-1.5, -1.5)--(1.5, 1.5);
\draw [white, line width=0.65cm] (-1.75, -1.75)--(1.75, 1.75);
\draw [white, line width=0.65cm]  (-3,0)--(3, 0);
\begin{scope}[shift={(3, 0)}]
\draw [fill=white] (0, 0) [partial ellipse=0:360:0.2 and 0.34];
\end{scope}
\begin{scope}[shift={(-3, 0)}]
\draw [fill=white] (0, 0) [partial ellipse=0:360:0.2 and 0.34];
\end{scope}
\begin{scope}[shift={(-1.5, -1.5)},rotate=45 ]
\draw [fill=white] (0, 0) [partial ellipse=0:360:0.2 and 0.34];
\end{scope}
\begin{scope}[shift={(1.5, 1.5)},rotate=45 ]
\draw [fill=white] (0, 0) [partial ellipse=0:360:0.2 and 0.34];
\end{scope}
\begin{scope}[shift={(1, 1)},rotate=45 ]
\end{scope}
\begin{scope}[shift={(1, -3)}, scale=1.8]
\draw (-1,0.5)--(2, 0.5) (-2,-0.5)--(1, -0.5);
\draw (-2,-0.5)--(-1, 0.5) (1,-0.5)--(2, 0.5) node [right] {\tiny Time Boundary} node [above] {\tiny $B$} ;
\end{scope}
\end{tikzpicture}}}}\right)
=\frac{1}{\mu}
\vcenter{\hbox{\scalebox{0.5}{
\begin{tikzpicture}
\begin{scope}[shift={(1, 3)}, scale=1.8]
\path [fill=brown!20!white] (-1,0.5)--(2, 0.5)--(1, -0.5)--(-2, -0.5)--cycle;
\draw (-1,0.5)--(2, 0.5) (-2,-0.5)--(1, -0.5);
\draw (-2,-0.5)--(-1, 0.5) (1,-0.5)--(2, 0.5)   node [below] {\tiny $B$} node [above] {\tiny $A$};
\node at (0,0) { $\mathcal{D}$};
\end{scope}
\draw [line width=0.7cm]  (-3,0)--(3, 0);
\draw [line width=0.7cm] (-1.5, -1.5)--(1.5, 1.5);
\draw [white, line width=0.65cm] (-1.75, -1.75)--(1.75, 1.75);
\draw [white, line width=0.65cm]  (-3,0)--(3, 0);
\begin{scope}[shift={(3, 0)}]
\draw [fill=white] (0, 0) [partial ellipse=0:360:0.2 and 0.34];
\end{scope}
\begin{scope}[shift={(-3, 0)}]
\draw [fill=white] (0, 0) [partial ellipse=0:360:0.2 and 0.34];
\end{scope}
\begin{scope}[shift={(-1.5, -1.5)},rotate=45 ]
\draw [fill=white] (0, 0) [partial ellipse=0:360:0.2 and 0.34];
\end{scope}
\begin{scope}[shift={(1.5, 1.5)},rotate=45 ]
\draw [fill=white] (0, 0) [partial ellipse=0:360:0.2 and 0.34];
\end{scope}
\begin{scope}[shift={(1, 1)},rotate=45 ]
\end{scope}
\begin{scope}[shift={(1, -3)}, scale=1.8]
\path [fill=brown!20!white] (-1,0.5)--(2, 0.5)--(1, -0.5)--(-2, -0.5)--cycle;
\draw (-1,0.5)--(2, 0.5) (-2,-0.5)--(1, -0.5);
\draw (-2,-0.5)--(-1, 0.5) (1,-0.5)--(2, 0.5) node [above] {\tiny $B$} node [below] {\tiny $A$};
\node at (0,0) { $\mathcal{E}$};
\end{scope}
\end{tikzpicture}}}}\,.
\end{align*}

\begin{remark}
Our definition of the double $\alpha$-induction here is well-defined for any spherical fusion category $\mathcal{C}$.
\end{remark}

We denote by $\MCG(\Sigma)$ the mapping class group of a closed oriented surface $\Sigma$, then $\MCG(\Sigma \times \{0,1\}) = \MCG(\Sigma) \times \MCG(\Sigma)$. For any $f \in \MCG(\Sigma \times \{0,1\})$, denote its mapping cylinder by $C_f$, then the action of $\MCG(\Sigma\times \{0,1\})$ on $V = \RT_{\mathcal{Z(C)}}(\Sigma\times \{0,1\})$ is given by the following group homomorphism:
\begin{equation}\label{eq:mcg-1}
\rho: \MCG(\Sigma\times\{0,1\}) \to \GL(V)\,, \quad \rho(f)(v) := v \cup_{\Sigma\times \{0,1\}} C_f(B)
\end{equation}
for all $f \in \MCG(\Sigma\times \{0,1\})$ and $v \in \RT_{\mathcal{Z(C)}}(\Sigma\times \{0, 1\})$, where by an abuse of notation, the right hand stands for the class in $V$ obtained from taking any 3-alterfold representative of $v$ and gluing it with a $B$-colored $C_f$ along the desired boundary. Here we remark that by construction, when we glue 3-alterfolds along their time boundaries, the two sides of the time boundaries are required to have the same color. Moreover, the gluing boundary will not be present in the resulting 3-alterfold with time boundary.

\begin{theorem}[Invariance]\label{thm:mcg0}
Suppose $\Sigma$ is an oriented surface with genus $g$. Then the double $\alpha$-induction is invariant under the mapping class group actions defined above. Precisely, for any $f \in \MCG(\Sigma\times \{0,1\})$, we have $\alpha_{g, \mathcal{D}, \mathcal{E}} \circ \rho(f)=\alpha_{g, \mathcal{D}, \mathcal{E}}$ as linear functionals on $V = \RT_{\mathcal{Z(C)}}(\Sigma\times \{0, 1\})$.
\end{theorem}
\begin{proof}
By definition, for any $f \in \MCG(\Sigma\times\{0,1\})$ and $v \in V$, since there is no tensor diagram on the time boundary of $v$, we have $\fM_{\cE}^{\cD}(\rho(f)(v))$ is homeomorphic to $\fM_{\cE}^{\cD}(v)$, and the theorem follows from the homeomorphism invariance of the partition function in the alterfold TQFT.
\end{proof}

When $\mathcal{C}$ is a modular fusion category, we have that $\mathcal{Z(C)}$ is braided isomorphic to $\mathcal{C}\boxtimes \mathcal{C}^{\op}$.
In this case $\RT_{\mathcal{Z(C)}}(\Sigma\times \{0,1\})= \RT_{\mathcal{C}}(\Sigma\times \{0,1\})\otimes \RT_{\mathcal{C}^{\op}}(\Sigma\times \{0,1\})$, and $\dim \RT_{\mathcal{C}}(\Sigma\times \{0,1\}) = \dim  \RT_{\mathcal{C}^{\op}}(\Sigma\times \{0,1\})$.
Hence we have the following specializations of the double $\alpha$-induction for modular fusion categories.

\begin{definition}[Double $\alpha$-Induction for Modular Fusion Category]
Suppose $\mathcal{C}$ is a modular fusion category.
We define the positive topological $\alpha$-induction $\alpha_{g, \mathcal{D}, \mathcal{E}}^+$ to be the linear map from $\RT_{\mathcal{C}}(\Sigma\times \{0,1\})$ to the dual space of $\RT_{\mathcal{C}^{\op}}(\Sigma\times \{0,1\})$ as follows:
\begin{align*}
\alpha^+_{g, \cD, \cE}: & \RT_{\mathcal{C}}(\Sigma\times \{0,1\}) \to \RT_{\mathcal{C}^{\op}}(\Sigma\times \{0,1\})^*\,,\\
\alpha^+_{g, \mathcal{D}, \mathcal{E}}(\xi) (\eta) =& \alpha_{g, \mathcal{D}, \mathcal{E}} (\xi\otimes \eta), 
\end{align*}
where $\xi\in \RT_{\mathcal{C}}(\Sigma\times \{0,1\})$ and $\eta \in \RT_{\mathcal{C}^{\op}}(\Sigma\times \{0,1\})$.
We call $(\alpha^+_{g,\cD, \cE})^*$ the negative topological $\alpha$-induction, and denote it by $\alpha^-_{g, \mathcal{D}, \mathcal{E}}$. Namely, we have 
\begin{align*}
\alpha^-_{g, \cD, \cE} = & (\alpha^+_{g,\cD,\cE})^*: \RT_{\mathcal{C}^{\op}}(\Sigma\times \{0,1\}) \to \RT_{\mathcal{C}}(\Sigma\times \{0,1\})^*\,,\\
\alpha^-_{g, \mathcal{D}, \mathcal{E}}(\eta) (\xi) = & \alpha_{g, \mathcal{D}, \mathcal{E}} (\xi\otimes \eta),
\end{align*}
where $\xi\in \RT_{\mathcal{C}}(\Sigma\times \{0,1\})$ and $\eta \in \RT_{\mathcal{C}^{\op}}(\Sigma\times \{0,1\})$.
\end{definition}

Note that the basis of $\RT_{\cC}(\Sigma\times\{0,1\})$ demonstrated in \cite{LMWW23b} to also be viewed as a basis for $\RT_{\cC^{\op}}(\Sigma\times\{0,1\})$ as our choice did not involve the data of the braidings of category under study. We will then denote a choice of basis (i.e., 3-alterfolds with time boundary) of $\RT_\cC(\Sigma \times \{0,1\})$ by $\{\xi_j\}$, and the same set of 3-alterfolds with time boundary will be our choice of a basis for $\RT_{\cC^{\op}}(\Sigma \times \{0,1\})$, denoted by $\{\xi_k^{\op}\}_k$.

\begin{definition}[$Z$-Transformation]\label{def:Zg}
Suppose $\mathcal{C}$ is a modular fusion category, $\{\xi_j\}_j$ be the basis for $ \RT_{\mathcal{C}}(\Sigma\times \{0,1\})$ and $\{\xi_j^{\op}\}_j$ the basis for $\RT_{\mathcal{C}^{\op}}(\Sigma\times \{0,1\})$ chosen above. Then the genus $g$ $Z$-matrix $Z_g^{\mathcal{D}, \mathcal{E}}$ is defined to be 
\begin{align*}
Z_g^{\mathcal{D}, \mathcal{E}}= \left(\alpha_{g, \mathcal{D}, \mathcal{E}}(\xi_j\otimes \xi_k^{\op})\right)_{j,k}.
\end{align*}
\end{definition}

\begin{remark}
The entries of the higher-genus $Z$-matrix are not necessarily integers.
The requirement that $\mathcal{C}$ is a modular fusion category is not necessary.
We could take a fusion subcategory $\mathcal{C}_1$ of $\mathcal{Z(C)}$ and the M\"uger centralizer $\mathcal{C}_2$ of $\mathcal{C}_1$ in $\mathcal{Z(C)}$ to present the definition of the $Z$-matrix.
\end{remark}

For any $f \in \MCG(\Sigma\times \{0,1\})$, recall from \eqref{eq:mcg-1} that $\rho(f)$ acts linearly on $\RT_{\mathcal{Z(C)}}(\Sigma\times \{0,1\})$. Let $\{\xi_{j}\}_j$ and $\{\xi_k^{\op}\}_k$ be the pair of bases described in Definition \ref{def:Zg}, then we can pick matrix representatives $f_\cC$ and $f_{\cC^{op}}$ of the (projective) actions of $f$ on $\RT_{\cC}(\Sigma\times\{0,1\})$ and $\RT_{\cC^{\op}}(\Sigma\times \{0,1\})$ respectively that satisfy the following two conditions:
\begin{itemize}
\item $f_{\cC} \otimes f_{\cC^{\op}} = \rho(f)$ for all $f \in \MCG(\Sigma\times\{0,1\})$;
\item if $f$ is one of the generators of $\MCG(\Sigma\times\{0,1\})$ described in Section 3.3 of \cite{Ker00}, then $f_{\cC^{\op}} = f_{\cC}^{-1}$.
\end{itemize}
(This is possible by the tangle presentation in loc.~cit.) Note that due to the well-known gluing anomaly, what may not be always possible is to extend such choices to linear actions of $f$ on $\RT_{\cC}(\Sigma\times\{0,1\})$ and $\RT_{\cC^{\op}}(\Sigma\times \{0,1\})$. Note that when $\Sigma$ is the torus, the S- and T-matrices are choices of matrix representatives in 
\[\operatorname{SL}_2(\mathbb{Z}) \cong \MCG(\Sigma) = \MCG(\Sigma)\times \{0\} \subset \MCG(\Sigma \times \{0,1\})\] 
satisfying the above conditions, which is a preferred choice due to its representation theoretical or physical meaning, which is not obvious in higher genus.

\begin{theorem}\label{thm:mcg1}
Let $\Sigma$ be a closed oriented surface of genus $g$. Using the above notations, we have equality of matrices
\begin{align*}
f_{\cC} \cdot Z_g^{\mathcal{D}, \mathcal{E}}= Z_g^{\mathcal{D}, \mathcal{E}} \cdot f_{\cC}\,.
\end{align*}
for any $f \in \MCG(\Sigma\times\{0,1\})$.
\end{theorem}
\begin{proof}
For simplicity, write $\alpha = \alpha_{g, \cD, \cE}$ and $Z = Z_g^{\mathcal{C}, \mathcal{E}}$. 
We first assume that $f$ is a generator of $\MCG(\Sigma\times \{0,1\})$ described in \cite[Section 3.3]{Ker00}. By our choice and Theorem \ref{thm:mcg0}, for any pair of basis elements $\xi_j$ and $\xi_k^{\op}$, we have
\[\alpha((f_{\cC}\xi_j) \otimes (f_{\cC^{\op}}\xi_k^{\op})) = \alpha(\rho(f)(\xi_j \otimes \xi_k^{\op})) = \alpha(\xi_j \otimes \xi_{k}^{\op})\,.\]
Therefore, by our choice of matrix representatives, we have
\[\alpha((f_{\cC}\xi_j) \otimes \xi_k^{\op}) = \alpha(\xi_j \otimes (f_{\cC}\xi_k^{\op}))\,,\]
and written in matrix form, we have $f_{\cC} \cdot Z = Z \cdot f_{\cC}$. 

Now suppose $f \in \MCG(\Sigma\times \{0,1\})$ is a finite product, i.e., $\displaystyle f = \prod_t u_t$, where $u_t$ are generators of the mapping class group. 
Then the projectivity of the mapping class group actions on $\RT_\cC(\Sigma\times \{0,1\})$ implies that the matrix representatives we pick are multiplicative up to a nonzero scalar, i.e., we have a equality of matrices $\displaystyle f_{\cC} = \lambda \cdot \prod_{t} (u_t)_{\cC}$ for some $\lambda \in \mathbb{C}^\times$. 
Then by the above argument, we have
\[f_{\cC} \cdot Z = \lambda (\prod_{t} (u_t)_{\cC}) \cdot Z = Z \cdot \lambda (\prod_t (u_t)_{\cC}) = Z \cdot f_{\cC}\]
as desired.
\end{proof}

\begin{theorem}\label{thm:tracevalue}
Suppose $\Sigma$ is an oriented surface with genus $g$.
Then we have that 
\begin{align*}
Tr(Z_g^{\mathcal{D}, \mathcal{E}})=
\frac{1}{\mu}
\vcenter{\hbox{\scalebox{0.5}{
\begin{tikzpicture}
\begin{scope}[shift={(1, 1)}, scale=1.8]
\path [fill=brown!20!white] (-1,0.5)--(2, 0.5)--(1, -0.5)--(-2, -0.5)--cycle;
\draw (-1,0.5)--(2, 0.5) (-2,-0.5)--(1, -0.5);
\draw (-2,-0.5)--(-1, 0.5) (1,-0.5)--(2, 0.5)   node [below] {\tiny $B$} node [above] {\tiny $A$};
\node at (0,0) { $\mathcal{D}$};
\end{scope}
\begin{scope}[shift={(1, -1)}, scale=1.8]
\draw (-1,0.5)--(2, 0.5) (-2,-0.5)--(1, -0.5);
\path [fill=brown!20!white] (-1,0.5)--(2, 0.5)--(1, -0.5)--(-2, -0.5)--cycle;
\draw (-2,-0.5)--(-1, 0.5) (1,-0.5)--(2, 0.5) node [above] {\tiny $B$} node [below] {\tiny $A$};
\node at (0,0) { $\mathcal{E}$};
\end{scope}
\end{tikzpicture}}}}.
\end{align*}
\end{theorem}
\begin{proof}
It follows from that the fact the identity in $\RT_{\mathcal{Z(C)}}(\Sigma)$ has decomposition $\sum_{j} \xi_j\otimes \xi_j^{\op}$, where $\{\xi_j\}_j$ is a basis in $\RT_{\mathcal{C}}(\Sigma)$.
\end{proof}

Suppose $\mathcal{E}=\mathcal{C}$.
Then the $\mathcal{Z(C)}$-diagram in the interior $\Sigma\times [0,1]$ can be contracted to the surface $\Sigma\times \{1\}$ colored by $\mathcal{C}$.
In this case, the double $\alpha$-induction is simply defined as follows:
\begin{align*}
\alpha_{g, \mathcal{D}}\left(\vcenter{\hbox{\scalebox{0.5}{
\begin{tikzpicture}
\begin{scope}[shift={(1, 1)}, scale=1.8]
\draw (-1,0.5)--(2, 0.5) (-2,-0.5)--(1, -0.5);
\draw (-2,-0.5)--(-1, 0.5) (1,-0.5)--(2, 0.5) node [right] {\tiny Time Boundary}  node [below] {\tiny $B$};
\end{scope}
\begin{scope}[shift={(1, -1)}, scale=1.8]
\draw (-1,0.5)--(2, 0.5) (-2,-0.5)--(1, -0.5);
\draw (-2,-0.5)--(-1, 0.5) (1,-0.5)--(2, 0.5) node [right] {\tiny Space Boundary} node [above] {\tiny $B$} node [below] {\tiny $A$} ;
\node at (0, 0) { $\Gamma$};
\end{scope}
\end{tikzpicture}}}}\right)
=\frac{1}{\mu}
\vcenter{\hbox{\scalebox{0.5}{
\begin{tikzpicture}
\begin{scope}[shift={(1, 1)}, scale=1.8]
\path [fill=brown!20!white] (-1,0.5)--(2, 0.5)--(1, -0.5)--(-2, -0.5)--cycle;
\draw (-1,0.5)--(2, 0.5) (-2,-0.5)--(1, -0.5);
\draw (-2,-0.5)--(-1, 0.5) (1,-0.5)--(2, 0.5)   node [below] {\tiny $B$} node [above] {\tiny $A$};
\node at (0,0) { $\mathcal{D}$};
\end{scope}
\begin{scope}[shift={(1, -1)}, scale=1.8]
\draw (-1,0.5)--(2, 0.5) (-2,-0.5)--(1, -0.5);
\draw (-2,-0.5)--(-1, 0.5) (1,-0.5)--(2, 0.5) node [above] {\tiny $B$} node [below] {\tiny $A$};
\node at (0,0) { $\Gamma$};
\end{scope}
\end{tikzpicture}}}}.
\end{align*}
Suppose $\{\xi_j\}_j, \{\xi_j^{\op}\}_j$ are base for $ \RT_{\mathcal{C}}(\Sigma\times \{0\})$ and $\RT_{\mathcal{C}^{\op}}(\Sigma\times \{0\})$.
The $Z$-matrix $Z_g^{\mathcal{D}}$ is 
\begin{align*}
Z_g^{\mathcal{D}}= (\alpha_{g, \mathcal{D} }(\xi_j\otimes \xi_k^{\op}))_{j,k}.
\end{align*}

\begin{proposition}\label{prop:zdecomposition}
Suppose that $\Sigma$ is an oriented surface with genus $g$.
Then we have that 
\begin{align*}
    Z_{g}^{\mathcal{D}, \mathcal{E}}=\mu^{-1} Z_g^{\mathcal{D}} (Z_{g}^{ \mathcal{E}})^t.
\end{align*}
\end{proposition}
\begin{proof}
We have that 
\begin{align*}
   \sum_{k} \alpha_{g, \mathcal{D}}(\xi_j \otimes \xi_k^{\op})\alpha_{g, \mathcal{E}}(\xi_\ell\otimes \xi_k^{\op})
   =\mu \alpha_{g, \mathcal{D}, \mathcal{E}}(\xi_j\otimes \xi_{\ell}^{\op})
\end{align*}
where the parallel surfaces colored by $\mathcal{C}$ with the inside colored by $A$ can be eliminated by Moves 2,3.
\end{proof}

\begin{corollary}\label{cor:mcgmatrix}
Suppose $\Sigma$ is an oriented surface with genus $g$ and $\Phi\in \MCG(\Sigma)$.
Then we have that 
\begin{align*}
(\Phi_{jk})_{jk} Z_g^{\mathcal{D}}= Z_g^{\mathcal{D}} (\Phi_{jk})_{jk}. 
\end{align*}
\end{corollary}
\begin{proof}
    It follows from Theorem \ref{thm:mcg1}.
\end{proof}

Suppose the genus $g=1$.
Then the surface $\Sigma$ is a torus and we shall write
$\alpha_{1, \mathcal{D}, \mathcal{E}}$ as $\alpha_{\mathcal{D}, \mathcal{E}}$, $\alpha_{1, \mathcal{D}, \mathcal{E}}^\pm$ as $\alpha_{\mathcal{D}, \mathcal{E}}^\pm$, $Z_{1}^{\mathcal{D}, \mathcal{E}}$ as $Z^{\mathcal{D},\mathcal{E}}$ etc.
If $\mathcal{E}=\mathcal{C}$, we shall write $\alpha_{1, \mathcal{D}}$ as $\alpha_{\mathcal{D}}$, $\alpha_{1, \mathcal{D}}^\pm$ as $\alpha_{\mathcal{D}}^\pm$, $Z_1^{\mathcal{D}}$ as $Z^{\mathcal{D}}$ etc.
The basis of $\RT_{\mathcal{C}}(\Sigma\times \{0\})$ can be taken as 
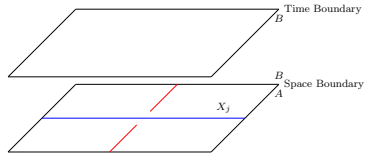
\begin{align*}
\vcenter{\hbox{\scalebox{0.5}{
\begin{tikzpicture}
\begin{scope}[shift={(1, 1)}, scale=1.8]
\draw (-1,0.5)--(2, 0.5) (-2,-0.5)--(1, -0.5);
\draw (-2,-0.5)--(-1, 0.5) (1,-0.5)--(2, 0.5) node [right] {\tiny Time Boundary}  node [below] {\tiny $B$};
\end{scope}
\begin{scope}[shift={(1, -1)}, scale=1.8]
\draw (-1,0.5)--(2, 0.5) (-2,-0.5)--(1, -0.5);
\draw (-2,-0.5)--(-1, 0.5) (1,-0.5)--(2, 0.5) node [right] {\tiny Space Boundary} node [above] {\tiny $B$} node [below] {\tiny $A$} ;
\draw [red] (-0.5, -0.5)--(0.5, 0.5);
\draw [white, line width=0.05cm] (-0.1, -0.1) --(0.1, 0.1);
\draw [blue]  (-1.5,0)--(1.5, 0);
\node [above right] at (1, 0) {\tiny $X_j$};
\end{scope}
\end{tikzpicture}}}}
\end{align*}
\begin{proposition}
    We have that the entry $\alpha_{\mathcal{D}}(X_j\otimes X_k^{\op})$ of $Z^{\mathcal{D}}$ is a positive integer for any $j,k$.
\end{proposition}
\begin{proof}
Note that 
\begin{align*}
\alpha_{\mathcal{D}} (X_j\otimes X_k^{\op})= &\frac{1}{\mu^2}
\vcenter{\hbox{\scalebox{0.6}{
\begin{tikzpicture}[scale=0.9]
\draw [line width=0.77cm, brown!20!white] (0,0) [partial ellipse=-0.1:180.1:2 and 1.5];
\begin{scope}[shift={(2.5, 0)}]
\draw [line width=0.8cm] (0,0) [partial ellipse=0:180:2 and 1.5];
\draw [white, line width=0.77cm] (0,0) [partial ellipse=-0.1:180.1:2 and 1.5];
\draw [blue] (0,0) [partial ellipse=0:180:2.15 and 1.65];
\draw [blue] (0,0) [partial ellipse=0:180:1.85 and 1.35];
\end{scope} 
\begin{scope}[shift={(2.5, 0)}]
\draw [line width=0.8cm] (0,0) [partial ellipse=180:360:2 and 1.5];
\draw [white, line width=0.77cm] (0,0) [partial ellipse=178:362:2 and 1.5];
\draw [blue, ->-=0.5] (0,0) [partial ellipse=178:362:2.15 and 1.65] node[black, pos=0.7,below ] {\tiny $X_k$}; 
\draw [blue, -<-=0.5] (0,0) [partial ellipse=178:362:1.85 and 1.35] node[black, pos=0.7,above ] {\tiny $X_j$};
\end{scope}
\draw [line width=0.77cm, brown!20!white] (0,0) [partial ellipse=180:360:2 and 1.5];
\begin{scope}[shift={(4.5, 0)}]
\draw [red, dashed](0,0) [partial ellipse=0:180:0.4 and 0.25] ;
\draw [white, line width=4pt] (0,0) [partial ellipse=290:270:0.4 and 0.25];
\draw [red] (0,0) [partial ellipse=260:360:0.4 and 0.25];
\draw [red] (0,0) [partial ellipse=180:230:0.4 and 0.25];
\end{scope}
\end{tikzpicture}}}} 
=
\frac{1}{\mu^3}\vcenter{\hbox{\scalebox{0.7}{
\begin{tikzpicture}[scale=0.7]
\draw [line width=0.6cm, brown!20!white] (0,0) [partial ellipse=-0.1:180.1:2 and 1.5];
\begin{scope}[shift={(2.5, 0)}]
\draw [double distance=0.57cm] (0,0) [partial ellipse=-0.1:180.1:2 and 1.5];
\draw [blue] (0,0) [partial ellipse=0:180:2 and 1.5];
\end{scope} 
\begin{scope}[shift={(2.5, 0)}]
\draw [line width=0.6cm] (0,0) [partial ellipse=180:360:2 and 1.5];
\draw [white, line width=0.57cm] (0,0) [partial ellipse=178:362:2 and 1.5];
\draw [blue, ->-=0.5] (0,0) [partial ellipse=178:362:2 and 1.5] node[black, pos=0.7,below ] {\tiny $X_k$};
\end{scope}
\begin{scope}[shift={(-2.5, 0)}]
\draw [double distance=0.57cm] (0,0) [partial ellipse=-0.1:180.1:2 and 1.5];
\draw [blue, ->-=0.5] (0,0) [partial ellipse=0:180:2 and 1.5] node[black, pos=0.6, above] {\tiny $X_j$};
\end{scope} 
\begin{scope}[shift={(-2.5, 0)}]
\draw [line width=0.6cm] (0,0) [partial ellipse=180:360:2 and 1.5];
\draw [white, line width=0.57cm] (0,0) [partial ellipse=178:362:2 and 1.5];
\draw [blue] (0,0) [partial ellipse=178:362:2 and 1.5];
\end{scope}
\draw [line width=0.6cm, brown!20!white] (0,0) [partial ellipse=180:360:2 and 1.5];
\begin{scope}[shift={(4.5, 0)}]
\draw [red, dashed](0,0) [partial ellipse=0:180:0.37 and 0.25] ;
\draw [white, line width=4pt] (0,0) [partial ellipse=270:250:0.37 and 0.25];
\draw [red] (0,0) [partial ellipse=180:360:0.37 and 0.25];
\end{scope}
\begin{scope}[shift={(-4.5, 0)}]
\draw [red, dashed] (0,0) [partial ellipse=0:180:0.37 and 0.25];
\draw [red] (0,0) [partial ellipse=180:250:0.37 and 0.25];
\draw [red] (0,0) [partial ellipse=290:360:0.37 and 0.25];
\end{scope}
\end{tikzpicture}}}}.
\end{align*}
We see the entries of $Z^{\mathcal{D}}$ are non-negative integers.
\end{proof}

\begin{corollary}\label{cor:mcg2}
We have that $Z^{\mathcal{D},\mathcal{E}}$ commutes with the $S$-matrix and the $T$-matrix, i.e. $Z^{\mathcal{D},\mathcal{E}}$ is a modular invariant matrix.   
\end{corollary}
\begin{proof}
    Note that $S$-matrix and $T$-matrix come from the mapping class group of the torus.
    By Theorem \ref{thm:mcg1}, we see the corollary is true.
\end{proof}

\begin{corollary}\label{cor:mcg3}
We have that $Z^{\mathcal{D}}$ commutes with the $S$-matrix and the $T$-matrix.    
\end{corollary}
\begin{proof}
    It follows from Corollary \ref{cor:mcg2}.
\end{proof}

\begin{corollary}
Suppose $\mathcal{C}$ is a modular fusion category.
We have that 
\begin{align*}
\left|\Irr(\ _\mathcal{D}\mathcal{M}_{\mathcal{E}})\right|
=Tr(Z^{\mathcal{D}}(Z^{\mathcal{E}})^t)
= \mu Tr(Z^{\mathcal{D}, \mathcal{E}}).
\end{align*}
\end{corollary}
\begin{proof}
 It follows from Theorem \ref{thm:tracevalue}, Proposition \ref{prop:zdecomposition}  and Proposition \ref{prop:dim}.
\end{proof}

In the following, we shall discuss the $\alpha$-induction functors.
We define the strict monoidal category $\widetilde{\mathcal{D}}_0$ as follows:
\begin{itemize}
    \item The objects are given by the following configuration over $\mathbb{R}\times [0, \infty)$:
    \begin{enumerate}
        \item A set of marked points labeled by objects in $\mathcal{D}$ over $\mathbb{R}\times \{0\}$,
        \item and a set of disjoint closed disks embedded in the region $\mathbb{R}\times (0, \infty)$ with marked points labeled by objects in $\mathcal{C}$ on the boundary.
    \end{enumerate}
The object should be considered as a two dimensional alterfold (See \cite{LMWW23b}) with its boundary left open for defining the monoidal structure. The $A$-colored region are exactly
$$\{(x, y)|y\in (-\infty, 0)\}\cup \{\text{disk interior}\}$$ 
\item The morphism are the linear span of the $3$-dimensional alterfold connecting the $2$-dimensional alterfold, modolo the moves and graphical calculus in the fusion category $\mathcal{C}, \mathcal{D}$.
\item Defining the composition of morphisms to be the stacking along $z$-axis, and the monoidal structure to be the juxtaposition along $x$-axis.
\end{itemize}
Let $\widetilde{\mathcal{D}}$ be the idempotent completion of $\widetilde{\mathcal{D}}_0$.
Then the category $\widetilde{\mathcal{D}}$ is equivalent to the category $\mathcal{D}$ via the forgetful functor $F:\mathcal{Z(C)} \to \mathcal{D}$ (see Section 4 in \cite{LMWW23} for details).
We denote the isomorphic functor by $\widetilde{F}$.
The classical $\alpha$-induction $\alpha_{O}$ is defined as follows in the alterfold TQFT:
\begin{align*}
\alpha_{O}\left(\vcenter{\hbox{\scalebox{0.5}{
\begin{tikzpicture}
\begin{scope}[shift={(1, 3)}, scale=1.8]
\draw (-1,0.5)--(2, 0.5) (-2,-0.5)--(1, -0.5);
\draw (-2,-0.5)--(-1, 0.5) (1,-0.5)--(2, 0.5) node [right] {\tiny Time Boundary}  node [below] {\tiny $B$};
\end{scope}
\draw [line width=0.7cm]  (-3,0)--(3, 0);
\draw [line width=0.7cm] (-1.5, -1.5)--(1.5, 1.5);
\draw [white, line width=0.65cm] (-1.75, -1.75)--(1.75, 1.75);
\draw [white, line width=0.65cm]  (-3,0)--(3, 0);
\begin{scope}[shift={(3, 0)}]
\draw [fill=white] (0, 0) [partial ellipse=0:360:0.2 and 0.34];
\end{scope}
\begin{scope}[shift={(-3, 0)}]
\draw [fill=white] (0, 0) [partial ellipse=0:360:0.2 and 0.34];
\end{scope}
\begin{scope}[shift={(-1.5, -1.5)},rotate=45 ]
\draw [fill=white] (0, 0) [partial ellipse=0:360:0.2 and 0.34];
\end{scope}
\begin{scope}[shift={(1.5, 1.5)},rotate=45 ]
\draw [fill=white] (0, 0) [partial ellipse=0:360:0.2 and 0.34];
\end{scope}
\begin{scope}[shift={(1, 1)},rotate=45 ]
\end{scope}
\begin{scope}[shift={(1, -3)}, scale=1.8]
\draw (-1,0.5)--(2, 0.5) (-2,-0.5)--(1, -0.5);
\draw (-2,-0.5)--(-1, 0.5) (1,-0.5)--(2, 0.5) node [right] {\tiny Time Boundary} node [above] {\tiny $B$} ;
\end{scope}
\end{tikzpicture}}}}\right)
=\frac{1}{\mu^2}
\vcenter{\hbox{\scalebox{0.5}{
\begin{tikzpicture}
\begin{scope}[shift={(1, 3)}, scale=1.8]
\path [fill=brown!20!white] (-1,0.5)--(2, 0.5)--(1, -0.5)--(-2, -0.5)--cycle;
\draw (-1,0.5)--(2, 0.5) (-2,-0.5)--(1, -0.5);
\draw (-2,-0.5)--(-1, 0.5) (1,-0.5)--(2, 0.5)   node [below] {\tiny $B$} node [above] {\tiny $A$};
\node at (0,0) { $\mathcal{D}$};
\end{scope}
\draw [line width=0.7cm]  (-3,0)--(3, 0);
\draw [line width=0.7cm] (-1.5, -1.5)--(1.5, 1.5);
\draw [white, line width=0.65cm] (-1.75, -1.75)--(1.75, 1.75);
\draw [white, line width=0.65cm]  (-3,0)--(3, 0);
\begin{scope}[shift={(3, 0)}]
\draw [fill=white] (0, 0) [partial ellipse=0:360:0.2 and 0.34];
\end{scope}
\begin{scope}[shift={(-3, 0)}]
\draw [fill=white] (0, 0) [partial ellipse=0:360:0.2 and 0.34];
\end{scope}
\begin{scope}[shift={(-1.5, -1.5)},rotate=45 ]
\draw [fill=white] (0, 0) [partial ellipse=0:360:0.2 and 0.34];
\end{scope}
\begin{scope}[shift={(1.5, 1.5)},rotate=45 ]
\draw [fill=white] (0, 0) [partial ellipse=0:360:0.2 and 0.34];
\end{scope}
\begin{scope}[shift={(1, 1)},rotate=45 ]
\end{scope}
\begin{scope}[shift={(1, -3)}, scale=1.8]
\draw (-1,0.5)--(2, 0.5) (-2,-0.5)--(1, -0.5);
\draw (-2,-0.5)--(-1, 0.5) (1,-0.5)--(2, 0.5) node [above] {\tiny $B$} node [below right] {\tiny Time Boundary};
\end{scope}
\end{tikzpicture}}}},
\end{align*}
The classical $\alpha$-induction $\alpha_{O}$ induces a tensor functor $\widetilde{\alpha}_{O}: \mathcal{Z(C)} \to \widetilde{\mathcal{D}}$ by adding a $\mathcal{D}$-colored surface to the time boundary.
Denote by $I^\pm:\mathcal{C}\to \mathcal{Z(C)}$ the two embedding functors.
Then the classical $\alpha$-induction induces two functors $\widetilde{F}\widetilde{\alpha}_{O}I^\pm: \mathcal{C}\to \mathcal{D}$. Pictorially, we have Equations \eqref{def:alpha1} and \eqref{def:alpha2}.

The double $\alpha$-induction $\alpha_{g, \mathcal{D}, \mathcal{E}}$ shall induce two tensor functors in a similar way.
We define the category $\langle\mathcal{D}, \mathcal{E}\rangle$ be the following strict monoidal category as follows:
\begin{itemize}
    \item The objects are given by the following configuration over $\mathbb{R}\times [0, 1]$:
    \begin{enumerate}
        \item A set of marked points labeled by objects in $\mathcal{D}$ over $\mathbb{R}\times \{0\}$,
        \item a set of marked points labeled by objects in $\mathcal{E}$ over $\mathbb{R}\times \{1\}$,
        \item and a set of disjoint closed disks embedded in the region $\mathbb{R}\times \{0, 1\}$ with marked points labeled by objects in $\mathcal{C}$ on the boundary.
    \end{enumerate}
The object should be considered as a two dimensional alterfold (See \cite{LMWW23b}) with its left and right boundary left open for defining the monoidal structure. The $A$-colored region are exactly
$$\{(x, y)|y\in (-\infty, 0)\cup(1, \infty)\}\cup \{\text{disk interior}\}$$ 
\item The morphism are the linear span of the $3$-dimensional alterfold connecting the $2$-dimensional alterfold, modolo the Moves and Graphical calculus in category $\mathcal{C}, \mathcal{D}$ and $\mathcal{E}$.
\item Defining the composition of morphisms to be the stacking along $z$-axis, and the monoidal structure to be the juxtaposition along $x$-axis.
\end{itemize}
We first remark that the unit object in this monoidal category is not simple. Indeed, the endomorphism space of tensor unit equals to the matrix algebra $\text{Mat}(|\text{Irr}({}_{\mathcal{D}}\mathcal{M}_{\mathcal{E}})|)$.

\begin{lemma}\label{lem:representative}
All objects in $\langle\mathcal{D}, \mathcal{E}\rangle$ are equivalent to a subobject of an object that contains no disks in $\mathbb{R}\times (0, 1)$.
\end{lemma}

\begin{proof}
Locally, we have the following isomorphism that reduces the number of disks in the bulk.
$$\sum_{M\in \text{Irr}({}_{\mathcal{D}}\mathcal{M}_{\mathcal{C}})}{d_{M}^{\frac{1}{2}}}\quad
\vcenter{\hbox{\scalebox{0.7}{\begin{tikzpicture}[scale=0.35]
\draw (4, -5)--(-4,-6)--(-4, 1)--(4, 2)  (4, -5)--(4, -2.5);
\draw [dashed](4, -2.5)--(4, 2);
\draw (3,1.5)--(3,0.5)..controls +(0,-1) and +(0,1)..(-2,-4);
\draw (7,1.5)--(7,0.5)..controls +(0,-1) and +(0,1)..(2,-4);
\draw [blue] (5,0.5)..controls +(0,-1) and +(0,1)..(0,-4);
\draw[dashed] (0,-4) [partial ellipse=0:180:2 and 0.8];
\draw [dashed] (0,-4) [partial ellipse=180:360:2 and 0.8];
\draw[blue, ->-=0.5] (0, -4)--(0, -5.5)node[below, black]{\tiny{${X}$}};
\draw [red](5, 1.25) [partial ellipse=-75:0:2 and 0.8];
\draw [black](5, 1.25) [partial ellipse=0:180:2 and 0.8];
\draw [black](5, 1.25) [partial ellipse=180:360:2 and 0.8];
\draw [blue] (0, -6) [partial ellipse=40:10:2.5 and 4] node [below, black] {\tiny $M^{*}$};
\draw [blue,dashed, ->-=0.5] (0, -6) [partial ellipse=90:40:2.5 and 4];
\draw [blue] (0, -6) [partial ellipse=90:179:2.5 and 4] node [below, black] {\tiny $M$};
\end{tikzpicture}}}}$$
On the other hand, $M\otimes X\otimes M^{*}$ is an object in $\mathcal{D}$. Hence all disks in the bulk could be pushed to the $\mathcal{D}$-colored boundary using the above isomorphism.
Diagramatically, we have the following injective map:
$$\vcenter{\hbox{\begin{tikzpicture}
\draw (-1, 0) to (1, 0) node[right]{$\mathcal{D}$};
\draw (-1, 1) to (1, 1) node[right]{$\mathcal{E}$};
\draw[fill=gray!50!white] (0, 0.5)circle(0.3);
\draw[fill=blue] (0, 0.8) circle(0.03);
\draw (0, 0.6) node{\tiny$X$};
\end{tikzpicture}}}\hookrightarrow \vcenter{\hbox{\begin{tikzpicture}
\draw (-1, 0) to (1, 0) node[right]{$\mathcal{D}$};
\draw (-1, 1) to (1, 1) node[right]{$\mathcal{E}$};
\draw[fill=blue] (-0.5, 0) circle(0.03) node[above]{\tiny$M^*$};
\draw[fill=blue] (0, 0) circle(0.03)node[above]{\tiny$X$}
;
\draw[fill=blue] (0.5, 0) circle(0.03)node[above]{\tiny$M$};
\end{tikzpicture}}}$$
\end{proof}

\begin{proposition}\label{prop:doublefusion}
The idempotent completion of category $\langle\mathcal{D}, \mathcal{E}\rangle$ is a multi-fusion category.
\end{proposition}
\begin{proof}
The rigidity follows from the rigidity of $\mathcal{Z(C)}$, $\mathcal{D}$ and $\mathcal{E}$. We only need to show the endomorphism spaces are matrix algebras.
We prove the endomorphism algebra are all matrix algebras. To be specific, we explicitly describe the endomorphism space of the following object:
$$\langle D, E\rangle :=\vcenter{\hbox{\begin{tikzpicture}
\draw (-1, 0) to (1, 0) node[right]{$\mathcal{D}$};
\draw (-1, 1) to (1, 1) node[right]{$\mathcal{E}$};
\draw[fill=blue] (0, 1) circle(0.03) node[below]{\tiny$E$};
\draw[fill=blue] (0, 0) circle(0.03)node[above]{\tiny$D$}
;
\end{tikzpicture}}}.$$
First notice that the following vectors form a basis of the morphism space:
$$\vcenter
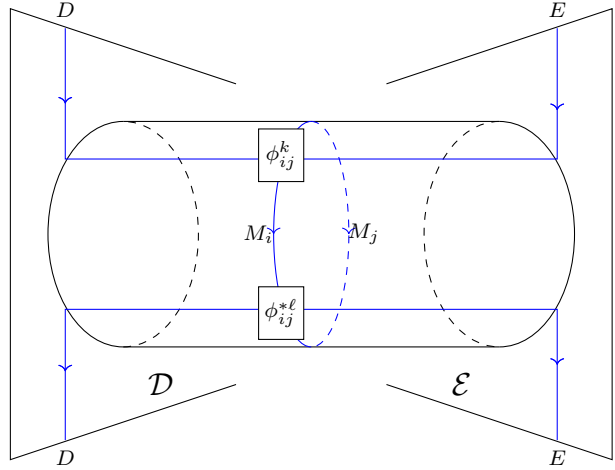
{\hbox{
\begin{tikzpicture}
\draw (-1, -2)--(-4, -3)--(-4, 3)--(-1, 2);
\draw (1, -2)--(4, -3)--(4, 3)--(1, 2);
\draw[dashed] (-2.5, 0) [partial ellipse=-90:90:1 and 1.5];
\draw (-2.5, 0) [partial ellipse=90:270:1 and 1.5];
\draw (2.5, 0) [partial ellipse=-90:90:1 and 1.5];
\draw[dashed] (2.5, 0) [partial ellipse=90:270:1 and 1.5];
\draw(-2.5, 1.5)--(2.5, 1.5);
\draw(-2.5, -1.5)--(2.5, -1.5);
\draw[blue, ->-=0.1, -<-=0.9] (-3.27, 2.74)node[above, black]{\tiny$D$}--(-3.27, 1)--(3.27, 1)--(3.27, 2.74)node[above, black]{\tiny$E$};
\draw[blue, -<-=0.1, ->-=0.9] (-3.27, -2.74)node[below, black]{\tiny$D$}--(-3.27, -1)--(3.27, -1)--(3.27, -2.74)node[below, black]{\tiny$E$};
\draw[blue, ->-=0.5] (0, 0)[partial ellipse=90:270:0.5 and 1.5];
\draw[blue, ->-=0.5, dashed] (0, 0)[partial ellipse=90:-90:0.5 and 1.5];
\draw (-.7, 0) node{\tiny$M_i$};
\draw (.7, 0) node{\tiny$M_j$};
\draw[fill=white] (-0.7, 0.7) rectangle (-.1, 1.4);
\draw (-0.4, 1.05) node{\tiny$\phi_{ij}^{k}$};
\draw[fill=white] (-0.7, -0.7) rectangle (-.1, -1.4);
\draw (-0.4, -1.05) node{\tiny$\phi_{ij}^{*\ell}$};
\draw (-2, -2) node{$\mathcal{D}$};
\draw (2, -2) node{$\mathcal{E}$};
\end{tikzpicture}
}}$$
where the vectors $\phi_{ij}^{k}$ and $\phi_{ij}^{*\ell}$ are chosen to be sets of dual basis in the corresponding morphisms spaces that are dual to each other due to the spherical structure. Fix $i, j$, this set of basis are elementary matrices in the corresponding block in the endomorphism algebra of $\langle D, E \rangle$. Therefore, we gave a matrix algebra description of endomorphism algebra of all objects of form $\langle D, E \rangle$. 
By Lemma \ref{lem:representative}, such objects dominate the category, therefore we finished the proof the Proposition.
\end{proof}

We shall characterize the minimal idempotents in the fusion algebra of $\mathcal{D}$ in Section 5.
Recall that in \cite{LMWW23}, we have a diagrammatical definition of the $\mathcal{Z(C)}$, the objects and the morphisms are exactly the configurations that could be put in the middle bulk $\mathbb{R}\times (0, 1)\times \mathbb{R}$. 
This produces a tensor functor $$\widehat{F}:\mathcal{Z(C)}\rightarrow \langle\mathcal{D}, \mathcal{E}\rangle.$$
The double $\alpha$-induction functor $\widetilde{\alpha}^\pm$ is defined to be $\widehat{F}I^\pm: \mathcal{C}\to \langle \mathcal{D}, \mathcal{E}\rangle.$
It is clear that $\widetilde{\alpha}^\pm$ are tensor functors.

\bibliographystyle{abbrv}
\bibliography{Trace}
\end{document}